\documentclass[12pt, a4paper, reqno]{amsart}
\usepackage{amsmath}
\usepackage{amsfonts}
\usepackage{amssymb}
\usepackage{color}
\usepackage{comment,graphicx}
\usepackage[T1]{fontenc}
\usepackage{url}
\usepackage{ae}
\usepackage{mathtools}
\usepackage{esint}
\usepackage{latexsym}
\usepackage{amscd}

\def\phi{\varphi}
\def\rho{\varrho}
\def\epsilon{\varepsilon}
\numberwithin{equation}{section}
\theoremstyle{plain}
\newtheorem{theorem}[equation]{Theorem}
\newtheorem{lemma}[equation]{Lemma}
\newtheorem{proposition}[equation]{Proposition}
\newtheorem{corollary}[equation]{Corollary}
\theoremstyle{definition}
\newtheorem{definition}[equation]{Definition}

\theoremstyle{remark}
\newtheorem{remark}[equation]{Remark}

\renewcommand{\leq}{\leqslant}
\renewcommand{\geq}{\geqslant}

\newcommand{\vertiii}[1]{{\left\vert\kern-0.25ex\left\vert\kern-0.25ex\left\vert \right\vert\kern-0.25ex\right\vert\kern-0.25ex\right\vert}}

\begin{document}
\title[Lorentz Herz-type Besov-Triebel-Lizorkin spaces]{Lorentz Herz-type
Besov-Triebel-Lizorkin spaces}
\author[D. Drihem]{Douadi Drihem}
\address{Douadi Drihem\\
Laboratory of Functional Analysis and Geometry of Spaces, Faculty of
Mathematics and Informatics, Department of Mathematics, M'sila University,
PO Box 166 Ichebelia, M'sila 28000, Algeria}
\email{douadidr@yahoo.fr, douadi.drihem@univ-msila.dz}
\thanks{ }
\date{\today }

\begin{abstract}
In this paper, we introduce a new family of function spaces of Besov-
Triebel-Lizorkin type. We present the $\varphi $-transform characterization
of these spaces in the sense of Frazier and Jawerth and we prove their
Sobolev and Franke-Jewarth embeddings. Also, we establish the smooth atomic,
molecular and wavelet decomposition of these function spaces.
Characterizations by ball means of differences are given. Finally, we
investigate a series of examples which play an important role in the study
of function spaces of Besov-Triebel-Lizorkin type.
\end{abstract}

\maketitle

\section{Introduction}

Function spaces have been a central topic in modern analysis, and are now of
increasing applications in many fields of mathematics especially harmonic
analysis and partial differential equations. The most known general scales
of function spaces are the scales of Besov spaces and Triebel-Lizorkin
spaces and it is known that they cover many well-known classical function
spaces such as H\"{o}lder-Zygmund spaces, Hardy spaces and Sobolev spaces.
For more details one can refer to Triebel's books \cite{T83}\ and\ \cite{T2}.

It is well known that Herz spaces $\dot{K}_{p}^{\alpha ,q}$, $\alpha \in 
\mathbb{R}$ and $0<p,q\leq \infty $, play an important role in harmonic
analysis. After they have been introduced in \cite{Herz68}, the theory of
these spaces had a remarkable development in part due to its usefulness in
applications. For instance, they appear in the characterization of
multipliers on Hardy spaces \cite{BS85}, in the summability of Fourier
transforms \cite{FeichtingerWeisz08} and in regularity theory for elliptic
equations in divergence form \cite{Rag09}. Y. Tsutsui \cite{T11} has
considered the Cauchy problem for Navier-Stokes equations on Herz spaces and
weak Herz spaces. They unify and generalize the classical Lebesgue spaces of
power weights. More precisely, if $\alpha =0$ and $p=q$, then $\dot{K}%
_{p}^{0,p}$ coincides with the Lebesgue spaces $L^{p}$ and 
\begin{equation*}
\dot{K}_{p}^{\alpha ,p}=L^{p}(\mathbb{R}^{n},|\cdot |^{\alpha p}),\quad 
\text{(Lebesgue space equipped with power weight).}
\end{equation*}

Based on Herz spaces, Besov and Triebel-Lizorkin spaces, the authors in \cite%
{XuYang05} and \cite{Xu05} introduced new family of function spaces called
Herz-type Besov spaces $\dot{K}_{p}^{\alpha ,q}B_{\beta }^{s}$ and
Triebel-Lizorkin spaces $\dot{K}_{p}^{\alpha ,q}F_{\beta }^{s}$. These
spaces unify and generalize the classical inhomogeneous Besov spaces and
Triebel-Lizorkin spaces. Several basic properties were established, such as
the Fourier analytical characterisation, lifting properties and embeddings
properties. When $\alpha =0$ and$\ p=q$ the spaces $\dot{K}%
_{p}^{0,p}B_{\beta }^{s}$ and $\dot{K}_{p}^{0,p}F_{\beta }^{s}$ coincide
with the usual function spaces $B_{p,\beta }^{s}$ and $F_{p,\beta }^{s}$,
respectively. The embeddings\ in $\dot{K}_{p}^{\alpha ,q}B_{\beta }^{s}$ and 
$\dot{K}_{p}^{\alpha ,q}F_{\beta }^{s}$ such us Sobolev, Franke and Jewarth,
extend and improve Sobolev, Franke and Jewarth embeddings of Besov and
Triebel-Lizorkin spaces; see \cite{Drihem1.13}, \cite{Drihem2.13} and \cite%
{drihem2016jawerth}.

The interest in Herz-type Besov-Triebel-Lizorkin spaces comes not only from
theoretical reasons but also from their applications to several classical
problems in analysis. In \cite{LuYang97}, Lu and Yang introduced the
Herz-type Sobolev and Bessel potential spaces. They gave some applications
to partial differential equations. In \cite{Dr22.Banach} the author studied
the Cauchy problem for the semilinear parabolic equations%
\begin{equation*}
\partial _{t}u-\Delta u=G(u)
\end{equation*}%
with initial data in Herz-type Triebel-Lizorkin spaces and under some
suitable conditions on $G$.

Based on Lorentz-Herz spaces, see below, and Herz-type
Besov-Triebel-Lizorkin spaces, this paper consists in giving a unified
treatment of function spaces of Besov-Triebel-Lizorkin type. We will define
and investigate the scales%
\begin{equation}
\dot{K}_{p,r}^{\alpha ,q}B_{\beta }^{s}\quad \text{and}\quad \dot{K}%
_{p,r}^{\alpha ,q}F_{\beta }^{s}.  \label{new}
\end{equation}%
Let us present briefly the idea to introduce the function spaces \eqref{new}%
. In \cite{Dr-inter-herz}\ the author studied the interpolation properties
of Herz-type Besov and Triebel-Lizorkin spaces. He proved the following
result. Let $\dot{K}_{{p,r}}^{{\alpha },q}$\ denote the homogeneous
Herz-type\ Lorentz space, see below. Let $0<\theta <1$, $1\leq p_{0}\neq
p_{1}\leq \infty ,1\leq q_{0},q_{1}<\infty ,1\leq \beta _{0},\beta
_{1},\beta \leq \infty $ and $\alpha _{0},\alpha _{1},s_{0},s_{1}\in \mathbb{%
R}$, with 
\begin{equation*}
-\frac{n}{p_{0}}<\alpha _{0}<n-\frac{n}{p_{0}}\quad \text{and}\quad -\frac{n%
}{p_{1}}<\alpha _{1}<n-\frac{n}{p_{1}}.
\end{equation*}%
Assume that 
\begin{equation*}
\alpha =(1-\theta )\alpha _{0}+\theta \alpha _{1},\text{\quad }\frac{1}{q}=%
\frac{1-\theta }{q_{0}}+\frac{\theta }{q_{1}}\text{\quad and\quad }\frac{1}{p%
}=\frac{1-\theta }{p_{0}}+\frac{\theta }{p_{1}}.
\end{equation*}%
(i) We have 
\begin{equation*}
(\dot{K}_{p_{0}}^{\alpha _{0},q_{0}}F_{\beta }^{s},\dot{K}_{p_{1}}^{\alpha
_{1},q_{1}}F_{\beta }^{s})_{\theta ,q}=\dot{K}_{p,q}^{\alpha ,q}F_{\beta
}^{s}
\end{equation*}%
hold in the sense of equivalent norms.$\newline
$(ii) Let $\frac{1}{q}=\frac{1-\theta }{\beta _{0}}+\frac{\theta }{\beta _{1}%
}$ and $s=(1-\theta )s_{0}+\theta s_{1}$. Then 
\begin{equation*}
(\dot{K}_{p_{0}}^{\alpha _{0},q_{0}}B_{\beta _{0}}^{s_{0}},\dot{K}%
_{p_{1}}^{\alpha _{1},q_{1}}B_{\beta _{1}}^{s_{1}})_{\theta ,q}=\dot{K}%
_{p,q}^{\alpha ,q}B_{q}^{s}
\end{equation*}%
hold in the sense of equivalent norms.

Therefore, it will be interesting to study such function spaces.

The paper contains\ six sections. Every section has an introduction which
explains what one will find there.

\textbf{Section 2}. We present some known results concerning Lorentz and
Herz spaces which are needed in the following parts of the paper. We provide
the necessary background information about Lorentz-Herz spaces. In this
section, we extend the vector-valued maximal inequality\ of Fefferman-Stein
and the classical Plancherel-Polya-Nikolskij inequality to the Lorentz-Herz
spaces.

\textbf{Section 3}. Using the Calder\'{o}n reproducing formulae, we
establish the $\varphi $-transform characterization of $\dot{K}%
_{p,r}^{\alpha ,q}B_{\beta }^{s}\ $and$\ \dot{K}_{p,r}^{\alpha ,q}F_{\beta
}^{s}$ spaces in the sense of Frazier and Jawerth. We continue by proving
Lifting property and the Fatou property of such spaces. Some interpolation
inequalities are established.

\textbf{Section 4}. We present some elementary embeddings. Sobolev and
Franke-Jewarth embeddings on such spaces are given. Also, we present new
embeddings between Besov and Herz spaces. All these results generalize and improve the existing classical results on Besov and Triebel-Lizorkin spaces.

\textbf{Section 5}. Firstly, we prove the boundedness of almost diagonal
operator in the sense of Frazier and Jawerth on sequence spaces $\dot{K}%
_{p,r}^{\alpha ,q}b_{\beta }^{s}$ and $\dot{K}_{p,r}^{\alpha ,q}f_{\beta
}^{s}$. Secondly, we establish characterizations by atomic and molecular
decompositions of $\dot{K}_{p,r}^{\alpha ,q}B_{\beta }^{s}\ $and$\ \dot{K}%
_{p,r}^{\alpha ,q}F_{\beta }^{s}$ spaces. Using the characterizations of $%
\dot{K}_{p,r}^{\alpha ,q}B_{\beta }^{s}\ $and$\ \dot{K}_{p,r}^{\alpha
,q}F_{\beta }^{s}$ spaces by atom, we establish characterizations of such
spaces\ by wavelets.

\textbf{Section 6}. In this section, we establish characterizations of $\dot{%
K}_{p,r}^{\alpha ,q}B_{\beta }^{s}\ $and$\ \dot{K}_{p,r}^{\alpha ,q}F_{\beta
}^{s}$\ by Peetre maximal function, by ball mean of differences and we will
present some useful examples, which play an important role in the study of
function spaces of Besov-Triebel-Lizorkin type.

\subsection{Basic spaces}

Throughout this paper, we denote by $\mathbb{R}^{n}$ the $n$-dimensional
real Euclidean space, $\mathbb{N}$ the collection of all natural numbers and 
$\mathbb{N}_{0}=\mathbb{N}\cup \{0\}$. The letter $\mathbb{Z}$ stands for
the set of all integer numbers.\ For a multi-index $\alpha =(\alpha
_{1},...,\alpha _{n})\in \mathbb{N}_{0}^{n}$, we write $\left\vert \alpha
\right\vert =\alpha _{1}+...+\alpha _{n}$. The Euclidean scalar product of $%
x=(x_{1},...,x_{n})$ and $y=(y_{1},...,y_{n})$ is given by $x\cdot
y=x_{1}y_{1}+...+x_{n}y_{n}$.

The expression $f\lesssim g$ means that $f\leq c\,g$ for some independent
constant $c$ (and non-negative functions $f$ and $g$), and $f\approx g$
means $f\lesssim g\lesssim f$. As usual for any $x\in \mathbb{R}$, $%
\left\lfloor x\right\rfloor $ stands for the largest integer smaller than or
equal to $x$.

For $x\in \mathbb{R}^{n}$ and $r>0$ we denote by $B(x,r)$ the open ball in $%
\mathbb{R}^{n}$ with center $x$ and radius $r$. By \textrm{supp}$f$ we
denote the support of the function $f$, i.e., the closure of its non-zero
set. If $E\subset {\mathbb{R}^{n}}$ is a measurable set, then $|E|$ stands
for the (Lebesgue) measure of $E$ and $\chi _{E}$ denotes its characteristic
function. By $c$ we denote generic positive constants, which may have
different values at different occurrences.

For $v\in \mathbb{N}_{0}$ and $m\in \mathbb{Z}^{n}$, denote by $Q_{v,m}$ the
dyadic cube,%
\begin{equation*}
Q_{v,m}=2^{-v}([0,1)^{n}+m).
\end{equation*}

For each cube $Q$, we denote by $x_{v,m}$ the lower left-corner $2^{-v}m$ of 
$Q=Q_{v,m}$. Also, we set $\chi _{v,m}=\chi _{Q_{v,m}},v\in \mathbb{N}_{0}$, 
$m\in \mathbb{Z}^{n}.$

The symbol $\mathcal{S}(\mathbb{R}^{n})$ is used in place of the set of all
Schwartz functions on $\mathbb{R}^{n}$, it is equipped with the family of
seminorms,%
\begin{equation*}
\big\|\varphi \big\|_{\mathcal{S}_{M}}=\sup_{\gamma \in \mathbb{N}%
_{0}^{n},|\gamma |\leq M}\sup_{x\in \mathbb{R}^{n}}|\partial ^{\alpha
}\varphi (x)|(1+|x|)^{n+M+|\gamma |}<\infty
\end{equation*}%
for all $M\in \mathbb{N}$. We denote by $\mathcal{S}^{\prime }(\mathbb{R}%
^{n})$ the dual space of all tempered distributions on $\mathbb{R}^{n}$. We
define the Fourier transform of a function $f\in \mathcal{S}(\mathbb{R}^{n})$
by 
\begin{equation*}
\mathcal{F(}f)(\xi )=\left( 2\pi \right) ^{-n/2}\int_{\mathbb{R}%
^{n}}e^{-ix\cdot \xi }f(x)dx,\quad \xi \in \mathbb{R}^{n}.
\end{equation*}%
Its inverse is denoted by $\mathcal{F}^{-1}f$. Both $\mathcal{F}$ and $%
\mathcal{F}^{-1}$ are extended to the dual Schwartz space $\mathcal{S}%
^{\prime }(\mathbb{R}^{n})$ in the usual way.

(i) Let $0<p\leq \infty $. By $L^{p}$ we denote the space of all measurable
functions $f$ such that%
\begin{equation*}
\big\|f\big\|_{p}=\Big(\int_{\mathbb{R}^{n}}\left\vert f(x)\right\vert ^{p}dx%
\Big)^{1/p}<\infty ,
\end{equation*}%
with $0<p<\infty $ and%
\begin{equation*}
\big\|f\big\|_{\infty }=\underset{x\in \mathbb{R}^{n}}{\text{ess-sup}}%
\left\vert f(x)\right\vert <\infty .
\end{equation*}

(ii) Let $\alpha \in \mathbb{R}$ and $0<p<\infty $. The weighted Lebesgue
space $L^{p}(\mathbb{R}^{n},|\cdot |^{\alpha })$ contains all measurable
functions $f$ such that 
\begin{equation*}
\big\|f\big\|_{L^{p}(\mathbb{R}^{n},|\cdot |^{\alpha })}=\Big(\int_{\mathbb{R%
}^{n}}\left\vert f(x)\right\vert ^{p}|x|^{\alpha }dx\Big)^{1/p}<\infty .
\end{equation*}%
If $\alpha =0$, then we put $L^{p}(\mathbb{R}^{n},|\cdot |^{0})=L^{p}.$

(iii) The space $C(\mathbb{R}^{n})$ consists of all uniformly continuous
functions $f$ such that%
\begin{equation*}
\big\|f\big\|_{C(\mathbb{R}^{n})}=\sup_{x\in \mathbb{R}^{n}}\left\vert
f(x)\right\vert <\infty .
\end{equation*}

(iv)\ Let $m\in \mathbb{N}$. The space $C^{m}(\mathbb{R}^{n})$ is defined as
the set of all of all functions $f\in C(\mathbb{R}^{n})$, having all
classical derivatives $\partial ^{\alpha }f\in C(\mathbb{R}^{n})$\ up to
order $|\alpha |\leq m$ and such that%
\begin{equation*}
\big\|f\big\|_{C^{m}(\mathbb{R}^{n})}=\sum_{|\alpha |\leq m}\big\|\partial
^{\alpha }f\big\|_{C(\mathbb{R}^{n})}<\infty .
\end{equation*}

\begin{definition}
$($H\"{o}lder spaces$)$ Let $m\in \mathbb{N}_{0}$ and $m<s<m+1$. The space $%
C^{s}$ is defined to be the set of all $f\in C^{m}(\mathbb{R}^{n})$ such that%
\begin{equation*}
\big\|f\big\|_{C^{s}}=\big\|f\big\|_{C^{m}(\mathbb{R}^{n})}+\sum_{|\alpha
|=m}\sup_{x\neq y}\frac{|\partial ^{\alpha }f(x)-\partial ^{\alpha }f(y)|}{%
|x-y|^{s-m}}<\infty .
\end{equation*}
\end{definition}

\begin{definition}
Let $1<p<\infty $ and $m\in \mathbb{N}_{0}$. We define the Sobolev space $%
W_{p}^{m}$ as the set of functions $f\in L^{p}$ with weak derivatives $%
\partial ^{\beta }f\in L^{p}$ for $|\beta |\leq m$. We define the norm of $%
W_{p}^{m}$ by%
\begin{equation*}
\big\|f\big\|_{W_{p}^{m}}=\sum_{|\beta |\leq m}\big\|\partial ^{\beta }f%
\big\|_{p}<\infty .
\end{equation*}
\end{definition}

As usual, we define $W_{p}^{0}=L^{p}$.

\subsection{Besov and Triebel-Lizorkin spaces}

We present the Fourier analytical definition of Besov space and
Triebel-Lizorkin spaces\ and recall their basic properties. We first need
the concept of a smooth dyadic resolution of unity. Let $\vartheta $ be a
function in $\mathcal{S}(\mathbb{R}^{n})$ satisfying 
\begin{equation}
\vartheta (x)=1\quad \text{for}\quad \lvert x\rvert \leq 1\quad \text{and}%
\quad \vartheta (x)=0\quad \text{for}\quad \lvert x\rvert \geq \frac{3}{2}.
\label{function-v}
\end{equation}%
We put $\mathcal{F}\varphi _{0}(x)=\vartheta (x),\,\mathcal{F}\varphi
_{1}(x)=\vartheta (\frac{x}{2})-\vartheta (x)$ and $\varphi _{k}(x)=\mathcal{%
F}\varphi _{1}(2^{-k+1}x)\ $for $k=2,3,....$ Then we have supp$\mathcal{F}%
\varphi _{k}\subset \{x\in {\mathbb{R}^{n}}:2^{k-1}\leq \lvert x\rvert \leq
3\cdot 2^{k-1}\}\ $and 
\begin{equation}
\sum_{k=0}^{\infty }\mathcal{F}\varphi _{k}(x)=1\quad \text{for all}\quad
x\in {\mathbb{R}^{n}}.  \label{partition}
\end{equation}%
The system of functions $\{\varphi _{k}\}_{k\in \mathbb{N}_{0}}$ is called a
smooth dyadic resolution of unity. Thus we obtain the Littlewood-Paley
decomposition 
\begin{equation*}
f=\sum_{k=0}^{\infty }\varphi _{k}\ast f
\end{equation*}%
for all $f\in \mathcal{S}^{\prime }(\mathbb{R}^{n})$ (convergence in $%
\mathcal{S}^{\prime }(\mathbb{R}^{n})$).

We are now in a position to state the definition of Besov and
Triebel-Lizorkin spaces.

\begin{definition}
\label{def-herz-Besov}Let $s\in \mathbb{R}$, $0<p\leq \infty $\ and $0<q\leq
\infty $. $\newline
\mathrm{(i)}$ The Besov space $B_{p,q}^{s}$ is the collection of all $f\in 
\mathcal{S}^{\prime }(\mathbb{R}^{n})$ such that 
\begin{equation*}
\big\Vert f\big\Vert_{B_{p,q}^{s}}=\Big(\sum_{k=0}^{\infty }2^{ksq}\big\Vert%
\varphi _{k}\ast f\big\Vert_{p}^{q}\Big)^{1/q}<\infty ,
\end{equation*}%
with the obvious modification if $q=\infty $.$\newline
\mathrm{(ii)}$ Let $0<p<\infty $. The Triebel-Lizorkin space $F_{p,q}^{s}$
is the collection of all $f\in \mathcal{S}^{\prime }(\mathbb{R}^{n})$\ such
that%
\begin{equation*}
\big\Vert f\big\Vert_{F_{p,q}^{s}}=\Big\|\Big(\sum\limits_{k=0}^{\infty
}2^{ksq}\left\vert \varphi _{k}\ast f\right\vert ^{q}\Big)^{1/q}\Big\|%
_{p}<\infty ,
\end{equation*}%
with the obvious modification if $q=\infty .$
\end{definition}

\begin{remark}
Let\textrm{\ }$s\in \mathbb{R},0<p<\infty $ and $0<q\leq \infty $. The spaces%
\textrm{\ }$B_{p,q}^{s}$ and $F_{p,q}^{s}$ are independent of the particular
choice of the smooth dyadic resolution of unity\textrm{\ }$\{\varphi
_{j}\}_{j\in \mathbb{N}_{0}}$ (in the sense of$\mathrm{\ }$equivalent
quasi-norms). In particular $B_{p,q}^{s}$ and $F_{p,q}^{s}$ are quasi-Banach
spaces and if\textrm{\ }$p,q\geq 1$, then $B_{p,q}^{s}$ and $F_{p,q}^{s}$\
are Banach spaces. In addition%
\begin{equation*}
F_{p,2}^{m}=W_{p}^{m},\quad m\in \mathbb{N}_{0},1<p<\infty ,
\end{equation*}%
and%
\begin{equation*}
B_{\infty ,\infty }^{s}=C^{s},\quad s>0,s\notin \mathbb{N},
\end{equation*}%
see \cite{Sawano18}, \cite{T83} and \cite{T2}\ for more details about these
function spaces.
\end{remark}

Let $\psi \in \mathcal{S}(\mathbb{R}^{n})$ be such that%
\begin{equation*}
\int_{\mathbb{R}^{n}}\psi (x)dx=1.
\end{equation*}%
The local Hardy space $h^{p}$\ consist of all distributions $f\in \mathcal{S}%
^{\prime }(\mathbb{R}^{n})$ for which 
\begin{equation*}
\big\|f\big\|_{h^{p}}=\big\|\sup_{0<t<1}|t^{-n}\mathcal{F}^{-1}\psi \left(
t^{-1}\cdot \right) \ast f|\big\|_{p}<\infty .
\end{equation*}%
We have 
\begin{equation*}
F_{p,2}^{0}=h^{p},\quad 0<p<\infty ,
\end{equation*}%
see \cite[Sect. 2.2.2]{T83}.

We would mention that if $s\in \mathbb{R}$ and $0<p,q<\infty $ then $%
\mathcal{S}(\mathbb{R}^{n})$\ is dense in $A_{p,q}^{s}$ spaces, see \cite%
{T83}. Further characterizations of such spaces can be fund in \cite%
{Sawano18}, \cite{T83} and \cite{T2}.

\section{Lorentz-Herz spaces}

The aim of this section is twofold. First, we provide the necessary
background information about Lorentz-Herz spaces. The second aim is to
present some technical results\ which are needed in the following parts of
the paper, such as the boundedness of class of sublinear operators and
Plancherel-Polya-Nikolskij inequality on such spaces. The results of this
section will play a crucial role in several other sections of this paper.

\subsection{Definition and some basic properties}

The main purpose of this subsection is to present\ some fundamental
properties of Lorentz-Herz spaces. Let $k\in \mathbb{Z}$. For convenience,
we set 
\begin{equation*}
B_{k}=B(0,2^{k})\quad \text{and}\quad \bar{B}_{k}=\{x\in {\mathbb{R}%
^{n}:|x|\leq }2^{k}\}.
\end{equation*}%
In addition, we put%
\begin{equation*}
R_{k}=B_{k}\backslash B_{k-1}\quad \text{and}\quad \chi _{k}=\chi _{R_{k}}.
\end{equation*}

\begin{definition}
\label{def:herz}Let $0<p,q\leq \infty $ and $\alpha \in \mathbb{R}$. The
homogeneous Herz space $\dot{K}_{{p}}^{{\alpha },q}$ is defined as the set
of all $f\in L_{\mathrm{loc}}^{{p}}({\mathbb{R}^{n}}\backslash \{0\})$ such
that 
\begin{equation}
\big\|f\big\|_{\dot{K}_{{p}}^{{\alpha },q}}=\Big(\sum\limits_{k=-\infty
}^{\infty }2^{k{\alpha q}}\big\|f\,\chi _{k}\big\|_{{p}}^{q}\Big)%
^{1/q}<\infty  \label{hnorm}
\end{equation}%
with the usual modification if $q=\infty $, i.e., 
\begin{equation*}
\big\|f\big\|_{\dot{K}_{{p}}^{{\alpha },\infty }}=\sup_{k\in \mathbb{Z}}\big(%
2^{k{\alpha }}\big\|f\,\chi _{k}\big\|_{{p}}\big).
\end{equation*}
\end{definition}

\begin{remark}
Herz spaces play an important role in Harmonic Analysis. After they have
been introduced in \cite{Herz68}, the theory of these spaces had a
remarkable development in part due to its usefulness in applications. For
instance, they appear in the characterization of multipliers on Hardy spaces 
\cite{BS85}, in the semilinear parabolic equations; see \cite{Drappl}, in
the summability of Fourier transforms \cite{FeichtingerWeisz08}, in
regularity theory for elliptic equations in divergence form \cite{Rag09}-%
\cite{Rag12}, and in the Cauchy problem for Navier-Stokes equations \cite%
{T11}. But, the study of the Herz spaces can be dated back to the work of
Beurling \cite{Be64}. Feichtinger in \cite{Fe} introduced another norm which
is equivalent to the norm defined by Beurling.
\end{remark}

\begin{remark}
A\ detailed discussion of the properties of Herz spaces may be found in \cite%
{HerYang99}, \cite{Ho18}, \cite{LYH08} and \cite{RS}, and references therein.
\end{remark}

Let $f$ be a measurable function on $\mathbb{R}^{n},t>0$ and $\lambda >0$.
We define the distribution function of $f$ by 
\begin{equation*}
m_{f}(\lambda )=|\{x\in \mathbb{R}^{n}:|f(x)|>\lambda \}|.
\end{equation*}%
The non-increasing rearrangement of $f$ is defined by 
\begin{equation*}
f^{\ast }(t)=\inf \{\lambda >0:m_{f}(\lambda )\leq t\}.
\end{equation*}

Next, we recall the Lorentz spaces.

\begin{definition}
Let $0<p<\infty $ and $0<r\leq \infty $. Then the Lorentz\ space $L^{p,r}$
is the set of all measurable function $f$ on $\mathbb{R}^{n}$ such that $%
\big\|f\big\|_{L^{p,r}}<\infty $, where%
\begin{equation*}
\big\|f\big\|_{L^{p,r}}=\Big(\int_{0}^{\infty }t^{\frac{r}{p}}(f^{\ast
}(t))^{r}\frac{dt}{t}\Big)^{1/r}\quad \text{if}\quad 0<r<\infty
\end{equation*}%
and%
\begin{equation*}
\big\|f\big\|_{L^{p,\infty }}=\sup_{t>0}t^{\frac{1}{p}}f^{\ast }(t)\quad 
\text{if}\quad r=\infty .
\end{equation*}
\end{definition}

\begin{remark}
We know that the Lorentz space is very important in harmonic analysis. A
much more detailed about such spaces can be found in \cite[Chapter 1]{L.
Graf14}. We put $L^{\infty ,\infty }=L^{\infty }.$
\end{remark}

We recall some basic properties of Lorentz space.

\begin{proposition}
\label{Holder's-convolution-lorentz}Let $0<p,p_{0},p_{1}<\infty $ and $%
0<r,r_{0},r_{1}\leq \infty $. $\newline
\mathrm{(i)}$ The Lorentz space $L^{p,r}$\ with the quasi-norm $\big\|\cdot %
\big\|_{L^{p,r}}$ is complete, quasi-Banach, for all $0<p<\infty $ and $%
0<r\leq \infty $.$\newline
\mathrm{(ii)}$ Let $0<s<\infty $ and $f\in L^{p,r}$. Then we have%
\begin{equation}
\big\||f|^{s}\big\|_{L^{p,r}}=\big\|f\big\|_{L^{ps,rs}}^{s}.
\label{pro1-lorentz}
\end{equation}%
$\mathrm{(iii)}$ We have $L^{p,p}=L^{p}$ in the sense of equivalent norms.$%
\newline
\mathrm{(iv)}$ Suppose $0<q<r\leq \infty $. Then $L^{p,q}\hookrightarrow
L^{p,r}.\newline
\mathrm{(v)}$\ Let $f\in L^{p_{0},r_{0}}$ and $g\in L^{p_{1},r_{1}}$.
Suppose 
\begin{equation*}
\frac{1}{p}=\frac{1}{p_{0}}+\frac{1}{p_{1}}\quad \text{and}\quad \frac{1}{r}=%
\frac{1}{r_{0}}+\frac{1}{r_{1}}.
\end{equation*}%
Then $fg\in L^{p,r}$ and 
\begin{equation*}
\big\|fg\big\|_{L^{p,r}}\leq c\big\|f\big\|_{L^{p_{0},r_{0}}}\big\|g\big\|%
_{L^{p_{1},r_{1}}}.
\end{equation*}%
$\mathrm{(vi)}$ Let $f\in L^{p_{0},r_{0}}$ and $g\in L^{p_{1},r_{1}}$.
Suppose\ $1<p,p_{0},p_{1}<\infty $ and $0<r_{0},r_{1}\leq \infty $\ with%
\begin{equation*}
\frac{1}{p}+1=\frac{1}{p_{0}}+\frac{1}{p_{1}}\quad \text{and}\quad \frac{1}{r%
}=\frac{1}{r_{0}}+\frac{1}{r_{1}}.
\end{equation*}%
Then $f\ast g\in L^{p,r}$\ and 
\begin{equation*}
\big\|f\ast g\big\|_{L^{p,r}}\leq c\big\|f\big\|_{L^{p_{0},r_{0}}}\big\|g%
\big\|_{L^{p_{1},r_{1}}}.
\end{equation*}%
$\mathrm{(vii)}$ Suppose $f\in L^{p,r},1<p<\infty ,1\leq r\leq \infty $ or $%
p=r=\infty $. We put 
\begin{equation*}
\big\|f\big\|_{L^{p,r}}^{\ast }=\Big(\int_{0}^{\infty }t^{\frac{r}{p}%
}(f^{\ast \ast }(t))^{r}\frac{dt}{t}\Big)^{1/r}\quad \text{if}\quad 1\leq
r<\infty
\end{equation*}%
and%
\begin{equation*}
\big\|f\big\|_{L^{p,\infty }}^{\ast }=\sup_{t>0}t^{\frac{1}{p}}f^{\ast \ast
}(t)\quad \text{if}\quad r=\infty ,
\end{equation*}%
where $f^{\ast \ast }(t)=\frac{1}{t}\int_{0}^{t}f^{\ast }(s)ds,s>0$. Then 
\begin{equation*}
\big\|f\big\|_{L^{p,r}}\leq \big\|f\big\|_{L^{p,r}}^{\ast }\leq \frac{p}{p-1}%
\big\|f\big\|_{L^{p,r}}.
\end{equation*}
\end{proposition}

\begin{proof}
For the proof, see \cite{L. Graf14} and \cite{O'Neil}.
\end{proof}

We recall the following Hardy-Littlewood inequality; see \cite{L. Graf14}.

\begin{lemma}
\label{Hardy-Littlewood inequality}Let $f$ and $g$ be two non-negative
functions on $\mathbb{R}^{n}$. The inequality%
\begin{equation*}
\int_{\mathbb{R}^{n}}f(x)g(x)dx\leq \int_{0}^{\infty }f^{\ast }(t)g^{\ast
}(t)dt
\end{equation*}%
holds.
\end{lemma}

Recall that 
\begin{equation}
\big\|\chi _{A}\big\|_{L^{{p,r}}}=(\frac{p}{r})^{\frac{1}{r}}|A|^{\frac{1}{p}%
},\quad \big\|\chi _{A}\big\|_{L^{{p,\infty }}}=|A|^{\frac{1}{p}},\quad
0<p,r<\infty  \label{est-function1}
\end{equation}%
for any measurable set $A\subset \mathbb{R}^{n}$\ of finite measure and%
\begin{equation}
\big\|f(\lambda \cdot )\big\|_{L^{p,r}}=\lambda ^{-n/p}\big\|f\big\|%
_{L^{p,r}},\quad f\in L^{p,r},0<\lambda <\infty ,  \label{dilation-lorentz}
\end{equation}

Now, we define the Lorentz-Herz spaces $\dot{K}_{{p,r}}^{{\alpha },q}$.

\begin{definition}
\label{Herz-lorentz}Let $0<p<\infty ,0<q,r\leq \infty $ and $\alpha \in 
\mathbb{R}$. The homogeneous Lorentz-Herz space $\dot{K}_{{p,r}}^{{\alpha }%
,q}$ is defined as the set of all functions $f\in L_{\mathrm{loc}}^{{p,r}}({%
\mathbb{R}^{n}}\backslash \{0\})$ such that 
\begin{equation*}
\big\|f\big\|_{\dot{K}_{{p,r}}^{{\alpha },q}}=\Big(\sum\limits_{k=-\infty
}^{\infty }2^{k{\alpha q}}\big\|f\,\chi _{k}\big\|_{L^{p,r}}^{q}\Big)%
^{1/q}<\infty
\end{equation*}%
with the usual modification if $q=\infty $, i.e., 
\begin{equation*}
\big\|f\big\|_{\dot{K}_{{p,r}}^{{\alpha },\infty }}=\sup_{k\in \mathbb{Z}}%
\big(2^{k{\alpha }}\big\|f\,\chi _{k}\big\|_{L^{p,r}}\big).
\end{equation*}
\end{definition}

\begin{remark}
Suppose $0<q\leq \infty $\ and $\alpha \in \mathbb{R}$. If either $%
0<p,r<\infty $ or $r=\infty $\ and $0<p<\infty $, then $\dot{K}_{{p,r}}^{{%
\alpha },q}$ is a quasi-Banach ideal space with the Fatou property. More
detailed about Lorentz-Herz spaces is given \cite{AKS}, \cite{FPV} and \cite%
{T11}. There is another definition of Lorentz-Herz spaces; see \cite{LYH08}.
\end{remark}

We now collect some inequalities in the theory of Lorentz-Herz spaces which
we will use throughout the paper. We begin with H\"{o}lder's inequality.

\begin{proposition}
\label{holder's}Let $0<p_{i}<\infty ,0<q_{i},r_{i}\leq \infty $ and $\alpha
_{i}\in \mathbb{R},i\in \{0,1\}$. Suppose 
\begin{equation*}
\alpha =\alpha _{0}+\alpha _{1},\quad \frac{1}{p}=\frac{1}{p_{0}}+\frac{1}{%
p_{1}},\quad \frac{1}{r}=\frac{1}{r_{0}}+\frac{1}{r_{1}}\quad \text{and}%
\quad \frac{1}{q}=\frac{1}{q_{0}}+\frac{1}{q_{1}}.
\end{equation*}%
Then 
\begin{equation}
\big\|fg\big\|_{\dot{K}_{p,r}^{\alpha ,q}}\leq \big\|f\big\|_{\dot{K}%
_{p_{0},r_{0}}^{\alpha _{0},q_{0}}}\big\|g\big\|_{\dot{K}_{p_{1},r_{1}}^{%
\alpha _{1},q_{1}}}  \label{Holder}
\end{equation}%
holds for all $f\in \dot{K}_{p_{0},r_{0}}^{\alpha _{0},q_{0}}$\ and all\ $%
g\in \dot{K}_{p_{1},r_{1}}^{\alpha _{1},q_{1}}.$
\end{proposition}

\begin{proof}
The estimate \eqref{Holder}\ follows from Proposition \ref%
{Holder's-convolution-lorentz}/(v) and H\"{o}lder's inequality in sequence
spaces $\ell ^{q}$.
\end{proof}

We present an interpolation inequality, namely if a function $f$ is in $\dot{%
K}_{p_{0},r_{0}}^{\alpha _{0},q_{0}}\cap \dot{K}_{p_{1},r_{1}}^{\alpha
_{1},q_{1}}$, then it also lies in $\dot{K}_{p,r}^{\alpha ,q}$, with some
suitable assumptions on the parameters of such spaces.

\begin{lemma}
Let $0<\theta <1,0<p_{i}<\infty ,0<q_{i},r_{i}\leq \infty $ and $\alpha
_{i}\in \mathbb{R},i\in \{0,1\}$. Suppose 
\begin{equation*}
\alpha =(1-\theta )\alpha _{0}+\theta \alpha _{1},\quad \frac{1}{p}=\frac{%
1-\theta }{p_{0}}+\frac{\theta }{p_{1}},\quad \frac{1}{r}=\frac{1-\theta }{%
r_{0}}+\frac{\theta }{r_{1}}\quad \text{and}\quad \frac{1}{q}=\frac{1-\theta 
}{q_{0}}+\frac{\theta }{q_{1}}.
\end{equation*}%
We have the so-called interpolation inequalities: 
\begin{equation}
\big\|f\big\|_{\dot{K}_{p,r}^{\alpha ,q}}\leq \big\|f\big\|_{\dot{K}%
_{p_{0},r_{0}}^{\alpha _{0},q_{0}}}^{1-\theta }\big\|f\big\|_{\dot{K}%
_{p_{1},r_{1}}^{\alpha _{1},q_{1}}}^{\theta }  \label{Interpolation}
\end{equation}%
holds for all $f\in \dot{K}_{p_{0},r_{0}}^{\alpha _{0},q_{0}}\cap \dot{K}%
_{p_{1},r_{1}}^{\alpha _{1},q_{1}}$.
\end{lemma}

\begin{proof}
From H\"{o}lder's inequality and \eqref{pro1-lorentz} 
\begin{equation*}
\big\|f\chi _{k}\big\|_{L^{p,r}}=\big\||f|^{1-\theta }|f|^{\theta }\chi _{k}%
\big\|_{L^{p,r}}\lesssim \big\|f\chi _{k}\big\|_{L^{p_{0},r_{0}}}^{1-\theta }%
\big\|f\chi _{k}\big\|_{L^{p_{1},r_{1}}}^{\theta }
\end{equation*}%
for any $k\in \mathbb{Z}$, where the implicit constant is independent of $k$%
. Using the H\"{o}lder inequality in sequence spaces $\ell ^{q}$ and the
fact that $\alpha =(1-\theta )\alpha _{0}+\theta \alpha _{1}$, we obtain the
desired estimate \eqref{Interpolation}.
\end{proof}

The following lemma and proposition give some preliminary results for
Lorentz-Herz spaces.

\begin{lemma}
\label{embeddings1-lorentz}$\mathrm{(i)}$ Let $0<p<\infty
,0<q_{1},q_{2},r\leq \infty $ and $\alpha \in \mathbb{R}$. Then 
\begin{equation}
\dot{K}_{p,r}^{\alpha ,q_{1}}\hookrightarrow \dot{K}_{p,r}^{\alpha ,q_{2}},
\label{emb1}
\end{equation}%
if and only if $0<q_{1}\leq q_{2}\leq \infty .\newline
\mathrm{(ii)}\ $Let $0<p<\infty ,0<q\leq \infty $\ and $\alpha \in \mathbb{R}
$. The space $\dot{K}_{{p,p}}^{{\alpha },q}$\ coincides with the Herz space $%
\dot{K}_{{p}}^{{\alpha },q}$.$\newline
\mathrm{(iii)}$\ Let $0<p<\infty ,0<q\leq \infty ,0<r_{1}\leq r_{2}\leq
\infty \ $and $\alpha \in \mathbb{R}$. Then 
\begin{equation*}
\dot{K}_{p,r_{1}}^{\alpha ,q}\hookrightarrow \dot{K}_{p,r_{2}}^{\alpha ,q}.
\end{equation*}%
$\mathrm{(iv)}$\ Let $0<p,s<\infty ,0<q\leq \infty \ $and $\alpha \in 
\mathbb{R}$. Then 
\begin{equation*}
\big\||f|^{s}\big\|_{\dot{K}_{p,r}^{\alpha ,q}}=\big\|f\big\|_{\dot{K}%
_{ps,rs}^{\alpha /s,qs}}^{s}.
\end{equation*}
\end{lemma}

\begin{proof}
We will do the proof in two steps.

\textit{Step 1.} We will prove (i). Assume that $0<q_{1}\leq q_{2}<\infty $.
Let $f\in \dot{K}_{p,r}^{\alpha ,q_{1}}$\ and 
\begin{equation*}
I=\Big(\sum\limits_{k=\mathbb{-\infty }}^{\infty }2^{k{\alpha q}_{1}}\big\|%
f\chi _{k}\big\|_{{L}^{p,r}}^{q_{1}}\Big)^{1/q_{1}}.
\end{equation*}%
If $I=0$, then nothing to prove. We have 
\begin{equation*}
\sum\limits_{k=\mathbb{-\infty }}^{\infty }2^{k{\alpha q}_{2}}\big\|\frac{f}{%
I}\chi _{k}\big\|_{{L}^{p,r}}^{q_{2}}=\sum\limits_{k=\mathbb{-\infty }%
}^{\infty }\big(2^{k{\alpha }}\big\|\frac{f}{I}\chi _{k}\big\|_{{L}^{p,r}}%
\big)^{q_{2}-q_{1}+q_{1}}.
\end{equation*}%
Observe that 
\begin{equation*}
\big(2^{k{\alpha }}\big\|\frac{f}{I}\chi _{k}\big\|_{{L}^{p,r}}\big)%
^{q_{2}-q_{1}}\leq 1
\end{equation*}%
for any $k\in \mathbb{Z}$. Therefore 
\begin{equation*}
\sum\limits_{k=\mathbb{-\infty }}^{\infty }2^{k{\alpha q}_{2}}\big\|\frac{f}{%
I}\chi _{k}\big\|_{{L}^{p,r}}^{q_{2}}\leq \sum\limits_{k=\mathbb{-\infty }%
}^{\infty }\big(2^{k{\alpha }}\big\|\frac{f}{I}\chi _{k}\big\|_{{L}^{p,r}}%
\big)^{q_{1}}=1,
\end{equation*}%
which gives the desired estimate. Now, let 
\begin{equation*}
f_{N}=\sum\limits_{j=1}^{N}2^{-(\alpha +\frac{n}{p})j}\chi _{j},\quad N\in 
\mathbb{N}.
\end{equation*}%
By \eqref{est-function1}, we have 
\begin{equation*}
\big\|f_{N}\chi _{k}\big\|_{{L}^{p,r}}=\left\{ 
\begin{array}{ccc}
0, & \text{if} & k\notin \{1,...,N\}, \\ 
c2^{-\alpha k}, & \text{if} & k\in \{1,...,N\},%
\end{array}%
\right.
\end{equation*}%
where the positive constant $c$ is independent of $k$ and $N$. Hence 
\begin{equation*}
\big\|f\big\|_{\dot{K}_{{p,r}}^{{\alpha },q_{i}}}=c\text{ }N^{\frac{1}{q_{i}}%
},\quad i\in \{1,2\}.
\end{equation*}%
If \eqref{emb1} holds, then $N^{\frac{1}{q_{2}}-\frac{1}{q_{1}}}\lesssim 1$,
where the implicit constant is independent of $N$. Observe for $N$, tends to
infinity, then $0<q_{1}\leq q_{2}<\infty $\ becomes necessary. Obviously
that $\dot{K}_{p,r}^{\alpha ,q_{1}}\hookrightarrow \dot{K}_{p,r}^{\alpha
,\infty }$ if and only if $0<q_{1}\leq \infty .$

\textit{Step 2.} We prove (ii), (iii) and (iv). The desired result follows
by the fact that ${L}^{p,p}=L^{p}$, ${L}^{p,r_{1}}\hookrightarrow {L}%
^{p,r_{2}},0<r_{1}\leq r_{2}\leq \infty $\ and \eqref{pro1-lorentz}. The
proof is complete.
\end{proof}

Let $K_{p,r}^{\alpha ,q}$ be the inhomogeneous Lorentz-Herz spaces. More
precisely, the set of all functions $f\in L_{\mathrm{loc}}^{{p,r}}({\mathbb{R%
}^{n}})$ such that 
\begin{equation*}
\big\|f\big\|_{K_{{p,r}}^{{\alpha },q}}=\big\|f\,\chi _{B_{0}}\big\|_{{L}%
^{p,r}}+\Big(\sum\limits_{k=1}^{\infty }2^{k{\alpha q}}\big\|f\,\chi _{k}%
\big\|_{{L}^{p,r}}^{q}\Big)^{1/q}<\infty .
\end{equation*}

\begin{proposition}
\label{embeddings1 copy(1)-lorentz}$\mathrm{(i)}$ Let $0<p<\infty ,0<r,q\leq
\infty $ and $\alpha >0$. Then 
\begin{equation*}
\dot{K}_{p,r}^{\alpha ,q}\cap L^{p,r}=K_{p,r}^{\alpha ,q},
\end{equation*}%
in the sense of equivalent quasi-norms.$\newline
\mathrm{(ii)}$ Let $0<q\leq \infty ,0<r_{2},r_{1}\leq \infty ,\alpha \in 
\mathbb{R}\ $and suppose $0<p_{2}<p_{1}<\infty $. Then 
\begin{equation*}
\dot{K}_{p_{1},r_{1}}^{\alpha ,q}\hookrightarrow \dot{K}_{p_{2},r_{2}}^{%
\alpha -\frac{n}{p_{2}}+\frac{n}{p_{1}},q}
\end{equation*}%
holds.
\end{proposition}

\begin{proof}
We proceed in two steps.

\textit{Step 1. Proof of }$\mathrm{(i)}$. Let $f\in \dot{K}_{p,r}^{\alpha
,q}\cap L^{p,r}$. Obviously 
\begin{equation*}
\big\|f\chi _{B_{0}}\big\|_{{L}^{p,r}}\leq \big\|f\big\|_{{L}^{p,r}}\quad 
\text{and}\quad \sum_{k=1}^{\infty }2^{k\alpha q}\big\|f\chi _{k}\big\|_{{L}%
^{p,r}}^{q}\leq \big\|f\big\|_{\dot{K}_{p,r}^{\alpha ,q}}^{q}.
\end{equation*}%
Hence, 
\begin{equation*}
\big\|f\big\|_{K_{p,r}^{\alpha ,q}}\leq \big\|f\big\|_{\dot{K}_{p,r}^{\alpha
,q}\cap L^{p,r}}.
\end{equation*}%
Now, let $f\in K_{p,r}^{\alpha ,q}$. Since\ $R_{k}\subset B_{0},k\in \mathbb{%
Z}\backslash \mathbb{N}$ and $\alpha >0$,\ we obtain

\begin{equation*}
\sum_{k=-\infty }^{0}2^{k\alpha q}\big\|f\chi _{k}\big\|_{{L}^{p,r}}^{q}\leq
\sum_{k=-\infty }^{0}2^{k\alpha q}\big\|f\chi _{B_{0}}\big\|_{{L}%
^{p,r}}^{q}\lesssim \big\|f\chi _{B_{0}}\big\|_{{L}^{p,r}}^{q}\leq \big\|f%
\big\|_{K_{p,r}^{\alpha ,q}}^{q}.
\end{equation*}%
Therefore $f\in \dot{K}_{p,r}^{\alpha ,q}$\ and 
\begin{equation*}
\big\|f\big\|_{\dot{K}_{p,r}^{\alpha ,q}}\lesssim \big\|f\big\|%
_{K_{p,r}^{\alpha ,q}}.
\end{equation*}%
We will prove that $f\in L^{p,r}$. Observe that 
\begin{equation*}
\big\|f\big\|_{{L}^{p,r}}\lesssim \big\|f\chi _{B_{0}}\big\|_{{L}^{p,r}}+%
\big\|f\chi _{\mathbb{R}^{n}\setminus B_{0}}\big\|_{{L}^{p,r}}.
\end{equation*}%
Let $0<\tau <\min (1,p,r)$. By \cite[(19)]{ST19}, we have 
\begin{align*}
\big\|f\chi _{\mathbb{R}^{n}\setminus B_{0}}\big\|_{{L}^{p,r}}\lesssim & %
\Big(\sum_{k=1}^{\infty }\big\|f\chi _{k}\big\|_{{L}^{p,r}}^{\tau }\Big)%
^{1/\tau } \\
\lesssim & \sup_{k\in \mathbb{N}}\big(2^{k\alpha }\big\|f\chi _{k}\big\|_{{L}%
^{p,r}}\big),
\end{align*}%
since $\alpha >0$. Consequently%
\begin{equation*}
\big\|f\big\|_{\dot{K}_{p}^{\alpha ,q}\cap {L}^{p,r}}\lesssim \big\|f\big\|%
_{K_{p,r}^{\alpha ,q}}.
\end{equation*}%
This estimate gives the desired result.

\textit{Step 2. Proof of }$\mathrm{(ii)}$. Let $\frac{1}{p_{2}}=\frac{1}{%
p_{1}}+\frac{1}{v},0<v<\infty $. By H\"{o}lder's inequality and %
\eqref{est-function1}, we obtain 
\begin{align*}
\big\|f\chi _{k}\big\|_{{L}^{p_{2},r_{2}}}& \lesssim \big\|f\chi _{k}\big\|%
_{L^{p_{1},\infty }}\big\|\chi _{k}\big\|_{L^{v,r_{1}}} \\
& \lesssim 2^{(\frac{n}{p_{2}}-\frac{n}{p_{1}})k}\big\|f\chi _{k}\big\|_{{L}%
^{p_{1},r_{1}}}
\end{align*}%
for any $k\in \mathbb{Z}$, where the implicit constant is independent of $k$
. This estimate yields the desired embeddings. This finishes the proof.
\end{proof}

Let $V_{\alpha ,p,r,q}$ be the set of $(\alpha ,p,r,q)\in \mathbb{R}\times
(1,\infty )^{2}\times \lbrack 1,\infty ]$\ such that:

\textbullet\ $\alpha <n-\frac{n}{p}$, $1<r,p<\infty $ and $1\leq q\leq
\infty ,$

\textbullet\ $\alpha =n-\frac{n}{p}$, $1<r,p<\infty \ $and$\ q=1,$

The next lemma gives a necessary and sufficient condition on the parameters $%
\alpha $, $p,r$ and $q$, in order to make sure that 
\begin{equation*}
\langle T_{f},\varphi \rangle =\int_{\mathbb{R}^{n}}f(x)\varphi (x)dx,\quad
\varphi \in \mathcal{D}(\mathbb{R}^{n}),f\in \dot{K}_{p,r}^{\alpha ,q}
\end{equation*}%
generates a regular distribution $T_{f}\in \mathcal{D}^{\prime }({\mathbb{R}%
^{n}})$.

\begin{lemma}
\label{regular-distribution-lorentz}Let $1<r,p<\infty ,1\leq q\leq \infty $
and $\alpha \in \mathbb{R}$. Then 
\begin{equation*}
\dot{K}_{p,r}^{\alpha ,q}\hookrightarrow \mathcal{D}^{\prime }(\mathbb{R}%
^{n}),
\end{equation*}%
holds if and only if $(\alpha ,p,r,q)\in V_{\alpha ,p,r,q}.$
\end{lemma}

\begin{proof}
The proof is a slight variant of \cite{Dr-Sobolev}. For the convenience of
the reader, we give some details. We divide the proof into two steps.

\textit{Step 1.} Assume that $(\alpha ,p,r,q)\in V_{\alpha ,p,r,q}$, $f\in 
\dot{K}_{p,r}^{\alpha ,q}$ and $B(0,2^{N})\subset \mathbb{R}^{n},N\in 
\mathbb{Z}$. By similarity we only consider the case $\alpha <n-\frac{n}{p}$%
, $1<r,p<\infty $ and $1\leq q\leq \infty $. H\"{o}lder's inequality and %
\eqref{est-function1} give 
\begin{align*}
\big\|f\big\|_{L^{1}(B(0,2^{N}))}=& \sum_{i=-\infty }^{N}\big\|f\chi _{i}%
\big\|_{1} \\
\lesssim & \sum_{i=-\infty }^{N}\big\|f\chi _{i}\big\|_{L^{p,r}}\big\|\chi
_{i}\big\|_{L^{p^{\prime },r^{\prime }}} \\
=& c2^{N(n-\frac{n}{p}-\alpha )}\sum_{i=-\infty }^{N}2^{(i-N)(n-\frac{n}{p}%
-\alpha )}2^{i\alpha }\big\|f\chi _{i}\big\|_{L^{p,r}} \\
\lesssim & 2^{N(n-\frac{n}{p}-\alpha )}\big\|f\big\|_{\dot{K}_{p,r}^{\alpha
,q}}.
\end{align*}

\textit{Step 2.} Assume that $(\alpha ,p,r,q)\notin V_{\alpha ,p,r,q}$. We
distinguish two cases.

\textit{Case 1.} $\alpha >n-\frac{n}{p}$. We set $f(x)=|x|^{-n}\chi
_{0<|\cdot |<1}(x)$. We obtain $f\in \dot{K}_{p,r}^{\alpha ,q}$ for any $%
1<r,p<\infty $ and $1\leq q\leq \infty $\ whereas $f\notin L_{\mathrm{loc}%
}^{1}(\mathbb{R}^{n})$. Indeed, by \eqref{est-function1}, we find 
\begin{align*}
\big\|f\big\|_{\dot{K}_{p,r}^{\alpha ,q}}^{q}& =\sum_{k\in \mathbb{Z}%
,2^{k}<2}2^{k\alpha q}\big\|f\chi _{k}\big\|_{L^{p,r}}^{q} \\
& \lesssim \sum_{k\in \mathbb{Z},2^{k}<2}2^{k(\alpha -n)q}\big\|\chi
_{k}\chi _{0<|\cdot |<1}\big\|_{L^{p,r}}^{q} \\
& \lesssim \sum_{k\in \mathbb{Z},2^{k}<2}2^{k(\alpha -n+\frac{n}{p})q} \\
& <\infty ,
\end{align*}%
with the usual modification if $q=\infty $. Obviously, $f\notin L_{\mathrm{%
loc}}^{1}(\mathbb{R}^{n})$.

\textit{Case 2.} $\alpha =n-\frac{n}{p}$, $1<r,p<\infty $ and $1<q\leq
\infty $. We consider the function $f$ defined by 
\begin{equation*}
f(x)=|x|^{-n}(|\log |x||)^{-1}\chi _{0<|\cdot |<\frac{1}{2}}(x).
\end{equation*}%
An easy computation yields that 
\begin{equation*}
\big\|f\big\|_{\dot{K}_{p,r}^{n-\frac{n}{p},q}}^{q}\lesssim
\sum_{k=1}^{\infty }k^{-q}<\infty ,
\end{equation*}%
which gives that $f\in \dot{K}_{p,r}^{n-\frac{n}{p},q}$, with the usual
modifications when $q=\infty $. It is easily seen that $f$ does not belong
to $L_{\mathrm{loc}}^{1}(\mathbb{R}^{n})$. The lemma is now proved.
\end{proof}

We collect some assertions which will be of some use for us. If $x\in 
\mathbb{R}^{n}$ and $R,N>0$, then we put $\eta _{R,N}(x)=R^{n}(1+R\left\vert
x\right\vert )^{-N}$.

\begin{lemma}
\label{est-eta}\textit{Let }$R>0,0<p<\infty $ and $0<r\leq \infty $\textit{.
Then there exists a constant }$c>0$ \textit{independent }$R$\textit{\ such
that }for any $N>\frac{n}{p}$ we have%
\begin{equation}
\big\|\eta _{R,N}\big\|_{L^{p,r}}\leq cR^{n-\frac{n}{p}}.  \label{est-teta}
\end{equation}
\end{lemma}

\begin{proof}
Simple calculation yields that%
\begin{equation*}
(\eta _{R,N})^{\ast }(t)=R^{n}(1+Rt^{1/n})^{-N},\quad t>0.
\end{equation*}%
Since $N>\frac{n}{p}$, we obtain the desired conclusion \eqref{est-teta}.
\end{proof}

\begin{lemma}
\label{Lp,r-estimate}L\textit{et }$0<p<\infty ,0<q,\beta \leq \infty
,0<r_{0}\leq r_{1}\leq \infty $\ and\ $\alpha \in \mathbb{R}$. Assume that $%
p\neq \beta $ or $p=\beta \geq r_{0}$. Then%
\begin{equation}
\Big\|\Big(\sum\limits_{j=0}^{\infty }|f_{j}|^{\beta }\Big)^{1/\beta }\Big\|%
_{\dot{K}_{p,r_{1}}^{\alpha ,q}}\lesssim \Big(\sum\limits_{j=0}^{\infty }%
\big\|f_{j}\big\|_{\dot{K}_{p,r_{0}}^{\alpha ,q}}^{\tau }\Big)^{1/\tau }
\label{lp,r-estimate}
\end{equation}%
for any $0<\tau \leq \min (p,r_{1},q,\beta )$, whenever the right-hand side
of \eqref{lp,r-estimate} is finite.
\end{lemma}

\begin{proof}
By the embedding ${L}^{p,r_{0}}\hookrightarrow {L}^{p,r_{1}}$, we only
consider the case $r_{0}=r_{1}$. Let\ $k\in \mathbb{Z}$. We have%
\begin{equation}
\Big\|\Big(\sum\limits_{j=0}^{\infty }|f_{j}|^{\beta }\Big)^{1/\beta }\chi
_{k}\Big\|_{{L}^{p,r_{1}}}\lesssim \Big(\sum\limits_{j=0}^{\infty }\big\|%
f_{j}\chi _{k}\big\|_{{L}^{p,r_{1}}}^{\tau }\Big)^{1/\tau },  \label{seeger}
\end{equation}%
see \cite[Proposition 4.1]{ST19}. The proof of \eqref{lp,r-estimate} follows
from the monotonicity in $q$ of the $\ell ^{q}$-norm. More precisely, by the
inequality 
\begin{equation*}
\Big\|\sum\limits_{j=0}^{\infty }g_{j}\Big\|_{\ell ^{\delta }}\leq \Big(%
\sum\limits_{j=0}^{\infty }\big\|g_{j}\big\|_{\ell ^{\delta }}^{v}\Big)%
^{1/v},\quad \{g_{j}\}_{j\in \mathbb{N}_{0}}\in \ell ^{\delta }
\end{equation*}%
for any $0<v\leq \min (1,\delta )$.
\end{proof}

\begin{lemma}
\label{Lp,r-estimate copy(1)}L\textit{et }$0<p<\infty ,0<q,\beta \leq \infty
,0<r_{0}\leq r_{1}\leq \infty $\ and\ $\alpha \in \mathbb{R}$. Assume that $%
p\neq \beta $ or $p=\beta \leq r_{1}$. Then%
\begin{equation}
\Big(\sum\limits_{j=0}^{\infty }\big\|f_{j}\big\|_{\dot{K}_{p,r_{1}}^{\alpha
,q}}^{\tau }\Big)^{1/\tau }\lesssim \Big\|\Big(\sum\limits_{j=0}^{\infty
}|f_{j}|^{\beta }\Big)^{1/\beta }\Big\|_{\dot{K}_{p,r_{0}}^{\alpha ,q}}
\label{lp,r-estimate1}
\end{equation}%
for any $\tau \geq \max (p,r_{0},q,\beta )$, whenever the right-hand side of %
\eqref{lp,r-estimate1} is finite.
\end{lemma}

\begin{proof}
Again, by the embedding ${L}^{p,r_{0}}\hookrightarrow {L}^{p,r_{1}}$, we
only consider the case $r_{0}=r_{1}$. Let\ $k\in \mathbb{Z}$. Since\ $\tau
\geq q$, we obtain%
\begin{align*}
& \Big(\sum\limits_{j=0}^{\infty }\Big(\sum\limits_{k=-\infty }^{\infty
}2^{k\alpha q}\big\|f_{j}\chi _{k}\big\|_{{L}^{p,r_{1}}}^{q}\Big)^{\tau /q}%
\Big)^{1/\tau } \\
& \lesssim \Big(\sum\limits_{k=-\infty }^{\infty }2^{k\alpha q}\Big(%
\sum\limits_{j=0}^{\infty }\big\|f_{j}\chi _{k}\big\|_{{L}^{p,r_{1}}}^{\tau }%
\Big)^{q/\tau }\Big)^{1/q}.
\end{align*}%
To prove \eqref{lp,r-estimate1} we use the inequality%
\begin{equation*}
\Big(\sum\limits_{j=0}^{\infty }\big\|f_{j}\chi _{k}\big\|_{{L}%
^{p,r_{1}}}^{\tau }\Big)^{1/\tau }\lesssim \Big\|\Big(\sum\limits_{j=0}^{%
\infty }|f_{j}|^{\beta }\Big)^{1/\beta }\chi _{k}\Big\|_{{L}^{p,r_{0}}},
\end{equation*}%
see \cite[Proposition 4.2]{ST19}.
\end{proof}

We shall also need the following elementary fact.

\begin{lemma}
\label{Lp-estimate}Let $0<p\leq \infty $ and $f_{k}\in L_{\mathrm{loc}}^{p}(%
\mathbb{R}^{n}),k\in \mathbb{N}_{0}$.\ Then, for any $0<\tau \leq \min
(1,p), $%
\begin{equation*}
\Big\|\sum\limits_{k=0}^{\infty }f_{k}\Big\|_{p}\leq \Big(%
\sum\limits_{k=0}^{\infty }\big\|f_{k}\big\|_{p}^{\tau }\Big)^{\frac{1}{\tau 
}}.
\end{equation*}
\end{lemma}

We finish this section with the following Hardy-type inequality.

\begin{lemma}
\label{lem:lq-inequality}Let $0<a<1$ and $0<q\leq \infty $. Let $%
\{\varepsilon _{k}\}_{k\in \mathbb{Z}}$ be a sequence of positive real
numbers, such that 
\begin{equation*}
\big\|\{\varepsilon _{k}\}_{k\in \mathbb{Z}}\big\|_{\ell ^{q}}=I<\infty .
\end{equation*}%
Then the sequences 
\begin{equation*}
\big\{\delta _{k}:\delta _{k}=\sum_{j=-\infty }^{k}a^{k-j}\varepsilon _{j}%
\big\}_{k\in \mathbb{Z}}\quad \text{and}\quad \big\{\eta _{k}:\eta
_{k}=\sum_{j=k}^{\infty }a^{j-k}\varepsilon _{j}\big\}_{k\in \mathbb{Z}}
\end{equation*}%
belong to $\ell ^{q}$, and 
\begin{equation*}
\big\|\{\delta _{k}\}_{k\in \mathbb{Z}}\big\|_{\ell ^{q}}+\big\|\{\eta
_{k}\}_{k\in \mathbb{Z}}\big\|_{\ell ^{q}}\leq c\,I,
\end{equation*}%
with $c>0$ only depending on $a$ and $q$.
\end{lemma}

\subsection{Maximal inequalities}

Various important results have been proved in Herz\ space $\dot{K}%
_{p}^{\alpha ,q}$ under some assumptions on $\alpha ,p$ and $q$. The
conditions $-\frac{n}{p}<\alpha <n(1-\frac{1}{p}),1<p<\infty $ and $0<q\leq
\infty $ is crucial in the study of the boundedness of classical operators
in $\dot{K}_{p}^{\alpha ,q}$ spaces. This fact was first realized by Li and
Yang \cite{LiYang96} with the proof of the boundedness of the maximal
function were the vector valued extension is given in \cite{TangYang2000}.
The aim is to extend some maximal inequalities\ to Lorentz-Herz spaces. Let
us recall the vector-valued maximal inequality in Lorentz spaces, \cite[%
Lemma 5.1]{ST19}.\ As usual, we put%
\begin{equation*}
\mathcal{M}(f)(x)=\sup_{B}\frac{1}{|B|}\int_{B}\left\vert f(y)\right\vert
dy,\quad f\in L_{\mathrm{loc}}^{1}(\mathbb{R}^{n}),
\end{equation*}%
where the supremum\ is taken over all balls of $\mathbb{R}^{n}$ and $x\in B$%
. Also we set 
\begin{equation*}
\mathcal{M}_{\sigma }(f)=\left( \mathcal{M}(\left\vert f\right\vert ^{\sigma
})\right) ^{\frac{1}{\sigma }},\quad 0<\sigma <\infty .
\end{equation*}

\begin{theorem}
\label{Maximal-Inq3-lorentz}Let $1<p<\infty ,0<r\leq \infty $ and $1<\beta
\leq \infty $. If $\{f_{k}\}_{k\in \mathbb{N}_{0}}$ is a sequence of locally
integrable functions on $\mathbb{R}^{n}$, then 
\begin{equation*}
\Big\|\Big(\sum_{k=0}^{\infty }(\mathcal{M}(f_{k}))^{\beta }\Big)^{1/\beta }%
\Big\|_{L^{{p,r}}}\lesssim \Big\|\Big(\sum_{k=0}^{\infty }|f_{k}|^{\beta }%
\Big)^{1/\beta }\Big\|_{L^{{p,r}}},
\end{equation*}
\end{theorem}

The extension of Fefferman-Stein vector-valued maximal inequality\ to
Lorentz-Herz spaces relies on Theorem \ref{Maximal-Inq3-lorentz}.

\begin{lemma}
\label{Maximal-Inq copy(2)-lorentz}Let $1<p<\infty ,1<\beta \leq \infty $\
and $0<r,q\leq \infty $. If $\{f_{k}\}_{k\in \mathbb{N}_{0}}$ is a sequence
of locally integrable functions on $\mathbb{R}^{n}$\ and $-\frac{n}{p}%
<\alpha <n(1-\frac{1}{p})$, then 
\begin{equation*}
\Big\|\Big(\sum_{k=0}^{\infty }(\mathcal{M}(f_{k}))^{\beta }\Big)^{1/\beta }%
\Big\|_{\dot{K}_{p,r}^{\alpha ,q}}\lesssim \Big\|\Big(\sum_{k=0}^{\infty
}|f_{k}|^{\beta }\Big)^{1/\beta }\Big\|_{\dot{K}_{p,r}^{\alpha ,q}},
\end{equation*}%
with the usual modification if $\beta =\infty .$
\end{lemma}

\begin{proof}
The proof follows easily by the same way as that the proof of vector-valued
maximal inequality in Herz spaces; see \cite{TangYang2000}, but now one has
to use the H\"{o}lder's inequality for Lorentz spaces, and Theorem \ref%
{Maximal-Inq3-lorentz}. The proof is complete.
\end{proof}

From Lemma \ref{Maximal-Inq copy(2)-lorentz} we immediately obtain the
following statement.

\begin{lemma}
\label{Maximal-Inq copy(1)-lorentz}Let $1<p<\infty \ $and $0<r,q\leq \infty $%
. . Let $f\in \dot{K}_{p,r}^{\alpha ,q}$ and $-\frac{n}{p}<\alpha <n(1-\frac{%
1}{p})$. Then 
\begin{equation*}
\big\|\mathcal{M}(f)\big\|_{\dot{K}_{p,r}^{\alpha ,q}}\lesssim \big\|f\big\|%
_{\dot{K}_{p,r}^{\alpha ,q}}
\end{equation*}%
holds.
\end{lemma}

\begin{remark}
We consider sublinear operators satisfying the size condition 
\begin{equation}
|Tf(x)|\lesssim \int_{{\mathbb{R}^{n}}}\frac{|f(y)|}{|x-y|^{n}}\,dy,\quad
x\notin \mathrm{supp}\,f,  \label{size1}
\end{equation}%
for integrable and compactly supported functions $f$. Condition \eqref{size1}
was first considered in \cite{SW94} and it is satisfied by several classical
operators in Harmonic Analysis, such as Calder\'{o}n-Zygmund operators, the
Carleson maximal operator and the Hardy-Littlewood maximal operator (see 
\cite{LuYang95}, \cite{SW94}). The results of this part can be extended to
sublinear operators satisfying the size condition \eqref{size1}; see \cite%
{LuYang95}. More precisely, we have the following statement.
\end{remark}

\begin{theorem}
\label{est-maximal copy(1)-lorentz}Let $1<p<\infty ,1<\beta \leq \infty $\
and $0<r,q\leq \infty $. Let $\{f_{k}\}_{k\in \mathbb{N}_{0}}$ be a sequence
of integrable and compactly\ supported functions\ on $\mathbb{R}^{n}$\ and $-%
\frac{n}{p}<\alpha <n(1-\frac{1}{p})$. Suppose a sublinear operator $T$
satisfies the size condition \eqref{size1}. Then, if $T$ is bounded on $%
L^{p,r}(\ell ^{\beta })$, that means%
\begin{equation*}
\Big\|\Big(\sum_{k=0}^{\infty }|Tf_{k}|^{\beta }\Big)^{1/\beta }\Big\|_{L^{{%
p,r}}}\lesssim \Big\|\Big(\sum_{k=0}^{\infty }|f_{k}|^{\beta }\Big)^{1/\beta
}\Big\|_{L^{{p,r}}},
\end{equation*}%
then we have 
\begin{equation}
\Big\|\Big(\sum_{k=0}^{\infty }|Tf_{k}|^{\beta }\Big)^{1/\beta }\Big\|_{\dot{%
K}_{p,r}^{\alpha ,q}}\lesssim \Big\|\Big(\sum_{k=0}^{\infty }|f_{k}|^{\beta }%
\Big)^{1/\beta }\Big\|_{\dot{K}_{p,r}^{\alpha ,q}}.  \label{new-est}
\end{equation}%
In particular, if $f$ is integrable and compactly supported function on $%
\mathbb{R}^{n}$ and $T$ satisfies the size condition \eqref{size1} which
bounded on $L^{p,r}$, then we have 
\begin{equation}
\big\|Tf\big\|_{\dot{K}_{p,r}^{\alpha ,q}}\lesssim \big\|f\big\|_{\dot{K}%
_{p,r}^{\alpha ,q}}.  \label{new-est1}
\end{equation}
\end{theorem}

\begin{remark}
\label{est-maximal copy(2)-lorentz}Let $1<p<\infty ,1<\beta \leq \infty
,0<r,q\leq \infty $\ and\ $-\frac{n}{p}<\alpha <n(1-\frac{1}{p})$. Theorem %
\ref{est-maximal copy(1)-lorentz} can be extended to the following way.
Suppose a sublinear operator $T$ satisfies the size conditions 
\begin{equation*}
|Tf(x)|\leq \frac{C}{|x|^{n}}\big\|f\big\|_{1},\quad \mathrm{supp}\,f\subset
R_{k},\quad |x|\geq 2^{k+1},k\in \mathbb{Z}
\end{equation*}%
and 
\begin{equation*}
|Tf(x)|\leq C2^{-kn}\big\|f\big\|_{1},\quad \mathrm{supp}\,f\subset
R_{k},\quad |x|\leq 2^{k-2},k\in \mathbb{Z}.
\end{equation*}%
Then, if $T$ is bounded on $L^{p,r}(\ell ^{\beta })$, then we have\ %
\eqref{new-est}. In particular, if $T$ is bounded on $L^{p,r}$, then we have %
\eqref{new-est1}.
\end{remark}

\begin{remark}
Let $0<p<\infty $ and $0<r,\beta \leq \infty $. We recall that the space $%
L^{p,r}(\ell ^{\beta })$ is defined to be the set of all sequences $%
\{f_{k}\}_{k\in \mathbb{N}_{0}}$ of functions such that%
\begin{equation*}
\big\|\{f_{k}\}_{k\in \mathbb{N}_{0}}\big\|_{L^{p,r}(\ell ^{\beta })}=\Big\|%
\Big(\sum_{k=0}^{\infty }|f_{k}|^{\beta }\Big)^{1/\beta }\Big\|%
_{L^{p,r}}<\infty
\end{equation*}%
with the usual modifications if $q=\infty $
\end{remark}

In what follows we use the following simple lemma.

\begin{lemma}
\label{est-maximal}Let $x\in \mathbb{R}^{n},N>0,m>n$\ and $\omega \in 
\mathcal{S(}\mathbb{R}^{n})$. Then there exists a positive constant\ $c>0$
independent of $N$ and $x$ such that for all $f\in L_{\mathrm{loc}%
}^{1}\left( \mathbb{R}^{n}\right) ,$ 
\begin{equation*}
|\omega _{N}\ast f\left( x\right) |\leq c\mathcal{M}(f)(x),
\end{equation*}%
where $\omega _{N}=N^{n}\omega (N\cdot )$.
\end{lemma}

\subsection{Plancherel-Polya-Nikolskij inequality}

The classical Plancherel-Polya-Nikolskij inequality (cf. \cite[1.3.2/5, Rem.
1.4.1/4]{T83}), says that $\big\|f\big\|_{q}$ can be estimated by%
\begin{equation*}
c\ R^{n\left( 1/p-1/q\right) }\big\|f\big\|_{p}
\end{equation*}%
for any $0<p\leq q\leq \infty $, $R>0$ and any $f\in L^{p}\cap \mathcal{S}%
^{\prime }(\mathbb{R}^{n})$\ with supp\ $\mathcal{F}f\subset \{\xi \in 
\mathbb{R}^{n}:|\xi |\leq R\}$. The constant $c>0$\ is independent of $R$.
This inequality plays an important role in theory of function spaces and
PDE's. Our aim is to extend this result to Lorentz-Herz spaces. Let us start
with the following lemma.

The following lemma is the $\dot{K}_{p}^{\alpha ,q}$-version of the
Plancherel-Polya-Nikolskij inequality. For the proof; see \cite{Drihem1.13}.

\begin{lemma}
\label{Bernstein-Herz-ine1}\textit{Let }$\alpha _{1},\alpha _{2}\in \mathbb{R%
}\mathit{\ }$\textit{and} $0<s,p,q,r\leq \infty $. \textit{We suppose that }$%
\alpha _{1}+\frac{n}{s}>0,0<p\leq s\leq \infty $ and $\alpha _{2}\geq \alpha
_{1}$. \textit{Then there exists a positive constant }$c>0$\textit{\
independent of }$R$\textit{\ such that for all }$f\in \dot{K}_{p}^{\alpha
_{2},\theta }\cap \mathcal{S}^{\prime }(\mathbb{R}^{n})$\textit{\ with }$%
\mathrm{\ supp}$\textit{\ }$\mathcal{F}f\subset \{\xi \in \mathbb{R}%
^{n}:|\xi |\leq R\}$\textit{, we have} 
\begin{equation*}
\big\|f\big\|_{\dot{K}_{s}^{\alpha _{1},r}}\leq c\text{ }R^{\frac{n}{p}-%
\frac{n}{s}+\alpha _{2}-\alpha _{1}}\big\|f\big\|_{\dot{K}_{p}^{\alpha
_{2},\theta }},
\end{equation*}%
where 
\begin{equation*}
\theta =\left\{ 
\begin{array}{ccc}
r, & \text{if} & \alpha _{2}=\alpha _{1}, \\ 
q, & \text{if} & \alpha _{2}>\alpha _{1}.%
\end{array}%
\right.
\end{equation*}
\end{lemma}

\begin{remark}
Lemma \ref{Bernstein-Herz-ine1}\ extends and improves classical\
Plancherel-Polya-Nikolskij inequality by taking $\alpha _{1}=\alpha _{2}=0,$ 
$r=s$ and by using the embedding$\ \ell ^{p}\hookrightarrow \ell ^{s}$.
\end{remark}

In the previous lemma we have not treated the case $s\leq p$. The next lemma
gives a positive answer; see also \cite{Drihem1.13}.

\begin{lemma}
\label{Bernstein-Herz-ine2}\textit{Let }$\alpha _{1},\alpha _{2}\in \mathbb{R%
}\mathit{\ }$\textit{and} $0<s,p,q,r\leq \infty $. \textit{We suppose that }$%
\alpha _{1}+\frac{n}{s}>0,0<s\leq p\leq \infty $ and $\alpha _{2}>\alpha
_{1}+\frac{n}{s}-\frac{n}{p}$. \textit{Then there exists a positive constant 
}$c$\textit{\ independent of }$R$\textit{\ such that for all }$f\in \dot{K}%
_{p}^{\alpha _{2},q}\cap \mathcal{S}^{\prime }(\mathbb{R}^{n})$\textit{\
with }$\mathrm{supp}$\textit{\ }$\mathcal{F}f\subset \{\xi \in \mathbb{R}%
^{n}:|\xi |\leq R\}$\textit{, we have} 
\begin{equation*}
\big\|f\big\|_{\dot{K}_{s}^{\alpha _{1},r}}\leq c\text{ }R^{\frac{n}{p}-%
\frac{n}{s}+\alpha _{2}-\alpha _{1}}\big\|f\big\|_{\dot{K}_{p}^{\alpha
_{2},q}}.
\end{equation*}
\end{lemma}

The following lemma plays a crucial role in our proofs.

\begin{lemma}
\label{r-trick}Let $r,R,N>0$, $m>n$ and $\theta ,\omega \in \mathcal{S}%
\left( \mathbb{R}^{n}\right) $ with $\mathrm{supp}$ $\mathcal{F}\omega
\subset \{\xi \in \mathbb{R}^{n}:|\xi |\leq 2\}$. Then there exists $%
c=c(r,m,n)>0$ such that for all $g\in \mathcal{S}^{\prime }(\mathbb{R}^{n})$%
, we have 
\begin{equation}
\left\vert \theta _{R}\ast \omega _{N}\ast g\left( x\right) \right\vert \leq
c\text{ }\max \Big(1,\Big(\frac{N}{R}\Big)^{m}\Big)(\eta _{N,m}\ast
\left\vert \omega _{N}\ast g\right\vert ^{r}(x))^{1/r},\quad x\in \mathbb{R}%
^{n},  \label{r-trick-est}
\end{equation}%
where $\theta _{R}=R^{n}\theta (R\cdot )$, $\omega _{N}=N^{n}\omega (N\cdot
) $ and $\eta _{N,m}=N^{n}(1+N\left\vert \cdot \right\vert )^{-m}$.
\end{lemma}

\begin{lemma}
\label{Key-est1-lorentz}\textit{Let }$\alpha \in \mathbb{R},0<p<\infty
,0<r,q\leq \infty $ \textit{and }$R\geq H>0$\textit{. Then there exists a
constant }$c>0$ \textit{independent of }$R$ and $H$\textit{\ such that for
all }$f\in \dot{K}_{p,r}^{\alpha ,q}\cap \mathcal{S}^{\prime }(\mathbb{R}%
^{n})$\textit{\ with }$\mathrm{supp}$\textit{\ }$\mathcal{F}f\subset \{\xi
\in \mathbb{R}^{n}:|\xi |\leq R\}$, we have 
\begin{equation*}
\sup_{x\in B(0,\frac{1}{H})}|f(x)|\leq c\ \big(\frac{R}{H}\big)^{\frac{n}{d}%
}H^{\frac{n}{p}+\alpha }\big\|f\big\|_{\dot{K}_{p,r}^{\alpha ,q}}
\end{equation*}%
for any $0<d<\min \big(p,r,\frac{n}{\frac{n}{p}+\alpha }\big).$
\end{lemma}

\begin{proof}
The proof follows by the same arguments as in \cite{Drihem1.13}.
\end{proof}

The following lemma is the $\dot{K}_{p,r}^{\alpha ,q}$-version of
Plancherel-Polya-Nikolskij inequality.

\begin{lemma}
\label{Bernstein-Herz-ine1-lorentz}\textit{Let }$\alpha _{1},\alpha _{2}\in 
\mathbb{R},0<p\leq s<\infty \mathit{\ }$\textit{and} $0<q,r,r_{1},r_{2}\leq
\infty $. \textit{We suppose that }$\alpha _{1}+\frac{n}{s}>0$ and $\alpha
_{2}\geq \alpha _{1}$. \textit{Then there exists a positive constant }$c>0$%
\textit{\ independent of }$R$\textit{\ such that for all }$f\in \dot{K}%
_{p,r_{2}}^{\alpha _{2},\theta }\cap \mathcal{S}^{\prime }(\mathbb{R}^{n})$%
\textit{\ with }$\mathrm{\ supp}$\textit{\ }$\mathcal{F}f\subset \{\xi \in 
\mathbb{R}^{n}:|\xi |\leq R\}$\textit{, we have} 
\begin{equation*}
\big\|f\big\|_{\dot{K}_{s,r_{1}}^{\alpha _{1},r}}\leq c\text{ }R^{\frac{n}{p}%
-\frac{n}{s}+\alpha _{2}-\alpha _{1}}\big\|f\big\|_{\dot{K}%
_{p,r_{2}}^{\alpha _{2},\theta }},
\end{equation*}%
where 
\begin{equation*}
\theta =\left\{ 
\begin{array}{ccc}
r, & \text{if} & \alpha _{2}=\alpha _{1}, \\ 
q, & \text{if} & \alpha _{2}>\alpha _{1}.%
\end{array}%
\right.
\end{equation*}
\end{lemma}

\begin{proof}
The proof is based on ideas of \cite{Drihem1.13}. By the embedding $%
L^{s,r_{2}}\hookrightarrow L^{s,r_{1}}$, when $0<r_{2}<r_{1}<\infty $, we
can assume only that $0<r_{1}\leq r_{2}<\infty $. We choose $N$ such that 
\begin{equation}
N>\max \Big(\frac{n}{s},\frac{n}{d},\frac{n}{s}-\alpha _{2}+\alpha _{1}+%
\frac{n}{d},\frac{n}{d}-\alpha _{2}\Big).  \label{aux6}
\end{equation}%
Write%
\begin{equation}
\sum\limits_{k=-\infty }^{\infty }2^{k\alpha _{1}r}\big\|f\chi _{k}\big\|%
_{L^{s,r_{1}}}^{r}=I_{R}+II_{R},  \label{aux5}
\end{equation}%
with 
\begin{equation*}
I_{R}=\sum\limits_{k\in \mathbb{Z},2^{k}\leq \frac{1}{R}}2^{k\alpha _{1}r}%
\big\|f\chi _{k}\big\|_{L^{s,r_{1}}}^{r},\quad II_{R}=\sum\limits_{k\in 
\mathbb{Z},2^{k}>\frac{1}{R}}2^{k\alpha _{1}r}\big\|f\chi _{k}\big\|%
_{L^{s,r_{1}}}^{r}.
\end{equation*}%
We will estimate each term separately.

\textit{Step 1. Estimate of }$I_{R}$. Lemma \ref{Key-est1-lorentz}\ and %
\eqref{est-function1} give for any $R>0$ 
\begin{equation*}
I_{R}\leq \sup_{x\in B(0,2/R)}|f(x)|^{r}\sum\limits_{k\in \mathbb{Z}%
,2^{k}\leq \frac{1}{R}}2^{k(\alpha _{1}+\frac{n}{s})r}\leq c\text{ }R^{(%
\frac{n}{p}-\frac{n}{s}+\alpha _{2}-\alpha _{1})r}\big\|f\big\|_{\dot{K}%
_{p,r_{2}}^{\alpha _{2},q}}^{r},
\end{equation*}%
because of $\alpha _{1}+\frac{n}{s}>0$ and $2^{k-1}R<1$.

\textit{Step 2. Estimate of }$II_{R}$. We set 
\begin{equation*}
\widetilde{C}_{k}=\left\{ x\in \mathbb{R}^{n}:\,2^{k-2}\leq |x|\leq
2^{k+2}\right\} ,\quad k\in \mathbb{Z}.
\end{equation*}%
Let $0<d<\min \big(p,r_{2},\frac{n}{\frac{n}{p}+\alpha _{2}}\big)$. By Lemma %
\ref{r-trick} and H\"{o}lder's inequality, we obtain 
\begin{align*}
\left\vert f(x)\right\vert \leq & c\text{ }\Big(\int_{\mathbb{R}%
^{n}}\left\vert f(y)\right\vert ^{d}\eta _{R,dN}(x-y)dy\Big)^{\frac{1}{d}} \\
\lesssim & V_{R,k}^{1}(x)+V_{R,k}^{2}(x)+V_{R,k}^{3}(x)
\end{align*}%
for any $R>0,N>\frac{n}{d}$ and any $x\in C_{k}$, where the implicit
constant is independent of $x,k$ and $R$, and 
\begin{equation*}
V_{R,k}^{1}(x)=\Big(\int_{B(0,2^{k-2})}\left\vert f(y)\right\vert ^{p}\eta
_{R,pN}(x-y)dy\Big)^{\frac{1}{p}},
\end{equation*}%
\begin{equation*}
V_{R,k}^{2}(x)=\Big(\int_{\widetilde{C}_{k}}\left\vert f(y)\right\vert
^{d}\eta _{R,dN}(x-y)dy\Big)^{\frac{1}{d}}
\end{equation*}%
and 
\begin{equation*}
V_{R,k}^{3}(x)=\Big(\int_{\mathbb{R}^{n}\backslash B(0,2^{k+2})}\left\vert
f(y)\right\vert ^{p}\eta _{R,pN}(x-y)dy\Big)^{\frac{1}{p}}.
\end{equation*}%
\textit{Substep 2.1. Estimate of }$V_{R,k}^{1}$. It is easy to verify that
if $x\in R_{k}$ and $y\in B(0,2^{k-2})$, then $\left\vert x-y\right\vert
>2^{k-2}$. This estimate and Lemma \ref{Key-est1-lorentz}, yield for any $%
x\in R_{k}$ and any $2^{k}R>1$ 
\begin{align}
V_{R,k}^{1}(x)\leq & c\text{ }\sup_{y\in B(0,2^{k-2})}\left\vert
f(y)\right\vert \Big(\int_{2^{k-2}<\left\vert z\right\vert <2^{k+1}}\eta
_{R,pN}(z)dz\Big)^{\frac{1}{p}}  \notag \\
\leq & c\text{ }R^{(\frac{n}{p}-N)}\left( 2^{k}R\right) ^{\frac{n}{d}%
}2^{-\left( \alpha _{2}+N\right) k}\big\|f\big\|_{\dot{K}_{p,r_{2}}^{\alpha
_{2},q}},  \label{est-V2-lorentz}
\end{align}%
where the positive constant $c$ is independent of $x,R,k$ and $f$. From %
\eqref{est-V2-lorentz}, \eqref{aux6} and \eqref{est-function1},\ we get 
\begin{align*}
& \sum\limits_{k\in \mathbb{Z},2^{k}>\frac{2}{R}}2^{k\alpha _{1}r}\big\|%
V_{R,k}^{1}\chi _{k}\big\|_{L^{s,r_{1}}}^{r} \\
\leq & c\text{ }R^{(\frac{n}{p}-N+\frac{n}{d})r}\big\|f\big\|_{\dot{K}%
_{p,r_{2}}^{\alpha _{2},q}}^{r}\sum\limits_{k\in \mathbb{Z},2^{k}>\frac{2}{R}%
}2^{k(\frac{n}{s}+\frac{n}{d}+\alpha _{1}-\alpha _{2}-N)r} \\
\leq & c\text{ }R^{(\frac{n}{p}-\frac{n}{s}+\alpha _{2}-\alpha _{1})r}\big\|f%
\big\|_{\dot{K}_{p,r_{2}}^{\alpha _{2},q}}^{r}.
\end{align*}%
\textit{Substep 2.2. Estimate of }$V_{R,k}^{2}$. Let $v_{1}$ and $v_{2}$ be
two positive real numbers such that $\frac{d}{s}+1=\frac{d}{p}+\frac{1}{v_{1}%
}$ and $\frac{d}{r_{1}}=\frac{d}{r_{2}}+\frac{1}{v_{2}}$. Since $N>\frac{n}{%
v_{1}}$, applying Proposition \ref{Holder's-convolution-lorentz}/(ii),(vi)
and Lemma\ \ref{est-eta}, we obtain 
\begin{align*}
\big\|V_{R,k}^{2}\chi _{k}\big\|_{L^{s,r_{1}}}\lesssim & \big\|\eta
_{R,dN}\ast (\left\vert f\right\vert ^{d}\chi _{\widetilde{C}_{k}})\big\|%
_{L^{\frac{s}{d},\frac{r_{1}}{d}}}^{\frac{1}{d}} \\
\lesssim & \big\|\eta _{R,dN}\big\|_{L^{v_{1},v_{2}}}^{\frac{1}{d}}\big\|%
\left\vert f\right\vert ^{d}\chi _{\widetilde{C}_{k}}\big\|_{L^{\frac{p}{d},%
\frac{r_{2}}{d}}}^{\frac{1}{d}} \\
\lesssim & R^{\frac{n}{p}-\frac{n}{s}}\big\|f\chi _{\widetilde{C}_{k}}\big\|%
_{L^{p,r_{2}}}.
\end{align*}%
This leads to 
\begin{align*}
& \Big(\sum\limits_{k\in \mathbb{Z},2^{k}>\frac{1}{R}}2^{k\alpha _{1}r}\big\|%
V_{R,k}^{2}\chi _{k}\big\|_{L^{s,r_{1}}}^{r}\Big)^{\frac{1}{r}} \\
\lesssim & \text{ }R^{\frac{n}{p}-\frac{n}{s}}\Big(\sum\limits_{k\in \mathbb{%
Z},2^{k}>\frac{1}{R}}2^{k\left( \alpha _{1}-\alpha _{2}\right) r}2^{k\alpha
_{2}r}\big\|f\chi _{\widetilde{C}_{k}}\big\|_{L^{p,r_{2}}}^{r}\Big)^{\frac{1%
}{r}} \\
\lesssim & \text{ }R^{\frac{n}{p}-\frac{n}{s}+\alpha _{2}-\alpha
_{1}}\sup_{k\in \mathbb{Z}}\Big(2^{k\alpha _{2}}\big\|f\chi _{\widetilde{C}%
_{k}}\big\|_{L^{p,r_{2}}}\Big)\Big(\sum\limits_{k\in \mathbb{Z},2^{k}>\frac{1%
}{R}}\left( 2^{k}R\right) ^{\left( \alpha _{1}-\alpha _{2}\right) r}\Big)^{%
\frac{1}{r}} \\
\lesssim & \text{ }R^{\frac{n}{p}-\frac{n}{s}+\alpha _{2}-\alpha _{1}}\big\|f%
\big\|_{\dot{K}_{p,r_{2}}^{\alpha _{2},q}},
\end{align*}%
if $\alpha _{2}>\alpha _{1}$, where the implicit constant is independent of $%
R$. The case $\alpha _{2}=\alpha _{1}$ can be easily solved.

\textit{Substep 2.3. Estimate of }$V_{R,k}^{3}$. Let $x\in C_{k}$ and $%
\varrho =\min (1,p)$. We see that $(V_{R,k}^{3}(x))^{\varrho }$ can be
estimated from above by 
\begin{equation*}
\sum\limits_{i=0}^{\infty }\Big(\int_{C_{k+i+3}}\left\vert f(y)\right\vert
^{p}\eta _{R,pN}(x-y)dy\Big)^{\frac{\varrho }{p}}.
\end{equation*}%
Since $\left\vert x-y\right\vert >3\cdot 2^{k+i}$ for any $x\in C_{k}$ and
any $y\in C_{k+i+3}$, the right-hand side of the last term\ is bounded by 
\begin{align*}
& c\text{ }R^{\varrho (\frac{n}{p}-N)}\sum\limits_{i=0}^{\infty
}2^{-(k+i)\varrho N}\big\|f\chi _{C_{k+i+3}}\big\|_{p}^{\varrho } \\
=& c\text{ }R^{\varrho (\frac{n}{p}-N)}\sum\limits_{j=k+3}^{\infty
}2^{-j\varrho N}\big\|f\chi _{C_{j}}\big\|_{p}^{\varrho } \\
\lesssim & \text{ }R^{\varrho (\frac{n}{p}-N)}\sum\limits_{j=k+3}^{\infty
}2^{j\varrho (\frac{n}{p}-N)}\sup_{x\in B(0,2^{j})}\left\vert
f(x)\right\vert ^{\varrho } \\
\lesssim & R^{\varrho (\frac{n}{p}-N+\frac{n}{d})}\sum\limits_{j=k+3}^{%
\infty }2^{j\varrho (\frac{n}{d}-N-\alpha _{2})}\big\|f\big\|_{\dot{K}%
_{p,r_{2}}^{\alpha _{2},q}}^{\varrho } \\
\lesssim & \text{ }R^{\varrho (\frac{n}{p}-N+\frac{n}{d})}2^{k\varrho (\frac{%
n}{d}-N-\alpha _{2})}\big\|f\big\|_{\dot{K}_{p,r_{2}}^{\alpha
_{2},q}}^{\varrho },
\end{align*}%
where we have used Lemma \ref{Key-est1-lorentz}, since $2^{j}>2^{k}>\frac{1}{%
R}$, and \eqref{aux6}. Consequently 
\begin{align*}
& \sum\limits_{k\in \mathbb{Z},2^{k}>\frac{1}{R}}2^{k\alpha _{1}r}\big\|%
V_{R,k}^{3}\chi _{k}\big\|_{L^{s,r_{1}}}^{r} \\
\lesssim &\text{ }R^{(\frac{n}{p}-N+\frac{n}{d})r}\big\|f\big\|_{\dot{K}%
_{p,r_{2}}^{\alpha _{2},q}}^{r}\sum\limits_{k\in \mathbb{Z},2^{k}>\frac{1}{R}%
}2^{k(\frac{n}{s}-\alpha _{2}+\alpha _{1}-N+\frac{n}{d})r} \\
\lesssim & \text{ }R^{(\frac{n}{p}-\frac{n}{s}+\alpha _{2}-\alpha _{1})r}%
\big\|f\big\|_{\dot{K}_{p,r_{2}}^{\alpha _{2},q}}^{r}\sum\limits_{k\in 
\mathbb{Z},2^{k}>\frac{1}{R}}\left( 2^{k}R\right) ^{(\frac{n}{s}-\alpha
_{2}+\alpha _{1}-N+\frac{n}{d})r} \\
\lesssim &\text{ }R^{(\frac{n}{p}-\frac{n}{s}+\alpha _{2}-\alpha _{1})r}%
\big\|f\big\|_{\dot{K}_{p,r_{2}}^{\alpha _{2},q}}^{r},
\end{align*}%
where we have used again \eqref{est-function1}\ and \eqref{aux6}. The proof
is complete.
\end{proof}

\begin{remark}
Lemma \ref{Bernstein-Herz-ine1-lorentz}\ improves Plancherel-Polya-Nikolskij
inequality in Herz spaces; see Lemma \ref{Bernstein-Herz-ine1}, where we
choose $r_{1}=s,p\leq r_{2}$ and we use the embedding$\ L^{p}\hookrightarrow
L^{p,r_{2}}$
\end{remark}

In the previous lemma we have not treated the case $s\leq p$. The next lemma
gives a positive answer.

\begin{lemma}
\label{Bernstein-Herz-ine2-lorentz}\textit{Let }$\alpha _{1},\alpha _{2}\in 
\mathbb{R}\ $\textit{and} $0<q,r,r_{1},r_{2}\leq \infty $. \textit{We
suppose that }$\alpha _{1}+\frac{n}{s}>0,0<s<p<\infty $\ and $\alpha
_{2}>\alpha _{1}+\frac{n}{s}-\frac{n}{p}$. \textit{Then there exists a
positive constant }$c$\textit{\ independent of }$R$\textit{\ such that for
all }$f\in \dot{K}_{p,r_{2}}^{\alpha _{2},q}\cap \mathcal{S}^{\prime }(%
\mathbb{R}^{n})$\textit{\ with }$\mathrm{\ supp}$\textit{\ }$\mathcal{F}%
f\subset \{\xi \in \mathbb{R}^{n}:|\xi |\leq R\}$\textit{, we have} 
\begin{equation*}
\big\|f\big\|_{\dot{K}_{s,r_{1}}^{\alpha _{1},r}}\leq c\text{ }R^{\frac{n}{p}%
-\frac{n}{s}+\alpha _{2}-\alpha _{1}}\big\|f\big\|_{\dot{K}%
_{p,r_{2}}^{\alpha _{2},q}}.
\end{equation*}
\end{lemma}

\begin{proof}
We employ the notations $II_{R}$ and $I_{R}$ from \eqref{aux5}. The estimate
of $I_{R}$ follows easily from the previous lemma. We only need to estimate
the part $II_{R}$. By the embedding $L^{s,r_{2}}\hookrightarrow L^{s,r_{1}}$%
, when $0<r_{2}<r_{1}<\infty $, we can assume only that $0<r_{1}\leq
r_{2}<\infty $. Let%
\begin{equation*}
\frac{1}{s}=\frac{1}{p}+\frac{1}{v_{1}}\text{\quad and\quad }\frac{1}{r_{1}}=%
\frac{1}{r_{2}}+\frac{1}{v_{2}}.
\end{equation*}%
H\"{o}lder's inequality and \eqref{est-function1} give 
\begin{equation}
\big\|f\chi _{k}\big\|_{L^{s,r_{1}}}\lesssim \big\|\chi _{k}\big\|%
_{L^{v_{1},v_{2}}}\big\|f\chi _{k}\big\|_{L^{p,r_{2}}}\lesssim 2^{kn(\frac{1%
}{s}-\frac{1}{p})}\big\|f\chi _{k}\big\|_{L^{p,r_{2}}},  \label{H-est1}
\end{equation}%
where the implicit constant is independent of $k$. Therefore, 
\begin{align*}
II_{R}\leq & \sum\limits_{k\in \mathbb{Z},2^{k}>\frac{1}{R}}2^{k(\frac{n}{s}-%
\frac{n}{p}-\alpha _{2}+\alpha _{1})r}2^{k\alpha _{2}r}\big\|f\chi _{k}\big\|%
_{L^{p,r_{2}}}^{r} \\
\leq & \sup_{k\in \mathbb{Z}}\big(2^{k\alpha _{2}}\big\|f\chi _{k}\big\|%
_{L^{p,r_{2}}}\big)^{r}\sum\limits_{k\in \mathbb{Z},2^{k}>\frac{1}{R}}2^{k(%
\frac{n}{s}-\frac{n}{p}-\alpha _{2}+\alpha _{1})r} \\
\lesssim & R^{(\frac{n}{p}-\frac{n}{s}+\alpha _{2}-\alpha _{1})r}\big\|f%
\big\|_{\dot{K}_{p,r_{2}}^{\alpha _{2},q}}^{r}\sum\limits_{k\in \mathbb{Z}%
,2^{k}>\frac{1}{R}}\left( 2^{k}R\right) ^{(\frac{n}{s}-\frac{n}{p}-\alpha
_{2}+\alpha _{1})r} \\
\lesssim & R^{(\frac{n}{p}-\frac{n}{s}+\alpha _{2}-\alpha _{1})r}\big\|f%
\big\|_{\dot{K}_{p,r_{2}}^{\alpha _{2},q}}^{r},
\end{align*}%
since $2^{k}R>1$. The proof is complete.
\end{proof}

\begin{remark}
Using the estimate \eqref{H-est1}, we easily obtain that Lemma {\ref%
{Bernstein-Herz-ine2-lorentz} }is true for $\alpha _{2}=\alpha _{1}+\frac{n}{%
s}-\frac{n}{p},r=q$ and any $f\in \dot{K}_{p,r_{2}}^{\alpha _{2},q}$. Also,
Lemma {\ref{Bernstein-Herz-ine2-lorentz} }extends and improves Lemma {\ref%
{Bernstein-Herz-ine2}.}
\end{remark}

\section{Lorentz Herz-type Besov and Triebel-Lizorkin spaces}

In this section, we present the spaces $\dot{K}_{p,r}^{\alpha ,q}B_{\beta
}^{s}$ and $\dot{K}_{p,r}^{\alpha ,q}F_{\beta }^{s}$ on which we work,
establish their $\varphi $-transform characterizations and interpolation
inequalities, lifting property and Fatou property.

\subsection{The $\protect\varphi $-transform of $\dot{K}_{p,r}^{\protect%
\alpha ,q}B_{\protect\beta }^{s}$ and $\dot{K}_{p,r}^{\protect\alpha ,q}F_{%
\protect\beta }^{s}$}

Select a pair of Schwartz functions $\Phi $ and $\varphi $ such that

\begin{equation}
\text{\textrm{supp}}\mathcal{F}\Phi \subset \{\xi \in \mathbb{R}^{n}:|\xi
|\leq 2\}\text{\quad and\quad }|\mathcal{F}\Phi (\xi )|\geq c>0,
\label{Ass1}
\end{equation}%
if $|\xi |\leq \frac{5}{3}$ and 
\begin{equation}
\text{\textrm{supp}}\mathcal{F}\varphi \subset \{\xi \in \mathbb{R}^{n}:%
\frac{1}{2}\leq |\xi |\leq 2\}\text{\quad and\quad }|\mathcal{F}\varphi (\xi
)|\geq c>0,  \label{Ass2}
\end{equation}%
if $\frac{3}{5}\leq |\xi |\leq \frac{5}{3}$, where $c>0$. Throughout the
section we put $\tilde{\varphi}(x)=\overline{\varphi (-x)},x\in \mathbb{R}%
^{n}$.

Now, we define the spaces under consideration.

\begin{definition}
\label{B-F-def-lorentz}Let $\alpha ,s\in \mathbb{R},0<p<\infty ,0<r,q,\beta
\leq \infty ,\Phi $ and $\varphi $\ satisfy $\mathrm{\eqref{Ass1}}$\ and\ $%
\mathrm{\eqref{Ass2}}$, respectively and we put $\varphi _{k}=2^{kn}\varphi
(2^{k}\cdot ),k\in \mathbb{N}$.\newline
$\mathrm{(i)}$ The\ Lorentz Herz-type Besov space $\dot{K}_{p,r}^{\alpha
,q}B_{\beta }^{s}$\ is defined to be the set of all $f\in \mathcal{S}%
^{\prime }(\mathbb{R}^{n})$\ such that 
\begin{equation*}
\big\|f\big\|_{\dot{K}_{p,r}^{\alpha ,q}B_{\beta }^{s}}=\Big(%
\sum\limits_{k=0}^{\infty }2^{ks\beta }\big\|\varphi _{k}\ast f\big\|_{\dot{K%
}_{p,r}^{\alpha ,q}}^{\beta }\Big)^{1/\beta }<\infty ,
\end{equation*}%
where $\varphi _{0}$ is replaced by $\Phi $, with the obvious modification if%
\textit{\ }$\beta =\infty $.\newline
$\mathrm{(ii)}$ Let $0<q<\infty $. The Lorentz Herz-type Triebel-Lizorkin
space $\dot{K}_{p,r}^{\alpha ,q}F_{\beta }^{s}$ is defined to be the set of
all $f\in \mathcal{S}^{\prime }(\mathbb{R}^{n})$\ such that 
\begin{equation*}
\big\|f\big\|_{\dot{K}_{p,r}^{\alpha ,q}F_{\beta }^{s}}=\Big\|\Big(%
\sum\limits_{k=0}^{\infty }2^{ks\beta }|\varphi _{k}\ast f|^{\beta }\Big)%
^{1/\beta }\Big\|_{\dot{K}_{p,r}^{\alpha ,q}}<\infty ,
\end{equation*}%
where $\varphi _{0}$ is replaced by $\Phi $, with the obvious modification
if $\beta =\infty .$
\end{definition}

\begin{remark}
One recognizes immediately that if $\alpha =0$ and $p=q$, then 
\begin{equation}
\dot{K}_{p,p}^{0,p}B_{\beta }^{s}=B_{p,\beta }^{s}\quad \text{and}\quad \dot{%
K}_{p,p}^{0,p}F_{\beta }^{s}=F_{p,\beta }^{s}.  \notag
\end{equation}
\end{remark}

Next, we present the definition of Herz-type Besov and Triebel-Lizorkin
spaces.

\begin{definition}
Let $\alpha ,s\in \mathbb{R},0<p,q,\beta \leq \infty ,\Phi $ and $\varphi $\
satisfy $\mathrm{\eqref{Ass1}}$\ and\ $\mathrm{\eqref{Ass2}}$, respectively
and we put $\varphi _{k}=2^{kn}\varphi (2^{k}\cdot ),k\in \mathbb{N}$.%
\newline
$\mathrm{(i)}$ The\ Herz-type Besov space $\dot{K}_{p}^{\alpha ,q}B_{\beta
}^{s}$\ is defined to be the set of all $f\in \mathcal{S}^{\prime }(\mathbb{R%
}^{n})$\ such that 
\begin{equation*}
\big\|f\big\|_{\dot{K}_{p}^{\alpha ,q}B_{\beta }^{s}}=\Big(%
\sum\limits_{k=0}^{\infty }2^{ks\beta }\big\|\varphi _{k}\ast f\big\|_{\dot{K%
}_{p}^{\alpha ,q}}^{\beta }\Big)^{1/\beta }<\infty ,
\end{equation*}%
where $\varphi _{0}$ is replaced by $\Phi $, with the obvious modification if%
\textit{\ }$\beta =\infty $.\newline
$\mathrm{(ii)}$ Let $0<p,q<\infty $. The Herz-type Triebel-Lizorkin space $%
\dot{K}_{p}^{\alpha ,q}F_{\beta }^{s}$ is defined to be the set of all $f\in 
\mathcal{S}^{\prime }(\mathbb{R}^{n})$\ such that 
\begin{equation*}
\big\|f\big\|_{\dot{K}_{p}^{\alpha ,q}F_{\beta }^{s}}=\Big\|\Big(%
\sum\limits_{k=0}^{\infty }2^{ks\beta }|\varphi _{k}\ast f|^{\beta }\Big)%
^{1/\beta }\Big\|_{\dot{K}_{p}^{\alpha ,q}}<\infty ,
\end{equation*}%
where $\varphi _{0}$ is replaced by $\Phi $, with the obvious modification
if $\beta =\infty .$
\end{definition}

\begin{remark}
$\mathrm{(i)}$ We have%
\begin{equation}
\dot{K}_{p,p}^{\alpha ,q}B_{\beta }^{s}=\dot{K}_{p}^{\alpha ,q}B_{\beta
}^{s}\quad \text{and}\quad \dot{K}_{p,p}^{\alpha ,p}F_{\beta }^{s}=\dot{K}%
_{p}^{\alpha ,q}F_{\beta }^{s},  \notag
\end{equation}%
We refer, in particular, to the papers \cite{Drihem1.13}, \cite{XuYang05}
and \cite{Xu05} for a comprehensive treatment of $\dot{K}_{p}^{\alpha
,q}B_{\beta }^{s}$ and $\dot{K}_{p}^{\alpha ,q}F_{\beta }^{s}$. \newline
$\mathrm{(ii)}$ Notice that the spaces $\dot{K}_{p,r}^{\alpha ,q}B_{\beta
}^{s}$ are also\ considered in \cite{FPV}\ to study the bilinear estimates
and uniqueness of mild solutions for the Navier-Stokes equations.
\end{remark}

Let $\Phi $ and $\varphi $ satisfy, respectively, \eqref{Ass1} and %
\eqref{Ass2}. From, \cite[Section 12]{FJ90}, there exist functions $\Psi \in 
\mathcal{S}(\mathbb{R}^{n})$ satisfying \eqref{Ass1} and $\psi \in \mathcal{S%
}(\mathbb{R}^{n})$ satisfying \eqref{Ass2} such that%
\begin{equation}
\mathcal{F}\tilde{\Phi}(\xi )\mathcal{F}\Psi (\xi )+\sum_{k=1}^{\infty }%
\mathcal{F}\tilde{\varphi}(2^{-k}\xi )\mathcal{F}\psi (2^{-k}\xi )=1,\quad
\xi \in \mathbb{R}^{n}.  \label{Ass3}
\end{equation}%
A basic tool to study the above function spaces is the following Calder\'{o}%
n reproducing formula, see \cite[(12.4)]{FJ90} and \cite[Lemma 2.3]{YSY10}.

\begin{lemma}
\label{DW-lemma1}Let $\Phi ,\Psi \in \mathcal{S}(\mathbb{R}^{n})$ satisfy %
\eqref{Ass1} and $\varphi ,\psi \in \mathcal{S}(\mathbb{R}^{n})$ satisfy %
\eqref{Ass2} such that \eqref{Ass3} holds. Then for all\ $f\in \mathcal{S}%
^{\prime }(\mathbb{R}^{n}),$%
\begin{align}
f=& \tilde{\Phi}\ast \Psi \ast f+\sum_{k=1}^{\infty }\widetilde{\varphi }%
_{k}\ast \psi _{k}\ast f  \notag \\
=& \sum_{m\in \mathbb{Z}^{n}}\tilde{\Phi}\ast f(m)\Psi
_{m}+\sum_{k=1}^{\infty }2^{-k\frac{n}{2}}\sum_{m\in \mathbb{Z}^{n}}%
\widetilde{\varphi }_{k}\ast f(2^{-k}m)\psi _{k,m},  \label{proc2}
\end{align}%
in$\ \mathcal{S}^{\prime }(\mathbb{R}^{n})$, where%
\begin{equation*}
\Psi _{m}=\Psi (\cdot -m)\quad \text{and}\quad \psi _{k,m}=2^{k\frac{n}{2}%
}\psi (2^{k}\cdot -m),\quad m\in \mathbb{Z}^{n},k\in \mathbb{N}.
\end{equation*}
\end{lemma}

Let $\Phi ,\Psi ,\varphi ,\psi \in \mathcal{S}(\mathbb{R}^{n})$ satisfying %
\eqref{Ass1}, \eqref{Ass2} and \eqref{Ass3}. The $\varphi $-transform $%
S_{\varphi }$ is defined by setting 
\begin{equation*}
(S_{\varphi }f)_{0,m}=\langle f,\Phi _{m}\rangle \quad \text{and}\quad
(S_{\varphi }f)_{k,m}=\langle f,\varphi _{k,m}\rangle ,
\end{equation*}%
where 
\begin{equation*}
\Phi _{m}=\Phi (\cdot -m)\quad \text{and}\quad \varphi _{k,m}=2^{k\frac{n}{2}%
}\varphi (2^{k}\cdot -m),\quad m\in \mathbb{Z}^{n},k\in \mathbb{N}.
\end{equation*}%
The inverse $\varphi $-transform $T_{\psi }$ is defined by 
\begin{equation*}
T_{\psi }\lambda =\sum_{m\in \mathbb{Z}^{n}}\lambda _{0,m}\Psi
_{m}+\sum_{k=1}^{\infty }\sum_{m\in \mathbb{Z}^{n}}\lambda _{k,m}\psi _{k,m},
\end{equation*}%
where $\lambda =\{\lambda _{k,m}\}_{k\in \mathbb{N}_{0},m\in \mathbb{Z}%
^{n}}\subset \mathbb{C}$, see \cite[p. 131]{FJ90}.

Now, we introduce the corresponding sequence spaces of $\dot{K}%
_{p,r}^{\alpha ,q}B_{\beta }^{s}$ and $\dot{K}_{p,r}^{\alpha ,q}F_{\beta
}^{s}$.

\begin{definition}
\label{sequence-space-lorentz}Let $\alpha ,s\in \mathbb{R},0<p<\infty
,0<r,q\leq \infty $ and $0<\beta \leq \infty $.\newline
$\mathrm{(i)}$ The\ space $\dot{K}_{p,r}^{\alpha ,q}b_{\beta }^{s}$\ is
defined to be the set of all complex valued sequences $\lambda =\{\lambda
_{k,m}\}_{k\in \mathbb{N}_{0},m\in \mathbb{Z}^{n}}$ such that%
\begin{equation*}
\big\|\lambda \big\|_{\dot{K}_{p,r}^{\alpha ,q}b_{\beta }^{s}}=\Big(%
\sum_{k=0}^{\infty }2^{k(s+\frac{n}{2})\beta }\big\|\sum\limits_{m\in 
\mathbb{Z}^{n}}\lambda _{k,m}\chi _{k,m}\big\|_{\dot{K}_{p,r}^{\alpha
,q}}^{\beta }\Big)^{1/\beta }<\infty .
\end{equation*}%
$\mathrm{(ii)}$ Let $0<q<\infty $. The\ space $\dot{K}_{p,r}^{\alpha
,q}f_{\beta }^{s}$\ is defined to be the set of all complex valued sequences 
$\lambda =\{\lambda _{k,m}\}_{k\in \mathbb{N}_{0},m\in \mathbb{Z}^{n}}$ such
that 
\begin{equation*}
\big\|\lambda \big\|_{\dot{K}_{p,r}^{\alpha ,q}f_{\beta }^{s}}=\Big\|\Big(%
\sum_{k=0}^{\infty }\sum\limits_{m\in \mathbb{Z}^{n}}2^{k(s+\frac{n}{2}%
)\beta }|\lambda _{k,m}|^{\beta }\chi _{k,m}\Big)^{1/\beta }\Big\|_{\dot{K}%
_{p,r}^{\alpha ,q}}<\infty .
\end{equation*}
\end{definition}

For simplicity, in what follows, we use $\dot{K}_{p,r}^{\alpha ,q}A_{\beta
}^{s}$ to denote either $\dot{K}_{p,r}^{\alpha ,q}B_{\beta }^{s}$ or $\dot{K}%
_{p,r}^{\alpha ,q}F_{\beta }^{s}$. The case $q=\infty $ is excluded when $%
\dot{K}_{p,r}^{\alpha ,q}A_{\beta }^{s}$ means $\dot{K}_{p,r}^{\alpha
,q}F_{\beta }^{s}$. In the same way we shall use the abbreviation $\dot{K}%
_{p,r}^{\alpha ,q}a_{\beta }^{s}$ in place of $\dot{K}_{p,r}^{\alpha
,q}b_{\beta }^{s}$ and $\dot{K}_{p,r}^{\alpha ,q}f_{\beta }^{s}$.

Notice that, the spaces $\dot{K}_{p,r}^{\alpha ,q}A_{\beta }^{s}$ are
quasi-normed spaces, it holds%
\begin{equation*}
\big\|f+g\big\|_{\dot{K}_{p,r}^{\alpha ,q}A_{\beta }^{s}}\lesssim \big\|f%
\big\|_{\dot{K}_{p,r}^{\alpha ,q}A_{\beta }^{s}}+\big\|g\big\|_{\dot{K}%
_{p,r}^{\alpha ,q}A_{\beta }^{s}}
\end{equation*}%
for all $f,g\in \dot{K}_{p,r}^{\alpha ,q}A_{\beta }^{s}$, where the implicit
constant is independent of $f$ and $g$.

The following lemma ensure that $\dot{K}_{p,r}^{\alpha ,q}a_{\beta }^{s}$ is
well defined.

\begin{lemma}
\label{Inv-phi-trans-lorentz}Let $s\in \mathbb{R},0<p<\infty ,0<r,q,\beta
\leq \infty $ and $\alpha >-\frac{n}{p}$. Let $\Psi $ and $\psi $\ satisfy,
respectively, $\mathrm{\eqref{Ass1}}$\ and\ $\mathrm{\eqref{Ass2}}$. Then
for all $\lambda \in \dot{K}_{p,r}^{\alpha ,q}a_{\beta }^{s}$ 
\begin{equation*}
T_{\psi }\lambda =\sum_{m\in \mathbb{Z}^{n}}\lambda _{0,m}\Psi
_{m}+\sum_{k=1}^{\infty }\sum_{m\in \mathbb{Z}^{n}}\lambda _{k,m}\psi _{k,m},
\end{equation*}%
converges in $\mathcal{S}^{\prime }(\mathbb{R}^{n})$; moreover, $T_{\psi }:%
\dot{K}_{p,r}^{\alpha ,q}a_{\beta }^{s}\rightarrow \mathcal{S}^{\prime }(%
\mathbb{R}^{n})$ is continuous.
\end{lemma}

\begin{proof}
Since the proof for $\dot{K}_{p,r}^{\alpha ,q}b_{\beta }^{s}$ is similar, we
only consider $\dot{K}_{p,r}^{\alpha ,q}f_{\beta }^{s}$. Let $0<h<\min \big(%
p,r,q,\frac{n}{\alpha +\frac{n}{p}}\big)$, with 
\begin{equation*}
\frac{1}{h}=\frac{1}{p}+\frac{1}{t}=\frac{1}{r}+\frac{1}{d}=\frac{1}{q}+%
\frac{1}{v},\quad t,d,v>0.
\end{equation*}%
Let $\lambda \in \dot{K}_{p,r}^{\alpha ,q}f_{\beta }^{s}$ and $\varphi \in 
\mathcal{S}(\mathbb{R}^{n})$. We set 
\begin{equation*}
I_{1}=\sum_{m\in \mathbb{Z}^{n}}|\lambda _{0,m}||\langle \Psi _{m},\varphi
\rangle |\quad \text{and}\quad I_{2}=\sum_{k=1}^{\infty }\sum_{m\in \mathbb{Z%
}^{n}}|\lambda _{k,m}||\langle \psi _{k,m},\varphi \rangle |.
\end{equation*}%
It suffices to show that both $I_{1}$ and $I_{2}$ are dominated by $c\big\|%
\varphi \big\|_{\mathcal{S}_{M}}\big\|\lambda \big\|_{\dot{K}_{p,r}^{\alpha
,q}f_{\beta }^{s}}$ for some $M\in \mathbb{N}$.

\textit{Estimate of }$I_{1}$. Let $M,L\in \mathbb{N}$ be such that $M>L+n$.
Since\ $\varphi ,\Psi \in \mathcal{S}(\mathbb{R}^{n})$, we obtain%
\begin{align*}
|\langle \Psi _{m},\varphi \rangle |\leq & \int_{\mathbb{R}^{n}}|\Psi
(x-m)||\varphi (x)|dx \\
\leq & \big\|\varphi \big\|_{\mathcal{S}_{M}}\big\|\Psi \big\|_{\mathcal{S}%
_{L}}\int_{\mathbb{R}^{n}}(1+|x-m|)^{-L-n}(1+|x|)^{-M-n}dx \\
\leq & \big\|\varphi \big\|_{\mathcal{S}_{M}}\big\|\Psi \big\|_{\mathcal{S}%
_{L}}(1+|m|)^{-L-n}.
\end{align*}%
The last estimate follow by the inequality%
\begin{equation*}
(1+|x-m|)^{-L-n}\leq (1+|m|)^{-L-n}(1+|x|)^{L+n},\quad x\in \mathbb{R}%
^{n},m\in \mathbb{Z}^{n}.
\end{equation*}%
By H\"{o}lder's inequality, we obtain 
\begin{align}
|\lambda _{0,m}|^{h}& =\frac{1}{|Q_{0,m}|}\sum_{j=-\infty }^{\infty }\big\|%
\lambda _{0,m}\chi _{0,m}\chi _{j}\big\|_{h}^{h}  \notag \\
& \leq c\sum_{j=-\infty }^{\infty }\big\|\lambda _{0,m}\chi _{0,m}\chi _{j}%
\big\|_{L^{p,r}}^{h}\big\|\chi _{0,m}\chi _{j}\big\|_{L^{t,d}}^{h}  \notag \\
& \leq c\sum_{j=-\infty }^{\infty }\big\|\lambda _{0,m}\chi _{0,m}\chi _{j}%
\big\|_{L^{p,r}}^{h}\big\|\chi _{j}\big\|_{L^{t,d}}^{h},  \label{sum1}
\end{align}%
where the\ positive constant $c$ is independent of $m\in \mathbb{Z}^{n}$.
Observe that the sum in $\mathrm{\eqref{sum1}}$ can be rewritten as%
\begin{equation}
\sum_{j\in \mathbb{Z},2^{j-1}\leq \sqrt{n}(\mathbf{1}+|m|)}\big\|\lambda
_{0,m}\chi _{0,m}\chi _{j}\big\|_{L^{p,r}}^{h}\big\|\chi _{j}\big\|%
_{L^{t,d}}^{h}.  \label{sum2}
\end{equation}%
H\"{o}lder's inequality in the Lebesgue sequence spaces gives that $\mathrm{%
\eqref{sum2}}$ is bounded from above by 
\begin{align}
& \Big(\sum_{j\in \mathbb{Z},2^{j-1}\leq \sqrt{n}(\mathbf{1}+|m|)}2^{j\alpha
q}\big\|\lambda _{0,m}\chi _{0,m}\chi _{j}\big\|_{L^{p,r}}^{q}\Big)^{h/q}%
\Big(\sum_{j\in \mathbb{Z},2^{j-1}\leq \sqrt{n}(\mathbf{1}+|m|)}2^{-j\alpha
v}\big\|\chi _{j}\big\|_{L^{t,d}}^{v}\Big)^{h/v}  \notag \\
\leq & c\big\|\lambda \big\|_{\dot{K}_{p,r}^{\alpha ,q}f_{\beta }^{s}}^{h}%
\Big(\sum_{j\in \mathbb{Z},2^{j-1}\leq \sqrt{n}(\mathbf{1}+|m|)}2^{j(\frac{n%
}{t}-\alpha )v}\Big)^{h/v}  \notag \\
\leq & c(1+|m|)^{(\frac{n}{t}-\alpha )h}\big\|\lambda \big\|_{\dot{K}%
_{p,r}^{\alpha ,q}f_{\beta }^{s}}^{h},  \label{sum3}
\end{align}%
since $\frac{n}{t}-\alpha >0$. Inserting \eqref{sum3} in \eqref{sum2}, we
obtain%
\begin{equation*}
|\lambda _{0,m}|\lesssim (1+|m|)^{\frac{n}{t}-\alpha }\big\|\lambda \big\|_{%
\dot{K}_{p,r}^{\alpha ,q}f_{\beta }^{s}},
\end{equation*}%
where the\ implicit constant is independent\ of $m\in \mathbb{Z}^{n}$. If we
choose $L$ large enough, then 
\begin{equation*}
I_{1}\lesssim \big\|\varphi \big\|_{\mathcal{S}_{M}}\big\|\lambda \big\|_{%
\dot{K}_{p,r}^{\alpha ,q}f_{\beta }^{s}}.
\end{equation*}

\textit{Estimate of }$I_{2}$. Let us recall the following estimate; see
Lemma 2.4 in \cite{YSY10}. Since $\psi $ has vanishing moments of any order,
we see that for any $L,M>0$ there exists a positive constant $C=C(M,n)$ such
that for all $k\in \mathbb{N}$ and all $x\in \mathbb{R}^{n},$%
\begin{equation}
|\psi _{k}\ast \varphi (x)|\leq C2^{-kL}\big\|\psi \big\|_{\mathcal{S}_{M+1}}%
\big\|\varphi \big\|_{\mathcal{S}_{M+1}}\big(1+|x|\big)^{-n-L},
\label{convolution}
\end{equation}%
where the positive constant $C$ is independent of $k$ and $x$. We set $%
\breve{\varphi}=\varphi (-\cdot )$. From \eqref{convolution}, we get 
\begin{align*}
\left\vert \langle \psi _{k,m},\varphi \rangle \right\vert & =2^{-k\frac{n}{2%
}}|\psi _{k}\ast \breve{\varphi}(-2^{-k}m)| \\
& \lesssim 2^{-k(\frac{n}{2}+L)}\big\|\psi \big\|_{\mathcal{S}_{M+1}}\big\|%
\varphi \big\|_{\mathcal{S}_{M+1}}\big(1+|2^{-k}m|\big)^{-n-L}.
\end{align*}%
We use the same schema as in the estimate of $I_{1}$ we arrive at the
inequality 
\begin{equation*}
|\lambda _{k,m}|^{h}\leq \frac{c}{|Q_{k,m}|}\sum_{j\in \mathbb{Z}%
,2^{j-1}\leq 2^{-k}\sqrt{n}+2^{-k}|m|}\big\|\lambda _{k,m}\chi _{k,m}\chi
_{j}\big\|_{L^{p,r}}^{h}\big\|\chi _{j}\big\|_{L^{t,d}}^{h},
\end{equation*}%
where the\ positive constant $c$ is independent of $k\in \mathbb{N}$ and $%
m\in \mathbb{Z}^{n}$. Again, by H\"{o}lder's inequality in the Lebesgue
sequence spaces we obtain 
\begin{align*}
|\lambda _{k,m}|^{h}\leq & c2^{k(n-sh-\frac{n}{2}h)}\big\|\lambda \big\|_{%
\dot{K}_{p,r}^{\alpha ,q}f_{\beta }^{s}}^{h}\Big(\sum_{j\in \mathbb{Z}%
,2^{j-1}\leq 2^{-k}\sqrt{n}+2^{-k}|m|}2^{j(\frac{n}{t}-\alpha )v}\Big)^{h/v}
\\
\leq & c2^{k(n-sh-\frac{n}{2}h)}(1+|2^{-k}m|)^{(\frac{n}{t}-\alpha )h}\big\|%
\lambda \big\|_{\dot{K}_{p,r}^{\alpha ,q}f_{\beta }^{s}}^{h},
\end{align*}%
since $\frac{n}{t}-\alpha >0$. Thus,%
\begin{align*}
I_{2}\lesssim & \big\|\varphi \big\|_{\mathcal{S}_{M+1}}\big\|\psi \big\|_{%
\mathcal{S}_{M+1}}\big\|\lambda \big\|_{\dot{K}_{p,r}^{\alpha ,q}f_{\beta
}^{s}}\sum_{k=1}^{\infty }\sum_{m\in \mathbb{Z}^{n}}2^{k(\frac{n}{h}%
-s-n-L)}(1+|2^{-k}m|)^{\frac{n}{t}-\alpha -n-L} \\
\lesssim & \big\|\varphi \big\|_{\mathcal{S}_{M+1}}\big\|\psi \big\|_{%
\mathcal{S}_{M+1}}\big\|\lambda \big\|_{\dot{K}_{p,r}^{\alpha ,q}f_{\beta
}^{s}},
\end{align*}%
if $L$ is sufficiently large. The proof is complete.
\end{proof}

For a sequence $\lambda =\{\lambda _{k,m}\}_{k\in \mathbb{N}_{0},m\in 
\mathbb{Z}^{n}}\subset \mathbb{C},0<\gamma <\infty $ and a fixed $d>0$, set%
\begin{equation*}
\lambda _{k,m,\gamma ,d}^{\ast }=\Big(\sum_{h\in \mathbb{Z}^{n}}\frac{%
|\lambda _{k,h}|^{\gamma }}{(1+2^{k}|2^{-k}h-2^{-k}m|)^{d}}\Big)^{1/\gamma }
\end{equation*}%
and $\lambda _{\gamma ,d}^{\ast }:=\{\lambda _{k,m,\gamma ,d}^{\ast
}\}_{k\in \mathbb{N}_{0},m\in \mathbb{Z}^{n}}\subset \mathbb{C}$ with the
usual modification if $\gamma =\infty $.

\begin{lemma}
\label{lamda-equi-lorentz}Let $s\in \mathbb{R},0<p<\infty ,0<r,q\leq \infty
,0<\beta \leq \infty $ and $\alpha >-\frac{n}{p}$. Let$\ $ 
\begin{equation*}
\gamma =\left\{ 
\begin{array}{ccc}
\min (p,\beta ), & \text{if} & \dot{K}_{p,r}^{\alpha ,q}a_{\beta }^{s}=\dot{K%
}_{p,r}^{\alpha ,q}f_{\beta }^{s} \\ 
p, & \text{if} & \dot{K}_{p,r}^{\alpha ,q}a_{\beta }^{s}=\dot{K}%
_{p,r}^{\alpha ,q}b_{\beta }^{s}%
\end{array}%
\right.
\end{equation*}%
and%
\begin{equation*}
d>\frac{n\gamma }{\min \big(\frac{n}{\alpha +\frac{n}{p}},\gamma \big)}.
\end{equation*}%
Then 
\begin{equation}
\big\|\lambda _{\gamma ,d}^{\ast }\big\|_{\dot{K}_{p,r}^{\alpha ,q}a_{\beta
}^{s}}\approx \big\|\lambda \big\|_{\dot{K}_{p,r}^{\alpha ,q}a_{\beta }^{s}}.
\label{First}
\end{equation}
\end{lemma}

\begin{proof}
By similarity, we only consider $\dot{K}_{p,r}^{\alpha ,q}f_{\beta }^{s}$.
Obviously, 
\begin{equation*}
\big\|\lambda \big\|_{\dot{K}_{p,r}^{\alpha ,q}f_{\beta }^{s}}\leq \big\|%
\lambda _{\min (p,\beta ),d}^{\ast }\big\|_{\dot{K}_{p,r}^{\alpha
,q}f_{\beta }^{s}}.
\end{equation*}%
Let $\frac{n\min (p,\beta )}{d}<a<\min \big(\frac{n}{\alpha +\frac{n}{p}}%
,\min (p,\beta )\big),{j}\in \mathbb{N}$ and $m\in \mathbb{Z}^{n}$. {Define }
\begin{equation*}
{\Omega _{j,m}=\{h\in \mathbb{Z}^{n}:2^{j-1}<|h-m|\leq 2^{j}\}\quad }\text{%
and}{\quad \Omega _{0,m}=\{h\in \mathbb{Z}^{n}:|h-m|\leq 1\}}.
\end{equation*}%
Then 
\begin{align*}
\sum_{h\in \mathbb{Z}^{n}}\frac{|\lambda _{k,h}|^{\min (p,\beta )}}{\big(%
1+|h-m|\big)^{d}}& =\sum\limits_{j=0}^{\infty }\sum\limits_{h\in {\Omega
_{j,m}}}\frac{|\lambda _{k,h}|^{\min (p,\beta )}}{\big(1+|h-m|\big)^{d}} \\
& \lesssim \sum\limits_{j=0}^{\infty }2^{-dj}\sum\limits_{h\in {\Omega _{j,m}%
}}|\lambda _{k,h}|^{\min (p,\beta )} \\
& \lesssim \sum\limits_{j=0}^{\infty }2^{-dj}\Big(\sum\limits_{h\in {\Omega
_{j,m}}}|\lambda _{k,h}|^{a}\Big)^{\min (p,\beta )/a},
\end{align*}%
where the last estimate follows by the embedding $\ell ^{a}\hookrightarrow
\ell ^{\min (p,\beta )}$. The last expression can be rewritten as 
\begin{equation}
c\sum\limits_{j=0}^{\infty }2^{(\frac{n\min (p,\beta )}{a}-d)j}\Big(%
2^{(k-j)n}\int_{\cup _{z\in {\Omega _{j,m}}}Q_{k,z}}\sum\limits_{h\in {%
\Omega _{j,m}}}|\lambda _{k,h}|^{a}\chi _{k,h}(y)dy\Big)^{\min (p,\beta )/a}.
\label{estimate-lamda}
\end{equation}%
Let $y\in \cup _{z\in {\Omega _{j,m}}}Q_{k,z}$ and $x\in Q_{k,m}$. Then $%
y\in Q_{k,z}$ for some $z\in {\Omega _{j,m}}$ which implies that ${\
2^{j-1}<|z-m|\leq 2^{j}}$. From this it follows that 
\begin{align*}
|y-x|& \leq |y-2^{-k}z|+|x-2^{-k}z| \\
& \leq \sqrt{n}\text{ }2^{-k}+|x-2^{-k}m|+2^{-k}|z-m| \\
& \leq 2^{j-k+\delta _{n}},\quad \delta _{n}\in \mathbb{N},
\end{align*}%
which implies that $y$ is located in the ball $B(x,2^{j-k+\delta _{n}})$.
Therefore, $\mathrm{\eqref{estimate-lamda}}$ can be estimated from above by 
\begin{equation*}
c\Big(\mathcal{M}_{a}\big(\sum\limits_{h\in \mathbb{Z}^{n}}\lambda
_{k,h}\chi _{k,h}\big)(x)\Big)^{\min (p,\beta )},
\end{equation*}%
where the positive constant $c$ is independent of $k$ and $x$. Consequently 
\begin{equation}
\big\|\lambda _{\min (p,\beta ),d}^{\ast }\big\|_{\dot{K}_{p,r}^{\alpha
,q}f_{\beta }^{s}}  \label{estimate-lamda1}
\end{equation}%
does not exceed 
\begin{equation*}
c\Big\|\Big(\sum_{k=0}^{\infty }2^{k(s+\frac{n}{2})\beta }\big(\mathcal{M}%
_{a}\big(\sum\limits_{h\in \mathbb{Z}^{n}}\lambda _{k,h}\chi _{k,h}\big)\big)%
^{\beta }\Big)^{1/\beta }\Big\|_{\dot{K}_{p,r}^{\alpha ,q}}.
\end{equation*}%
Applying Lemma {\ref{Maximal-Inq copy(2)-lorentz}} we obtain that %
\eqref{estimate-lamda1} is dominated by 
\begin{equation*}
c\Big\|\Big(\sum_{k=0}^{\infty }2^{k(s+\frac{n}{2})\beta }\sum\limits_{h\in 
\mathbb{Z}^{n}}|\lambda _{k,h}|^{\beta }\chi _{k,h}\Big)^{1/\beta }\Big\|_{%
\dot{K}_{p,r}^{\alpha ,q}}=c\big\|\lambda \big\|_{\dot{K}_{p,r}^{\alpha
,q}f_{\beta }^{s}},
\end{equation*}%
which completes the proof of Lemma {\ref{lamda-equi-lorentz}.}
\end{proof}

Now, we have the following result which is called the $\varphi $-transform
characterization in the sense of Frazier and Jawerth. It will play an
important role in the rest of this section.

\begin{theorem}
\label{phi-tran-lorentz}Let $s\in \mathbb{R},0<p<\infty ,0<r,q\leq \infty
,0<\beta \leq \infty $ and $\alpha >-\frac{n}{p}$. Let $\Phi ,\Psi \in 
\mathcal{S}(\mathbb{R}^{n})$ satisfy \eqref{Ass1} and $\varphi ,\psi \in 
\mathcal{S}(\mathbb{R}^{n})$ satisfy \eqref{Ass2} such that \eqref{Ass3}
holds. The operators 
\begin{equation*}
S_{\varphi }:\dot{K}_{p,r}^{\alpha ,q}A_{\beta }^{s}\rightarrow \dot{K}%
_{p,r}^{\alpha ,q}a_{\beta }^{s}
\end{equation*}%
and 
\begin{equation*}
T_{\psi }:\dot{K}_{p,r}^{\alpha ,q}a_{\beta }^{s}\rightarrow \dot{K}%
_{p,r}^{\alpha ,q}A_{\beta }^{s}
\end{equation*}%
are bounded. Furthermore, $T_{\psi }\circ S_{\varphi }$ is the identity on $%
\dot{K}_{p,r}^{\alpha ,q}A_{\beta }^{s}$.
\end{theorem}

\begin{proof}
The proof is a straightforward adaptation of {\cite[Theorem 2.2]{FJ90} with
the help of Lemma \ref{lamda-equi-lorentz}. }The proof is complete.
\end{proof}

\begin{remark}
This theorem can then be exploited to obtain a variety of results for the
spaces $\dot{K}_{p,r}^{\alpha ,q}A_{\beta }^{s}$, where arguments can be
equivalently transferred to the sequence space, which is often more
convenient to handle. More precisely, under the same hypothesis of Theorem %
\ref{phi-tran-lorentz}, we obtain 
\begin{equation*}
\big\|\{\langle f,\varphi _{k,m}\rangle \}_{k\in \mathbb{N}_{0},m\in \mathbb{%
Z}^{n}}\big\|_{\dot{K}_{p,r}^{\alpha ,q}a_{\beta }^{s}}\approx \big\|f\big\|%
_{\dot{K}_{p,r}^{\alpha ,q}A_{\beta }^{s}}.
\end{equation*}
\end{remark}

From Theorem \ref{phi-tran-lorentz}, we obtain the next important property
of the spaces\ $\dot{K}_{p,r}^{\alpha ,q}A_{\beta }^{s}$.

\begin{corollary}
\label{Indpendent-lorentz}Let $s\in \mathbb{R},0<p<\infty ,0<r,q\leq \infty
,0<\beta \leq \infty $ and $\alpha >-\frac{n}{p}$. The definition of the
spaces $\dot{K}_{p,r}^{\alpha ,q}A_{\beta }^{s}$ is independent of the
choices of $\Phi $ and $\varphi $.
\end{corollary}

Let $\{\varphi _{k}\}_{k\in \mathbb{N}_{0}}$ be a resolution of unity; see %
\eqref{partition}. We set

\begin{equation*}
\big\|\lambda \big\|_{\dot{K}_{p,r}^{\alpha ,q}B_{\beta }^{s}}^{\varphi
_{0},\varphi _{1}}=\Big(\sum_{k=0}^{\infty }2^{ks\beta }\big\|\mathcal{F}%
^{-1}\varphi _{k}\ast f\big\|_{\dot{K}_{p,r}^{\alpha ,q}}^{\beta }\Big)%
^{1/\beta }
\end{equation*}%
and 
\begin{equation*}
\big\|\lambda \big\|_{\dot{K}_{p,r}^{\alpha ,q}F_{\beta }^{s}}^{\varphi
_{0},\varphi _{1}}=\Big\|\Big(\sum_{k=0}^{\infty }2^{ks\beta }|\mathcal{F}%
^{-1}\varphi _{k}\ast f|^{\beta }\Big)^{1/\beta }\Big\|_{\dot{K}%
_{p,r}^{\alpha ,q}}.
\end{equation*}

\begin{theorem}
\label{partition-equi-lorentz}Let $s\in \mathbb{R},0<p<\infty ,0<r,q\leq
\infty ,0<\beta \leq \infty $ and $\alpha >-\frac{n}{p}$. A tempered
distribution $f$ belongs to $\dot{K}_{p,r}^{\alpha ,q}A_{\beta }^{s}$ if and
only if%
\begin{equation*}
\big\|f\big\|_{\dot{K}_{p,r}^{\alpha ,q}A_{\beta }^{s}}^{\varphi
_{0},\varphi _{1}}<\infty .
\end{equation*}%
Furthermore, the quasi-norms\ $\big\|f\big\|_{\dot{K}_{p,r}^{\alpha
,q}A_{\beta }^{s}}$\ and $\big\|f\big\|_{\dot{K}_{p,r}^{\alpha ,q}A_{\beta
}^{s}}^{\varphi _{0},\varphi _{1}}$ are equivalent.
\end{theorem}

\begin{proof}
Let $\Phi ,\Psi \in \mathcal{S}(\mathbb{R}^{n})$ satisfy \eqref{Ass1} and $%
\varphi ,\psi \in \mathcal{S}(\mathbb{R}^{n})$ satisfy \eqref{Ass2} such
that \eqref{Ass3} holds. From Lemma {\ref{DW-lemma1} and by inspecting the
support conditions we obtain }%
\begin{equation*}
\mathcal{F}^{-1}\varphi _{k}\ast f=\sum_{j=k-1}^{k+1}\mathcal{F}^{-1}\varphi
_{k}\ast \widetilde{\varphi }_{j}\ast \psi _{j}\ast f+\left\{ 
\begin{array}{ccc}
0, & \text{if} & k\geq 3 \\ 
\mathcal{F}^{-1}\varphi _{k}\ast \tilde{\Phi}\ast \Psi \ast f, & \text{if} & 
k\in \{1,2\}%
\end{array}%
\right.
\end{equation*}%
and%
\begin{equation*}
\mathcal{F}^{-1}\varphi _{0}\ast f=\mathcal{F}^{-1}\varphi _{0}\ast 
\widetilde{\varphi }_{1}\ast \psi _{1}\ast f+\mathcal{F}^{-1}\varphi
_{0}\ast \tilde{\Phi}\ast \Psi \ast f.
\end{equation*}%
Let $j\in \{k-1,k,k+1\},k\geq 3$. Applying Lemmas {\ref{r-trick} and \ref%
{est-maximal}}, we conclude that%
\begin{equation*}
|\mathcal{F}^{-1}\varphi _{k}\ast \widetilde{\varphi }_{j}\ast \psi _{j}\ast
f|\lesssim \mathcal{M}_{\tau }(\widetilde{\varphi }_{j}\ast f),\quad 0<\tau
<\infty ,
\end{equation*}%
where the implicit constant is independent of $j$ and $k$. Similarly, when $%
k\in \{0,1,2\}$, we see that 
\begin{equation*}
|\mathcal{F}^{-1}\varphi _{k}\ast \tilde{\Phi}\ast \Psi \ast f|+|\mathcal{F}%
^{-1}\varphi _{0}\ast \widetilde{\varphi }_{1}\ast \psi _{1}\ast f|\lesssim 
\mathcal{M}_{\tau }(\tilde{\Phi}\ast f)+\mathcal{M}_{\tau }(\widetilde{%
\varphi }_{1}\ast f),\quad 0<\tau <\infty .
\end{equation*}%
If we choose $0<\tau <\min (\frac{n}{\alpha +\frac{n}{p}},p,\beta )$, then
by Lemma {\ref{Maximal-Inq copy(2)-lorentz}, we get}%
\begin{equation*}
\big\|f\big\|_{\dot{K}_{p,r}^{\alpha ,q}A_{\beta }^{s}}^{\varphi
_{0},\varphi _{1}}\lesssim \big\|f\big\|_{\dot{K}_{p,r}^{\alpha ,q}A_{\beta
}^{s}}.
\end{equation*}%
The opposite inequality follows by the same argument, with the help of the
smooth resolution of unity\ \eqref{partition}. The proof is complete.
\end{proof}

As an immediate conclusion of Theorem {\ref{partition-equi-lorentz} }we
obtain the next important property of the spaces $\dot{K}_{p,r}^{\alpha
,q}A_{\beta }^{s}$.

\begin{corollary}
Let $\{\varpi _{k}\}_{k\in \mathbb{N}_{0}}$ and $\{\varphi _{k}\}_{k\in 
\mathbb{N}_{0}}$ be two resolutions of unity. Let $s\in \mathbb{R}%
,0<p<\infty ,0<q,r,\beta \leq \infty $ and $\alpha >-\frac{n}{p}$. Let $f\in 
\dot{K}_{p,r}^{\alpha ,q}A_{\beta }^{s}$. Then 
\begin{equation*}
\big\|f\big\|_{\dot{K}_{p,r}^{\alpha ,q}A_{\beta }^{s}}^{\{\varpi
_{k}\}_{k\in \mathbb{N}_{0}}}\approx \big\|f\big\|_{\dot{K}_{p,r}^{\alpha
,q}A_{\beta }^{s}}^{\{\varphi _{k}\}_{k\in \mathbb{N}_{0}}}\approx \big\|f%
\big\|_{\dot{K}_{p,r}^{\alpha ,q}A_{\beta }^{s}}.
\end{equation*}
\end{corollary}

\begin{remark}
The function $\vartheta $ defined in \eqref{function-v} can be replaced by 
\begin{equation*}
\mu (x)=1\quad \text{for}\quad \lvert x\rvert \leq 1\quad \text{and}\quad
\mu (x)=0\quad \text{for}\quad \lvert x\rvert \geq 2.
\end{equation*}%
We put $\varphi _{0}(x)=\mu (x),\,\varphi _{1}(x)=\mu (x)-\mu (2x)$ and $%
\varphi _{k}(x)=\varphi _{1}(2^{-k}x)\ $for $k=2,3,....$ Then we have $%
\mathrm{supp}\varphi _{k}\subset \{x\in {\mathbb{R}^{n}}:2^{k-1}\leq \lvert
x\rvert \leq 2^{k}\}\ $and \eqref{partition} is true.
\end{remark}

\begin{lemma}
Let $s\in \mathbb{R},0<p<\infty ,0<r,q\leq \infty ,0<\beta \leq \infty $ and 
$\alpha >-\frac{n}{p}$. The spaces $\dot{K}_{p,r}^{\alpha ,q}a_{\beta }^{s}$
are quasi-Banach spaces.
\end{lemma}

\begin{proof}
The proof is very similar as in \cite{Dr-AOT23}.
\end{proof}

Applying this lemma and Theorem \ref{phi-tran-lorentz} we obtain the
following useful properties of the spaces $\dot{K}_{p,r}^{\alpha ,q}A_{\beta
}^{s}$.

\begin{theorem}
Let $s\in \mathbb{R},0<p<\infty ,0<r,q\leq \infty ,0<\beta \leq \infty $ and 
$\alpha >-\frac{n}{p}$. The spaces $\dot{K}_{p,r}^{\alpha ,q}A_{\beta }^{s}$
are quasi-Banach spaces.
\end{theorem}

\begin{proof}
Let $\{U_{i}\}_{i\in \mathbb{N}_{0}}$ be a Cauchy sequence in $\dot{K}%
_{p,r}^{\alpha ,q}A_{\beta }^{s}$. From Theorem \ref{phi-tran-lorentz}, $%
\{S_{\varphi }U_{i}\}_{i\in \mathbb{N}_{0}}$ is Cauchy sequence in $\dot{K}%
_{p,r}^{\alpha ,q}a_{\beta }^{s}$, this has a limit $\lambda =\{\lambda
_{j,m}\}_{j\in \mathbb{N}_{0},m\in \mathbb{Z}^{n}}$ by the completeness of
the sequence space $\dot{K}_{p,r}^{\alpha ,q}a_{\beta }^{s}$. Using again
Theorem \ref{phi-tran-lorentz}, we easily obtain 
\begin{equation*}
T_{\psi }\lambda =\lim_{i\rightarrow \infty }T_{\psi }S_{\varphi
}U_{i}=\lim_{i\rightarrow \infty }U_{i},
\end{equation*}%
where the limit is in $\dot{K}_{p,r}^{\alpha ,q}A_{\beta }^{s}$.
\end{proof}

\begin{remark}
Let $s_{0},s_{1}\in \mathbb{R},0<p_{i}<\infty ,0<q_{i},\beta _{i},r_{i}\leq
\infty ,\alpha _{i}>-\frac{n}{p_{i}},i\in \{0,1\}$ and $0<\theta <1$. Put%
\begin{equation}
\alpha =(1-\theta )\alpha _{0}+\theta \alpha _{1},\quad \frac{1}{p}=\frac{%
1-\theta }{p_{0}}+\frac{\theta }{p_{1}},\quad \frac{1}{q}=\frac{1-\theta }{%
q_{0}}+\frac{\theta }{q_{1}},  \label{interpolation-ine-lorentz1}
\end{equation}%
\begin{equation}
\frac{1}{r}=\frac{1-\theta }{r_{0}}+\frac{\theta }{r_{1}},\quad s=(1-\theta
)s_{0}+\theta s_{1}  \label{interpolation-ine-lorentz2}
\end{equation}%
and%
\begin{equation*}
\frac{1}{\beta }=\frac{1-\theta }{\beta _{0}}+\frac{\theta }{\beta _{1}}.
\end{equation*}%
As an immediate consequence of H\"{o}lder's inequality we have the so-called
interpolation inequalities:%
\begin{equation}
\big\|f\big\|_{\dot{K}_{p,r}^{\alpha ,q}A_{\beta }^{s}}\leq \big\|f\big\|_{%
\dot{K}_{p_{0},r_{0}}^{\alpha _{0},q_{0}}A_{\beta _{0}}^{s_{0}}}^{1-\theta }%
\big\|f\big\|_{\dot{K}_{p_{1},r_{1}}^{\alpha _{1},q_{1}}A_{\beta
_{1}}^{s_{1}}}^{\theta }  \label{interpolation-ine-lorentz}
\end{equation}%
holds for all $f\in \dot{K}_{p_{0},r_{0}}^{\alpha _{0},q_{0}}A_{\beta
_{0}}^{s_{0}}\cap \dot{K}_{p_{1},r_{1}}^{\alpha _{1},q_{1}}A_{\beta
_{1}}^{s_{1}}$.
\end{remark}

For Lorentz Herz-type Triebel-Lizorkin spaces inequality %
\eqref{interpolation-ine-lorentz} can be improved by using the following
statement which can be found in \cite{BM01}.

\begin{lemma}
\label{BM-lemma}Let real numbers $s_{1}<s_{0}$ be given, and $0<\sigma <1$.
For $0<q\leq \infty $ there is $c>0$ such that%
\begin{equation*}
\Big(\sum_{j=0}^{\infty }2^{\left( \sigma s_{0}+\left( 1-\sigma \right)
s_{1}\right) qj}|a_{j}|^{q}\Big)^{1/q}\leq c\sup_{j\in \mathbb{N}%
_{0}}(2^{s_{0}j}|a_{j}|)^{\sigma }\sup_{j\in \mathbb{N}%
_{0}}(2^{s_{1}j}|a_{j}|)^{1-\sigma }
\end{equation*}%
holds for all complex sequences $\left\{ 2^{s_{0}j}a_{j}\right\} _{j\in 
\mathbb{N}_{0}}$\ in $\ell ^{\infty }$ with the usual modification if $%
q=\infty .$
\end{lemma}

\begin{lemma}
Let $s_{0},s_{1}\in \mathbb{R}$ be such that $s_{0}<s_{1}$. Let $%
0<p_{i}<\infty ,0<q_{i},\beta ,r_{i}\leq \infty ,\alpha _{i}>-\frac{n}{p_{i}}%
,i\in \{0,1\}$ and $0<\theta <1$. Under the same additional restrictions %
\eqref{interpolation-ine-lorentz1} and \eqref{interpolation-ine-lorentz2} we
have%
\begin{equation*}
\big\|f\big\|_{\dot{K}_{p,r}^{\alpha ,q}F_{\beta }^{s}}\leq \big\|f\big\|_{%
\dot{K}_{p_{0},r_{0}}^{\alpha _{0},q_{0}}F_{\infty }^{s_{0}}}^{1-\theta }%
\big\|f\big\|_{\dot{K}_{p_{1},r_{1}}^{\alpha _{1},q_{1}}F_{\infty
}^{s_{1}}}^{\theta }
\end{equation*}%
holds for all $f\in \dot{K}_{p_{0},r_{0}}^{\alpha _{0},q_{0}}F_{\infty
}^{s_{0}}\cap \dot{K}_{p_{1},r_{1}}^{\alpha _{1},q_{1}}F_{\infty }^{s_{1}}$.
\end{lemma}

\begin{proof}
Let $\{\varphi _{k}\}_{k\in \mathbb{N}_{0}}$ be a smooth dyadic resolution
of unity. By Lemma \ref{BM-lemma}, we obtain 
\begin{equation*}
\Big(\sum_{k=0}^{\infty }2^{ks\beta }|\mathcal{F}^{-1}\varphi _{k}\ast
f|^{\beta }\Big)^{1/\beta }\leq \sup_{k\in \mathbb{N}_{0}}(2^{ks_{0}}|%
\mathcal{F}^{-1}\varphi _{k}\ast f|)^{1-\theta }\sup_{k\in \mathbb{N}%
_{0}}(2^{ks_{1}}|\mathcal{F}^{-1}\varphi _{k}\ast f|)^{\theta }.
\end{equation*}%
The\ rest is an immediate consequence of H\"{o}lder's inequality.
\end{proof}

\subsection{Lifting property and Fatou property}

Let $\sigma $ be a real number. Recall that the lifting operator $I_{\sigma
} $ is defined by%
\begin{equation*}
\mathcal{F}(I_{\sigma }f)\equiv (1+|\cdot |^{2})^{\sigma /2}\mathcal{F}%
(f),\quad f\in \mathcal{S}^{\prime }(\mathbb{R}^{n}),
\end{equation*}%
see, for example, \cite[p. 58]{T83}. It is well known that $I_{\sigma }$ is
a one-to-one mapping from $\mathcal{S}^{\prime }(\mathbb{R}^{n})$ onto
itself. We have the following result, where the proof can be obtained as in 
\cite[Theorem 4.5]{XuYang05}.

\begin{theorem}
\label{Lifting-lorentz}Let $s,\sigma \in \mathbb{R},m\in \mathbb{N}%
,0<p<\infty ,0<q,r,\beta \leq \infty $ and $\alpha >-\frac{n}{p}$. Then the
operator $I_{\sigma }$ maps $\dot{K}_{p,r}^{\alpha ,q}A_{\beta }^{s}$
isomorphically onto $\dot{K}_{p,r}^{\alpha ,q}A_{\beta }^{s-\sigma }$ and $%
\big\|I_{\sigma }\big\|_{\dot{K}_{p,r}^{\alpha ,q}A_{\beta }^{s-\sigma }}$
is an equivalent quasi-norm on $\dot{K}_{p,r}^{\alpha ,q}A_{\beta
}^{s-\sigma }$. Furthermore%
\begin{equation*}
\sum\limits_{|\gamma |\leq m}\big\|D^{\gamma }f\big\|_{\dot{K}_{p,r}^{\alpha
,q}A_{\beta }^{s-m}}
\end{equation*}%
and%
\begin{equation*}
\big\|f\big\|_{\dot{K}_{p,r}^{\alpha ,q}A_{\beta
}^{s-m}}+\sum\limits_{j=0}^{n}\Big\|\frac{\partial ^{m}f}{\partial x_{j}^{m}}%
\Big\|_{\dot{K}_{p,r}^{\alpha ,q}A_{\beta }^{s-m}},
\end{equation*}%
\textit{are an equivalent quasi-norm in }$\dot{K}_{p,r}^{\alpha ,q}A_{\beta
}^{s}$.
\end{theorem}

Next, we prove that the spaces $\dot{K}_{p,r}^{\alpha ,q}A_{\beta }^{s}$
satisfy the Fatou property. First we recall the definition of the Fatou
property; see, e.g., \cite{Fr86} and \cite[p. 48]{YSY10}.

\begin{definition}
Let $(A,\big\Vert\cdot \big\Vert_{A})$ be a Banach space with $\mathcal{S}(%
\mathbb{R}^{n})\hookrightarrow A\hookrightarrow \mathcal{S}^{\prime }(%
\mathbb{R}^{n})$. We say $A$ has the Fatou property if there exists a
constant $c$ such that from 
\begin{equation*}
g_{m}\rightharpoonup g\quad \text{if}\quad m\longrightarrow \infty \quad 
\text{(weak convergence in }\mathcal{S}^{\prime }(\mathbb{R}^{n})\text{)}
\end{equation*}%
and%
\begin{equation*}
\underset{m\longrightarrow \infty }{\mathrm{lim}\text{ }\mathrm{inf}}\text{ }%
\big\Vert g_{m}\big\Vert_{A}\leq M
\end{equation*}%
it follows $g\in A$ and $\big\Vert g\big\Vert_{A}\leq c$ $M$ with $c$
independent of $g$ and $\{g_{m}\}_{m\in \mathbb{N}_{0}}\subset A$.
\end{definition}

\begin{proposition}
\label{Fatou-lorentz}Let $0<p,q,r<\infty ,0<\beta <\infty ,s\in \mathbb{R}%
^{n}$ and $\alpha >-\frac{n}{p}$. The spaces $\dot{K}_{p,r}^{\alpha
,q}A_{\beta }^{s}$ have the Fatou property.
\end{proposition}

\begin{proof}
By similarity, we only consider the space $\dot{K}_{p,r}^{\alpha ,q}F_{\beta
}^{s}$. Let $\Phi $ and $\varphi $\ satisfy \eqref{Ass1}\ and\ \eqref{Ass2},
respectively. By the assumption it follows that for all $k\in \mathbb{N}_{0}$
\begin{equation*}
\varphi _{k}\ast f_{m}\rightarrow \varphi _{k}\ast f
\end{equation*}%
as $m\rightarrow \infty $, where when $k=0$, $\varphi _{0}$ is replaced by $%
\Phi $. Fatou's lemma yields%
\begin{equation*}
\Big\|\Big(\sum\limits_{k=0}^{N}2^{ks\beta }|\varphi _{k}\ast f|^{\beta }%
\Big)^{1/\beta }\Big\|_{\dot{K}_{p,r}^{\alpha ,q}}\leq \underset{%
m\longrightarrow \infty }{\mathrm{lim}\text{ }\mathrm{inf}}\Big\|\Big(%
\sum\limits_{k=0}^{N}2^{ks\beta }|\varphi _{k}\ast f_{m}|^{\beta }\Big)%
^{1/\beta }\Big\|_{\dot{K}_{p,r}^{\alpha ,q}}.
\end{equation*}%
This combined with Beppo Levi's lemma yields the desired conclusion. The
proof is complete.
\end{proof}

\begin{remark}
$\mathrm{(i)}$ The Fatou property of Besov and Triebel-Lizorkin spaces has
been proved by Franke \cite{Fr86}; see also Franke and Runst \cite{FR65}. 
\newline
$\mathrm{(ii)}$ Bourdaud and Meyer \cite{BM91} gave an independent proof
restricted to Besov spaces.\newline
$\mathrm{(iii)}$ There are spaces which do not have the Fatou property. For
example, $L^{1}$ and $C$; see \cite{Fr86}.\newline
$\mathrm{(iv)}$ Fatou property plays an essential role in mathematical
analysis such as nonlinear problems; see \cite{BM91} and \cite{RuSi96}.
\end{remark}

\section{Embeddings}

In this section, we establish basic embeddings, Sobolev, Jawerth and Franke
embeddings for the spaces under consideration. The following theorem gives
basic embeddings of the spaces $\dot{K}_{p,r}^{\alpha ,q}A_{\beta }^{s}$.

\begin{theorem}
\label{embeddings1.1-lorentz}\textit{Let }$s\in \mathbb{R},0<p<\infty
,0<r,q\leq \infty $\textit{\ and }$\alpha >-\frac{n}{p}$\textit{.}\newline
$\mathrm{(i)}$\textit{\ If }$0<\beta _{1}\leq \beta _{2}\leq \infty $, then%
\begin{equation}
\dot{K}_{p,r}^{\alpha ,q}A_{\beta _{1}}^{s}\hookrightarrow \dot{K}%
_{p,r}^{\alpha ,q}A_{\beta _{2}}^{s}.  \label{embed1}
\end{equation}%
$\mathrm{(ii)}$\textit{\ If }$0<\beta _{1},\beta _{2}\leq \infty $ and $%
\varepsilon >0$, then%
\begin{equation}
\dot{K}_{p,r}^{\alpha ,q}A_{\beta _{1}}^{s+\varepsilon }\hookrightarrow \dot{%
K}_{p,r}^{\alpha ,q}A_{\beta _{2}}^{s}.  \label{embed2}
\end{equation}%
$\mathrm{(iii)}$\textit{\ If }$0<q_{1}\leq q_{2}\leq \infty $, then%
\begin{equation}
\dot{K}_{p,r}^{\alpha ,q_{1}}A_{\beta }^{s}\hookrightarrow \dot{K}%
_{p,r}^{\alpha ,q_{2}}A_{\beta }^{s}.  \label{embed3}
\end{equation}%
$\mathrm{(iv)}$\textit{\ }Let $0<r_{2},r_{1}\leq \infty ,\alpha \in \mathbb{R%
}\ $and suppose $0<p_{1}<p_{2}<\infty $, then%
\begin{equation}
\dot{K}_{p_{2},r_{2}}^{\alpha ,q}A_{\beta }^{s}\hookrightarrow \dot{K}%
_{p_{1},r_{1}}^{m,q}A_{\beta }^{s},  \label{embed4}
\end{equation}%
where $m=\alpha -n\big(\frac{1}{p_{1}}-\frac{1}{p_{2}}\big).$\newline
$\mathrm{(v)}$\textit{\ If }$0<r_{1}\leq r_{2}\leq \infty $, then%
\begin{equation}
\dot{K}_{p,r_{1}}^{\alpha ,q}A_{\beta }^{s}\hookrightarrow \dot{K}%
_{p,r_{2}}^{\alpha ,q}A_{\beta }^{s}.  \label{embed5}
\end{equation}
\end{theorem}

\begin{proof}
The emdeddings \eqref{embed1}, \eqref{embed3} and \eqref{embed5} are ready
consequence of the embeddings between Lebesgue sequence spaces and Lemma \ref%
{embeddings1-lorentz}. Let $\Phi $ and $\varphi $\ satisfy $\mathrm{%
\eqref{Ass1}}$\ and\ $\mathrm{\eqref{Ass2}}$, respectively and $f\in \dot{K}%
_{p,r}^{\alpha ,q}F_{\beta _{1}}^{s+\varepsilon }$. To prove \eqref{embed2},
since $\varepsilon >0$ we see that 
\begin{equation*}
\Big\|\Big(\sum\limits_{k=0}^{\infty }2^{ks\beta _{2}}|\varphi _{k}\ast
f|^{\beta _{2}}\Big)^{1/\beta _{2}}\Big\|_{\dot{K}_{p,r}^{\alpha ,q}}\leq c%
\Big\|\sup_{k\in \mathbb{N}_{0}}\big(2^{k(s+\varepsilon )}|\varphi _{k}\ast
f|\big)\Big\|_{\dot{K}_{p,r}^{\alpha ,q}}.
\end{equation*}%
The desired estimate follows by the embeddings $\ell ^{\beta
_{1}}\hookrightarrow \ell ^{\infty }$. The $B$-case follows from a similar
argument. The emdeddings \eqref{embed4}, follows immediately from
Proposition \ref{embeddings1 copy(1)-lorentz}.
\end{proof}

Similarly as in \cite{Drihem1.13} and \cite[Proposition. 2.3.2/2]{T83}, we
obtain the following basic embeddings between the spaces $\dot{K}%
_{p,r}^{\alpha ,q}B_{\beta }^{s}$ and $\dot{K}_{p,r}^{\alpha ,q}F_{\beta
}^{s}$.

\begin{theorem}
\label{embeddings2-lorentz}\textit{Let }$s\in \mathbb{R},0<p<\infty
,0<q,\beta \leq \infty ,0<r_{0}\leq r_{1}\leq \infty $\textit{\ and }$\alpha
>-\frac{n}{p}$\textit{. \newline
}$\mathrm{(i)}$\textit{\ }Assume that $p\neq \beta $ or $p=\beta \geq r_{0}$%
. \textit{Then}%
\begin{equation*}
\dot{K}_{p,r_{0}}^{\alpha ,q}B_{\min \left( p,\beta ,r_{1},q\right)
}^{s}\hookrightarrow \dot{K}_{p,r_{1}}^{\alpha ,q}F_{\beta }^{s}.
\end{equation*}%
$\mathrm{(ii)}$\textit{\ }Assume that $p\neq \beta $ or $p=\beta \leq r_{1}$%
. \textit{Then}%
\begin{equation*}
\dot{K}_{p,r_{0}}^{\alpha ,q}F_{\beta }^{s}\hookrightarrow \dot{K}%
_{p,r_{1}}^{\alpha ,q}B_{\max \left( p,\beta ,r_{0},q\right) }^{s}.
\end{equation*}
\end{theorem}

\begin{proof}
The proof of (i) is a consequence of Lemma \ref{Lp,r-estimate}. To prove
(ii), we use Lemma \ref{Lp,r-estimate copy(1)}.
\end{proof}

\begin{remark}
Theorem \ref{embeddings2-lorentz}\ \ when $\alpha =0,p=q=r$ generalizes the
corresponding results on Besov and Triebel-Lizorkin spaces established in 
\cite[ Section 2.3]{T83}.
\end{remark}

The same arguments as in \cite{Drihem1.13} yield the following theorem.

\begin{theorem}
\label{embeddings-S-inf-lorentz}Let $s\in \mathbb{R},0<p<\infty ,0<r,q\leq
\infty ,0<\beta \leq \infty $ and $\alpha >-\frac{n}{p}$.\newline
$\mathrm{(i)}$ We have the embedding 
\begin{equation}
\mathcal{S}(\mathbb{R}^{n})\hookrightarrow \dot{K}_{p,r}^{\alpha ,q}A_{\beta
}^{s}.  \label{embedding}
\end{equation}%
In addition if $0<q,r<\infty $ and $0<\beta <\infty $, then $\mathcal{S}(%
\mathbb{R}^{n})$ is dense in $\dot{K}_{p,r}^{\alpha ,q}A_{\beta }^{s}$.%
\newline
$\mathrm{(ii)}$ We have the embedding 
\begin{equation}
\dot{K}_{p,r}^{\alpha ,q}A_{\beta }^{s}\hookrightarrow \mathcal{S}^{\prime }(%
\mathbb{R}^{n}).  \label{embeddingsSch}
\end{equation}
\end{theorem}

\subsection{Sobolev embeddings for the spaces $\dot{K}_{p,r}^{\protect\alpha %
,q}B_{\protect\beta }^{s}$}

We next consider embeddings of Sobolev-type in $\dot{K}_{p,r}^{\alpha
,q}B_{\beta }^{s}$. It is well-known that%
\begin{equation}
B_{q,\beta }^{s_{2}}\hookrightarrow B_{s,\beta }^{s_{1}},
\label{Sobolev-emb-lorentz}
\end{equation}%
if $s_{1}-\frac{n}{s}=s_{2}-\frac{n}{q}$, where $0<q\leq s\leq \infty $ and $%
0<\beta \leq \infty $; see, e.g., \cite[Theorem 2.7.1]{T83}). In the
following theorem we generalize these embeddings to Lorentz Herz-type Besov
spaces.

\begin{theorem}
\label{embeddings3-lorentz}\textit{Let }$\alpha _{1},\alpha
_{2},s_{1},s_{2}\in \mathbb{R},0<s,p<\infty ,0<q,r,r_{1},r_{2},\beta \leq
\infty ,\alpha _{1}>-\frac{n}{s}\ $\textit{and }$\alpha _{2}>-\frac{n}{p}$. 
\textit{We suppose that }%
\begin{equation}
s_{1}-\frac{n}{s}-\alpha _{1}\leq s_{2}-\frac{n}{p}-\alpha _{2}.
\label{newexp1-lorentz}
\end{equation}%
\textit{Let }$0<p\leq s<\infty $ and $\alpha _{2}\geq \alpha _{1}$ or $%
0<s<p<\infty $\ and 
\begin{equation}
\alpha _{2}+\frac{n}{p}\geq \alpha _{1}+\frac{n}{s}.  \label{newexp2-lorentz}
\end{equation}%
Then%
\begin{equation}
\dot{K}_{p,r_{2}}^{\alpha _{2},\theta }B_{\beta }^{s_{2}}\hookrightarrow 
\dot{K}_{s,r_{1}}^{\alpha _{1},r}B_{\beta }^{s_{1}},
\label{Sobolev-emb1-lorentz}
\end{equation}%
where%
\begin{equation*}
\theta =\left\{ 
\begin{array}{ccc}
r, & \text{if} & \alpha _{2}+\frac{n}{p}=\alpha _{1}+\frac{n}{s},\quad
s<p\quad \text{or}\quad \alpha _{2}=\alpha _{1},\quad p\leq s \\ 
q, & \text{if} & \alpha _{2}+\frac{n}{p}>\alpha _{1}+\frac{n}{s},\quad
s<p\quad \text{or}\quad \alpha _{2}>\alpha _{1},\quad p\leq s.%
\end{array}%
\right.
\end{equation*}%
The conditions \eqref{newexp1-lorentz} and \eqref{newexp2-lorentz} become
necessary.
\end{theorem}

\begin{proof}
\textit{Step 1. Sufficiency.} Let $\{\varphi _{j}\}_{j\in \mathbb{N}_{0}}$\
be a smooth dyadic resolution of unity and $f\in \dot{K}_{p,r_{2}}^{\alpha
_{2},\theta }B_{\beta }^{s_{2}}$. By Lemmas \ref{Bernstein-Herz-ine1-lorentz}%
\ and\ \ref{Bernstein-Herz-ine2-lorentz}, we obtain%
\begin{equation}
\big\|\mathcal{F}^{-1}\varphi _{j}\ast f\big\|_{\dot{K}_{s,r_{1}}^{\alpha
_{1},r}}\leq c\text{ }2^{j(\alpha _{2}+\frac{n}{p}-\frac{n}{s}-\alpha _{1})}%
\big\|\mathcal{F}^{-1}\varphi _{j}\ast f\big\|_{\dot{K}_{p,r_{2}}^{\alpha
_{2},\theta }},  \label{bernstein2-lorentz}
\end{equation}%
where $c>0$ is independent of $j\in \mathbb{N}_{0}$. However the desired
embedding is an immediate consequence of \eqref{bernstein2-lorentz}.

\textit{Step 2. }We prove the necessity\ of \eqref{newexp1-lorentz}. Let $%
\omega \in \mathcal{S}(\mathbb{R}^{n})$ be a function such that \textrm{supp}%
$\mathcal{F}\omega \subset \{\xi \in \mathbb{R}^{n}:\frac{3}{4}<|\xi |<1\}$.
For $x\in \mathbb{R}^{n}$ and $N\in \mathbb{N}$ we put $f_{N}(x)=\omega
(2^{N}x)$. First we have $\omega \in \dot{K}_{p,r_{2}}^{\alpha _{2},\theta
}\cap \dot{K}_{s,r_{1}}^{\alpha _{1},r}$. Due to the support properties of
the function $\omega $ we have for any $j\in \mathbb{N}_{0}$%
\begin{equation*}
\mathcal{F}^{-1}\varphi _{j}\ast f_{N}=\left\{ 
\begin{array}{cc}
f_{N}, & j=N, \\ 
0, & \text{otherwise.}%
\end{array}%
\right.
\end{equation*}%
This leads to%
\begin{align*}
\big\|f_{N}\big\|_{\dot{K}_{s,r_{1}}^{\alpha _{1},r}B_{\beta }^{s_{1}}}=&
2^{s_{1}N}\big\|f_{N}\big\|_{\dot{K}_{s,r_{1}}^{\alpha _{1},r}} \\
=& 2^{s_{1}N}\Big(\sum_{k=-\infty }^{\infty }2^{k\alpha _{1}r}\big\|%
f_{N}\chi _{k}\big\|_{L^{s,r_{1}}}^{r}\Big)^{1/r} \\
=& 2^{(s_{1}-\frac{n}{s})N}\Big(\sum_{k=-\infty }^{\infty }2^{k\alpha _{1}r}%
\big\|\omega \chi _{k+N}\big\|_{L^{s,r_{1}}}^{r}\Big)^{1/r} \\
=& 2^{(s_{1}-\alpha _{1}-\frac{n}{s})N}\big\|\omega \big\|_{\dot{K}%
_{s,r_{1}}^{\alpha _{1},r}},
\end{align*}%
with the help of \eqref{dilation-lorentz}, since 
\begin{align*}
\big\|f_{N}\chi _{k}\big\|_{L^{s,r_{1}}}& =\big\|\omega (2^{N}\cdot )\chi
_{R_{k}}\big\|_{L^{s,r_{1}}} \\
& =2^{-\frac{n}{s}N}\big\|\omega \chi _{R_{k}}(2^{-N}\cdot )\big\|%
_{L^{s,r_{1}}} \\
& =2^{-\frac{n}{s}N}\big\|\omega \chi _{k+N}\big\|_{L^{s,r_{1}}}
\end{align*}%
for any $k\in \mathbb{Z}$. The same arguments give 
\begin{equation*}
\big\|f_{N}\big\|_{\dot{K}_{p,r_{2}}^{\alpha _{2},\theta }B_{\beta
}^{s_{2}}}=2^{(s_{2}-\alpha _{2}-\frac{n}{p})N}\big\|\omega \big\|_{\dot{K}%
_{p,r_{2}}^{\alpha _{2},\theta }}.
\end{equation*}%
If the embeddings \eqref{Sobolev-emb1-lorentz}$\ $holds then for any $N\in 
\mathbb{N}$%
\begin{equation*}
2^{(s_{1}-s_{2}-\alpha _{1}+\alpha _{2}-\frac{n}{s}+\frac{n}{p})N}\leq c.
\end{equation*}%
Thus, we conclude that \eqref{newexp1-lorentz} must necessarily hold by
letting $N\rightarrow +\infty $.

\textit{Step 3. }We prove the necessity\ of \eqref{newexp2-lorentz}. Let $%
\varpi \in \mathcal{S}(\mathbb{R}^{n})$ be a function such that \textrm{supp 
}$\mathcal{F}\varpi \subset \left\{ \xi \in \mathbb{R}^{n}:|\xi |<1\right\} $%
. For $x\in \mathbb{R}^{n}$ and $N\in \mathbb{Z}\backslash \mathbb{N}_{0}$
we put $f_{N}(x)=\varpi (2^{N}x)$. We have $\varpi \in \dot{K}%
_{p,r_{2}}^{\alpha _{2},\theta }\cap \dot{K}_{s,r_{1}}^{\alpha _{1},r}$. It
is easy to see that%
\begin{equation*}
\mathcal{F}^{-1}\varphi _{j}\ast f_{N}=\left\{ 
\begin{array}{cc}
f_{N}, & j=0, \\ 
0, & \text{otherwise}.%
\end{array}%
\right.
\end{equation*}%
This yields%
\begin{equation*}
\big\|f_{N}\big\|_{\dot{K}_{s,r_{1}}^{\alpha _{1},r}B_{\beta }^{s_{1}}}=%
\big\|f_{N}\big\|_{\dot{K}_{s,r_{1}}^{\alpha _{1},r}}=2^{-\left( \alpha _{1}+%
\frac{n}{s}\right) N}\big\|\varpi \big\|_{\dot{K}_{s,r_{1}}^{\alpha _{1},r}}.
\end{equation*}%
Similarly, we have%
\begin{equation*}
\big\|f_{N}\big\|_{\dot{K}_{p,r_{2}}^{\alpha _{2},\theta }B_{\beta
}^{s_{2}}}=2^{-(\alpha _{2}+\frac{n}{p})N}\big\|\varpi \big\|_{\dot{K}%
_{p,r_{2}}^{\alpha _{2},\theta }}.
\end{equation*}%
If the embedding \eqref{Sobolev-emb1-lorentz}$\ $holds, then for any $N\in 
\mathbb{Z}\backslash \mathbb{N}_{0}$ 
\begin{equation*}
2^{-(\alpha _{1}-\alpha _{2}+\frac{n}{s}-\frac{n}{p})N}\leq c.
\end{equation*}%
Thus, we conclude that \eqref{newexp2-lorentz} must necessarily hold by
letting $N\rightarrow -\infty $. The proof is complete.
\end{proof}

\begin{remark}
If $\alpha _{1}=\alpha _{2}=0$, $p=q=r_{2}$ and $r=s=r_{1}$, then Theorem %
\ref{embeddings3-lorentz} reduces to the known results on $B_{p,\beta }^{s}$%
; see \eqref{Sobolev-emb-lorentz}, by using the embedding $\ell
^{q}\hookrightarrow \ell ^{s}$. Also under the hypothesis of such theorem,
we have $s_{1}\leq s_{2}$ becomes necessary.
\end{remark}

\begin{corollary}
\label{embeddings3-cor}\textit{Under the hypotheses of }Theorem \ref%
{embeddings3-lorentz}, with $0<p\leq r_{2}\leq \infty $, we have 
\begin{equation*}
\dot{K}_{p}^{\alpha _{2},\theta }B_{\beta }^{s_{2}}\hookrightarrow \dot{K}%
_{p,r_{2}}^{\alpha _{2},\theta }B_{\beta }^{s_{2}}\hookrightarrow \dot{K}%
_{s}^{\alpha _{1},r}B_{\beta }^{s_{1}}.
\end{equation*}
\end{corollary}

\begin{proof}
From Theorem \ref{embeddings3-lorentz}, we obtain%
\begin{equation*}
\dot{K}_{p}^{\alpha _{2},\theta }B_{\beta }^{s_{2}}=\dot{K}_{p,p}^{\alpha
_{2},\theta }B_{\beta }^{s_{2}}\hookrightarrow \dot{K}_{p,r_{2}}^{\alpha
_{2},\theta }B_{\beta }^{s_{2}}\hookrightarrow \dot{K}_{s,s}^{\alpha
_{1},r}B_{\beta }^{s_{1}}=\dot{K}_{s}^{\alpha _{1},r}B_{\beta }^{s_{1}}.
\end{equation*}
\end{proof}

\begin{remark}
Corollary \ref{embeddings3-cor}\ extends and improves\ Sobolev embeddings of
Herz-type Besov spaces\ given\ in \cite{Drihem1.13}. In particular Sobolev
embeddings for Besov spaces\ of power weight\ obtained\ in \cite{MM12}.
\end{remark}

In the following theorems, we compare our spaces above with classical Besov
spaces. From Theorem \ref{embeddings3-lorentz}\ and the fact that $\dot{K}%
_{s}^{0,s}B_{\beta }^{s_{1}}=B_{s,\beta }^{s_{1}}$\ we immediately arrive at
the following result.

\begin{theorem}
\label{embeddings4-lorentz}\textit{Let }$\alpha ,s_{1},s_{2}\in \mathbb{R}%
,0<s,p<\infty ,0<q,r_{2}\leq \infty ,s_{1}-\frac{n}{s}\leq s_{2}-\frac{n}{p}%
-\alpha $\textit{\ and }$0<\beta \leq \infty $\textit{.} \textit{I}f 
\begin{equation*}
\alpha \geq 0,\quad 0<p\leq s<\infty \quad \text{or}\quad \alpha +\frac{n}{p}%
\geq \frac{n}{s}\quad \text{and}\quad 0<s<p<\infty ,
\end{equation*}%
then%
\begin{equation*}
\dot{K}_{p,r_{2}}^{\alpha ,\theta }B_{\beta }^{s_{2}}\hookrightarrow
B_{s,\beta }^{s_{1}},
\end{equation*}%
where%
\begin{equation}
\theta =\left\{ 
\begin{array}{ccc}
s, & \text{if} & \alpha +\frac{n}{p}=\frac{n}{s},\quad s<p\quad \text{or}%
\quad \alpha =0,\quad p\leq s, \\ 
q, & \text{if} & \alpha +\frac{n}{p}>\frac{n}{s},\quad s<p\quad \text{or}%
\quad \alpha >0,\quad p\leq s.%
\end{array}%
\right.  \label{aux10-lorentz}
\end{equation}
\end{theorem}

Using Corollary \ref{embeddings3-cor}, we have the following useful
consequence.

\begin{corollary}
\label{embeddings-besov1}Let $s_{1},s_{2}\in \mathbb{R},0<p\leq s<\infty
,0<q\leq \infty ,s_{1}-\frac{n}{s}\leq s_{2}-\frac{n}{p}$\ and $0<\beta \leq
\infty $. Then%
\begin{equation*}
B_{p,\beta }^{s_{2}}\hookrightarrow \dot{K}_{p,s}^{0,s}B_{\beta
}^{s_{2}}\hookrightarrow B_{s,\beta }^{s_{1}}.
\end{equation*}
\end{corollary}

\begin{proof}
By Corollary \ref{embeddings3-cor}, the desired embeddings are an immediate
consequence of the fact that 
\begin{equation*}
B_{p,\beta }^{s_{2}}=\dot{K}_{p}^{0,p}B_{\beta }^{s_{2}}\hookrightarrow \dot{%
K}_{p,s}^{0,s}B_{\beta }^{s_{2}}\hookrightarrow \dot{K}_{s,s}^{0,s}B_{\beta
}^{s_{1}}=B_{s,\beta }^{s_{1}}.
\end{equation*}%
The proof is complete.
\end{proof}

Let us define 
\begin{equation*}
\sigma _{p}=\frac{n}{\min (1,p)}-n\quad \text{and}\quad \overline{p}=\max
(1,p).
\end{equation*}%
By Theorem \ref{embeddings4-lorentz} and the Sobolev-type embeddings\ %
\eqref{Sobolev-emb-lorentz}, we get%
\begin{equation*}
\dot{K}_{p,r}^{\alpha ,q}B_{\beta }^{s_{2}}\hookrightarrow B_{p,\beta
}^{s_{1}}\hookrightarrow B_{\overline{p},1}^{0}
\end{equation*}%
for any $0<p<\infty ,0<q,\beta ,r\leq \infty ,\alpha >0,\sigma
_{p}<s_{1}\leq s_{2}-\alpha $. Let $\{\varphi _{j}\}_{j\in \mathbb{N}_{0}}$
be the smooth dyadic resolution of unity. We further conclude that%
\begin{equation*}
\big\|f\big\|_{\overline{p}}\leq \sum\limits_{j=0}^{\infty }\big\|\mathcal{F}%
^{-1}\varphi _{j}\ast f\big\|_{\overline{p}}=\big\|f\big\|_{B_{\overline{p}%
,1}^{0}}\leq c\big\|f\big\|_{\dot{K}_{p,r}^{\alpha ,q}B_{\beta }^{s_{2}}}
\end{equation*}%
This shows that under the above assumptions the elements from $\dot{K}%
_{p,r}^{\alpha ,q}B_{\beta }^{s_{2}}$\ are regular distributions.

\begin{proposition}
\textit{Let }$\alpha >0,0<s,p<\infty ,0<q,r\leq \infty $ \textit{and }$%
0<\beta \leq \infty $\textit{. }If $s>\sigma _{p}+\alpha $,\ \textit{then}%
\begin{equation*}
\dot{K}_{p,r}^{\alpha ,q}B_{\beta }^{s}\hookrightarrow L^{\overline{p}}.
\end{equation*}
\end{proposition}

Concerning embeddings into $L^{\infty }$, we have the following result.

\begin{theorem}
\label{bounded1-lorentz}\textit{Let }$\alpha \geq 0,0<p<\infty \ $\textit{%
and }$0<q,r\leq \infty $\textit{. Then}%
\begin{equation*}
\dot{K}_{p,r}^{\alpha ,q}B_{\beta }^{s}\hookrightarrow L^{\infty },
\end{equation*}%
if and only 
\begin{equation*}
s>\alpha +\frac{n}{p}\quad \text{or}\quad s=\alpha +\frac{n}{p}\text{ and }%
0<\beta \leq 1.
\end{equation*}
\end{theorem}

\begin{proof}
Let $0<p<v<\infty $. It follows from Theorem \ref{embeddings4-lorentz} that 
\begin{equation*}
\dot{K}_{p,r}^{\alpha ,q}B_{\beta }^{\alpha +\frac{n}{p}}\hookrightarrow 
\dot{K}_{v}^{\alpha ,q}B_{1}^{\alpha +\frac{n}{v}}\hookrightarrow B_{\infty
,1}^{0},
\end{equation*}%
where the second embeddings follows by Lemma \ref{Bernstein-Herz-ine1}.
Hence the result follows by the embedding $B_{\infty ,1}^{0}\hookrightarrow
L^{\infty }$; see \cite[Proposition 2.5.7]{T83}. Let $\varrho $ be a $%
C^{\infty }$ function on $\mathbb{R}$ such that $\varrho (x)=1$ for $x\leq
e^{-3}$ and $\varrho (x)=0$ for $x\geq e^{-2}$. Let $(\lambda ,\sigma )\in 
\mathbb{R}^{2}$ and%
\begin{equation*}
f_{\lambda ,\sigma }(x)=|\log |x||^{\lambda }|\log |\log |x|||^{-\sigma
}\varrho (|x|).
\end{equation*}%
Let $U_{\beta }$ be the set of $(\lambda ,\sigma )\in \mathbb{R}^{2}$ such
that:

\textbullet\ $\lambda =1-\frac{1}{\beta }$ and $\sigma >\frac{1}{\beta }$,
or $\lambda <1-\frac{1}{\beta }$, in case $1<\beta <\infty $,

\textbullet\ $\lambda =0$ and $\sigma >0$, or $\lambda <0$, in case $\beta
=1 $,

\textbullet\ $\lambda =1$ and $\sigma >1$, or $\lambda <1$, in case $\beta
=\infty .$

Let $(\lambda ,\sigma )\in \mathbb{R}^{2},0<p<\infty ,0<r,q\leq \infty
,1\leq \beta \leq \infty ,\alpha >-\frac{n}{p}$ and 
\begin{equation*}
(\lambda ,\sigma )\in U_{\beta }.
\end{equation*}%
Let $f_{\lambda ,\sigma }$ be the function defined by %
\eqref{triebel-function}; see below. In Subsection 6.3, we will prove that $%
f_{\lambda ,\sigma }\in \dot{K}_{p,r}^{\alpha ,q}B_{\beta }^{\alpha +\frac{n%
}{p}}$ if and only if $(\lambda ,\sigma )\in U_{\beta }$. We choose $\lambda
=1-\frac{1}{\beta }$ and $\sigma =\frac{1}{2}+\frac{1}{\beta }$. Then 
\begin{equation*}
f_{1-\frac{1}{\beta },\frac{1}{2}+\frac{1}{\beta }}\in \dot{K}_{p,r}^{\alpha
,q}B_{\beta }^{\alpha +\frac{n}{p}},
\end{equation*}%
but $f\notin L^{\infty }(\mathbb{R}^{n})$.
\end{proof}

The following statement holds by Theorem \ref{embeddings3-lorentz} and the
fact that $\dot{K}_{p,p}^{0,p}B_{\beta }^{s_{2}}=B_{p,\beta }^{s_{2}}$.

\begin{theorem}
\label{embeddings5-lorentz}\textit{Let }$\alpha ,s_{1},s_{2}\in \mathbb{R}%
,0<s,p<\infty ,0<r_{1}\leq \infty ,s_{1}-\frac{n}{s}-\alpha \leq s_{2}-\frac{%
n}{p}$\textit{\ and }$0<\beta ,r\leq \infty $\textit{.\ I}f 
\begin{equation*}
-\frac{n}{s}<\alpha \leq 0,\quad 0<p\leq s<\infty
\end{equation*}%
or%
\begin{equation*}
-\frac{n}{s}<\alpha \leq \frac{n}{p}-\frac{n}{s},\quad 0<s<p<\infty ,
\end{equation*}%
then%
\begin{equation*}
B_{p,\beta }^{s_{2}}\hookrightarrow \dot{K}_{s,r_{1}}^{\alpha ,\theta
}B_{\beta }^{s_{1}},
\end{equation*}%
where%
\begin{equation*}
\theta =\left\{ 
\begin{array}{ccc}
p, & \text{if} & \alpha =\frac{n}{p}-\frac{n}{s},\quad s<p\quad \text{or}%
\quad \alpha =0,\quad p\leq s, \\ 
r, & \text{if} & -\frac{n}{s}<\alpha <\frac{n}{p}-\frac{n}{s},\quad s<p\quad 
\text{or}\quad -\frac{n}{s}<\alpha <0,\quad p\leq s.%
\end{array}%
\right.
\end{equation*}
\end{theorem}

As a consequence, one obtains the following corollary.

\begin{corollary}
\label{embeddings-besov2}Let $s_{1},s_{2}\in \mathbb{R},0<\max (p,r_{1})\leq
s<\infty ,s_{1}-\frac{n}{s}\leq s_{2}-\frac{n}{p}$\ and $0<\beta \leq \infty 
$. Then%
\begin{equation}
B_{p,\beta }^{s_{2}}\hookrightarrow \dot{K}_{s,r_{1}}^{0,p}B_{\beta
}^{s_{1}}\hookrightarrow B_{s,\beta }^{s_{1}}.  \label{new-emb}
\end{equation}
\end{corollary}

\begin{proof}
To prove \eqref{new-emb} it is sufficient to\ choose in Theorem \ref%
{embeddings5-lorentz}, $\theta =p$ and $\alpha =0$. Then the desired
embedding is an immediate consequence of the fact that 
\begin{equation*}
\dot{K}_{s,r_{1}}^{0,p}B_{\beta }^{s_{1}}\hookrightarrow \dot{K}%
_{s,s}^{0,s}B_{\beta }^{s_{1}}=B_{s,\beta }^{s_{1}}.
\end{equation*}
\end{proof}

\begin{remark}
Corollaries \ref{embeddings-besov1}\ and \ref{embeddings-besov2} extend and
improve\ Sobolev embeddings of Besov spaces.
\end{remark}

\subsection{Sobolev embeddings for the spaces $\dot{K}_{p,r}^{\protect\alpha %
,q}F_{\protect\beta }^{s}$}

It is well-known that%
\begin{equation}
F_{q,\infty }^{s_{2}}\hookrightarrow F_{s,\beta }^{s_{1}}
\label{Sobolev-Tr-Li-lorentz}
\end{equation}%
if $s_{1}-\frac{n}{s}=s_{2}-\frac{n}{q}$, where $0<q<s<\infty $ and $0<\beta
\leq \infty $; see, e.g., \cite[Theorem 2.7.1]{T83}. In this subsection, we
generalize these embeddings to Lorentz-Herz-type Triebel-Lizorkin spaces. We
need the Sobolev embeddings properties of the sequence\ spaces $\dot{K}%
_{p,r_{1}}^{\alpha ,r}f_{\infty }^{s}$. Put $c_{n}=1+\lfloor \log _{2}(2%
\sqrt{n}+1)\rfloor $, which will be fixed throughout this section.

\begin{theorem}
\label{embeddings3 copy(1)-lorentz}\textit{Let }$\alpha _{1},\alpha
_{2},s_{1},s_{2}\in \mathbb{R},0<s,r,p,q<\infty ,0<\theta ,r_{1}\leq \infty
,\alpha _{1}>-\frac{n}{s}\ $\textit{and }$\alpha _{2}>-\frac{n}{p}$. \textit{%
We suppose that\ }%
\begin{equation}
s_{1}-\frac{n}{s}-\alpha _{1}=s_{2}-\frac{n}{p}-\alpha _{2}.
\label{newexp1.1-lorentz}
\end{equation}%
\textit{Let }$0<p<s<\infty $ and $\alpha _{2}>\alpha _{1}$. Then%
\begin{equation}
\dot{K}_{p,\infty }^{\alpha _{2},r}f_{\infty }^{s_{2}}\hookrightarrow \dot{K}%
_{s,r_{1}}^{\alpha _{1},q}f_{\theta }^{s_{1}},
\label{Sobolev-emb1.1-lorentz}
\end{equation}%
if and only if $0<r\leq q<\infty $.
\end{theorem}

\begin{proof}
First the necessity of \eqref{newexp1.1-lorentz} follows by using the same
type of arguments as in the proof of Theorem \ref{embeddings3-lorentz}. The
rest of the proof is in two steps

\textit{Step 1.} Let us prove that $0<r\leq q<\infty $ is necessary. In the
calculations below we consider the 1-dimensional case for simplicity. For
any $v\in \mathbb{N}_{0}$ and $N\geq 1$, we put 
\begin{equation*}
\lambda _{v,m}^{N}=\left\{ 
\begin{array}{ccc}
2^{-(s_{1}-\frac{1}{s}-\alpha _{1}+\frac{n}{2})v}\sum_{i=1}^{N}\chi
_{i}(2^{v-1}), & \text{if} & m=1, \\ 
0, &  & \text{otherwise,}%
\end{array}%
\right.
\end{equation*}%
$\lambda ^{N}=\{\lambda _{v,m}^{N}:v\in \mathbb{N}_{0},m\in \mathbb{Z}\}$.
Let $0<\beta <\infty $. We have 
\begin{equation*}
\big\|\lambda ^{N}\big\|_{\dot{K}_{s,r_{1}}^{\alpha _{1},q}f_{\beta
}^{s_{1}}}^{q}=\sum_{k=-\infty }^{\infty }2^{\alpha _{1}kq}\Big\|\Big(%
\sum_{v=0}^{\infty }\sum\limits_{m\in \mathbb{Z}}2^{v(s_{1}+\frac{n}{2}%
)\beta }|\lambda _{v,m}^{N}|^{\beta }\chi _{v,m}\Big)^{1/\beta }\chi _{k}%
\Big\|_{L^{s,r_{1}}}^{q}.
\end{equation*}%
We can rewrite the last statement as follows: 
\begin{align*}
& \sum_{k=1-N}^{0}2^{\alpha _{1}kq}\Big\|\Big(\sum_{v=1}^{N}2^{(\frac{1}{s}%
+\alpha _{1})v\beta }\chi _{v,1}\Big)^{1/\beta }\chi _{k}\Big\|%
_{L^{s,r_{1}}}^{q} \\
=& \sum_{k=1-N}^{0}2^{\alpha _{1}kq}\big\|2^{(\frac{1}{s}+\alpha
_{1})(1-k)}\chi _{1-k,1}\big\|_{L^{s,r_{1}}}^{q} \\
=& c\text{ }N,
\end{align*}%
where the constant $c>0$ does not depend on $N$. Now%
\begin{equation*}
\big\|\lambda ^{N}\big\|_{\dot{K}_{p,r_{2}}^{\alpha _{2},r}f_{\theta
}^{s_{2}}}^{r}=\sum_{k=-\infty }^{\infty }2^{\alpha _{2}kr}\Big\|\Big(%
\sum_{v=0}^{\infty }2^{v(s_{2}+\frac{n}{2})\theta }|\lambda
_{v,1}^{N}|^{\theta }\chi _{v,1}\Big)^{1/\theta }\chi _{k}\Big\|%
_{L^{p,r_{2}}}^{r}.
\end{equation*}%
Again we can rewrite the last statement as follows:%
\begin{align*}
& \sum_{k=1-N}^{0}2^{\alpha _{2}kr}\Big\|\Big(\sum_{v=1}^{N}2^{(s_{2}-s_{1}+%
\frac{1}{s}+\alpha _{1})v\theta }\chi _{v,1}\Big)^{1/\theta }\chi _{k}\Big\|%
_{L^{p,r_{2}}}^{r} \\
=& \sum_{k=1-N}^{0}2^{\alpha _{2}kr}\big\|2^{(s_{2}-s_{1}+\frac{1}{s}+\alpha
_{1})(1-k)}\chi _{1-k,1}\big\|_{L^{p,r_{2}}}^{r} \\
=& c\text{ }N,
\end{align*}%
where the constant $c>0$ does not depend on $N$ and we have used %
\eqref{est-function1}. If the embeddings \eqref{Sobolev-emb1.1-lorentz}
holds then for any $N\in \mathbb{N}$, $N^{\frac{1}{q}-\frac{1}{r}}\leq C$.
Thus, we conclude that $0<r\leq q<\infty $ must necessarily hold by letting $%
N\rightarrow +\infty $.

\textit{Step 2.} We consider the sufficiency of the conditions. In view of
the embedding $\ell ^{r}\hookrightarrow \ell ^{q}$, it is sufficient to
prove that 
\begin{equation*}
\dot{K}_{p,\infty }^{\alpha _{2},r}f_{\infty }^{s_{2}}\hookrightarrow \dot{K}%
_{s,r_{1}}^{\alpha _{1},r}f_{\theta }^{s_{1}}.
\end{equation*}%
Let $\lambda \in \dot{K}_{p,\infty }^{\alpha _{2},r}f_{\theta }^{s_{2}}$. We
have%
\begin{align}
\big\|\lambda \big\|_{\dot{K}_{s,r_{1}}^{\alpha _{1},r}f_{\theta
}^{s_{1}}}\lesssim & \Big(\sum\limits_{k=-\infty }^{0}2^{k\alpha _{1}r}\Big\|%
\Big(\sum_{v=0}^{\infty }2^{v(s_{1}+\frac{n}{2})\theta }\sum_{m\in \mathbb{Z}%
^{n}}|\lambda _{v,m}|^{\theta }\chi _{v,m}\chi _{k}\Big)^{1/\theta }\Big\|%
_{L^{s,r_{1}}}^{r}\Big)^{1/r}  \label{est2-lorentz} \\
& +\Big(\sum\limits_{k=1}^{\infty }2^{k\alpha _{1}r}\Big\|\Big(%
\sum_{v=0}^{\infty }2^{v(s_{1}+\frac{n}{2})\theta }\sum_{m\in \mathbb{Z}%
^{n}}|\lambda _{v,m}|^{\theta }\chi _{v,m}\chi _{k}\Big)^{1/\theta }\Big\|%
_{L^{s,r_{1}}}^{r}\Big)^{1/r}.  \label{est 1-lorentz}
\end{align}%
The right-hand side of \eqref{est2-lorentz} can be estimated\ from above by%
\begin{align*}
& c\Big(\sum\limits_{k=-\infty }^{0}2^{k\alpha _{1}r}\Big\|\Big(%
\sum_{v=0}^{1+c_{n}-k}2^{v(s_{1}+\frac{n}{2})\theta }\sum_{m\in \mathbb{Z}%
^{n}}|\lambda _{v,m}|^{\theta }\chi _{v,m}\chi _{k}\Big)^{1/\theta }\Big\|%
_{L^{s,r_{1}}}^{r}\Big)^{1/r} \\
& +c\Big(\sum\limits_{k=-\infty }^{0}2^{k\alpha _{1}r}\Big\|\Big(%
\sum_{v=2+c_{n}-k}^{\infty }2^{v(s_{1}+\frac{n}{2})\theta }\sum_{m\in 
\mathbb{Z}^{n}}|\lambda _{v,m}|^{\theta }\chi _{v,m}\chi _{k}\Big)^{1/\theta
}\Big\|_{L^{s,r_{1}}}^{r}\Big)^{1/r} \\
& =I+II.
\end{align*}%
\textit{Estimation of }$I$\textit{.} Let $x\in R_{k}\cap Q_{v,m}$ and $y\in
Q_{v,m}$. We have $|x-y|\leq 2\sqrt{n}2^{-v}<2^{c_{n}-v}$ and from this it
follows that $|y|<2^{c_{n}-v}+2^{k}\leq 2^{c_{n}-v+2}$, which implies that $%
y $ is located in the ball $B(0,2^{c_{n}-v+2})$. This leads to%
\begin{equation*}
|\lambda _{v,m}|^{t}\chi _{R_{k}\cap Q_{v,m}}(x)\leq 2^{nv}\int_{\mathbb{R}%
^{n}}|\lambda _{v,m}|^{t}\chi _{v,m}(y)dy\leq
2^{nv}\int_{B(0,2^{c_{n}-v+2})}|\lambda _{v,m}|^{t}\chi _{v,m}(y)dy,
\end{equation*}%
where $t>0$. Therefore for any $x\in R_{k}$, we obtain that 
\begin{align*}
\sum_{m\in \mathbb{Z}^{n}}|\lambda _{v,m}|^{t}\chi _{v,m}(x)\leq &
2^{nv}\int_{B(0,2^{c_{n}-v+2})}\sum_{m\in \mathbb{Z}^{n}}|\lambda
_{v,m}|^{t}\chi _{v,m}(y)dy \\
=& 2^{nv}\Big\|\sum_{m\in \mathbb{Z}^{n}}|\lambda _{v,m}|\chi _{v,m}\chi
_{B(0,2^{c_{n}-v+2})}\Big\|_{L^{t,t}}^{t}.
\end{align*}%
This yields%
\begin{align*}
& 2^{\alpha _{1}k}\Big\|\Big(\sum_{v=0}^{1+c_{n}-k}2^{v(s_{1}+\frac{n}{2}%
)\theta }\sum_{m\in \mathbb{Z}^{n}}|\lambda _{v,m}|^{\theta }\chi _{v,m}\chi
_{k}\Big)^{1/\theta }\Big\|_{L^{s,r_{1}}} \\
\lesssim & 2^{(\alpha _{1}+\frac{n}{s})k}\Big(%
\sum_{v=0}^{1+c_{n}-k}2^{v(s_{1}+\frac{n}{2}+\frac{n}{t})\theta }\Big\|%
\sum_{m\in \mathbb{Z}^{n}}|\lambda _{v,m}|\chi _{v,m}\chi
_{B(0,2^{c_{n}-v+2})}\Big\|_{L^{t,t}}^{\theta }\Big)^{1/\theta },
\end{align*}%
with the help of \eqref{est-function1}, where the implicit\ constant is
independent of $k$. We may choose $t>0$ such that $\frac{1}{t}>\max (\frac{1%
}{p},\frac{1}{r_{2}},\frac{1}{p}+\frac{\alpha _{2}}{n})$. Put $\varkappa
=\min (1,t)$ and 
\begin{equation*}
\frac{1}{t}=\frac{1}{p}+\frac{1}{h}=\frac{1}{\infty }+\frac{1}{t},\quad 
\frac{n}{h}=\alpha _{2}+\frac{n}{d},\quad 0<d<\infty .
\end{equation*}%
Using \eqref{newexp1.1-lorentz} and Lemmas \ref{Lp-estimate} and \ref%
{lem:lq-inequality} we estimate $I^{r}$ by%
\begin{align}
& c\sum_{v=0}^{\infty }2^{v(s_{2}-\frac{n}{p}-\alpha _{2}+\frac{n}{t}+\frac{n%
}{2})r}\Big\|\sum_{m\in \mathbb{Z}^{n}}|\lambda _{v,m}|\chi _{v,m}\chi
_{B(0,2^{c_{n}-v+2})}\Big\|_{L^{t,t}}^{r}  \notag \\
\leq & c\sum_{v=0}^{\infty }2^{v(s_{2}-\frac{n}{p}-\alpha _{2}+\frac{n}{t}+%
\frac{n}{2})r}\Big(\sum_{i\leq -v}\Big\|\sum_{m\in \mathbb{Z}^{n}}|\lambda
_{v,m}|\chi _{v,m}\chi _{i+c_{n}+2}\Big\|_{L^{t,t}}^{\varkappa }\Big)%
^{r/\varkappa }.  \label{sobolev-lorentz}
\end{align}%
By H\"{o}lder's inequality and \eqref{est-function1}, we obtain%
\begin{align*}
& \Big\|\sum_{m\in \mathbb{Z}^{n}}|\lambda _{v,m}|\chi _{v,m}\chi
_{i+c_{n}+2}\Big\|_{L^{t,t}} \\
& \lesssim \Big\|\sum_{m\in \mathbb{Z}^{n}}|\lambda _{v,m}|\chi _{v,m}\chi
_{i+c_{n}+2}\Big\|_{L^{p,\infty }}\big\|\chi _{i+c_{n}+2}\big\|_{L^{h,t}} \\
& \lesssim 2^{i(\frac{n}{d}+\alpha _{2})}\Big\|\sum_{m\in \mathbb{Z}%
^{n}}|\lambda _{v,m}|\chi _{v,m}\chi _{i+c_{n}+2}\Big\|_{L^{p,\infty }} \\
& \lesssim 2^{i(\frac{n}{d}+\alpha _{2})-(s_{2}+\frac{n}{2})v}\Big\|%
\sup_{j\in \mathbb{N}_{0}}2^{(s_{2}+\frac{n}{2})j}\sum_{m\in \mathbb{Z}%
^{n}}|\lambda _{j,m}|\chi _{j,m}\chi _{i+c_{n}+2}\Big\|_{L^{p,\infty }},
\end{align*}%
where the implicit constant is independent of $i$\ and $v$. Inserting this
estimate in \eqref{sobolev-lorentz} and applying Lemma \ref%
{lem:lq-inequality}, we get 
\begin{align*}
I^{r}& \lesssim \sum_{v=0}^{\infty }2^{v\frac{nr}{d}}\Big(\sum_{i\leq
-v}2^{i(\frac{n}{d}+\alpha _{2})\varkappa }\Big\|\sup_{j\in \mathbb{N}%
_{0}}2^{(s_{2}+\frac{n}{2})j}\sum_{m\in \mathbb{Z}^{n}}|\lambda _{j,m}|\chi
_{j,m}\chi _{i+c_{n}+2}\Big\|_{L^{p,\infty }}^{\varkappa }\Big)^{r/\varkappa
} \\
& \lesssim \sum_{i=0}^{\infty }2^{-\alpha _{2}ir}\Big\|\sup_{j\in \mathbb{N}%
_{0}}\Big(2^{(s_{2}+\frac{n}{2})j}\sum_{m\in \mathbb{Z}^{n}}|\lambda
_{j,m}|\chi _{j,m}\chi _{-i+c_{n}+2}\Big)\Big\|_{L^{p,\infty }}^{r} \\
& \lesssim \big\|\lambda \big\|_{\dot{K}_{p,\infty }^{\alpha
_{2},r}f_{\infty }^{s_{2}}}^{r}.
\end{align*}

\textit{Estimation of }$II$\textit{.} Since $\alpha _{2}>\alpha _{1}$, by %
\eqref{seeger} we obtain 
\begin{align}
& 2^{k\alpha _{1}}\Big\|\Big(\sum_{v=2+c_{n}-k}^{\infty }2^{v(s_{1}+\frac{n}{%
2})\theta }\sum_{m\in \mathbb{Z}^{n}}|\lambda _{v,m}|^{\theta }\chi
_{v,m}\chi _{k}\Big)^{1/\theta }\Big\|_{L^{s,r_{1}}}  \notag \\
& \lesssim \sup_{v\in \mathbb{N}_{0}}\Big\|2^{v(s_{1}+\alpha _{2}-\alpha
_{1}+\frac{n}{2})+k\alpha _{2}}\sum_{m\in \mathbb{Z}^{n}}|\lambda
_{v,m}|\chi _{v,m}\chi _{k}\Big\|_{L^{s,r_{1}}},  \label{Lorentz1}
\end{align}%
where the implicit constant is independent of $k$. We see that it suffices
to show that \eqref{Lorentz1} can be estimated from above by 
\begin{equation*}
c\Big\|\sup_{v\in \mathbb{N}_{0}}\Big(2^{v(s_{2}+\frac{n}{2})+k\alpha
_{2}}\sum_{m\in \mathbb{Z}^{n}}|\lambda _{v,m}|\chi _{v,m}\chi _{k}\Big)%
\Big\|_{L^{p,\infty }}
\end{equation*}%
for any $k\leq 0$, where the positive constant $c$ is independent of $k$.
Observe that%
\begin{equation*}
\left\vert 2^{-v}m\right\vert \leq \left\vert x-2^{-v}m\right\vert
+\left\vert x\right\vert \leq \sqrt{n}2^{-v}+2^{k}\leq 2^{k+1}
\end{equation*}%
and 
\begin{equation*}
\left\vert 2^{-v}m\right\vert \geq \left\vert \left\vert
x-2^{-v}m\right\vert -\left\vert x\right\vert \right\vert \geq 2^{k-1}-\sqrt{%
n}2^{-v}\geq 2^{k-2}
\end{equation*}%
if $x\in R_{k}\cap Q_{v,m}$ and $v\geq c_{n}+2-k$. Hence $m$ is located in%
\begin{equation*}
A_{k+v}=\{m\in \mathbb{Z}^{n}:2^{k+v-2}\leq \left\vert m\right\vert \leq
2^{k+v+1}\}.
\end{equation*}%
Observe that $cardA_{k+v}\leq 2^{2n(k+v+1)}$. Let%
\begin{equation*}
\tilde{\lambda}_{v,m_{1}}^{1,k}=\max_{m\in A_{k+v}}\left\vert \lambda
_{v,m}\right\vert ,\quad m_{1}\in \mathbb{Z}^{n}
\end{equation*}%
and (decreasing rearrangement of $\{\lambda _{v,m}\}_{m\in A_{k+v}}$) 
\begin{equation*}
\tilde{\lambda}_{v,m_{j}}^{j,k}=\max_{m^{i}\in
A_{k+v},i=1,...,j}\sum_{i=1}^{j}\left\vert \lambda _{v,m^{i}}\right\vert
-\sum_{i=1}^{j-1}\tilde{\lambda}_{v,m_{i}}^{i,k},\quad m_{j}\in \mathbb{Z}%
^{n},j\geq 2.
\end{equation*}%
Then%
\begin{equation*}
2^{v(s_{2}+\frac{n}{2})+k\alpha _{2}}\sum_{m\in A_{k+v}}|\lambda _{v,m}|\chi
_{v,m}=2^{v(s_{2}+\frac{n}{2})+k\alpha _{2}}\sum_{i=1}^{cardA_{k+v}}\tilde{%
\lambda}_{v,m_{i}}^{i,k}\chi _{v,m_{i}}=\varpi _{v,k}.
\end{equation*}%
It is not difficult to see that%
\begin{equation*}
\varpi _{v,k}^{\ast }(t)=2^{v(s_{2}+\frac{n}{2})+k\alpha
_{2}}\sum_{i=1}^{cardA_{k+v}}\tilde{\lambda}_{v,m_{i}}^{i,k}\tilde{\chi}%
_{[B_{i-1,v},B_{i,v})}(t),
\end{equation*}%
with 
\begin{equation*}
B_{0,v}=0,\quad B_{i,v}=\sum_{j=1}^{i}\left\vert Q_{v,m_{j}}\right\vert
=2^{-vn}i,\quad i=1,...,cardA_{k+v},
\end{equation*}%
where $\tilde{\chi}_{[B_{i-1,v},B_{i,v})}$ is the characteristic function of
the interval $[B_{i-1,v},B_{i,v})$.\ In addition, we have 
\begin{equation*}
Q_{k,m}\subset \breve{R}_{k}\quad \text{if}\quad v\geq c_{n}+2-k\quad \text{%
and}\quad m\in A_{k+v},
\end{equation*}
where $\breve{R}_{k}=\cup _{i=-2}^{3}R_{k+i}$, and 
\begin{align*}
\varpi _{v,k}& \leq 2^{v(s_{2}+\frac{n}{2})+k\alpha _{2}}\sum_{m\in \mathbb{Z%
}^{n}}|\lambda _{v,m}|\chi _{v,m}\chi _{\breve{R}_{k}} \\
& \leq \sup_{v\in \mathbb{N}_{0}}\big(2^{v(s_{2}+\frac{n}{2})+k\alpha
_{2}}\sum_{m\in \mathbb{Z}^{n}}|\lambda _{v,m}|\chi _{v,m}\chi _{\breve{R}%
_{k}}\big) \\
& =\digamma _{k}.
\end{align*}%
Using \eqref{newexp1.1-lorentz}, we get 
\begin{equation}
\Big\|2^{v(s_{1}+\alpha _{2}-\alpha _{1}-s_{2})}\varpi _{v,k}\Big\|%
_{L^{s,r_{1}}}^{\theta _{1}}=\Big\|2^{v(\frac{n}{s}-\frac{n}{p})\theta
_{1}}\varpi _{v,k}^{\theta _{1}}\Big\|_{L^{s/\theta _{1},r_{1}/\theta _{1}}}
\label{Lorentz2}
\end{equation}%
for any $0<\theta _{1}<\infty $ and any $v\in \mathbb{N}_{0},k\in \mathbb{Z}$%
. We choose $\theta _{1}<\min (s,r_{1})$. Using duality, the right-hand side
of \eqref{Lorentz2} is dominated by%
\begin{equation*}
c\sup \int_{\mathbb{R}^{n}}2^{v(\frac{n}{s}-\frac{n}{p})\theta _{1}}(\varpi
_{v,k}(x))^{\theta _{1}}g(x)dx,
\end{equation*}%
where the supremum is taken\ over all $g\in L^{(s/\theta _{1})^{\prime
},(r_{1}/\theta _{1})^{\prime }}$ such that $\big\|g\big\|_{L^{(s/\theta
_{1})^{\prime },(r_{1}/\theta _{1})^{\prime }}}\leq 1$. It follows from
Lemma \ref{Hardy-Littlewood inequality} that 
\begin{equation*}
2^{v(\frac{n}{s}-\frac{n}{p})\theta _{1}}\int_{\mathbb{R}^{n}}(\varpi
_{v,k}(x))^{\theta _{1}}g(x)dx\leq 2^{v(\frac{n}{s}-\frac{n}{p})\theta
_{1}}\int_{0}^{\infty }(\varpi _{v,k}^{\ast }(t))^{\theta _{1}}g^{\ast
}(t)dt.
\end{equation*}%
We have%
\begin{equation}
\int_{0}^{\infty }(\varpi _{v,k}^{\ast }(t))^{\theta _{1}}g^{\ast
}(t)dt=\int_{0}^{2^{-vn}}(\varpi _{v,k}^{\ast }(t))^{\theta _{1}}g^{\ast
}(t)dt+\sum_{l=0}^{\infty }\int_{2^{(l-v)n}}^{2^{(l-v)n+n}}(\varpi
_{v,k}^{\ast }(t))^{\theta _{1}}g^{\ast }(t)dt.  \label{Lorentz2.1}
\end{equation}

We see that $\varpi _{v,k}^{\ast }$ is constant in $[0,2^{-vn})$\ and $%
\varpi _{v,k}^{\ast }\leq \digamma _{k}^{\ast }$. Then first term on the
right-hand side of \eqref{Lorentz2.1} is bounded by%
\begin{align*}
(\varpi _{v,k}^{\ast }(2^{-vn-1}))^{\theta _{1}}\int_{0}^{2^{-vn}}g^{\ast
}(t)dt& \leq 2^{-vn}(\varpi _{v,k}^{\ast }(2^{-vn-1}))^{\theta _{1}}g^{\ast
\ast }(2^{-vn}) \\
& \leq 2^{-vn}(\digamma _{k}^{\ast }(2^{-vn-1}))^{\theta _{1}}g^{\ast \ast
}(2^{-vn}) \\
& \leq 2^{-vn(1-\frac{1}{p})\theta _{1}}\sup_{v\in \mathbb{N}_{0}}\big(2^{-%
\frac{vn}{p}}\digamma _{k}^{\ast }(2^{-vn-1})\big)^{\theta _{1}}g^{\ast \ast
}(2^{-vn}) \\
& \leq 2^{v(\frac{n}{p}-\frac{n}{s})\theta _{1}}\sup_{v\in \mathbb{Z}}\big(%
2^{-\frac{vn}{p}}\digamma _{k}^{\ast }(2^{-vn-1})\big)^{\theta
_{1}}\sup_{v\in \mathbb{Z}}\big(2^{-vn(1-\frac{\theta _{1}}{s})}g^{\ast \ast
}(2^{-vn})\big) \\
& \leq 2^{v(\frac{n}{p}-\frac{n}{s})\theta _{1}}\big\|\digamma _{k}\big\|%
_{L^{p,\infty }}^{\theta _{1}}\big\|g\big\|_{L^{(s/\theta _{1})^{\prime
},\infty }}.
\end{align*}%
Now, the second term on the right-hand side of \eqref{Lorentz2.1} can be
estimated from above by%
\begin{align}
& c\sum_{l=0}^{\infty }(\digamma _{k}^{\ast }(2^{(l-v)n}))^{\theta
_{1}}2^{(l-v)n}g^{\ast }(2^{(l-v)n})  \notag \\
& =c2^{v(\frac{n}{p}-\frac{n}{s})\theta _{1}}\sum_{l=0}^{\infty }(\digamma
_{k}^{\ast }(2^{(l-v)n}))^{\theta _{1}}2^{(l-v)n}2^{v(\frac{n}{s}-\frac{n}{p}%
)\theta _{1}}g^{\ast }(2^{(l-v)n})  \notag \\
& =c2^{v(\frac{n}{p}-\frac{n}{s})\theta _{1}}\sum_{l=0}^{\infty }2^{(l-v)n%
\frac{\theta _{1}}{p}}(\digamma _{k}^{\ast }(2^{(l-v)n}))^{\theta
_{1}}2^{(l-v)n(1-\frac{\theta _{1}}{p})}2^{v(\frac{n}{s}-\frac{n}{p})\theta
_{1}}g^{\ast }(2^{(l-v)n}).  \label{Lorentz3}
\end{align}%
The term inside the sum in \eqref{Lorentz3} is dominated by%
\begin{align}
& \sup_{v\in \mathbb{N}_{0}}\big(2^{(l-v)\frac{n}{p}}(\digamma _{k}^{\ast
}(2^{(l-v)n})\big)^{\theta _{1}}\sup_{v\in \mathbb{N}_{0}}\Big(2^{(l-v)n(1-%
\frac{\theta _{1}}{p})}2^{v(\frac{n}{s}-\frac{n}{p})\theta _{1}}g^{\ast
}(2^{(l-v)n})\Big)  \notag \\
& \leq 2^{l(\frac{n}{s}-\frac{n}{p})\theta _{1}}\big\|\digamma _{k}\big\|%
_{L^{p,\infty }}^{\theta _{1}}\sup_{v\in \mathbb{N}_{0}}\big(2^{(l-v)n(1-%
\frac{\theta _{1}}{s})}g^{\ast }(2^{(l-v)n})\big)  \notag \\
& \leq 2^{l(\frac{n}{s}-\frac{n}{p})\theta _{1}}\big\|\digamma _{k}\big\|%
_{L^{p,\infty }}^{\theta _{1}}\big\|g\big\|_{L^{(s/\theta _{1})^{\prime
},\infty }}.  \label{Lorentz4}
\end{align}%
We insert \eqref{Lorentz4} in \eqref{Lorentz3} we get \eqref{Lorentz2} is
bounded by $\big\|\digamma _{k}\big\|_{L^{p,\infty }}^{\theta _{1}}$. This
leads to 
\begin{equation*}
II\lesssim \big\|\lambda \big\|_{\dot{K}_{p,\infty }^{\alpha
_{2},r}f_{\infty }^{s_{2}}}.
\end{equation*}

\textit{Estimate of \eqref{est 1-lorentz}}. The arguments here are quite
similar to those used in the estimation of $II$. The proof is complete.
\end{proof}

Now, we deal with $\alpha _{1}=\alpha _{2}$ in Theorem \ref{embeddings3
copy(1)-lorentz}.

\begin{theorem}
\label{embeddings3 copy(2)-lorentz}\textit{Let }$\alpha ,s_{1},s_{2}\in 
\mathbb{R},0<s,r,p,q<\infty ,0<\theta ,r_{1}\leq \infty $ \textit{and }$%
\alpha >-\frac{n}{s}$. \textit{We suppose that\ }%
\begin{equation}
s_{1}-\frac{n}{s}=s_{2}-\frac{n}{p}.  \label{newexp1.1-lorentz-bis}
\end{equation}%
\textit{Let }$0<p<s<\infty $. Then%
\begin{equation*}
\dot{K}_{p,r_{1}}^{\alpha ,r}f_{\infty }^{s_{2}}\hookrightarrow \dot{K}%
_{s,r_{1}}^{\alpha ,q}f_{\theta }^{s_{1}},
\end{equation*}%
if and only if $0<r\leq q<\infty $.
\end{theorem}

\begin{proof}
In view the proof of Theorem \ref{embeddings3 copy(1)-lorentz}, we consider
only the sufficiency of the conditions and we employ the same notations as
in such theorem, but with $\alpha _{1}=\alpha _{2}$. We only need to
estimate\ $II$ of Theorem \ref{embeddings3 copy(1)-lorentz}. In view of the
embedding $\ell ^{r}\hookrightarrow \ell ^{q}$, it is sufficient to prove
that 
\begin{equation*}
\dot{K}_{p,r_{1}}^{\alpha ,r}f_{\infty }^{s_{2}}\hookrightarrow \dot{K}%
_{s,r_{1}}^{\alpha ,r}f_{\theta }^{s_{1}}.
\end{equation*}%
We can suppose that $\theta \leq p$, since the opposite case can be obtained
by the fact that\ $\ell ^{p}\hookrightarrow \ell ^{\theta }$, if $p\leq
\theta $. Let $\lambda \in \dot{K}_{p,r_{1}}^{\alpha ,r}f_{\infty }^{s_{2}}$%
. Here the estimates are inspired by \cite{ST19}\ and\ \cite{Vybiral08}. We
distinguish two cases.

\textit{Case 1. }$\theta <r_{1}$. We need to prove that 
\begin{equation}
2^{k\alpha }\Big\|\Big(\sum_{v=c_{n}+2-k}^{\infty }2^{v(s_{1}+\frac{n}{2}%
)\theta }\sum_{m\in \mathbb{Z}^{n}}|\lambda _{v,m}|^{\theta }\chi _{v,m}\chi
_{k}\Big)^{1/\theta }\Big\|_{L^{s,r_{1}}}\lesssim \big\|\digamma _{k}\big\|%
_{L^{p,r_{1}}},  \label{Lorentz10}
\end{equation}%
where the implicit constant is independent of $k$. The left-hand side of %
\eqref{Lorentz10} with power $\theta $ can be estimated from above by%
\begin{equation}
\Big\|\Big(\sum_{v=c_{n}+2-k}^{\infty }2^{v(s_{1}-s_{2})\theta }\varpi
_{v,k}^{\theta }\Big)^{1/\theta }\Big\|_{L^{s,r_{1}}}^{\theta }=\Big\|%
\sum_{v=c_{n}+2-k}^{\infty }2^{v(s_{1}-s_{2})\theta }\varpi _{v,k}^{\theta }%
\Big\|_{L^{s/\theta ,r_{1}/\theta }}.  \label{Lorentz6}
\end{equation}%
Using duality the right-hand side of \eqref{Lorentz6} is comparable to%
\begin{equation}
\sup_{g\in L^{(s/\theta )^{\prime },(r_{1}/\theta )^{\prime }},\big\|g\big\|%
_{L^{(s/\theta )^{\prime },(r_{1}/\theta )^{\prime }}}\leq 1}\int_{\mathbb{R}%
^{n}}\sum_{v=c_{n}+2-k}^{\infty }2^{v(s_{1}-s_{2})\theta }(\varpi
_{v,k}(x))^{\theta }g(x)dx.  \label{Lorentz5}
\end{equation}%
It follows from Lemma \ref{Hardy-Littlewood inequality} that%
\begin{equation}
\sum_{v=c_{n}+2-k}^{\infty }2^{v(s_{1}-s_{2})\theta }\int_{\mathbb{R}%
^{n}}(\varpi _{v,k}(x))^{\theta }g(x)dx\leq \sum_{v=0}^{\infty
}2^{v(s_{1}-s_{2})\theta }\int_{0}^{\infty }(\varpi _{v,k}^{\ast
}(t))^{\theta }g^{\ast }(t)dt.  \label{sum3-lorentz}
\end{equation}%
Since, $\varpi _{v,k}^{\ast }$ is constant in $[0,2^{-vn})$, we have%
\begin{align*}
\int_{0}^{\infty }(\varpi _{v,k}^{\ast }(t))^{\theta }g^{\ast }(t)dt& \leq
\varpi _{v,k}^{\ast }(2^{-vn-1})\int_{0}^{2^{-vn}}g^{\ast
}(t)dt+\sum_{l=0}^{\infty }\int_{2^{(l-v)n}}^{2^{(l-v)n+n}}(\varpi
_{v,k}^{\ast }(t))^{\theta }g^{\ast }(t)dt \\
& \lesssim \sum_{l=0}^{\infty }(\varpi _{v,k}^{\ast }(2^{(l-v)n-1}))^{\theta
}2^{(l-v)n}g^{\ast \ast }(2^{(l-v)n+n}).
\end{align*}%
Inserting this estimate in \eqref{sum3-lorentz}, we get%
\begin{align}
& \sum_{v=c_{n}+2-k}^{\infty }2^{v(s_{1}-s_{2})\theta }\int_{\mathbb{R}%
^{n}}(\varpi _{v,k}(x))^{\theta }g(x)dx  \notag \\
& \lesssim \sum_{v=c_{n}+2-k}^{\infty }\sum_{l=0}^{\infty
}2^{v(s_{1}-s_{2})\theta }(\digamma _{k}^{\ast }(2^{(l-v)n-1}))^{\theta
}2^{(l-v)n}g^{\ast \ast }(2^{(l-v)n+n}),  \label{lorentz1-bis}
\end{align}%
where the implicit constant is independent of $k$. Since $s_{1}-s_{2}=\tfrac{%
n}{s}-\frac{n}{p}$, we obtain that \eqref{lorentz1-bis} is just%
\begin{align}
& c\sum_{l=0}^{\infty }\sum_{v=c_{n}+2-k}^{\infty }(\digamma _{k}^{\ast
}((2^{(l-v)n-1}))^{\theta }2^{(l-v)n}2^{v(\frac{n}{s}-\frac{n}{p})\theta
}g^{\ast \ast }(2^{(l-v)n+n})  \notag \\
& =c\sum_{l=0}^{\infty }\sum_{v=c_{n}+2-k}^{\infty }2^{(l-v)n\frac{\theta }{p%
}}(\digamma _{k}^{\ast }(2^{(l-v)n-1}))^{\theta }2^{(l-v)n(1-\frac{\theta }{p%
})}2^{v(\frac{n}{s}-\frac{n}{p})\theta }g^{\ast \ast }(2^{(l-v)n+n}).
\label{Lorentz7}
\end{align}%
H\"{o}lder's inequality implies that the\ second sum in \eqref{Lorentz7} can
be estimated from above by%
\begin{align}
& \Big(\sum_{v=c_{n}+2-k}^{\infty }2^{(l-v)n\frac{r_{1}}{p}}(\digamma
_{k}^{\ast }(2^{(l-v)n-1}))^{r_{1}}\Big)^{\theta /r_{1}}  \notag \\
& \times \Big(\sum_{v=0}^{\infty }\big(2^{(l-v)n(1-\frac{\theta }{p})}2^{v(%
\frac{n}{s}-\frac{n}{p})\theta }g^{\ast \ast }(2^{(l-v)n+n})\big)%
^{(r_{1}/\theta _{1})^{\prime }}\Big)^{1/(r_{1}/\theta )^{\prime }}  \notag
\\
& \leq \big\|\digamma _{k}\big\|_{L^{p,r_{1}}}^{\theta }  \notag \\
& \times \Big(\sum_{v=0}^{\infty }\big(2^{(l-v)n(1-\frac{\theta }{p})}2^{v(%
\frac{n}{s}-\frac{n}{p})\theta }g^{\ast \ast }(2^{(l-v)n+n})\big)%
^{(r_{1}/\theta )^{\prime }}\Big)^{1/(r_{1}/\theta )^{\prime }}.
\label{Lorentz9}
\end{align}%
Observe that%
\begin{align}
& \sum_{v=0}^{\infty }\big(2^{(l-v)n(1-\frac{\theta }{p})}2^{v(\frac{n}{s}-%
\frac{n}{p})\theta }g^{\ast \ast }(2^{(l-v)n+n})\big)^{(r_{1}/\theta
)^{\prime }}  \notag \\
& \leq 2^{l(\frac{n}{s}-\frac{n}{p})(\frac{r_{1}}{\theta })^{\prime }\theta
}\sum_{v=0}^{\infty }\big(2^{(l-v)n(1-\frac{\theta }{p})}2^{(v-l)(\frac{n}{s}%
-\frac{n}{p})\theta }g^{\ast \ast }(2^{(l-v)n+n})\big)^{(r_{1}/\theta
)^{\prime }}  \notag \\
& \leq 2^{l(\frac{n}{s}-\frac{n}{p})(\frac{r_{1}}{\theta })^{\prime }\theta
}\sum_{v=0}^{\infty }\big(2^{(l-v)n(1-\frac{\theta }{s})}g^{\ast \ast
}(2^{(l-v)n+n})\big)^{(r_{1}/\theta )^{\prime }}  \notag \\
& \leq 2^{l(\frac{n}{s}-\frac{n}{p})(\frac{r_{1}}{\theta })^{\prime }\theta }%
\big\|g\big\|_{L^{(s/\theta )^{\prime },(r_{1}/\theta )^{\prime
}}}^{(r_{1}/\theta )^{\prime }}.  \label{Lorentz8}
\end{align}%
We insert \eqref{Lorentz8} in \eqref{Lorentz9}, we get \eqref{Lorentz5} is
bounded by $c\big\|\digamma _{k}\big\|_{L^{p,r_{1}}}^{\theta }$. This prove %
\eqref{Lorentz10}.

\textit{Case 2. }$\theta \geq r_{1}$. Let $r_{2}>0$ be such that $%
r_{2}<r_{1} $. The left-hand side of \eqref{Lorentz10} is bounded by 
\begin{equation}
\Big\|\Big(\sum_{v=0}^{\infty }2^{v(s_{1}-s_{2})r_{2}}\varpi _{v,k}^{r_{2}}%
\Big)^{1/r_{2}}\Big\|_{L^{s,r_{1}}}^{r_{2}}.  \label{Lorentz11}
\end{equation}%
Now, repeating the arguments of Case 1, we obtain that \eqref{Lorentz11} is
bounded by $c\big\|\digamma _{k}\big\|_{L^{p,r_{1}}}$. The proof is complete.
\end{proof}

Finally, we deal with $0<s\leq p<\infty $ in Theorem \ref{embeddings3
copy(1)-lorentz}.

\begin{theorem}
\label{embeddings3 copy(3)-lorentz}\textit{Let }$\alpha _{1},\alpha
_{2},s_{1},s_{2}\in \mathbb{R},0<s,p,q,r<\infty ,0<\beta ,r_{1},r_{2}\leq
\infty ,\alpha _{1}>-\frac{n}{s}\ $\textit{and }$\alpha _{2}>-\frac{n}{p}$. 
\textit{Assume }\eqref{newexp1.1-lorentz} and%
\begin{equation}
\alpha _{2}+\frac{n}{p}\geq \alpha _{1}+\frac{n}{s}.
\label{newexp1.3-lorentz}
\end{equation}%
\textit{Let }$0<s\leq p<\infty $. Then%
\begin{equation*}
\dot{K}_{p,r_{2}}^{\alpha _{2},r}f_{\theta }^{s_{2}}\hookrightarrow \dot{K}%
_{s,r_{1}}^{\alpha _{1},q}f_{\beta }^{s_{1}},
\end{equation*}%
if and only if $0<r\leq q<\infty $, where%
\begin{equation*}
\theta =\left\{ 
\begin{array}{ccc}
\beta , & \text{if} & \alpha _{2}+\frac{n}{p}=\alpha _{1}+\frac{n}{s}, \\ 
\infty , &  & \text{otherwise,}%
\end{array}%
\right.
\end{equation*}%
and $r_{1}=r_{2}$ if $s=p.$
\end{theorem}

\begin{proof}
First the necessity of \eqref{newexp1.1-lorentz} and %
\eqref{newexp1.3-lorentz} follow by using the same type of arguments as in
the proof of Theorem \ref{embeddings3-lorentz}. We need only estimate only $%
II$ of Theorem \ref{embeddings3 copy(1)-lorentz}. For simplicity, we put $%
\beta =1$. H\"{o}lder's inequality and \eqref{est-function1} imply that%
\begin{align*}
II\leq & \Big\|\sum_{v=0}^{\infty }2^{(\frac{n}{s}-\frac{n}{p}+\alpha
_{1}-\alpha _{2})(v+k)}2^{v(s_{2}+\frac{n}{2})+k\alpha _{2}}\sum_{m\in 
\mathbb{Z}^{n}}|\lambda _{v,m}|\chi _{v,m}\chi _{k}\Big\|_{L^{p,r_{2}}}^{s}
\\
\leq & \Big\|\sup_{v\in \mathbb{N}_{0}}\sum_{m\in \mathbb{Z}^{n}}2^{v(s_{2}+%
\frac{n}{2})+k\alpha _{2}}|\lambda _{v,m}|\chi _{v,m}\chi _{k}\Big\|%
_{L^{p,r_{2}}}^{s},
\end{align*}%
whenever $\alpha _{2}+\frac{n}{p}>\alpha _{1}+\frac{n}{s}$. The remaining
case can be easily solved. The proof is complete.
\end{proof}

From Theorems \ref{phi-tran-lorentz}, \ref{embeddings3 copy(1)-lorentz},\ref%
{embeddings3 copy(2)-lorentz} and \ref{embeddings3 copy(3)-lorentz}, we have
the following Sobolev embedding for\ spaces\ $\dot{K}_{p,r}^{\alpha
,q}F_{\beta }^{s}$.

\begin{theorem}
\label{embeddings3 copy(4)-lorentz}\textit{Let }$\alpha ,\alpha _{1},\alpha
_{2},s_{1},s_{2}\in \mathbb{R},0<s,r,p,q<\infty ,0<\theta ,r_{1},r_{2},\beta
\leq \infty ,\alpha _{1}>-\frac{n}{s}\ $\textit{and }$\alpha _{2}>-\frac{n}{p%
}$. \newline
$\mathrm{(i)}$ Under the hypothesis of Theorem \ref{embeddings3
copy(1)-lorentz}\ we have%
\begin{equation*}
\dot{K}_{p,\infty }^{\alpha _{2},r}F_{\infty }^{s_{2}}\hookrightarrow \dot{K}%
_{s,r_{1}}^{\alpha _{1},q}F_{\theta }^{s_{1}}.
\end{equation*}%
The condition \eqref{newexp1.1-lorentz} becomes necessary.\newline
$\mathrm{(ii)}$\ Under the hypothesis of Theorem \ref{embeddings3
copy(2)-lorentz} we have%
\begin{equation*}
\dot{K}_{p,r_{1}}^{\alpha ,r}F_{\infty }^{s_{2}}\hookrightarrow \dot{K}%
_{s,r_{1}}^{\alpha ,q}F_{\theta }^{s_{1}}.
\end{equation*}%
The condition \eqref{newexp1.1-lorentz-bis} becomes necessary.\newline
$\mathrm{(iii)}$ Under the hypothesis of Theorem \ref{embeddings3
copy(3)-lorentz} we have%
\begin{equation*}
\dot{K}_{p,r_{2}}^{\alpha _{2},r}F_{\theta }^{s_{2}}\hookrightarrow \dot{K}%
_{s,r_{1}}^{\alpha _{1},q}F_{\beta }^{s_{1}}.
\end{equation*}%
The conditions \eqref{newexp1.1-lorentz} and \eqref{newexp1.3-lorentz}
become necessary.
\end{theorem}

From Theorem\ \ref{embeddings3 copy(4)-lorentz}\ and the fact that $\dot{K}%
_{s}^{0,s}F_{\beta }^{s_{1}}=F_{s,\beta }^{s_{1}}$\ we immediately arrive at
the following results.

\begin{theorem}
\label{embeddings4.1-lorentz}Let $s_{1},s_{2}\in \mathbb{R},0<s,p<\infty
,0<\theta ,r_{1},r_{2},\beta \leq \infty ,\alpha >-\frac{n}{p}$ and%
\begin{equation*}
s_{1}-\frac{n}{s}=s_{2}-\frac{n}{p}-\alpha .
\end{equation*}%
$\mathrm{(i)}$ Assume that $0<p<s<\infty ,0<r\leq s<\infty $\ and $\alpha >0$%
. Then%
\begin{equation*}
\dot{K}_{p,\infty }^{\alpha ,r}F_{\theta }^{s_{2}}\hookrightarrow F_{s,\beta
}^{s_{1}}.
\end{equation*}%
$\mathrm{(ii)}$\ Assume that $0<p<s<\infty \ $and $0<\max (r,r_{1})\leq
s<\infty $. Then%
\begin{equation*}
\dot{K}_{p,r_{1}}^{0,r}F_{\theta }^{s_{2}}\hookrightarrow F_{s,\beta
}^{s_{1}}.
\end{equation*}%
$\mathrm{(iii)}$\ Assume that $0<s\leq p<\infty ,0<r\leq s<\infty $ and\ $%
\alpha \geq \frac{n}{s}-\frac{n}{p}$.. Then 
\begin{equation*}
\dot{K}_{p,r_{2}}^{\alpha ,r}F_{\theta }^{s_{2}}\hookrightarrow F_{s,\beta
}^{s_{1}},
\end{equation*}%
where $r_{1}=r_{2}$ if $s=p$ and%
\begin{equation*}
\theta =\left\{ 
\begin{array}{ccc}
\beta , & \text{if} & 0<s\leq p<\infty \text{ and }\alpha =\frac{n}{s}-\frac{%
n}{p}, \\ 
\infty , &  & \text{otherwise.}%
\end{array}%
\right.
\end{equation*}
\end{theorem}

Using the fact that $F_{p,\theta }^{s}=\dot{K}_{p,p}^{\alpha ,p}F_{\theta
}^{s}$, we obtain from Theorem\ \ref{embeddings3 copy(4)-lorentz} the
following results.

\begin{theorem}
\label{embeddings5.1-lorentz}Let $s_{1},s_{2}\in \mathbb{R},0<s,p<\infty
,0<\theta ,r_{1},\beta \leq \infty ,\alpha >-\frac{n}{s}$ and%
\begin{equation*}
s_{1}-\frac{n}{s}-\alpha =s_{2}-\frac{n}{p}.
\end{equation*}%
$\mathrm{(i)}$ Assume that $0<p<s<\infty ,0<p\leq q<\infty $\ and $\alpha <0$%
.\ Then%
\begin{equation*}
F_{p,\theta }^{s_{2}}\hookrightarrow \dot{K}_{s,r_{1}}^{\alpha ,q}F_{\beta
}^{s_{1}}.
\end{equation*}%
$\mathrm{(ii)}$\ Assume that $0<p<s<\infty $ and $0<p\leq \min
(r_{1},q)<\infty $. Then%
\begin{equation*}
F_{p,\theta }^{s_{2}}\hookrightarrow \dot{K}_{s,r_{1}}^{0,q}F_{\beta
}^{s_{1}}.
\end{equation*}%
$\mathrm{(iii)}$\ Assume that $0<s<p\leq q<\infty $\ and\ $\alpha \leq \frac{%
n}{p}-\frac{n}{s}$. Then 
\begin{equation*}
F_{p,\theta }^{s_{2}}\hookrightarrow \dot{K}_{s,r_{1}}^{\alpha ,q}F_{\beta
}^{s_{1}},
\end{equation*}%
where%
\begin{equation*}
\theta =\left\{ 
\begin{array}{ccc}
\beta , & \text{if} & 0<s<p<\infty \text{ and }\alpha =\frac{n}{p}-\frac{n}{s%
}, \\ 
\infty , &  & \text{otherwise.}%
\end{array}%
\right.
\end{equation*}
\end{theorem}

\begin{remark}
Theorem \ref{embeddings5.1-lorentz}/(ii) extends and improves Sobolev
embeddings of Triebel-Lizorkin spaces. Indeed, we choose $0<r_{1},q<\infty $
such that 
\begin{equation*}
0<p\leq \min (r_{1},q)\leq \max (r_{1},q)<s<\infty .
\end{equation*}%
Then, we have 
\begin{equation*}
F_{p,\theta }^{s_{2}}\hookrightarrow \dot{K}_{s,r_{1}}^{0,q}F_{\beta
}^{s_{1}}\hookrightarrow \dot{K}_{s,s}^{0,s}F_{\beta }^{s_{1}}=F_{s,\beta
}^{s_{1}}.
\end{equation*}%
In particular%
\begin{equation*}
W_{p}^{s_{2}}\hookrightarrow \dot{K}_{s,r_{1}}^{0,q}F_{2}^{s_{1}}%
\hookrightarrow W_{s}^{s_{1}},
\end{equation*}%
whenever $s_{1},s_{2}\in \mathbb{N}_{0}.$
\end{remark}

From Theorem \ref{embeddings3 copy(4)-lorentz} and the fact that 
\begin{equation*}
\dot{K}_{p,p}^{\alpha ,r}F_{2}^{0}=\dot{K}_{p}^{\alpha ,r}
\end{equation*}%
for $1<r,p,r_{1}<\infty $ and $-\frac{n}{p}<\alpha <n-\frac{n}{p}$; see \cite%
{XuYang03}, we obtain the following embeddings between Herz and
Triebel-Lizorkin spaces.

\begin{corollary}
\label{corollary-lorentz 1}Let $0<s,p<\infty ,0<\theta ,\beta \leq \infty
,1<r<\infty $ and $-\frac{n}{p}<\alpha <n-\frac{n}{p}.\newline
\mathrm{(i)}$\ Assume that $1<p<s<\infty $ and $\alpha >0$. Then%
\begin{equation*}
\dot{K}_{p}^{\alpha ,r}\hookrightarrow \dot{K}_{s,p}^{0,r}F_{\beta }^{\frac{n%
}{s}-\frac{n}{p}-\alpha }.
\end{equation*}%
$\mathrm{(ii)}$\ Assume that $1<p<s<\infty $. Then%
\begin{equation*}
\dot{K}_{p}^{0,r}\hookrightarrow \dot{K}_{s,p}^{0,r}F_{\beta }^{\frac{n}{s}-%
\frac{n}{p}}.
\end{equation*}%
$\mathrm{(iii)}$\ Assume that $\max (1,s)<p<\infty $\ and\ $\alpha \geq 
\frac{n}{s}-\frac{n}{p}$. Then 
\begin{equation*}
\dot{K}_{p}^{\alpha ,r}\hookrightarrow \dot{K}_{s,p}^{0,r}F_{\beta }^{\frac{n%
}{s}-\frac{n}{p}-\alpha },
\end{equation*}%
where $\beta =2$\ if\ $\max (1,s)<p<\infty $ and $\alpha =\frac{n}{s}-\frac{n%
}{p}$.
\end{corollary}

\begin{corollary}
\label{corollary-lorentz 2}Let $0<s,p<\infty ,0<\theta \leq \infty
,1<r<\infty \ $and $-\frac{n}{s}<\alpha <n-\frac{n}{s}.\newline
\mathrm{(i)}$ Assume that $\max (1,p)<s<\infty $ and $\alpha <0$. Then%
\begin{equation*}
\dot{K}_{p,s}^{0,r}F_{\theta }^{\frac{n}{p}-\frac{n}{s}-\alpha
}\hookrightarrow \dot{K}_{s}^{\alpha ,r}.
\end{equation*}%
$\mathrm{(ii)}$\ Assume that $\max (1,p)<s<\infty $. Then%
\begin{equation*}
\dot{K}_{p,s}^{0,r}F_{\theta }^{\frac{n}{p}-\frac{n}{s}}\hookrightarrow \dot{%
K}_{s}^{0,r}.
\end{equation*}%
$\mathrm{(iii)}$\ Assume that $1<s<p<\infty $ and\ $\alpha \leq \frac{n}{p}-%
\frac{n}{s}$. Then 
\begin{equation*}
\dot{K}_{p,s}^{0,r}F_{\theta }^{\frac{n}{p}-\frac{n}{s}-\alpha
}\hookrightarrow \dot{K}_{s}^{\alpha ,r},
\end{equation*}%
where $\theta =2$ if $1<s<p<\infty $ and $\alpha =\frac{n}{s}-\frac{n}{p}.$
\end{corollary}

\begin{remark}
Corollaries \ref{corollary-lorentz 1} and \ref{corollary-lorentz 2} extend
and improve the corresponding results of \cite{Drihem2.13}. In particular
Sobolev embeddings for Triebel-Lizorkin spaces\ of power weight\ obtained\
in \cite{MM12}.
\end{remark}

\begin{corollary}
Let $s_{1},s_{2},s_{3}\in \mathbb{R},0<t\leq p<s<\infty ,0<\beta \leq \infty 
$ are real numbers such that%
\begin{equation*}
s_{1}-\frac{n}{s}=s_{2}-\frac{n}{p}=s_{3}-\frac{n}{t}.
\end{equation*}%
Then%
\begin{equation*}
F_{t,\infty }^{s_{3}}\hookrightarrow \dot{K}_{p,s}^{0,s}F_{\infty
}^{s_{2}}\hookrightarrow F_{s,\beta }^{s_{1}}.
\end{equation*}
\end{corollary}

\begin{proof}
To prove this result, it is sufficient to choose in Theorem \ref%
{embeddings4.1-lorentz}/(ii) $r=s=r_{1}$.\ However, the desired embeddings
are an immediate consequence of the fact that 
\begin{equation*}
F_{t,\infty }^{s_{3}}\hookrightarrow F_{p,\infty }^{s_{2}}=\dot{K}%
_{p,p}^{0,p}F_{\infty }^{s_{2}}\hookrightarrow \dot{K}_{p,s}^{0,s}F_{\infty
}^{s_{2}}\hookrightarrow F_{s,\beta }^{s_{1}}.
\end{equation*}
\end{proof}

\begin{corollary}
Let $s_{1},s_{2}\in \mathbb{R},0<p\leq r_{1}<s<\infty ,s_{1}-\frac{n}{s}%
=s_{2}-\frac{n}{p}$\ and $0<\beta \leq \infty $. Then%
\begin{equation*}
F_{p,\infty }^{s_{2}}\hookrightarrow \dot{K}_{s,r_{1}}^{0,p}F_{\beta
}^{s_{1}}\hookrightarrow F_{s,\beta }^{s_{1}}.
\end{equation*}
\end{corollary}

\begin{proof}
In Theorem\ \ref{embeddings5.1-lorentz}/(ii)\ we choose $p=q,r_{1}=s$. Then
the desired embeddings are an immediate consequence of the fact that 
\begin{equation*}
F_{p,\infty }^{s_{2}}\hookrightarrow \dot{K}_{s,s}^{0,p}F_{\beta
}^{s_{1}}\hookrightarrow \dot{K}_{s,s}^{0,s}F_{\beta }^{s_{1}}=F_{s,\beta
}^{s_{1}}.
\end{equation*}
\end{proof}

By Theorem \ref{embeddings4.1-lorentz}/(ii), we get%
\begin{equation*}
\dot{K}_{p,\infty }^{\alpha ,r}F_{\infty }^{s_{2}}\hookrightarrow F_{p,\beta
}^{s_{1}}
\end{equation*}%
for any $\alpha >0,s_{1}\leq s_{2}-\alpha ,0<r\leq p<\infty $\ and\ $0<\beta
\leq \infty $. Let $\{\varphi _{j}\}_{j\in \mathbb{N}_{0}}$\ be a smooth
dyadic resolution of unity. Recall that%
\begin{equation*}
\big\|f\big\|_{\overline{p}}\leq \sum\limits_{j=0}^{\infty }\big\|\mathcal{F}%
^{-1}\varphi _{j}\ast f\big\|_{\overline{p}}=\big\|f\big\|_{B_{\overline{p}%
,1}^{0}}
\end{equation*}%
for any $f\in B_{\overline{p},1}^{0}$. In addition from the fact that 
\begin{equation*}
F_{p,\beta }^{s}\hookrightarrow B_{\overline{p},1}^{0}
\end{equation*}%
for any $s>\max (0,\frac{n}{p}-n)$, where $\overline{p}=\max (1,p)$, we get 
\begin{equation*}
\big\|f\big\|_{\overline{p}}\lesssim \big\|f\big\|_{F_{p,\beta }^{s}}
\end{equation*}%
for any $f\in F_{p,\beta }^{s}$. This shows that under some suitable
assumptions the elements of $\dot{K}_{p,\infty }^{\alpha ,r}F_{\beta }^{s}$\
are regular distributions.

\begin{proposition}
\textit{Let }$\alpha >0,s\in \mathbb{R},0<r\leq p<\infty $ \textit{and }$%
0<\beta \leq \infty $\textit{.}\ If $s>\frac{n}{p}-n+\alpha $ and $0<p\leq 1$%
\ or\ $s>\alpha $ and $1<p<\infty $, \textit{then}%
\begin{equation*}
\dot{K}_{p,\infty }^{\alpha ,r}F_{\beta }^{s}\hookrightarrow L^{\overline{p}%
}.
\end{equation*}
\end{proposition}

\subsection{Jawerth embedding}

The classical Jawerth embedding says that:%
\begin{equation*}
F_{q,\infty }^{s_{2}}\hookrightarrow B_{s,q}^{s_{1}}
\end{equation*}%
if $s_{1}-\frac{n}{s}=s_{2}-\frac{n}{q}$ and $0<q<s<\infty $; see e.g., \cite%
{Ja77}. We will extend this embeddings to Lorentz Herz-type
Besov-Triebel-Lizorkin spaces.\ We follow some ideas of Vyb\'{\i}ral, \cite[%
p. 76]{Vybiral08}, where it is used the technique of non-increasing
rearrangement. First, we will prove the discrete version of Jawerth
embedding.

\begin{theorem}
\label{embeddings6-lorentz}\textit{Let }$\alpha _{1},\alpha
_{2},s_{1},s_{2}\in \mathbb{R},0<r_{1},r_{2}\leq \infty ,0<s,p,q,r<\infty
,\alpha _{1}>-\frac{n}{s}\ $\textit{and }$\alpha _{2}>-\frac{n}{p}$. \textit{%
We suppose that\ }%
\begin{equation}
s_{1}-\frac{n}{s}-\alpha _{1}=s_{2}-\frac{n}{p}-\alpha _{2}.
\label{new-exp1-lorentz}
\end{equation}%
\textit{Under the following assumptions} \textit{\ }%
\begin{equation*}
0<p<s\leq \infty \quad \text{and}\quad \alpha _{2}>\alpha _{1}
\end{equation*}%
we have%
\begin{equation}
\dot{K}_{p,r_{2}}^{\alpha _{2},r}f_{\infty }^{s_{2}}\hookrightarrow \dot{K}%
_{s,r_{1}}^{\alpha _{1},q}b_{r}^{s_{1}}.  \label{FJ-emb1-lorentz}
\end{equation}
\end{theorem}

\begin{proof}
Put $c_{n}=1+\lfloor \log _{2}(2\sqrt{n}+1)\rfloor $. Let $\lambda \in \dot{K%
}_{p,r_{2}}^{\alpha _{2},r}f_{\infty }^{s_{2}}$. We have%
\begin{align*}
\big\|\lambda \big\|_{\dot{K}_{s,r_{1}}^{\alpha _{1},q}b_{r}^{s_{1}}}^{r}=&
\sum_{v=0}^{\infty }\Big(\sum\limits_{k=-\infty }^{\infty }2^{(k\alpha
_{1}+v(s_{1}+\frac{n}{2}))q}\big\|\sum\limits_{m\in \mathbb{Z}^{n}}\lambda
_{v,m}\chi _{v,m}\chi _{k}\big\|_{L^{s,r_{1}}}^{q}\Big)^{r/q} \\
\lesssim & I+II,
\end{align*}%
where%
\begin{equation*}
I=\sum_{v=0}^{\infty }\Big(\sum\limits_{k=-\infty }^{c_{n}+1-v}2^{(k\alpha
_{1}+v(s_{1}+\frac{n}{2}))q}\big\|\sum\limits_{m\in \mathbb{Z}^{n}}\lambda
_{v,m}\chi _{v,m}\chi _{k}\big\|_{L^{s,r_{1}}}^{q}\Big)^{r/q}
\end{equation*}%
and%
\begin{equation*}
II=\sum_{v=0}^{\infty }\Big(\sum\limits_{k=c_{n}+2-v}^{\infty }2^{(k\alpha
_{1}+v(s_{1}+\frac{n}{2}))q}\big\|\sum\limits_{m\in \mathbb{Z}^{n}}\lambda
_{v,m}\chi _{v,m}\chi _{k}\big\|_{L^{s,r_{1}}}^{q}\Big)^{r/q}.
\end{equation*}

\textit{Step 1.}\textbf{\ }We will estimate $I$ and $II$, respectively.

\textit{Estimation of }$I$\textit{.} Let $x\in R_{k}\cap Q_{v,m}$ and $y\in
Q_{v,m}$. We have $|x-y|\leq 2\sqrt{n}2^{-v}<2^{c_{n}-v}$ and from this it
follows that $|y|<2^{c_{n}-v}+2^{k}\leq 2^{c_{n}-v+2}$, which implies that $%
y $ is located in the ball $B(0,2^{c_{n}-v+2})$ and 
\begin{equation*}
|\lambda _{v,m}|^{t}\lesssim 2^{nv}\int_{B(0,2^{c_{n}-v+2})}|\lambda
_{v,m}|^{t}\chi _{v,m}(y)dy,
\end{equation*}%
where $t>0$. Then for any $x\in R_{k}$ we obtain%
\begin{align*}
\sum_{m\in \mathbb{Z}^{n}}|\lambda _{v,m}|^{t}\chi _{v,m}(x)\lesssim &
2^{nv}\int_{B(0,2^{c_{n}-v+2})}\sum_{m\in \mathbb{Z}^{n}}|\lambda
_{v,m}|^{t}\chi _{v,m}(y)dy \\
=& c2^{nv}\big\|\sum_{m\in \mathbb{Z}^{n}}\lambda _{v,m}\chi _{v,m}\chi
_{B(0,2^{c_{n}-v+2})}\big\|_{L^{t,t}}^{t},
\end{align*}%
where the positive constant $c$ is independent of $v$ and $x$. Consequently,
with the help of \eqref{est-function1}, we obtain 
\begin{align*}
& 2^{k\alpha _{1}+v(s_{1}+\frac{n}{2})}\Big\|\sum_{m\in \mathbb{Z}%
^{n}}\lambda _{v,m}\chi _{v,m}\chi _{k}\Big\|_{L^{s,r_{1}}} \\
& \lesssim 2^{v(s_{1}+\frac{n}{t}+\frac{n}{2})+k(\alpha _{1}+\frac{n}{s})}%
\Big\|\sum_{m\in \mathbb{Z}^{n}}\lambda _{v,m}\chi _{v,m}\chi
_{B(0,2^{c_{n}-v+2})}\Big\|_{L^{t,t}}.
\end{align*}%
We may choose $t>0$ such that $\frac{1}{t}>\max (\frac{1}{p},\frac{1}{r_{2}},%
\frac{1}{p}+\frac{\alpha _{2}}{n})$. Therefore, since $\alpha _{1}+\frac{n}{s%
}>0$, 
\begin{equation*}
I\lesssim \sum_{v=0}^{\infty }2^{v(s_{1}+\frac{n}{t}-\alpha _{1}-\frac{n}{s}%
-s_{2})r}\sup_{j\in \mathbb{N}_{0}}2^{(s_{2}+\frac{n}{2})jr}\Big\|\sum_{m\in 
\mathbb{Z}^{n}}\lambda _{j,m}\chi _{j,m}\chi _{B(0,2^{c_{n}-v+2})}\Big\|%
_{L^{t,t}}^{r},
\end{equation*}%
which can be estimated from above by 
\begin{equation*}
c\sum_{v=0}^{\infty }2^{v\frac{nr}{d}}\Big(\sum_{i=-\infty }^{-v}2^{(\frac{n%
}{d}+\alpha _{2})\delta i}\sup_{j\in \mathbb{N}_{0}}2^{j(s_{2}+\frac{n}{2}%
)\delta }\Big\|\sum_{m\in \mathbb{Z}^{n}}\lambda _{j,m}\chi _{j,m}\chi
_{i+c_{n}+2}\Big\|_{L^{q,r_{2}}}^{\delta }\Big)^{r/\delta },
\end{equation*}%
by Lemma \ref{Lp-estimate}, \eqref{new-exp1-lorentz} and H\"{o}lder's
inequality, with $\delta =\min (1,t)$ and $\frac{n}{d}=\frac{n}{t}-\frac{n}{p%
}-\alpha _{2}$. Hence Lemma \ref{lem:lq-inequality} implies that 
\begin{equation*}
I\lesssim \sum_{i=0}^{\infty }2^{-\alpha _{2}ir}\sup_{j\in \mathbb{N}%
_{0}}2^{j(s_{2}+\frac{n}{2})r}\Big\|\sum\limits_{m\in \mathbb{Z}^{n}}\lambda
_{j,m}\chi _{j,m}\chi _{2-i+c_{n}}\Big\|_{L^{p,r_{2}}}^{r}\lesssim \big\|%
\lambda \big\|_{\dot{K}_{p,r_{2}}^{\alpha _{2},r}f_{\infty }^{s_{2}}}^{r}.
\end{equation*}

\textit{Estimation of }$II$\textit{.} Let $\alpha _{3}=\alpha _{1}-\alpha
_{2}$. We have 
\begin{align*}
& II \\
& =\sum_{v=0}^{\infty }\Big(\sum\limits_{k=c_{n}+2-v}^{\infty
}2^{(k+v)\alpha _{3}q+v(\frac{n}{s}-\frac{n}{p}+s_{2}+\frac{n}{2})q+k\alpha
_{2}q}\big\|\sum\limits_{m\in \mathbb{Z}^{n}}\lambda _{v,m}\chi _{v,m}\chi
_{k}\big\|_{L^{s,r_{1}}}^{q}\Big)^{r/q},
\end{align*}%
which is bounded by 
\begin{equation*}
\sum_{v=-\infty }^{\infty }\Big(\sum\limits_{k=c_{n}+2-v}^{\infty
}2^{(k+v)\alpha _{3}+k\alpha _{2}q}\sup_{j\geq c_{n}+2-k}\Big(2^{j(\frac{n}{s%
}-\frac{n}{p}+s_{2}+\frac{n}{2})}\big\|\sum\limits_{m\in \mathbb{Z}%
^{n}}\lambda _{j,m}\chi _{j,m}\chi _{k}\big\|_{L^{s,r_{1}}}\Big)^{q}\Big)%
^{r/q}.
\end{equation*}%
Since $\alpha _{2}>\alpha _{1}$, by Lemma \ref{lem:lq-inequality}, we
estimate $II$ by%
\begin{align}
& \sum\limits_{k=-\infty }^{\infty }2^{k\alpha _{2}r}\sup_{j\geq c_{n}+2-k}%
\Big(2^{j(\frac{n}{s}-\frac{n}{p}+s_{2}+\frac{n}{2})}\big\|\sum\limits_{m\in 
\mathbb{Z}^{n}}\lambda _{j,m}\chi _{j,m}\chi _{k}\big\|_{L^{s,r_{1}}}\Big)%
^{r}  \notag \\
\leq & \sum\limits_{k=-\infty }^{\infty }2^{k\alpha _{2}r}\Big(%
\sum_{v=c_{n}+2-k}^{\infty }2^{v(\frac{n}{s}-\frac{n}{p}+s_{2}+\frac{n}{2}%
)r_{2}}\Big\|\sum\limits_{m\in \mathbb{Z}^{n}}\lambda _{v,m}\chi _{v,m}\chi
_{k}\Big\|_{L^{s,r_{1}}}^{r_{2}}\Big)^{r/r_{2}}.  \label{sum-lorentz1}
\end{align}%
We use some decomposition techniques already used in \cite{Vybiral08}. Let $%
\breve{R}_{k}$ and $A_{k+v}$ be as in the proof of Theorem \ref{embeddings3
copy(1)-lorentz}. Put 
\begin{equation*}
h_{k}(x)=\sup_{v\in \mathbb{N}_{0}}2^{v(s_{2}+\frac{n}{2})}\sum\limits_{m\in 
\mathbb{Z}^{n}}\left\vert \lambda _{v,m}\right\vert \chi _{v,m}(x)\chi _{%
\breve{R}_{k}}(x).
\end{equation*}%
Then 
\begin{equation*}
\big\|\lambda \big\|_{\dot{K}_{p,r_{2}}^{\alpha _{2},r}f_{\infty
}^{s_{2}}}\approx \Big(\sum\limits_{k=-\infty }^{\infty }2^{k\alpha _{2}r}%
\big\|h_{k}\big\|_{L^{p,r_{2}}}^{r}\Big)^{1/r}.
\end{equation*}%
Let $x\in Q_{v,m}\cap R_{k}$,\ with $m\in \mathbb{Z}^{n}$, $v\geq c_{n}+2-k$
and $k\in \mathbb{Z}$. Recall that 
\begin{align*}
2^{v(s_{2}+\frac{n}{2})}\sum\limits_{m\in \mathbb{Z}^{n}}|\lambda
_{v,m}|\chi _{v,m}(x)\chi _{k}(x)& \leq 2^{v(s_{2}+\frac{n}{2}%
)}\sum\limits_{m\in A_{k+v}}|\lambda _{v,m}|\chi _{v,m}(x) \\
& =g_{v,k}(x) \\
& \leq h_{k}(x).
\end{align*}%
We have%
\begin{align}
& 2^{v(s_{2}+\frac{n}{2})r_{1}}\Big\|\sum\limits_{m\in \mathbb{Z}%
^{n}}\lambda _{v,m}\chi _{v,m}\chi _{k}\Big\|_{L^{s,r_{1}}}^{r_{1}}  \notag
\\
& \leq \int_{0}^{\infty }(y^{\frac{1}{s}}g_{v,k}^{\ast }(y))^{r_{1}}\frac{dy%
}{y}  \notag \\
& =\int_{0}^{2^{-vn}}(y^{\frac{1}{s}}g_{v,k}^{\ast }(y))^{r_{1}}\frac{dy}{y}%
+\int_{2^{-vn}}^{\infty }(y^{\frac{1}{s}}g_{v,k}^{\ast }(y))^{r_{1}}\frac{dy%
}{y}  \notag \\
& \lesssim 2^{-\frac{vnr_{1}}{s}}(g_{v,k}^{\ast
}(2^{-vn-n}))^{r_{1}}+\int_{2^{-vn}}^{\infty }(y^{\frac{1}{s}}h_{k}^{\ast
}(y))^{r_{1}}\frac{dy}{y},  \label{corr1-lorentz}
\end{align}%
where the implicit constant is independent of $v$ and $k$. By the
monotonicity of $h^{\ast }$, we get%
\begin{align}
\int_{2^{-vn}}^{\infty }(y^{\frac{1}{s}}\left( h_{k}\right) ^{\ast
}(y))^{r_{1}}\frac{dy}{y}& =\sum_{l=0}^{\infty
}\int_{2^{(l-v)n}}^{2^{(l-v)n+n}}(y^{\frac{1}{s}}h_{k}^{\ast }(y))^{r_{1}}%
\frac{dy}{y}  \notag \\
& \leq 2^{-\frac{vnr_{1}}{s}}\sum_{l=0}^{\infty }2^{\frac{nr_{1}}{s}%
l}(h_{k}^{\ast }(2^{(l-v)n}))^{r_{1}}.  \label{corr2-lorentz}
\end{align}%
Inserting \eqref{corr2-lorentz} in \eqref{corr1-lorentz} and using $%
g_{v,k}^{\ast }(2^{-vn-n})\leq h_{k}^{\ast }(2^{-vn-n})$, we obtain that the
sum $\sum_{v=c_{n}+2-k}^{\infty }\cdot \cdot \cdot $ in \eqref{sum-lorentz1}
can be estimated from above by 
\begin{equation}
c\sum_{v=c_{n}+2-k}^{\infty }2^{-\frac{vnr_{2}}{p}}\Big(\sum\limits_{l=0}^{%
\infty }2^{l\frac{nr_{1}}{s}}\left( h_{k}^{\ast }(2^{(l-v)n-n})\right)
^{r_{1}}\Big)^{r_{2}/r_{1}}.  \label{corr4.1-lorentz}
\end{equation}%
We have \eqref{corr4.1-lorentz} can be rewritten as%
\begin{align}
& c\sum_{v=c_{n}+2-k}^{\infty }\Big(\sum\limits_{l=0}^{\infty }2^{l(\frac{1}{%
s}-\frac{1}{p})r_{1}n}2^{(l-v)\frac{nr_{1}}{p}}\left( h_{k}^{\ast
}(2^{(l-v)n-n})\right) ^{r_{1}}\Big)^{r_{2}/r_{1}}  \notag \\
& =c\sum_{v=c_{n}+2-k}^{\infty }\Big(\sum\limits_{j=-v}^{\infty }2^{(j+v)(%
\frac{1}{s}-\frac{1}{p})r_{1}n}2^{j\frac{nr_{1}}{p}}\left( h_{k}^{\ast
}(2^{jn-n})\right) ^{r_{1}}\Big)^{r_{2}/r_{1}}.  \label{corr5-lorentz}
\end{align}%
Applying Lemma \ref{lem:lq-inequality}, we find that \eqref{corr5-lorentz}
is bounded by 
\begin{equation*}
c\sum_{j=-\infty }^{\infty }2^{\frac{njr_{2}}{p}}(h_{k}^{\ast
}(2^{nj}))^{r_{2}}\approx \big\|h_{k}\big\|_{L^{p,r_{2}}}^{r_{2}}.
\end{equation*}%
Consequently, we obtain $II\lesssim \big\|\lambda \big\|_{\dot{K}%
_{p,r_{2}}^{\alpha _{2},r}f_{\infty }^{s_{2}}}.$ The proof is complete.
\end{proof}

Now, we deal with the case $\alpha _{2}=\alpha _{1}.$

\begin{theorem}
\label{embeddings6-lorentz copy(1)}\textit{Let }$\alpha ,s_{1},s_{2}\in 
\mathbb{R},0<r_{1}\leq \infty ,0<s,p,q,r<\infty $ \textit{and }$\alpha >-%
\frac{n}{s}$. \textit{We\ suppose that\ }$0<p<s\leq \infty ,0<q\leq r<\infty
\ $and%
\begin{equation}
s_{1}-\frac{n}{s}=s_{2}-\frac{n}{p}.  \label{new-exp2-lorentz}
\end{equation}%
Then 
\begin{equation*}
\dot{K}_{p,r}^{\alpha ,q}f_{\infty }^{s_{2}}\hookrightarrow \dot{K}%
_{s,r_{1}}^{\alpha ,q}b_{r}^{s_{1}}.
\end{equation*}
\end{theorem}

\begin{proof}
Put $c_{n}=1+\lfloor \log _{2}(2\sqrt{n}+1)\rfloor $. In view the proof of
Theorem \ref{embeddings6-lorentz}, we estimate only $II$. By the assumption %
\eqref{new-exp2-lorentz} we estimate $II^{\frac{q}{r}}$ by%
\begin{equation}
\sum\limits_{k=-\infty }^{\infty }2^{k\alpha q}\Big(\sum_{v=c_{n}+2-k}^{%
\infty }2^{v(\frac{n}{s}-\frac{n}{p}+s_{2}+\frac{n}{2})r}\Big\|%
\sum\limits_{m\in A_{k+v}}\lambda _{v,m}\chi _{v,m}\chi _{k}\Big\|%
_{L^{s,r_{1}}}^{r}\Big)^{q/r}.  \label{sum-lorentz}
\end{equation}%
Let $h_{k}$ be as in the proof of Theorem \ref{embeddings6-lorentz}. The sum 
$\sum_{v=c_{n}+2-k}^{\infty }\cdot \cdot \cdot $ in \eqref{sum-lorentz} can
be estimated from above by 
\begin{align*}
& \sum_{v=0}^{\infty }2^{-\frac{vnr}{p}}\Big(\sum\limits_{l=0}^{\infty }2^{l%
\frac{nr_{1}}{s}}\left( h_{k}^{\ast }(2^{(l-v)n-n})\right) ^{r_{1}}\Big)%
^{r/r_{1}} \\
& =\sum_{v=0}^{\infty }\Big(\sum\limits_{l=0}^{\infty }2^{l(\frac{1}{s}-%
\frac{1}{p})nr_{1}}2^{(l-v)\frac{nr_{1}}{q}}\left( h_{k}^{\ast
}(2^{(l-v)n-n})\right) ^{r_{1}}\Big)^{r/r_{1}} \\
& =\sum_{v=0}^{\infty }\Big(\sum\limits_{j=-v}^{\infty }2^{(j+v)(\frac{1}{s}-%
\frac{1}{p})nr_{1}}2^{j\frac{nr_{1}}{q}}\left( h_{k}^{\ast
}(2^{jn-n})\right) ^{r_{1}}\Big)^{r/r_{1}} \\
& \lesssim \sum\limits_{j=-\infty }^{\infty }2^{j\frac{nr}{p}}\left(
h_{k}^{\ast }(2^{jn})\right) ^{r},
\end{align*}%
by using Lemma\ \ref{lem:lq-inequality}. Hence, we obtain $II\lesssim \big\|%
\lambda \big\|_{\dot{K}_{p,r}^{\alpha _{2},q}f_{\infty }^{s_{2}}}$.
\end{proof}

\begin{theorem}
\label{embeddings6-lorentz copy(2)}\textit{Let }$\alpha ,s_{1},s_{2}\in 
\mathbb{R},0<r_{1}\leq \infty ,0<s,p,q,r<\infty $ \textit{and }$\alpha >-%
\frac{n}{s}$. \textit{We\ suppose that\ }$0<p<s\leq \infty ,0<r<q\leq \infty 
$\ and%
\begin{equation*}
s_{1}-\frac{n}{s}=s_{2}-\frac{n}{p}.
\end{equation*}%
Then%
\begin{equation*}
\dot{K}_{p,r}^{\alpha ,r}f_{\theta }^{s_{2}}\hookrightarrow \dot{K}%
_{s,r_{1}}^{\alpha ,q}b_{r}^{s_{1}}.
\end{equation*}
\end{theorem}

\begin{proof}
In this case, we estimate $II$ by%
\begin{equation*}
\sum\limits_{k=-\infty }^{\infty }2^{k\alpha r}\sum_{v=c_{n}+2-k}^{\infty
}2^{v(\frac{n}{s}-\frac{n}{p}+s_{2}+\frac{n}{2})r}\Big\|\sum\limits_{m\in 
\mathbb{Z}^{n}}\lambda _{v,m}\chi _{v,m}\chi _{k}\Big\|_{L^{s,r_{1}}}^{r}.
\end{equation*}%
As in Theorem \ref{embeddings6-lorentz copy(1)} we arrive at the desired
estimate. The proof is complete.
\end{proof}

\begin{theorem}
\label{embeddings6-lorentz copy(3)}\textit{Let }$\alpha _{1},\alpha
_{2},s_{1},s_{2}\in \mathbb{R}$, $0<r,r_{1},r_{2}\leq \infty ,0<s,p,q<\infty
,\alpha _{1}>-\frac{n}{s}\ $\textit{and }$\alpha _{2}>-\frac{n}{p}$. \textit{%
We suppose \eqref{new-exp1-lorentz} and }$0<s\leq p<\infty $. Assume that%
\textit{,}%
\begin{equation*}
\alpha _{2}+\frac{n}{p}>\alpha _{1}+\frac{n}{s}.
\end{equation*}%
Then%
\begin{equation*}
\dot{K}_{p,r_{2}}^{\alpha _{2},r}f_{\theta }^{s_{2}}\hookrightarrow \dot{K}%
_{s,r_{1}}^{\alpha _{1},q}b_{r}^{s_{1}},
\end{equation*}%
where $r_{1}=r_{2}$ if $s=p.$
\end{theorem}

\begin{proof}
Again, we need only to estimate $II$. By H\"{o}lder's inequality we get%
\begin{equation*}
2^{vs_{1}}\Big\|\sum\limits_{m\in \mathbb{Z}^{n}}\lambda _{v,m}\chi
_{v,m}\chi _{k}\Big\|_{L^{s,r_{1}}}\lesssim 2^{(\frac{n}{s}-\frac{n}{p}%
)k+vs_{1}}\Big\|\sum\limits_{m\in \mathbb{Z}^{n}}\lambda _{v,m}\chi
_{v,m}\chi _{k}\Big\|_{L^{p,r_{2}}},
\end{equation*}%
where the implicit constant is independent of $v$ and $k$. Hence $II$ can
be\ estimated from above by%
\begin{align*}
& c\sum_{v=0}^{\infty }2^{v(s_{1}+\frac{n}{2})r}\Big(\sum\limits_{k=-v}^{%
\infty }2^{k(\alpha _{1}+\frac{n}{s}-\frac{n}{p})q}\Big\|\sum\limits_{m\in 
\mathbb{Z}^{n}}\lambda _{v,m}\chi _{v,m}\chi _{k}\Big\|_{L^{p,r_{2}}}^{q}%
\Big)^{r/q} \\
\leq & \sum_{v=0}^{\infty }2^{v(s_{2}+\frac{n}{2})r}\Big(\sum%
\limits_{k=-v}^{\infty }2^{(k+v)(\alpha _{1}-\alpha _{2}+\frac{n}{s}-\frac{n%
}{p})q}2^{k\alpha _{2}q}\Big\|\sum\limits_{m\in \mathbb{Z}^{n}}\lambda
_{v,m}\chi _{v,m}\chi _{k}\Big\|_{L^{p,r_{2}}}^{q}\Big)^{r/q} \\
\leq & \sum_{v=0}^{\infty }\Big(\sum\limits_{k=-v}^{\infty }2^{(k+v)(\alpha
_{1}-\alpha _{2}+\frac{n}{s}-\frac{n}{p})q}2^{k\alpha _{2}q}\Big\|\sup_{j\in 
\mathbb{N}_{0}}\big(2^{j(s_{2}+\frac{n}{2})}\sum\limits_{m\in \mathbb{Z}%
^{n}}|\lambda _{j,m}|\chi _{j,m}\chi _{k}\big)\Big\|_{L^{p,r_{2}}}^{q}\Big)%
^{r/q} \\
\lesssim & \sum_{k=-\infty }^{\infty }2^{k\alpha _{2}r}\Big\|\sup_{j\in 
\mathbb{N}_{0}}\big(2^{j(s_{2}+\frac{n}{2})}\sum\limits_{m\in \mathbb{Z}%
^{n}}|\lambda _{j,m}|\chi _{j,m}\chi _{k}\big)\Big\|_{L^{p,r_{2}}}^{r} \\
\lesssim & \big\|\lambda \big\|_{\dot{K}_{p,r_{2}}^{\alpha _{2},r}f_{\infty
}^{s_{2}}}^{r},
\end{align*}%
by Lemma\ \ref{lem:lq-inequality}.
\end{proof}

\begin{remark}
\label{optimal-lorentz}We have $r$ on the right-hand side of %
\eqref{FJ-emb1-lorentz} is optimal. Indeed, for $v\in \mathbb{N}_{0}$ and $%
N\geq 1$, we put 
\begin{equation*}
\lambda _{v,m}^{N}=\left\{ 
\begin{array}{ccc}
2^{-(s_{1}-\frac{1}{s}-\alpha _{1}+\frac{n}{2})v}\sum_{i=1}^{N}\chi
_{i}(2^{v-1}) & \text{if} & m=1 \\ 
0 &  & \text{otherwise,}%
\end{array}%
\right.
\end{equation*}%
and $\lambda ^{N}=\{\lambda _{v,m}^{N}\}_{v\in \mathbb{N}_{0},m\in \mathbb{Z}%
}$. As in Theorem \ref{embeddings3 copy(1)-lorentz}, we obtain%
\begin{equation*}
\big\|\lambda ^{N}\big\|_{\dot{K}_{p,r_{2}}^{\alpha _{2},\nu }f_{\theta
}^{s_{2}}}^{\nu }=c\text{ }N,
\end{equation*}%
where the constant $c>0$ does not depend on $N$. Now%
\begin{equation*}
\big\|\lambda ^{N}\big\|_{\dot{K}_{s,r_{1}}^{\alpha
_{1},q}b_{r}^{s_{1}}}^{r}=\sum_{v=0}^{\infty }2^{v(s_{1}+\frac{n}{2})r}\Big(%
\sum_{k=-\infty }^{\infty }2^{\alpha _{1}kq}\Big\|\sum\limits_{m\in \mathbb{Z%
}}|\lambda _{v,m}^{N}|\chi _{v,m}\chi _{k}\Big\|_{L^{s,r_{1}}}^{q}\Big)%
^{r/q}.
\end{equation*}%
We rewrite the last statement as follows:%
\begin{equation*}
\big\|\lambda ^{N}\big\|_{\dot{K}_{s,r_{1}}^{\alpha
_{1},q}b_{r}^{s_{1}}}^{r}=\sum_{v=1}^{N}2^{(\frac{1}{s}+\alpha _{1})vr}\Big(%
\sum_{k=1-N}^{0}2^{\alpha _{1}kq}\big\|\chi _{v,1}\chi _{k}\big\|%
_{L^{s,r_{1}}}^{q}\Big)^{r/q}=cN,
\end{equation*}%
where the constant $c>0$ does not depend on $N$. If the embeddings %
\eqref{FJ-emb1-lorentz} holds then for any $N\in \mathbb{N}$, $N^{\frac{1}{r}%
-\frac{1}{\nu }}\leq C$. Thus, we conclude that $0<\nu \leq r<\infty $ must
necessarily hold by letting $N\rightarrow +\infty $.
\end{remark}

Using Theorems \ref{phi-tran-lorentz}, \ref{embeddings6-lorentz}, \ref%
{embeddings6-lorentz copy(1)}, \ref{embeddings6-lorentz copy(2)} and\ \ref%
{embeddings6-lorentz copy(3)}, we have the following Jawerth embedding.

\begin{theorem}
\label{embeddings6.1-lorentz}\textit{Let }$\alpha ,\alpha _{1},\alpha
_{2},s_{1},s_{2}\in \mathbb{R}$, $0<r_{1},r_{2}\leq \infty ,0<s,p,q,r<\infty
,\alpha >-\frac{n}{s},\alpha _{1}>-\frac{n}{s}\ $\textit{and }$\alpha _{2}>-%
\frac{n}{p}$. \newline
$\mathrm{(i)}$ Under the hypothesis of Theorem \ref{embeddings6-lorentz}\ we
have%
\begin{equation}
\dot{K}_{p,r_{2}}^{\alpha _{2},r}F_{\infty }^{s_{2}}\hookrightarrow \dot{K}%
_{s,r_{1}}^{\alpha _{1},q}B_{r}^{s_{1}}.  \label{Jawerth-lorentz}
\end{equation}%
$\mathrm{(ii)}$ Under the hypothesis of Theorem \ref{embeddings6-lorentz
copy(1)} we have%
\begin{equation}
\dot{K}_{p,r}^{\alpha ,q}F_{\infty }^{s_{2}}\hookrightarrow \dot{K}%
_{s,r_{1}}^{\alpha ,q}B_{r}^{s_{1}}.  \label{Jawerth1-lorentz}
\end{equation}%
$\mathrm{(iii)}$ Under the hypothesis of Theorem \ref{embeddings6-lorentz
copy(2)} we have%
\begin{equation}
\dot{K}_{p,r}^{\alpha ,r}F_{\infty }^{s_{2}}\hookrightarrow \dot{K}%
_{s,r_{1}}^{\alpha ,q}B_{r}^{s_{1}}.  \label{Jawerth3-lorentz}
\end{equation}%
$\mathrm{(iv)}$ Under the hypothesis of Theorem \ref{embeddings6-lorentz
copy(3)} we have%
\begin{equation}
\dot{K}_{p,r_{2}}^{\alpha _{2},r}F_{\infty }^{s_{2}}\hookrightarrow \dot{K}%
_{s,r_{1}}^{\alpha _{1},q}B_{r}^{s_{1}}.  \label{Jawerth2-lorentz}
\end{equation}
\end{theorem}

By Theorem \ref{embeddings6.1-lorentz}/(ii) and the fact that 
\begin{equation*}
F_{p,\infty }^{s_{2}}=\dot{K}_{p,p}^{0,p}F_{\infty }^{s_{2}}\hookrightarrow 
\dot{K}_{p,q}^{0,q}F_{\infty }^{s_{2}}\quad \text{and}\quad \dot{K}%
_{s,p}^{0,q}B_{p}^{s_{1}}\hookrightarrow \dot{K}%
_{s,s}^{0,s}B_{p}^{s_{1}}=B_{s,p}^{s_{2}},
\end{equation*}%
with $0<p<q<s<\infty $, we obtain the following embeddings.

\begin{corollary}
Let $0<p<q<s<\infty $\ and\ $s_{1}-\frac{n}{s}=s_{2}-\frac{n}{p}$. Then 
\begin{equation*}
F_{p,\infty }^{s_{2}}\hookrightarrow \dot{K}_{s,p}^{0,q}B_{p}^{s_{1}}%
\hookrightarrow B_{s,p}^{s_{2}}.
\end{equation*}
\end{corollary}

From Theorem \ref{embeddings6.1-lorentz} and the fact that $\dot{K}%
_{p,p}^{\alpha ,q}F_{2}^{0}=\dot{K}_{p}^{\alpha ,q}$ for $1<p,q<\infty $ and 
$-\frac{n}{p}<\alpha <n-\frac{n}{p}$\ we immediately arrive at the following
embedding between Herz and Besov spaces.

\begin{theorem}
\label{embeddings6.1-lorentz copy(1)}\textit{Let }$0<s<\infty ,0<q,r_{1}\leq
\infty ,1<r,p<\infty \ $\textit{and }$0<\alpha <n-\frac{n}{p}$. \textit{We
suppose that}%
\begin{equation*}
1<p<s<\infty
\end{equation*}%
or%
\begin{equation*}
0<\max (1,s)<p<\infty \quad \text{and}\quad \alpha >\tfrac{n}{s}-\tfrac{n}{p}%
.
\end{equation*}%
Then%
\begin{equation*}
\dot{K}_{p}^{\alpha ,r}\hookrightarrow \dot{K}_{s,r_{1}}^{0,q}B_{r}^{\tfrac{n%
}{s}-\tfrac{n}{p}-\alpha }.
\end{equation*}%
In addition, we have%
\begin{equation*}
\dot{K}_{p}^{0,q}\hookrightarrow \dot{K}_{s,r_{1}}^{0,q}B_{p}^{\tfrac{n}{s}-%
\tfrac{n}{p}},\quad 1<q\leq p<s<\infty
\end{equation*}%
and%
\begin{equation*}
L^{p}\hookrightarrow \dot{K}_{s,r_{1}}^{0,q}B_{p}^{\tfrac{n}{s}-\tfrac{n}{p}%
},\quad 1<p<\min (s,q)<\infty .
\end{equation*}
\end{theorem}

From Theorem \ref{embeddings6.1-lorentz copy(1)} we obtain the following
result.

\begin{corollary}
Under the hypothesis of Theorem \ref{embeddings6.1-lorentz copy(1)}, we have%
\begin{equation*}
\dot{K}_{p}^{\alpha ,r}\hookrightarrow \dot{K}_{s,r_{1}}^{0,q}B_{r}^{\tfrac{n%
}{s}-\tfrac{n}{p}-\alpha }\hookrightarrow B_{s,r}^{\tfrac{n}{s}-\tfrac{n}{p}%
-\alpha }.
\end{equation*}%
In addition, we have%
\begin{equation*}
\dot{K}_{p}^{0,q}\hookrightarrow \dot{K}_{s,r_{1}}^{0,q}B_{p}^{\tfrac{n}{s}-%
\tfrac{n}{p}}\hookrightarrow B_{s,p}^{\tfrac{n}{s}-\tfrac{n}{p}},
\end{equation*}%
whenever $1<q\leq p<s<\infty $ and%
\begin{equation*}
L^{p}\hookrightarrow \dot{K}_{s,r_{1}}^{0,q}B_{p}^{\tfrac{n}{s}-\tfrac{n}{p}%
}\hookrightarrow B_{s,p}^{\tfrac{n}{s}-\tfrac{n}{p}},
\end{equation*}%
whenever $1<p<q<s<\infty .$
\end{corollary}

\subsection{Franke embedding}

The classical Franke embedding may be rewritten as follows: 
\begin{equation*}
B_{p,s}^{s_{2}}\hookrightarrow F_{s,\infty }^{s_{1}},
\end{equation*}%
if $s_{1}-\frac{n}{s}=s_{2}-\frac{n}{p}$ and $0<p<s<\infty $, see e.g. \cite%
{Fr86}. As in Section 3 we will extend this embeddings to Lorentz-Herz-type
Besov-Triebel-Lizorkin spaces. Again, we follow some ideas of \cite{ST19}
and \cite[p. 76]{Vybiral08}. We will prove the discrete version of Franke
embedding.

\begin{theorem}
\label{F-emb3-lorentz}\textit{Let }$\alpha _{1},\alpha _{2},s_{1},s_{2}\in 
\mathbb{R},0<s,p,q<\infty ,0<\theta ,r_{1}\leq \infty ,\alpha _{1}>-\frac{n}{%
s}\ $\textit{and }$\alpha _{2}>-\frac{n}{p}$. \textit{We suppose that }%
\begin{equation*}
s_{1}-\tfrac{n}{s}-\alpha _{1}=s_{2}-\tfrac{n}{p}-\alpha _{2}.
\end{equation*}%
\textit{Let }%
\begin{equation}
0<p<s<\infty ,\text{\quad }\alpha _{2}>\alpha _{1}.  \label{Newexp21-lorentz}
\end{equation}%
Then%
\begin{equation}
\dot{K}_{p,\infty }^{\alpha _{2},q}b_{q}^{s_{2}}\hookrightarrow \dot{K}%
_{s,r_{1}}^{\alpha _{1},q}f_{\theta }^{s_{1}}.
\label{Sobolev-emb1.2-lorentz}
\end{equation}
\end{theorem}

\begin{proof}
Put $c_{n}=1+\lfloor \log _{2}(2\sqrt{n}+1)\rfloor $. Let $\lambda \in \dot{K%
}_{p,\infty }^{\alpha _{2},q}b_{q}^{s_{2}}$. We have%
\begin{align*}
\big\|\lambda \big\|_{\dot{K}_{s,r_{1}}^{\alpha _{1},q}f_{\theta
}^{s_{1}}}^{q}=& \sum\limits_{k=-\infty }^{\infty }2^{k\alpha _{1}q}\Big\|%
\Big(\sum_{v=0}^{\infty }\sum\limits_{m\in \mathbb{Z}^{n}}2^{v(s_{1}+\frac{n%
}{2})\theta }|\lambda _{v,m}|^{\theta }\chi _{v,m}\chi _{k}\Big)^{1/\theta }%
\Big\|_{L^{s,r_{1}}}^{q} \\
=& J_{1,\alpha _{1}}+J_{2,\alpha _{1}},
\end{align*}%
where%
\begin{equation}
J_{1,\alpha _{1}}=\sum\limits_{k=-\infty }^{0}2^{k\alpha _{1}q}\Big\|\Big(%
\sum_{v=0}^{\infty }\sum\limits_{m\in \mathbb{Z}^{n}}2^{v(s_{1}+\frac{n}{2}%
)\theta }|\lambda _{v,m}|^{\theta }\chi _{v,m}\chi _{k}\Big)^{1/\theta }%
\Big\|_{L^{s,r_{1}}}^{q}  \label{J1-lorentz}
\end{equation}%
and%
\begin{equation}
J_{2,\alpha _{1}}=\sum\limits_{k=1}^{\infty }2^{k\alpha _{1}q}\Big\|\Big(%
\sum_{v=0}^{\infty }\sum\limits_{m\in \mathbb{Z}^{n}}2^{v(s_{1}+\frac{n}{2}%
)\theta }|\lambda _{v,m}|^{\theta }\chi _{v,m}\chi _{k}\Big)^{1/\theta }%
\Big\|_{L^{s,r_{1}}}^{q}.  \label{J2-lorentz}
\end{equation}

\textit{Step 1. Estimation of }$J_{1,\alpha _{1}}$\textit{.} We split the
sum $\sum_{v=0}^{\infty }$ in \eqref{J1-lorentz} into two sums one over $%
0\leq v\leq 1+c_{n}-k$ and one over $v\geq 2+c_{n}-k$. The first term is
denoted by $T_{1,\alpha _{1}}$ and the second term by $T_{2,\alpha _{1}}$.
Obviously 
\begin{equation*}
J_{1,\alpha _{1}}\lesssim T_{1,\alpha _{1}}+T_{2,\alpha _{1}}.
\end{equation*}%
The same analysis as in the proof of Theorem \ref{embeddings6-lorentz} shows
that 
\begin{equation*}
\sum_{m\in \mathbb{Z}^{n}}|\lambda _{v,m}|^{t}\chi _{v,m}(x)\lesssim 2^{nv}%
\Big\|\sum_{m\in \mathbb{Z}^{n}}|\lambda _{v,m}|\chi _{v,m}\chi
_{B(0,2^{c_{n}-v+2})}\Big\|_{L^{t,t}}^{t}
\end{equation*}%
for any $x\in R_{k}$. From Lemma \ref{lem:lq-inequality}, since $\alpha _{1}+%
\frac{n}{s}>0$, we have 
\begin{equation*}
T_{1,\alpha _{1}}\leq c\sum\limits_{v=0}^{\infty }2^{v(s_{1}-\alpha _{1}-%
\frac{n}{s}+\frac{n}{t}+\frac{n}{2})q}\Big\|\sum_{m\in \mathbb{Z}%
^{n}}|\lambda _{v,m}|\chi _{v,m}\chi _{B(0,2^{c_{n}-v+2})}\Big\|%
_{L^{t,t}}^{q}.
\end{equation*}%
We may choose $t>0$ such that $\frac{1}{t}>\max (\frac{1}{p},\frac{1}{r_{2}},%
\frac{1}{p}+\frac{\alpha _{2}}{n})$, $\varkappa =\min (1,t)$ and $\frac{n}{d}%
=\frac{n}{t}-\frac{n}{p}-\alpha _{2}$. By Lemma \ref{Lp-estimate}\ and H\"{o}%
lder's inequality, $T_{1,\alpha _{1}}$ is bounded by%
\begin{equation*}
c\sum\limits_{v=0}^{\infty }2^{v\frac{n}{d}q}\Big(\sum_{i=-\infty }^{-v}2^{i(%
\frac{n}{d}+\alpha _{2})\varkappa i}\sup_{j\in \mathbb{N}_{0}}\Big\|%
\sum_{m\in \mathbb{Z}^{n}}2^{(s_{2}+\frac{n}{2})j}|\lambda _{j,m}|\chi
_{j,m}\chi _{i+c_{n}+2}\Big\|_{L^{p,\infty }}^{\varkappa }\Big)^{q/\varkappa
}.
\end{equation*}%
Using Lemma \ref{lem:lq-inequality}, the last term is bounded by%
\begin{equation*}
c\sum_{i=0}^{\infty }2^{-\alpha _{2}iq}\sup_{j\in \mathbb{N}_{0}}\Big\|%
\sum\limits_{m\in \mathbb{Z}^{n}}2^{(s_{2}+\frac{n}{2})j}|\lambda
_{j,m}|\chi _{j,m}\chi _{2-i+c_{n}}\Big\|_{L^{p,\infty }}^{q}\lesssim \big\|%
\lambda \big\|_{\dot{K}_{p,\infty }^{\alpha _{2},q}b_{\infty }^{s_{2}}}^{q}.
\end{equation*}%
\textit{Estimate of }$T_{2,\alpha _{1}}$. We can suppose that $\theta \leq p$%
, since the opposite cases can be obtained by the fact that\ $\ell
^{p}\hookrightarrow \ell ^{\theta }$. We set 
\begin{equation*}
f_{v,k}=\sum_{m\in \mathbb{Z}^{n}}|\lambda _{v,m}|\chi _{v,m}\chi _{k}.
\end{equation*}%
Let $0<u<\min (\frac{s}{\theta },\frac{r_{1}}{\theta },1)$. Since $\alpha
_{2}>\alpha _{1}$ and $s_{1}=s_{2}+\tfrac{n}{s}-\frac{n}{p}+\alpha
_{1}-\alpha _{2}$, we have 
\begin{align*}
\Big\|\Big(\sum_{v=2+c_{n}-k}^{\infty }2^{vs_{1}\theta }f_{v,k}^{\theta }%
\Big)^{1/\theta }\Big\|_{L^{s,r_{1}}}^{u\theta }& =\Big\|%
\sum_{v=2+c_{n}-k}^{\infty }2^{vs_{1}\theta }f_{v,k}^{\theta }\Big\|%
_{L^{s/\theta ,r_{1}/\theta }}^{u} \\
& \leq \sum_{v=2+c_{n}-k}^{\infty }2^{v(s_{2}+\frac{n}{s}-\frac{n}{p}+\alpha
_{1}-\alpha _{2})\theta u}\big\|f_{v,k}\big\|_{L^{s,r_{1}}}^{u\theta } \\
& \leq 2^{k(\alpha _{2}-\alpha _{1})\theta u}\sup_{v\geq
2+c_{n}-k}2^{v(s_{2}+\frac{n}{s}-\frac{n}{p})\theta u}\big\|f_{v,k}\big\|%
_{L^{s,r_{1}}}^{u\theta } \\
& =2^{k(\alpha _{2}-\alpha _{1})\theta u}\sup_{v\geq 2+c_{n}-k}2^{v(s_{2}+%
\frac{n}{s}-\frac{n}{p})\theta u}\big\|f_{v,k}^{\theta _{1}}\big\|%
_{L^{s/\theta _{1},r_{1}/\theta _{1}}}^{u\theta /\theta _{1}}
\end{align*}%
for any $0<\theta _{1}<\infty $. Here the estimates are inspired by \cite%
{ST19}\ and\ \cite{Vybiral08}. We choose $\theta _{1}<\min (s,r_{1})$. Using
duality, 
\begin{equation*}
2^{v(s_{2}+\frac{n}{s}-\frac{n}{p})\theta _{1}}\big\|f_{v,k}^{\theta _{1}}%
\big\|_{L^{s/\theta _{1},r_{1}/\theta _{1}}}
\end{equation*}%
is comparable to%
\begin{equation}
\sup_{g\in L^{(s/\theta _{1})^{\prime },(r_{1}/\theta _{1})^{\prime }},\big\|%
g\big\|_{L^{(s/\theta _{1})^{\prime },(r_{1}/\theta _{1})^{\prime }}}\leq
1}\int_{\mathbb{R}^{n}}2^{v(s_{2}+\frac{n}{s}-\frac{n}{p})\theta
_{1}}(f_{v,k}(x))^{\theta _{1}}g(x)dx.  \label{duality-lorentz}
\end{equation}%
Put 
\begin{equation*}
w_{v,k}(x)=\sum\limits_{m\in \mathbb{Z}^{n}}\left\vert \lambda
_{v,m}\right\vert \chi _{v,m}(x)\chi _{\breve{R}_{k}}(x).
\end{equation*}%
Let $A_{k+v}$ be as in the proof of Theorem \ref{embeddings3 copy(1)-lorentz}%
, $v\geq c_{n}+2-k$ and $k\in \mathbb{Z}$. Recall that 
\begin{equation*}
f_{v,k}\leq \sum\limits_{m\in A_{k+v}}|\lambda _{v,m}|\chi _{v,m}=\Omega
_{v,k}\leq w_{v,k},
\end{equation*}%
It follows from Lemma \ref{Hardy-Littlewood inequality} that 
\begin{align}
& 2^{v(s_{2}+\frac{n}{s}-\frac{n}{p})\theta _{1}}\int_{\mathbb{R}%
^{n}}(f_{v,k}(x))^{\theta _{1}}g(x)dx  \notag \\
& \leq 2^{v(s_{2}+\frac{n}{s}-\frac{n}{p})\theta _{1}}\int_{0}^{\infty
}(f_{v,k}^{\ast }(t))^{\theta _{1}}g^{\ast }(t)dt  \notag \\
& =2^{v(s_{2}+\frac{n}{s}-\frac{n}{p})\theta
_{1}}\int_{0}^{2^{-vn}}(f_{v,k}^{\ast }(t))^{\theta _{1}}g^{\ast
}(t)dt+2^{v(s_{2}+\frac{n}{s}-\frac{n}{p})\theta _{1}}\int_{2^{-vn}}^{\infty
}(f_{v,k}^{\ast }(t))^{\theta _{1}}g^{\ast }(t)dt.  \label{lorentz1.2}
\end{align}%
We see that $\Omega _{v,k}^{\ast }$ is constant in $[0,2^{-vn})$. Using H%
\"{o}lder's inequality, we obtain%
\begin{align*}
2^{v(s_{2}+\frac{n}{s}-\frac{n}{p})\theta
_{1}}\int_{0}^{2^{-vn}}(f_{v,k}^{\ast }(t))^{\theta _{1}}g^{\ast }(t)dt&
\leq 2^{v(s_{2}+\frac{n}{s}-\frac{n}{p})\theta
_{1}}\int_{0}^{2^{-vn}}(\Omega _{v,k}^{\ast }(t))^{\theta _{1}}g^{\ast }(t)dt
\\
& \leq 2^{v(s_{2}+\frac{n}{s}-\frac{n}{p}-\frac{n}{\theta _{1}})\theta
_{1}}(\Omega _{v,k}^{\ast }(2^{vn-1}))^{\theta _{1}}g^{\ast \ast }(2^{-vn})
\\
& =2^{vs_{2}\theta _{1}}2^{-v\frac{n}{p}\theta _{1}}(\Omega _{v,k}^{\ast
}(2^{vn-1}))^{\theta _{1}}2^{-vn(\frac{s}{\theta _{1}})^{\prime }}g^{\ast
\ast }(2^{-vn}) \\
& \leq 2^{vs_{2}\theta _{1}}\big\|\Omega _{v,k}\big\|_{L^{p,\infty
}}^{\theta _{1}}\big\|g\big\|_{L^{(s/\theta _{1})^{\prime },(r_{1}/\theta
_{1})^{\prime }}} \\
& \leq 2^{vs_{2}\theta _{1}}\big\|w_{v,k}\big\|_{L^{p,\infty }}^{\theta
_{1}}.
\end{align*}

The second term of \eqref{lorentz1.2} is comparable to%
\begin{align}
& c2^{v(s_{2}+\frac{n}{s}-\frac{n}{p})\theta _{1}}\sum_{l=0}^{\infty
}(f_{v,k}^{\ast }(2^{(l-v)n}))^{\theta _{1}}2^{(l-v)n}g^{\ast }(2^{(l-v)n}) 
\notag \\
& =c\sum_{l=0}^{\infty }2^{vs_{2}\theta _{1}}(f_{v,k}^{\ast
}(2^{(l-v)n}))^{\theta _{1}}2^{(l-v)n}2^{v(\frac{n}{s}-\frac{n}{p})\theta
_{1}}g^{\ast }(2^{(l-v)n})  \notag \\
& =c\sum_{l=0}^{\infty }2^{vs_{2}\theta _{1}}2^{(l-v)n\frac{\theta _{1}}{p}%
}(f_{v,k}^{\ast }(2^{(l-v)n}))^{\theta _{1}}2^{(l-v)n(1-\frac{\theta _{1}}{p}%
)}2^{v(\frac{n}{s}-\frac{n}{p})\theta _{1}}g^{\ast }(2^{(l-v)n}),
\label{lorentz2}
\end{align}%
where the positive constant is independent of $v$ and $k$. We have%
\begin{align*}
2^{(l-v)n\frac{\theta _{1}}{p}}(f_{v,k}^{\ast }(2^{(l-v)n}))^{\theta _{1}}&
\leq \sup_{z\geq 0}(2^{(z-v)\frac{n}{p}}(f_{v,k}^{\ast
}(2^{(z-v)n-2}))^{\theta _{1}} \\
& \leq \big\|f_{v,k}\big\|_{L^{p,\infty }}^{\theta _{1}}
\end{align*}%
and

\begin{align*}
2^{(l-v)n(1-\frac{\theta _{1}}{p})}2^{v(\frac{n}{s}-\frac{n}{p})\theta
_{1}}g^{\ast }(2^{(l-v)n})& =2^{l(\frac{n}{s}-\frac{n}{p})\theta
_{1}}2^{(l-v)n(1-\frac{\theta _{1}}{p})}2^{(v-l)(\frac{n}{s}-\frac{n}{p}%
)\theta _{1}}g^{\ast }(2^{(l-v)n}) \\
& =2^{l(\frac{n}{s}-\frac{n}{p})\theta _{1}}2^{(l-v)n(1-\frac{\theta _{1}}{s}%
)}g^{\ast }(2^{(l-v)n}) \\
& \leq 2^{l(\frac{n}{s}-\frac{n}{p})\theta _{1}}\big\|g\big\|_{L^{(s/\theta
_{1})^{\prime },\infty }}.
\end{align*}%
Thus \eqref{lorentz2} is bounded by 
\begin{equation*}
2^{vs_{2}\theta _{1}}\big\|f_{v,k}\big\|_{L^{p,\infty }}^{\theta _{1}}.
\end{equation*}%
Inserting this estimate in \eqref{duality-lorentz} we get%
\begin{align*}
\Big\|\Big(\sum_{v=c_{n}+2-k}^{\infty }2^{vs_{1}\theta }f_{v,k}^{\theta }%
\Big)^{1/\theta }\Big\|_{L^{s,r_{1}}}& \leq 2^{k(\alpha _{2}-\alpha
_{1})}\sup_{v\geq c_{n}+2-k}2^{vs_{2}}\big\|w_{v,k}\big\|_{L^{p,\infty }} \\
& \leq 2^{k(\alpha _{2}-\alpha _{1})}\Big(\sum_{v=c_{n}+2-k}^{\infty
}2^{vs_{2}q}\big\|w_{v,k}\big\|_{L^{p,\infty }}^{q}\Big)^{1/q}.
\end{align*}%
Consequently, we obtain that $T_{2,\alpha _{1}}$ can be estimated from above
by $c\big\|\lambda \big\|_{\dot{K}_{p,\infty }^{\alpha
_{2},q}b_{q}^{s_{2}}}^{q}$.

\textit{Step 2. Estimation of }$J_{2,\alpha _{1}}$\textit{.} The arguments
here are quite similar to those used in the estimation of $T_{2,\alpha _{1}}$%
. The proof is complete.
\end{proof}

\begin{remark}
As in Remark \ref{optimal-lorentz}, $q$\ on the right-hand side of %
\eqref{Sobolev-emb1.2-lorentz} is optimal.
\end{remark}

\begin{theorem}
\label{F-emb3-copy(1)-lorentz}\textit{Let }$\alpha ,s_{1},s_{2}\in \mathbb{R}%
,0<s,p,q<\infty ,0<\theta ,r_{1}\leq \infty $ \textit{and }$\alpha >-\frac{n%
}{s}$. \textit{We suppose that }$0<p<s<\infty $\textit{\ and}%
\begin{equation*}
s_{1}-\tfrac{n}{s}=s_{2}-\tfrac{n}{p}.
\end{equation*}%
Then%
\begin{equation*}
\dot{K}_{p,\infty }^{\alpha ,q}b_{p_{1}}^{s_{2}}\hookrightarrow \dot{K}%
_{s,r_{1}}^{\alpha ,q}f_{\theta }^{s_{1}},
\end{equation*}%
where%
\begin{equation*}
p_{1}=\left\{ 
\begin{array}{ccc}
q, & \text{if} & q\leq r_{1}, \\ 
r_{1}, & \text{if} & q>r_{1}.%
\end{array}%
\right.
\end{equation*}
\end{theorem}

\begin{proof}
Put $c_{n}=1+\lfloor \log _{2}(2\sqrt{n}+1)\rfloor $. Let $\lambda \in \dot{K%
}_{p,\infty }^{\alpha ,q}b_{p_{1}}^{s_{2}}$. We write as in Theorem \ref%
{F-emb3-lorentz}, 
\begin{equation*}
\big\|\lambda \big\|_{\dot{K}_{s,r_{1}}^{\alpha ,q}f_{\theta
}^{s_{1}}}^{q}=J_{1,\alpha }+J_{2,\alpha }.
\end{equation*}%
We split the sum $\sum_{v=0}^{\infty }$ in \eqref{J1-lorentz}\ with $\alpha
_{1}=\alpha $ into two sums one over $0\leq v\leq c_{n}+1-k$ and one over $%
v\geq c_{n}+2-k$. The first term is denoted by $T_{1,\alpha }$ and the
second term by $T_{2,\alpha }$. Obviously 
\begin{equation*}
J_{1,\alpha }\lesssim T_{1,\alpha }+T_{2,\alpha }.
\end{equation*}%
The same analysis as in the proof of Theorem \ref{F-emb3-lorentz} shows that%
\begin{equation*}
T_{1,\alpha }\lesssim \big\|\lambda \big\|_{\dot{K}_{p,\infty }^{\alpha
,q}b_{\infty }^{s_{2}}}^{q}.
\end{equation*}%
\textit{Estimate of }$T_{2,\alpha }$. We can suppose that $\theta \leq p$,
since the opposite cases can be obtained by the fact that\ $\ell
^{p}\hookrightarrow \ell ^{\theta }$, if $p\leq \theta $. We distinguish two
cases.

\textit{Case 1. }$\theta <r_{1}$. We have\ $T_{2,\alpha }$ can be estimated
from above by%
\begin{equation*}
\sum\limits_{k=-\infty }^{0}2^{k\alpha _{1}q}\Big\|\Big(\sum_{v=c_{n}+2-k}^{%
\infty }\sum_{m\in \mathbb{Z}^{n}}2^{vs_{1}\theta }|\lambda _{v,m}|^{\theta
}\chi _{v,m}\Big)^{1/\theta }\Big\|_{L^{s,r_{1}}}^{q}.
\end{equation*}%
We have 
\begin{equation}
\Big\|\Big(\sum_{v=c_{n}+2-k}^{\infty }2^{vs_{1}\theta }f_{v,k}^{\theta }%
\Big)^{1/\theta }\Big\|_{L^{s,r_{1}}}^{\theta }=\Big\|\sum_{v=c_{n}+2-k}^{%
\infty }2^{vs_{1}\theta }f_{v,k}^{\theta }\Big\|_{L^{s/\theta ,r_{1}/\theta
}}  \label{lorentz3}
\end{equation}%
Using duality, the right-hand side of \eqref{lorentz3} is comparable to%
\begin{equation*}
\sup_{g\in L^{(s/\theta )^{\prime },(r_{1}/\theta )^{\prime }},\big\|g\big\|%
_{L^{(s/\theta )^{\prime },(r_{1}/\theta )^{\prime }}}\leq 1}\int_{\mathbb{R}%
^{n}}\sum_{v=c_{n}+2-k}^{\infty }2^{vs_{1}\theta }(f_{v,k}(x))^{\theta
}g(x)dx.
\end{equation*}%
Put 
\begin{equation*}
w_{v,k}(x)=\sum\limits_{m\in \mathbb{Z}^{n}}\left\vert \lambda
_{v,m}\right\vert \chi _{v,m}(x)\chi _{\breve{R}_{k}}(x).
\end{equation*}%
Let $A_{k+v}$ be as in the proof of Theorem \ref{embeddings3 copy(2)-lorentz}%
, $v\geq c_{n}+2-k$ and $k\in \mathbb{Z}$. We have 
\begin{equation*}
f_{v,k}\leq \sum\limits_{m\in A_{k+v}}|\lambda _{v,m}|\chi _{v,m}=\Omega
_{v,k}\leq w_{v,k}
\end{equation*}%
and $\Omega _{v,k}^{\ast }$ is constant in $[0,2^{-vn})$. It follows from
Lemma \ref{Hardy-Littlewood inequality} that 
\begin{align}
& \sum_{v=c_{n}+2-k}^{\infty }2^{vs_{1}\theta }\int_{\mathbb{R}%
^{n}}(f_{v,k}(x))^{\theta }g(x)dx  \notag \\
& =\sum_{v=c_{n}+2-k}^{\infty }2^{vs_{1}\theta }\int_{0}^{\infty }(\Omega
_{v,k}^{\ast }(t))^{\theta }g^{\ast }(t)dt  \notag \\
& \leq \sum_{v=c_{n}+2-k}^{\infty }\sum_{l=0}^{\infty }2^{vs_{1}\theta
}(\Omega _{v,k}^{\ast }(2^{(l-v)n-1})^{\theta }2^{(l-v)n}g^{\ast \ast
}(2^{(l-v)n+n}).  \label{lorentz1.1}
\end{align}%
Since $s_{1}-s_{2}=\tfrac{n}{s}-\frac{n}{p}$, we obtain that %
\eqref{lorentz1.1} is just%
\begin{align}
& \sum_{l=0}^{\infty }\sum_{v=c_{n}+2-k}^{\infty }2^{vs_{2}\theta }(\Omega
_{v,k}^{\ast }(2^{(l-v)n-1}))^{\theta }2^{(l-v)n}2^{v(\frac{n}{s}-\frac{n}{p}%
)\theta }g^{\ast \ast }(2^{(l-v)n})  \notag \\
& =\sum_{l=0}^{\infty }\sum_{v=c_{n}+2-k}^{\infty }2^{vs_{2}\theta }2^{(l-v)n%
\frac{\theta }{p}}(\Omega _{v,k}^{\ast }(2^{(l-v)n-1}))^{\theta }2^{(l-v)n(1-%
\frac{\theta }{p})}2^{v(\frac{n}{s}-\frac{n}{p})\theta }g^{\ast \ast
}(2^{(l-v)n}).  \label{lorentz2.1}
\end{align}%
H\"{o}lder's inequality implies that the\ second sum in \eqref{lorentz2.1}
can be estimated from above by%
\begin{align*}
& \Big(\sum_{v=0}^{\infty }2^{vs_{2}r_{1}}2^{(l-v)n\frac{r_{1}}{p}%
}(w_{v,k}^{\ast }(2^{(l-v)n-1}))^{r_{1}}\Big)^{\theta /r_{1}} \\
& \times \Big(\sum_{v=0}^{\infty }(2^{(l-v)n(1-\frac{\theta }{p})}2^{v(\frac{%
n}{s}-\frac{n}{p})\theta }g^{\ast \ast }(2^{(l-v)n}))^{(r_{1}/\theta
_{1})^{\prime }}\Big)^{1/(r_{1}/\theta )^{\prime }} \\
& \leq \Big(\sum_{v=0}^{\infty }2^{vs_{2}r_{1}}\sup_{h\geq 0}(2^{(h-v)n\frac{%
r_{1}}{p}}(w_{v,k}^{\ast }(2^{(h-v)n-1}))^{r_{1}})\Big)^{\theta /r_{1}} \\
& \times \Big(\sum_{v=0}^{\infty }(2^{(l-v)n(1-\frac{\theta }{p})}2^{v(\frac{%
n}{s}-\frac{n}{p})\theta }g^{\ast \ast }(2^{(l-v)n}))^{(r_{1}/\theta
)^{\prime }}\Big)^{1/(r_{1}/\theta )^{\prime }} \\
& \leq \Big(\sum_{v=0}^{\infty }2^{vs_{2}r_{1}}\big\|w_{v,k}\big\|%
_{L^{p,\infty }}^{r_{1}}\Big)^{\theta /r_{1}} \\
& \times \Big(\sum_{v=0}^{\infty }\big(2^{(l-v)n(1-\frac{\theta }{p})}2^{v(%
\frac{n}{s}-\frac{n}{p})\theta }g^{\ast \ast }(2^{(l-v)n})\big)%
^{(r_{1}/\theta )^{\prime }}\Big)^{1/(r_{1}/\theta )^{\prime }}.
\end{align*}

Observe that

\begin{align*}
& \sum_{v=0}^{\infty }\big(2^{(l-v)n(1-\frac{\theta }{p})}2^{v(\frac{n}{s}-%
\frac{n}{p})\theta }g^{\ast \ast }(2^{(l-v)n})\big)^{(r_{1}/\theta )^{\prime
}} \\
& \leq 2^{l(\frac{n}{s}-\frac{n}{p})(r_{1}/\theta )^{\prime }\theta
}\sum_{v=0}^{\infty }\big(2^{(l-v)n(1-\frac{\theta }{p})}2^{(v-l)(\frac{n}{s}%
-\frac{n}{p})\theta }g^{\ast \ast }(2^{(l-v)n})\big)^{(r_{1}/\theta
)^{\prime }} \\
& \leq 2^{l(\frac{n}{s}-\frac{n}{p})(r_{1}/\theta )^{\prime }\theta
}\sum_{v=0}^{\infty }\big(2^{(l-v)n(1-\frac{\theta }{s})}g^{\ast \ast
}(2^{(l-v)n})\big)^{(r_{1}/\theta )^{\prime }} \\
& \leq 2^{l(\frac{n}{s}-\frac{n}{p})(r_{1}/\theta )^{\prime }\theta }\big\|g%
\big\|_{L^{(s/\theta )^{\prime },(r_{1}/\theta )^{\prime }}}^{(r_{1}/\theta
)^{\prime }}.
\end{align*}%
Thus \eqref{lorentz2.1} is bounded by 
\begin{equation*}
c\Big(\sum_{v=0}^{\infty }2^{vs_{2}r_{1}}\big\|w_{v,k}\big\|_{L^{p,\infty
}}^{r_{1}}\Big)^{\theta /r_{1}}.
\end{equation*}%
Using the well-known inequality 
\begin{equation*}
\Big(\sum_{j=0}^{\infty }\left\vert a_{j}\right\vert \Big)^{\varrho }\leq
\sum_{j=0}^{\infty }\left\vert a_{j}\right\vert ^{\varrho },\quad \left\{
a_{j}\right\} _{j}\subset \mathbb{C},\text{ }\varrho \in (0,1]
\end{equation*}%
if $q\leq r_{1}$ and Minkowski inequality if $q>r_{1}$ we obtain that $%
T_{2,\alpha }$ can be estimated from above by $c\big\|\lambda \big\|_{\dot{K}%
_{p,\infty }^{\alpha _{2},q}b_{p_{1}}^{s_{2}}}^{q}$.

\textit{Case 2. }$\theta \geq r_{1}$. Let $r_{2}>0$ be such that $%
r_{2}<r_{1} $. The left-hand side of \eqref{lorentz3} is bounded by 
\begin{equation*}
\Big\|\Big(\sum_{v=c_{n}+2-k}^{\infty }2^{vs_{1}r_{2}}f_{v,k}^{r_{2}}\Big)%
^{1/r_{2}}\Big\|_{L^{s,r_{1}}}^{r_{2}}.
\end{equation*}%
Now, repeating the arguments of Case 1, we deduce that $T_{2,\alpha
}\lesssim \big\|\lambda \big\|_{\dot{K}_{p,\infty }^{\alpha
_{2},q}b_{p_{1}}^{s_{2}}}^{q}.$

\textit{Estimation of }$J_{2,\alpha }$\textit{.} We use the same arguments
as in the estimation of $J_{2,\alpha _{1}}$ of Theorem \ref{F-emb3-lorentz}
to obtain $J_{2,\alpha }\lesssim \big\|\lambda \big\|_{\dot{K}_{p,\infty
}^{\alpha _{2},q}b_{p_{1}}^{s_{2}}}^{q}$. The proof is complete.
\end{proof}

\begin{theorem}
\label{F-emb3 copy(2)-lorentz}\textit{Let }$\alpha _{1},\alpha
_{2},s_{1},s_{2}\in \mathbb{R},0<s,p,q<\infty ,0<\theta ,r_{1}\leq \infty
,\alpha _{1}>-\frac{n}{s}\ $\textit{and }$\alpha _{2}>-\frac{n}{p}$. \textit{%
We suppose that }%
\begin{equation*}
s_{1}-\tfrac{n}{s}-\alpha _{1}=s_{2}-\tfrac{n}{p}-\alpha _{2}.
\end{equation*}%
\textit{Let}%
\begin{equation}
0<s<p<\infty \text{\quad and\quad }\alpha _{2}+\tfrac{n}{p}>\alpha _{1}+%
\tfrac{n}{s}.  \label{newexp2.2-lorentz}
\end{equation}%
Then%
\begin{equation*}
\dot{K}_{p,\infty }^{\alpha _{2},q}b_{q}^{s_{2}}\hookrightarrow \dot{K}%
_{s,r_{1}}^{\alpha _{1},q}f_{\theta }^{s_{1}}.
\end{equation*}
\end{theorem}

\begin{proof}
We prove our embedding under the conditions \eqref{newexp2.2-lorentz}.
Obviously, we have $T_{1,\alpha _{1}}\lesssim \big\|\lambda \big\|_{\dot{K}%
_{p,\infty }^{\alpha _{2},q}b_{q}^{s_{2}}}$, so we need only to estimate $%
T_{2,\alpha _{1}}$. Let $0<\tau <\min (1,\frac{s}{\theta },\frac{r_{1}}{%
\theta })$. Minkowski's inequality; see \eqref{seeger}, yields 
\begin{align*}
& \Big\|\Big(\sum_{v=c_{n}+2-k}^{\infty }\sum\limits_{m\in \mathbb{Z}%
^{n}}2^{v(s_{1}+\frac{n}{2})\theta }|\lambda _{v,m}|^{\theta }\chi
_{v,m}\chi _{k}\Big)^{1/\theta }\Big\|_{L^{s,r_{1}}} \\
& \lesssim \Big(\sum_{v=c_{n}+2-k}^{\infty }2^{v(s_{1}+\frac{n}{2})\tau
\theta }\Big\|\sum\limits_{m\in \mathbb{Z}^{n}}|\lambda _{v,m}|\chi
_{v,m}\chi _{k}\Big\|_{L^{s,r_{1}}}^{\tau \theta }\Big)^{1/\theta \tau }.
\end{align*}%
By H\"{o}lder's inequality we obtain%
\begin{equation*}
2^{vs_{1}}\Big\|\sum\limits_{m\in \mathbb{Z}^{n}}\lambda _{v,m}\chi
_{v,m}\chi _{k}\Big\|_{L^{s,r_{1}}}\lesssim 2^{(\frac{n}{s}-\frac{n}{p}%
)k+vs_{1}}\Big\|\sum\limits_{m\in \mathbb{Z}^{n}}\lambda _{v,m}\chi
_{v,m}\chi _{k}\Big\|_{L^{p,\infty }},
\end{equation*}%
where the implicit constant is independent of $v$ and $k$. Put 
\begin{equation*}
\mu =\alpha _{1}+\frac{n}{s}-\frac{n}{p}-\alpha _{2}+s_{2}+\frac{n}{2}\quad 
\text{and}\quad \eta =\alpha _{1}+\frac{n}{s}-\frac{n}{p}-\alpha _{2}.
\end{equation*}%
\ Hence $T_{2,\alpha _{1}}$ can be\ estimated from above by%
\begin{equation*}
c\sum\limits_{k=-\infty }^{0}2^{k(\alpha _{1}+\frac{n}{s}-\frac{n}{p})q}\Big(%
\sum_{v=1-k}^{\infty }2^{v\mu \tau \theta }\Big\|\sum\limits_{m\in \mathbb{Z}%
^{n}}|\lambda _{v,m}|\chi _{v,m}\chi _{k}\Big\|_{L^{p,\infty }}^{\tau \theta
}\Big)^{q/\theta \tau },
\end{equation*}%
which is just%
\begin{align*}
& c\sum\limits_{k=-\infty }^{0}2^{k\alpha _{2}q}\Big(\sum_{v=c_{n}+2-k}^{%
\infty }2^{(v+k)\eta \tau \theta }2^{v(s_{2}+\frac{n}{2})\tau \theta }\Big\|%
\sum\limits_{m\in \mathbb{Z}^{n}}|\lambda _{v,m}|\chi _{v,m}\chi _{k}\Big\|%
_{L^{p,\infty }}^{\tau \theta }\Big)^{q/\theta \tau } \\
& \lesssim \sum\limits_{k=-\infty }^{0}2^{k\alpha _{2}q}\Big(\sup_{v\in 
\mathbb{N}_{0}}2^{v(s_{2}+\frac{n}{2})}\Big\|\sum\limits_{m\in \mathbb{Z}%
^{n}}|\lambda _{v,m}|\chi _{v,m}\chi _{k}\Big\|_{L^{p,\infty }}\Big)^{q} \\
& \lesssim \big\|\lambda \big\|_{\dot{K}_{p,\infty }^{\alpha
_{2},q}b_{q}^{s_{2}}}.
\end{align*}%
H\"{o}lder's inequality, Minkowski's inequality and the fact that $\eta <0$
lead to%
\begin{align*}
& \Big\|\Big(\sum_{v=0}^{\infty }\sum\limits_{m\in \mathbb{Z}^{n}}2^{v(s_{1}+%
\frac{n}{2})\theta }|\lambda _{v,m}|^{\theta }\chi _{v,m}\chi _{k}\Big)%
^{1/\theta }\Big\|_{L^{s,r_{1}}} \\
& \lesssim 2^{k(\frac{n}{s}-\frac{n}{p})}\sup_{v\in \mathbb{N}%
_{0}}2^{v(s_{2}+\frac{n}{2})}\Big\|\sum\limits_{m\in \mathbb{Z}^{n}}|\lambda
_{v,m}|\chi _{v,m}\chi _{k}\Big\|_{L^{p,\infty }} \\
& \lesssim 2^{k(\frac{n}{s}-\frac{n}{p}-\alpha _{2})}\big\|\lambda \big\|_{%
\dot{K}_{p,\infty }^{\alpha _{2},q}b_{q}^{s_{2}}}
\end{align*}%
for any $k\in \mathbb{N}_{0}$, where the implicit constant is independent of 
$k$. Thus, $J_{2,\alpha _{1}}\lesssim \big\|\lambda \big\|_{\dot{K}%
_{p,\infty }^{\alpha _{2},q}b_{q}^{s_{2}}}$. The proof is complete.
\end{proof}

Using Theorems \ref{phi-tran-lorentz} and \ref{F-emb3-lorentz}, we have the
following Franke embedding.

\begin{theorem}
\label{embeddings6.2-lorentz}\textit{Let }$\alpha ,\alpha _{1},\alpha
_{2},s_{1},s_{2}\in \mathbb{R}$, $0<r_{1}\leq \infty ,0<s,p,q,r<\infty
,\alpha >-\frac{n}{s},\alpha _{1}>-\frac{n}{s}\ $\textit{and }$\alpha _{2}>-%
\frac{n}{p}$.\ \newline
$\mathrm{(i)}$ Under the hypothesis of Theorem \ref{F-emb3-lorentz}\ we have%
\begin{equation*}
\dot{K}_{p,\infty }^{\alpha _{2},q}B_{q}^{s_{2}}\hookrightarrow \dot{K}%
_{s,r_{1}}^{\alpha _{1},q}F_{\theta }^{s_{1}}.
\end{equation*}%
$\mathrm{(ii)}$ Under the hypothesis of Theorem \ref{F-emb3-copy(1)-lorentz}
we have%
\begin{equation*}
\dot{K}_{p,\infty }^{\alpha ,q}B_{p_{1}}^{s_{2}}\hookrightarrow \dot{K}%
_{s,r_{1}}^{\alpha ,q}F_{\theta }^{s_{1}},
\end{equation*}%
where%
\begin{equation*}
p_{1}=\left\{ 
\begin{array}{ccc}
q, & \text{if} & q\leq r_{1}, \\ 
r_{1}, & \text{if} & q>r_{1}.%
\end{array}%
\right.
\end{equation*}%
$\mathrm{(iii)}$ Under the hypothesis of Theorem \ref{embeddings6-lorentz
copy(2)} we have%
\begin{equation*}
\dot{K}_{p,\infty }^{\alpha _{2},q}B_{p}^{s_{2}}\hookrightarrow \dot{K}%
_{s,r_{1}}^{\alpha _{1},q}F_{\theta }^{s_{1}}.
\end{equation*}
\end{theorem}

We observe that from Theorem \ref{embeddings6.2-lorentz}/(ii) we obtain the
following statement.

\begin{corollary}
Let $0<p<s<\infty ,0<\theta \leq \infty $ and\textit{\ }%
\begin{equation*}
s_{1}-\tfrac{n}{s}=s_{2}-\tfrac{n}{p}.
\end{equation*}%
Then%
\begin{equation*}
B_{p,s}^{s_{2}}\hookrightarrow \dot{K}_{p,\infty
}^{0,s}B_{s}^{s_{2}}\hookrightarrow F_{s,\theta }^{s_{1}}.
\end{equation*}
\end{corollary}

Again by Theorem \ref{embeddings6.2-lorentz}, we immediately arrive at the
following embedding between Herz and Besov spaces.

\begin{theorem}
\label{embeddings6.2-lorentz copy(1)}\textit{Let }$\alpha \in \mathbb{R}%
,1<s,q,p<\infty \ $\textit{and }$-\frac{n}{s}<\alpha \leq 0$. \newline
$\mathrm{(i)}$ \textit{We suppose that }$\mathit{\max }(1,p)<s<\infty \ $and 
$-\frac{n}{s}<\alpha <0$.\textit{\ Then}%
\begin{equation}
\dot{K}_{p,\infty }^{0,q}B_{q}^{\tfrac{n}{p}-\tfrac{n}{s}-\alpha
}\hookrightarrow \dot{K}_{s}^{\alpha ,q}.  \label{lorentz5}
\end{equation}%
In addition, if $1<p\leq q<\infty $, then we have%
\begin{equation*}
B_{p,q}^{\tfrac{n}{p}-\tfrac{n}{s}-\alpha }\hookrightarrow \dot{K}_{p,\infty
}^{0,q}B_{q}^{\tfrac{n}{p}-\tfrac{n}{s}-\alpha }\hookrightarrow \dot{K}%
_{s}^{\alpha ,q}.
\end{equation*}%
$\mathrm{(ii)}$ \textit{We suppose that }$1<p<s<\infty $.\textit{\ Then}%
\begin{equation}
B_{p,p_{1}}^{\tfrac{n}{p}-\tfrac{n}{s}}\hookrightarrow \dot{K}_{p,\infty
}^{0,q}B_{p_{1}}^{\tfrac{n}{p}-\tfrac{n}{s}}\hookrightarrow \dot{K}%
_{s}^{0,q}.  \label{lorentz6}
\end{equation}%
where%
\begin{equation*}
p_{1}=\left\{ 
\begin{array}{ccc}
q, & \text{if} & q\leq s, \\ 
s, & \text{if} & q>s.%
\end{array}%
\right.
\end{equation*}%
In addition if $1<p\leq q<\infty $, then we have%
\begin{equation*}
B_{p,p_{1}}^{\tfrac{n}{p}-\tfrac{n}{s}}\hookrightarrow \dot{K}_{p,\infty
}^{0,q}B_{p_{1}}^{\tfrac{n}{p}-\tfrac{n}{s}}\hookrightarrow \dot{K}%
_{s}^{0,q}.
\end{equation*}%
\textit{\ }$\mathrm{(iii)}$ \textit{We suppose that }$1<s<p<\infty \ $and $-%
\frac{n}{s}<\alpha <\frac{n}{p}-\frac{n}{s}$.\textit{\ Then the embeddings }%
\eqref{lorentz5} holds. In addition if $1<p\leq q<\infty $, then we have the
embeddings \eqref{lorentz6}.
\end{theorem}

\begin{remark}
Theorem \ref{embeddings6.2-lorentz copy(1)} extends and improves the
corresponding results obtained in \cite{drihem2016jawerth}.
\end{remark}

\begin{remark}
The same analysis as in Theorem \ref{embeddings3-lorentz} can be used to
prove that in Theorem \ref{embeddings6.2-lorentz} the assumptions 
\begin{equation*}
s_{1}-\tfrac{n}{s}-\alpha _{1}\leq s_{2}-\tfrac{n}{p}-\alpha _{2}\quad \text{%
and}\quad \alpha _{2}+\tfrac{n}{p}\geq \alpha _{1}+\tfrac{n}{s}
\end{equation*}%
\ become necessary.
\end{remark}

We now present an immediate consequence of the Franke embeddings.

\begin{corollary}
\textit{Let }$1<p<s<\infty $\ with $1<p<n$.\ \textit{Let }$\alpha =\frac{n}{p%
}-\frac{n}{s}-1<0$. There is a constant $c>0$ such that for all $f\in
B_{p,s}^{1}$ 
\begin{equation*}
\int_{\mathbb{R}^{n}}\Big(\frac{|f(x)|}{|x|^{-\alpha }}\Big)^{s}dx\leq c%
\big\|f\big\|_{\dot{K}_{p,\infty }^{0,s}B_{s}^{1}}^{s}\leq c\big\|f\big\|%
_{B_{p,s}^{1}}^{s}.
\end{equation*}
\end{corollary}

Concerning embeddings $\dot{K}_{q,r}^{\alpha ,\theta }F_{\beta }^{s}$ into $%
L^{\infty }$, we have the following result.

\begin{theorem}
\label{bounded1 copy(1)}\textit{Let }$\alpha \geq 0$ \textit{and }$%
0<q,p<\infty $ and $0<\theta \leq \infty $\textit{. }\newline
$\mathrm{(i)}$\textit{\ Let }$\alpha >0$. \textit{Assume that }$s>\alpha +%
\frac{n}{p}\ $or$\ s=\alpha +\frac{n}{p}$\ and $0<q\leq 1$. \textit{We have}%
\begin{equation*}
\dot{K}_{p,\infty }^{\alpha ,q}F_{\theta }^{s}\hookrightarrow L^{\infty }.
\end{equation*}%
$\mathrm{(ii)}$\textit{\ Assume that }$s>\frac{n}{p}\ $or$\ s=\frac{n}{p}$\
and $0<r,q\leq 1$. Then%
\begin{equation*}
\dot{K}_{p,r}^{0,q}F_{\theta }^{s}\hookrightarrow L^{\infty },
\end{equation*}%
holds.
\end{theorem}

\begin{proof}
First assume that $\alpha >0$. Let $0<p<v<\infty $. It follows from Theorem %
\ref{embeddings6.1-lorentz}/(i) that%
\begin{equation*}
\dot{K}_{p,\infty }^{\alpha ,q}F_{\infty }^{\alpha +\frac{n}{p}%
}\hookrightarrow B_{v,q}^{\frac{n}{v}}.
\end{equation*}%
Hence the result follows by the embedding $B_{v,q}^{\frac{n}{v}%
}\hookrightarrow L^{\infty }$; see \cite{SiTr}.\ Now we study the case $%
\alpha =0$. If $q\leq r$, then by Theorems \ref{embeddings6.1-lorentz}/(ii)
and \ref{bounded1-lorentz} we have%
\begin{equation*}
\dot{K}_{p,r}^{0,q}F_{\theta }^{s}\hookrightarrow \dot{K}%
_{v,r_{1}}^{0,q}B_{r}^{s+\frac{n}{v}-\frac{n}{p}}\hookrightarrow L^{\infty }.
\end{equation*}%
If $q>r$, again, by Theorem \ref{embeddings6.1-lorentz}/(ii) we obtain%
\begin{equation*}
\dot{K}_{p,r}^{0,q}F_{\theta }^{s}\hookrightarrow \dot{K}_{p,q}^{0,q}F_{%
\theta }^{s}\hookrightarrow \dot{K}_{v,r_{1}}^{0,q}B_{q}^{s+\frac{n}{v}-%
\frac{n}{p}}\hookrightarrow L^{\infty }.
\end{equation*}%
The proof is complete.
\end{proof}

\begin{remark}
The results obtained in Subsections 4.3 and 4.4 extend and improve the
corresponding results of \cite{drihem2016jawerth}. In particular
Franke-Jawerth embeddings for Besov and Triebel-Lizorkin spaces\ of power
weight\ obtained\ in \cite{MM12}.
\end{remark}

\section{Atomic, molecular and wavelet characterizations}

In the first part of this section we will prove that under certain
restrictions on the parameters the spaces $\dot{K}_{p,r}^{\alpha ,q}A_{\beta
}^{s}$ can be characterized by smooth molecules and smooth atoms. The second
part is devoted to the characterization of the spaces $\dot{K}_{p,r}^{\alpha
,q}A_{\beta }^{s}$ by wavelet. The contents of this section are based on 
\cite{FJ90}, \cite{triebel08}.

\subsection{Atomic and molecular characterizations}

We will use the notation of \cite{FJ90}. We shall say that an operator $A$
is associated with the matrix $\{a_{Q_{k,m}P_{v,h}}\}_{k,v\in \mathbb{N}%
_{0},m,h\in \mathbb{Z}^{n}}$, if for all sequences $\lambda =\{\lambda
_{k,m}\}_{k\in \mathbb{N}_{0},m\in \mathbb{Z}^{n}}\subset \mathbb{C}$,%
\begin{equation*}
A\lambda =\{(A\lambda )_{k,m}\}_{k\in \mathbb{Z},m\in \mathbb{Z}^{n}}=\Big\{%
\sum_{v=0}^{\infty }\sum_{h\in \mathbb{Z}^{n}}a_{Q_{k,m}P_{v,h}}\lambda
_{v,h}\Big\}_{k\in \mathbb{N}_{0},m\in \mathbb{Z}^{n}}.
\end{equation*}%
We will use the notation%
\begin{equation*}
J=\left\{ 
\begin{array}{ccc}
\frac{n}{\min (1,p,\frac{n}{\alpha +\frac{n}{p}})}, & \text{if} & \dot{K}%
_{p,r}^{\alpha ,q}A_{\beta }^{s}=\dot{K}_{p,r}^{\alpha ,q}B_{\beta }^{s}, \\ 
\frac{n}{\min (1,p,\beta ,\frac{n}{\alpha +\frac{n}{p}})}, & \text{if} & 
\dot{K}_{p,r}^{\alpha ,q}A_{\beta }^{s}=\dot{K}_{p,r}^{\alpha ,q}F_{\beta
}^{s}.%
\end{array}%
\right. .
\end{equation*}%
We say that $A$, with associated matrix $\{a_{Q_{k,m}P_{v,h}}\}_{k,v\in 
\mathbb{N}_{0},m,h\in \mathbb{Z}^{n}}$, is almost diagonal on $\dot{K}%
_{p,r}^{\alpha _{2},q}a_{\beta }^{s}$ if there exists $\varepsilon >0$ such
that%
\begin{equation*}
\sup_{k,v\in \mathbb{N}_{0},m,h\in \mathbb{Z}^{n}}\frac{|a_{Q_{k,m}P_{v,h}}|%
}{\omega _{Q_{k,m}P_{v,h}}(\varepsilon )}<\infty ,
\end{equation*}%
where

\begin{align}
& \omega _{Q_{k,m}P_{v,h}}(\varepsilon )  \notag \\
& =\Big(1+\frac{|x_{Q_{k,m}}-x_{P_{v,h}}|}{\max \big(2^{-k},2^{-v}\big)}\Big)%
^{-J-\varepsilon }\left\{ 
\begin{array}{ccc}
2^{(v-k)(s+\frac{n+\varepsilon }{2})}, & \text{if} & v\leq k, \\ 
2^{(v-k)(s-\frac{n+\varepsilon }{2}-J+n)}, & \text{if} & v>k.%
\end{array}%
\right.  \label{omega-assumption}
\end{align}

The following theorem is a generalization of {\cite[Theorem 3.3]{FJ90}.}

\begin{theorem}
\label{almost-diag-est}Let $s\in \mathbb{R},0<p<\infty ,0<r,q\leq \infty
,0<\beta <\infty $ and $\alpha >-\frac{n}{p}$. Any almost diagonal operator $%
A$ on $\dot{K}_{p,r}^{\alpha ,q}a_{\beta }^{s}$ is bounded.
\end{theorem}

\begin{proof}
By similarity, we consider only the spaces $\dot{K}_{p,r}^{\alpha
,q}f_{\beta }^{s}$. We write $A\equiv A_{0}+A_{1}$ with\ 
\begin{equation*}
(A_{0}\lambda )_{k,m}=\sum_{v=0}^{k}\sum_{h\in \mathbb{Z}%
^{n}}a_{Q_{k,m}P_{v,h}}\lambda _{v,h},\quad k\in \mathbb{N}_{0},m\in \mathbb{%
Z}^{n}
\end{equation*}%
and 
\begin{equation*}
(A_{1}\lambda )_{k,m}=\sum_{v=k+1}^{\infty }\sum_{h\in \mathbb{Z}%
^{n}}a_{Q_{k,m}P_{v,h}}\lambda _{v,h},\quad k\in \mathbb{N}_{0},m\in \mathbb{%
Z}^{n}.
\end{equation*}

\textit{Estimate of }$A_{0}$. From $\mathrm{\eqref{omega-assumption}}$,\ we
obtain 
\begin{align*}
\big|(A_{0}\lambda )_{k,m}\big|& \leq \sum_{v=0}^{k}\sum_{h\in \mathbb{Z}%
^{n}}2^{(v-k)(\alpha _{2}+\frac{n+\varepsilon }{2})}\frac{|\lambda _{v,h}|}{%
\big(1+2^{v}|x_{k,m}-x_{v,h}|\big)^{J+\varepsilon }} \\
& =\sum_{v=0}^{k}2^{(v-k)(\alpha _{2}+\frac{n+\varepsilon }{2})}S_{k,v,m}.
\end{align*}%
{For each }$j\in \mathbb{N},k\in \mathbb{N}_{0}$\ and $m\in \mathbb{Z}^{n}${%
\ we define } 
\begin{equation*}
{\Omega _{j,k,m}=\{h\in \mathbb{Z}^{n}:2^{j-1}<2^{v}|x_{k,m}-x_{v,h}|\leq
2^{j}\}}
\end{equation*}%
{and } 
\begin{equation*}
{\Omega _{0,k,m}=\{h\in \mathbb{Z}^{n}:2^{v}|x_{k,m}-x_{v,h}|\leq 1\}.}
\end{equation*}%
Let $\frac{n}{J+\frac{\varepsilon }{2}}<\tau <\min (1,p,\beta ,\frac{n}{%
\alpha +\frac{n}{p}})$. We rewrite $S_{k,v,m}$ as follows 
\begin{align*}
S_{k,v,m}& =\sum\limits_{j=0}^{\infty }\sum\limits_{h\in {\Omega _{j,k,m}}}%
\frac{|\lambda _{v,h}|}{\big(1+2^{v}|x_{k,m}-x_{v,h}|\big)^{J+\varepsilon }}
\\
& \leq \sum\limits_{j=0}^{\infty }2^{-(J+\varepsilon )j}\sum\limits_{h\in {\
\Omega _{j,k,m}}}|\lambda _{v,h}|.
\end{align*}%
By the embedding $\ell _{\tau }\hookrightarrow \ell _{1}$ we deduce that 
\begin{align*}
S_{k,v,m}& \leq \sum\limits_{j=0}^{\infty }2^{-(J+\varepsilon )j}\big(%
\sum\limits_{h\in {\Omega _{j,k,m}}}|\lambda _{v,h}|^{\tau }\big)^{1/\tau }
\\
& =\sum\limits_{j=0}^{\infty }2^{(\frac{n}{\tau }-J-\varepsilon )j}\Big(%
2^{(v-j)n}\int\limits_{\cup _{z\in {\Omega _{j,k,m}}}Q_{v,z}}\sum\limits_{h%
\in {\Omega _{j,k,m}}}|\lambda _{v,h}|^{\tau }\chi _{v,h}(y)dy\Big)^{1/\tau
}.
\end{align*}%
Let $y\in \cup _{z\in {\Omega _{j,k,m}}}Q_{v,z}$ and $x\in Q_{k,m}$. It
follows that $y\in Q_{v,z}$ for some $z\in {\Omega _{j,k,m}}$ and ${2^{j-1}<2%
}^{v}{|2}^{-k}m{-2}^{-v}{z|\leq 2^{j}}$. From this we obtain that 
\begin{align*}
\left\vert y-x\right\vert & \leq \left\vert y-{2}^{-k}m\right\vert
+\left\vert x-{2}^{-k}m\right\vert \\
& \lesssim 2^{-v}+2^{j-v}+2^{-k} \\
& \leq 2^{j-v+\delta _{n}},\quad \delta _{n}\in \mathbb{N},
\end{align*}%
which implies that $y$ is located in the ball $B(x,2^{j-v+\delta _{n}})$.
Consequently 
\begin{equation*}
S_{k,v,m}\lesssim \mathcal{M}_{\tau }\big(\sum\limits_{h\in \mathbb{Z}%
^{n}}\lambda _{v,h}\chi _{v,h}\big)(x)
\end{equation*}%
for any $x\in Q_{k,m}$ and any $k\leq v$. Applying Lemmas \ref{Maximal-Inq
copy(2)-lorentz} and \ref{Maximal-Inq copy(1)-lorentz}, we obtain that 
\begin{equation*}
{{\big\|}A_{0}\lambda {\big\|}}_{\dot{K}_{p,r}^{\alpha ,q}a_{\beta
}^{s}}\lesssim {{\big\|}\lambda {\big\|}}_{\dot{K}_{p,r}^{\alpha ,q}a_{\beta
}^{s}}{.}
\end{equation*}

\textit{Estimate of }$A_{1}$. Again from $\mathrm{\eqref{omega-assumption}}$
,\ we see that 
\begin{align*}
\big|(A_{1}\lambda )_{v,h}\big|& \leq \sum_{v=k+1}^{\infty }\sum_{h\in 
\mathbb{Z}^{n}}2^{(v-k)(\alpha _{1}-\frac{\varepsilon }{2}-J+\frac{n}{2})}%
\frac{|\lambda _{v,h}|}{\big(1+2^{k}|x_{k,m}-x_{v,h}|\big)^{J+\varepsilon }}
\\
& =\sum_{v=k+1}^{\infty }2^{(v-k)(\alpha _{1}-\frac{\varepsilon }{2}-J+\frac{%
n}{2})}T_{k,v,m}.
\end{align*}%
We proceed as in the estimate of $A_{0}$ we can prove that 
\begin{equation*}
T_{k,v,m}\leq c2^{(v-k)n/\tau }\mathcal{M}_{\tau }\big(\sum\limits_{h\in 
\mathbb{Z}^{n}}\lambda _{v,h}\chi _{v,h}\big)(x),\quad v>k,x\in Q_{k,m},
\end{equation*}%
where $\frac{n}{J+\frac{\varepsilon }{2}}<\tau <\min (1,p,\beta ,\frac{n}{%
\alpha +\frac{n}{p}})$ and the positive constant $c$ is independent of $v$, $%
k$ and $m$. Again applying Lemmas \ref{Maximal-Inq copy(2)-lorentz} and \ref%
{Maximal-Inq copy(1)-lorentz} we obtain 
\begin{equation*}
{{\big\|}A_{1}\lambda {\big\|}}_{\dot{K}_{p,r}^{\alpha ,q}f_{\beta
}^{s}}\lesssim {{\big\|}\lambda {\big\|}}_{\dot{K}_{p,r}^{\alpha ,q}f_{\beta
}^{s}}{.}
\end{equation*}%
Hence the theorem is proved.
\end{proof}

The following two lemmas are from \cite[Lemmas B.1-B.2]{FJ90}.

\begin{lemma}
\label{FJ901}Let $R>n,0<\theta \leq 1,j,k\in \mathbb{Z},j\leq k,L\in \mathbb{%
Z},L\geq 0,$%
\begin{equation*}
S>L+n+\theta \quad \text{and}\quad x_{1},x,y\in \mathbb{R}^{n}.
\end{equation*}%
Suppose that $g,h\in L^{1}$ satisfy%
\begin{equation*}
|\partial ^{\gamma }g(x)|\leq 2^{j(\frac{n}{2}+|\gamma
|)}(1+2^{j}|x|)^{-R},\quad |\gamma |\leq L,
\end{equation*}%
\begin{equation*}
|\partial ^{\gamma }g(x)-\partial ^{\gamma }g(y)|\leq 2^{j(\frac{n}{2}%
+L+\theta )}|x-y|^{\theta }\sup_{|z|\leq |x-y|}(1+2^{j}|z-x|)^{-R},\quad
|\gamma |=L,
\end{equation*}%
\begin{equation*}
|h(x)|\leq 2^{k\frac{n}{2}}(1+2^{k}|x-x_{1}|)^{-\max (R,S)},\quad |\gamma
|\leq L,
\end{equation*}%
and%
\begin{equation*}
\int_{\mathbb{R}^{n}}h(x)dx=0,\quad |\gamma |\leq L.
\end{equation*}%
Then%
\begin{equation*}
|h\ast g(x)|\lesssim 2^{-(k-j)(\frac{n}{2}+L+\theta )}(1+2^{j}|x-x_{1}|)^{-R}
\end{equation*}%
where the implicit constant is independent of $k,j,x_{1},x$ and $y.$
\end{lemma}

\begin{lemma}
\label{FJ902}Let $R>n,j,k\in \mathbb{Z},j\leq k\ $and $x_{1},x\in \mathbb{R}%
^{n}$. Suppose that $g,h\in L^{1}$ satisfy%
\begin{equation*}
|g(x)|\leq 2^{j\frac{n}{2}}(1+2^{j}|x|)^{-R},
\end{equation*}%
\begin{equation*}
|h(x)|\leq 2^{k\frac{n}{2}}(1+2^{k}|x-x_{1}|)^{-R}.
\end{equation*}%
Then%
\begin{equation*}
|h\ast g(x)|\lesssim 2^{-(k-j)\frac{n}{2}}(1+2^{j}|x-x_{1}|)^{-R},
\end{equation*}%
where the implicit constant is independent of $k,j,x_{1}$ and $x.$
\end{lemma}

Next we present the definition of inhomogeneous smooth synthesis and
analysis molecules for $\dot{K}_{p,r}^{\alpha ,q}A_{\beta }^{s}$, see {\cite%
{FJ90} and \cite{YSY10} for Besov-Triebel-Lizorkin type spaces.}

\begin{definition}
\label{Atom-Def}Let\ $s\in \mathbb{R},0<p<\infty ,0<r,q\leq \infty ,0<\beta
<\infty $ and $\alpha >-\frac{n}{p}$. Let $N=\max \{\lfloor J-n-s\rfloor
,-1\}$ and $s^{\ast }=s-\lfloor s\rfloor $.\newline
$\mathrm{(i)}$\ Let $k\in \mathbb{N}_{0}$ and $m\in \mathbb{Z}^{n}$. A
function $\varrho _{Q_{k,m}}$ is called an inhomogeneous smooth synthesis
molecule for $\dot{K}_{p,r}^{\alpha ,q}A_{\beta }^{s}$\ supported near $%
Q_{k,m}$ if there exist a real number $\delta \in (s^{\ast },1]$ and a real
number $M\in (J,\infty )$ such that 
\begin{equation}
\int_{\mathbb{R}^{n}}x^{\gamma }\varrho _{Q_{k,m}}(x)dx=0\text{\quad if\quad 
}0\leq |\gamma |\leq N,\quad k\in \mathbb{N},  \label{mom-cond}
\end{equation}%
\begin{equation}
|\varrho _{Q_{0,m}}(x)|\leq (1+|x-x_{Q_{0,m}}|)^{-M},  \label{cond1bis}
\end{equation}%
\begin{equation}
|\varrho _{Q_{k,m}}(x)|\leq 2^{\frac{kn}{2}}(1+2^{k}|x-x_{Q_{k,m}}|)^{-\max
(M,M-s)},\quad k\in \mathbb{N},  \label{cond1}
\end{equation}%
\begin{equation}
|\partial ^{\gamma }\varrho _{Q_{k,m}}(x)|\leq 2^{k(|\gamma |+\frac{1}{2}%
)}(1+2^{k}|x-x_{Q_{k,m}}|)^{-M}\quad \text{if}\quad |\gamma |\leq \lfloor
s\rfloor  \label{cond2}
\end{equation}%
and 
\begin{align}
& |\partial ^{\gamma }\varrho _{Q_{k,m}}(x)-\partial ^{\gamma }\varrho
_{Q_{k,m}}(y)|  \label{cond3} \\
& \leq 2^{k(|\gamma |+\frac{1}{2}+\delta )}|x-y|^{\delta }\sup_{|z|\leq
|x-y|}(1+2^{k}|x-z-x_{Q_{k,m}}|)^{-M}\text{\quad if\quad }|\gamma |=\lfloor
s\rfloor .  \notag
\end{align}%
A collection $\{\varrho _{Q_{k,m}}\}_{k\in \mathbb{N}_{0},m\in \mathbb{Z}%
^{n}}$ is called a family of inhomogeneous smooth synthesis molecules for $%
\dot{K}_{p,r}^{\alpha ,q}A_{\beta }^{s}$, if each $\varrho _{Q_{k,m}}$, $%
k\in \mathbb{N}_{0},m\in \mathbb{Z}^{n}$, is an homogeneous smooth synthesis
molecule for $\dot{K}_{p,r}^{\alpha ,q}A_{\beta }^{s}$ supported near $%
Q_{k,m}$. \newline
$\mathrm{(ii)}$\ Let $k\in \mathbb{N}_{0}$ and $m\in \mathbb{Z}^{n}$. A
function $b_{Q_{k,m}}$ is called an inhomogeneous smooth analysis molecule
for $\dot{K}_{p,r}^{\alpha ,q}A_{\beta }^{s}$ supported near $Q_{k,m}$ if
there exist a $\kappa \in ((J-s)^{\ast },1]$ and an $M\in (J,\infty )$ such
that 
\begin{equation}
\int_{\mathbb{R}^{n}}x^{\gamma }b_{Q_{k,m}}(x)dx=0\text{\quad if\quad }0\leq
|\gamma |\leq \left\lfloor s\right\rfloor ,\quad k\in \mathbb{N}
\label{mom-cond2}
\end{equation}%
\begin{equation}
|\varrho _{Q_{0,m}}(x)|\leq (1+|x-x_{Q_{0,m}}|)^{-M},  \label{cond1.1bis}
\end{equation}%
\begin{equation}
|b_{Q_{k,m}}(x)|\leq 2^{\frac{kn}{2}}(1+2^{k}|x-x_{Q_{k,m}}|)^{-\max
(M,M+n+s-J)},\quad k\in \mathbb{N}  \label{cond1.1}
\end{equation}%
\begin{equation}
|\partial ^{\gamma }b_{Q_{k,m}}(x)|\leq 2^{k(|\gamma |+\frac{n}{2}%
)}(1+2^{k}|x-x_{Q_{k,m}}|)^{-M}\quad \text{if}\quad |\gamma |\leq N
\label{cond1.2}
\end{equation}%
and 
\begin{align}
& |\partial ^{\gamma }b_{Q_{k,m}}(x)-\partial ^{\gamma }b_{Q_{k,m}}(y)|
\label{cond1.3} \\
& \leq 2^{k(|\gamma |+\frac{n}{2}+\kappa )}|x-y|^{\kappa }\sup_{|z|\leq
|x-y|}(1+2^{k}|x-z-x_{Q_{k,m}}|)^{-M}\text{\quad if\quad }|\gamma |=N. 
\notag
\end{align}%
A collection $\{b_{Q_{k,m}}\}_{k\in \mathbb{N}_{0},m\in \mathbb{Z}^{n}}$ is
called a family of inhomogeneous smooth analysis molecules for $\dot{K}%
_{p,r}^{\alpha ,q}A_{\beta }^{s}$, if each $b_{Q_{k,m}}$, $k\in \mathbb{N}%
_{0},m\in \mathbb{Z}^{n}$, is an homogeneous smooth synthesis molecule for $%
\dot{K}_{p,r}^{\alpha ,q}A_{\beta }^{s}$ supported near $Q_{k,m}$.
\end{definition}

We will use the notation $\{b_{k,m}\}_{k\in \mathbb{N}_{0},m\in \mathbb{Z}%
^{n}}$ instead of $\{b_{Q_{k,m}}\}_{k\in \mathbb{N}_{0},m\in \mathbb{Z}^{n}}$%
{. }To establish the homogeneous smooth atomic and molecular decomposition
characterizations of $\dot{K}_{p,r}^{\alpha ,q}A_{\beta }^{s}$ spaces, we
need the following key lemma.

\begin{lemma}
\label{matrix-est}Let\ $s,\alpha ,J,M,N,\delta ,\kappa ,p,q$ and $\beta $ be
as in Definition {\ref{Atom-Def}}. Suppose that $\{\varrho _{v,h}\}_{v\in 
\mathbb{N}_{0},h\in \mathbb{Z}^{n}}$ is a family of smooth synthesis
molecules for $\dot{K}_{p,r}^{\alpha ,q}A_{\beta }^{s}$ and $%
\{b_{k,m}\}_{k\in \mathbb{N}_{0},m\in \mathbb{Z}^{n}}$\ is a family of
homogeneous smooth analysis molecules for\ $\dot{K}_{p,r}^{\alpha
,q}A_{\beta }^{s}$. Then there exist a positive real number $\varepsilon
_{1} $ and a positive constant $c$ such that 
\begin{equation*}
\left\vert \langle \varrho _{v,h},b_{k,m}\rangle \right\vert \leq c\text{ }%
\omega _{Q_{k,m}P_{v,h}}(\varepsilon ),\quad k,v\in \mathbb{N}_{0},h,m\in 
\mathbb{Z}^{n}
\end{equation*}%
if $\varepsilon \leq \varepsilon _{1}$.
\end{lemma}

\begin{proof}
The proof is a slight modification of {\cite[Corollary \ B.3]{FJ90}. }%
Possibly reducing $\delta $, $\varrho $, or $M$, we may assume that $\delta
-s^{\ast }=\frac{M-J}{2}=\kappa -(J-s)^{\ast }>0$ . First we suppose that $%
k\geq v$ and $s\geq 0$. We have 
\begin{equation*}
\langle \varrho _{v,h},b_{k,m}\rangle =g_{v,h}\ast b_{k,m}(x_{P_{v,h}})
\end{equation*}%
with $g_{v,h}(x)=\overline{\varrho _{v,h}(x_{P_{v,h}}-x)}$. Applying Lemma {%
\ref{FJ901},\ }we obtain 
\begin{align*}
\left\vert \langle \varrho _{v,h},b_{k,m}\rangle \right\vert & \leq c\text{ }%
2^{-(k-v)(\lfloor s\rfloor +\frac{n}{2}+\delta
)}(1+2^{v}|x_{P_{v,h}}-x_{Q_{k,m}}|)^{-M} \\
& \leq c\text{ }2^{-(k-v)(s+\frac{n+\varepsilon }{2}%
)}(1+2^{v}|x_{P_{v,h}}-x_{Q_{k,m}}|)^{-M}
\end{align*}%
if $\lfloor s\rfloor +\delta \geq s+\frac{\varepsilon }{2}$\ for some $%
\varepsilon >0$ small enough, but this is possible since $\delta >s^{\ast }$%
. In view if the fact that $\delta \leq 1$, we will take $\varepsilon
<2(\delta -s^{\ast })$.

Now if $k\geq v$ and $s<0$, then by Lemma {\ref{FJ902}}, we find that 
\begin{align*}
\left\vert \langle \varrho _{v,h},b_{k,m}\rangle \right\vert & \leq c\text{ }%
2^{-(k-v)\frac{n}{2}}(1+2^{v}|x_{P_{v,h}}-x_{Q_{k,m}}|)^{-M} \\
& \leq c\text{ }2^{-(k-v)(s+\frac{n+\varepsilon }{2}%
)}(1+2^{v}|x_{P_{v,h}}-x_{Q_{k,m}}|)^{-M}
\end{align*}%
if $0<\varepsilon <-2s$.

\noindent We suppose that $k<v$ and $N\geq 0$. We have $\langle \varrho
_{v,h},b_{k,m}\rangle =g_{k,m}\ast \varrho _{v,h}(x_{Q_{k,m}})$, with $%
g_{k,m}(x)=\overline{b_{k,m}(x_{Q_{k,m}}-x)}$. Again, using Lemma {\ref%
{FJ901}, }we obtain 
\begin{align*}
\left\vert \langle \varrho _{v,h},b_{k,m}\rangle \right\vert & \leq c\text{ }%
2^{-(v-k)(N+\frac{n}{2}+\kappa )}(1+2^{k}|x_{Q_{v,h}}-x_{Q_{k,m}}|)^{-M} \\
& \leq c\text{ }2^{(v-k)(s-J-\frac{\varepsilon -n}{2}%
)}(1+2^{k}|x_{Q_{v,h}}-x_{Q_{k,m}}|)^{-M},
\end{align*}%
since 
\begin{equation*}
N+\frac{n}{2}+\kappa >\frac{\varepsilon }{2}+J-\frac{n}{2}-s
\end{equation*}%
for any $0<\varepsilon <2\kappa .$

Now if that $k<v$ and $N=-1$, then we apply Lemma {\ref{FJ902}}, since $N=-1$
implies $n+s>J$ so that $n>-s+\frac{\varepsilon }{2}+J$, and obtain 
\begin{align*}
\left\vert \langle \varrho _{v,h},b_{k,m}\rangle \right\vert & \leq c\text{ }%
2^{-(v-k)\frac{n}{2}}(1+2^{k}|x_{Q_{v,h}}-x_{Q_{k,m}}|)^{-M} \\
& \leq c\text{ }2^{(v-k)(s-J-\frac{\varepsilon -n}{2}%
)}(1+2^{k}|x_{Q_{v,h}}-x_{Q_{k,m}}|)^{-M}
\end{align*}%
if $0<\varepsilon <2(s-J+n)$. The proof is complete.
\end{proof}

As an immediate consequence, we have the following analogues of the
corresponding results on \cite[Corollary\ B.3]{FJ90}.

\begin{corollary}
Let\ $s,\alpha ,J,M,N,\delta ,\kappa ,p,q\ $and $\beta $ be as in Definition 
{\ref{Atom-Def}}. Let $\Phi $ and $\varphi $ satisfy, respectively $\mathrm{%
\eqref{Ass1}}$ and $\mathrm{\eqref{Ass2}}$.\newline
$\mathrm{(i)}$\ If $\{\varrho _{k,m}\}_{k\in \mathbb{N}_{0},m\in \mathbb{Z}%
^{n}}$ is a family of homogeneous synthesis molecules for the
Triebel-Lizorkin spaces $\dot{K}_{p,r}^{\alpha ,q}A_{\beta }^{s}$, then the
operator $A$ with matrix\ $a_{Q_{k,m}P_{v,h}}=\langle \varrho _{v,h},\varphi
_{k,m}\rangle $, $k,v\in \mathbb{N}_{0},m,h\in \mathbb{Z}^{n}$, is almost
diagonal.\newline
$\mathrm{(ii)}$ If $\{b_{k,m}\}_{k\in \mathbb{N}_{0},m\in \mathbb{Z}^{n}}$
is a family of homogeneous smooth analysis molecules for the
Triebel-Lizorkin spaces $\dot{K}_{p,r}^{\alpha ,q}A_{\beta }^{s}$, then \
the operator\ $A$, with matrix $a_{Q_{k,m}P_{v,h}}=\langle \varphi
_{v,h},b_{Q_{k,m}}\rangle $, $k,v\in \mathbb{N}_{0},m,h\in \mathbb{Z}^{n}$,
is almost diagonal.
\end{corollary}

Let $f\in \dot{K}_{p,r}^{\alpha ,q}A_{\beta }^{s}\ $and $\{b_{k,m}\}_{k\in 
\mathbb{N}_{0},m\in \mathbb{Z}^{n}}$ be a family of homogeneous\ smooth
analysis molecules. To prove that $\langle f,b_{Q_{k,m}}\rangle $, $k\in 
\mathbb{N}_{0},m\in \mathbb{Z}^{n}$, is well defined for all homogeneous
smooth analysis molecules for $\dot{K}_{p,r}^{\alpha ,q}A_{\beta }^{s}$, we
need the following result, which proved in {\cite[Lemma 5.4]{BoHo06}}.
Suppose that $\Phi $ is a smooth analysis (or synthesis) molecule supported
near $Q\in \mathcal{Q}$. Then there exists a sequence $\{\varphi
_{k}\}_{k\in \mathbb{N}}\subset \mathcal{S(}\mathbb{R}^{n}\mathcal{)}$ and $%
c>0$ such that $c\varphi _{k}$ is a smooth analysis (or synthesis) molecule
supported near $Q$ for every $k$,and $\varphi _{k}(x)\rightarrow \Phi (x)$
uniformly on $\mathbb{R}^{n}$ as $k\rightarrow \infty $.

Now we have the following smooth molecular characterization of the spaces $%
\dot{K}_{p,r}^{\alpha ,q}A_{\beta }^{s}$.

\begin{theorem}
\label{molecules-dec}Let $s\in \mathbb{R},0<p<\infty ,0<r,q\leq \infty
,0<\beta <\infty $ and $\alpha >-\frac{n}{p}$. Let $J,M,N,\delta $ and $%
\kappa $ be as in Definition {\ref{Atom-Def}}. \newline
$\mathrm{(i)}$\ If $f=\sum_{v=0}^{\infty }\sum_{h\in \mathbb{Z}^{n}}\varrho
_{v,h}\lambda _{v,h}$, where $\{\varrho _{v,h}\}_{v\in \mathbb{N}_{0},h\in 
\mathbb{Z}^{n}}$ is a family of homogeneous smooth synthesis molecules for $%
\dot{K}_{p,r}^{\alpha ,q}A_{\beta }^{s}$, then for all $\lambda \in \dot{K}%
_{p,r}^{\alpha ,q}a_{\beta }^{s}$ 
\begin{equation*}
{{\big\|}f{\big\|}}_{\dot{K}_{p,r}^{\alpha ,q}A_{\beta }^{s}}{\lesssim {%
\big\|}\lambda {\big\|}}_{\dot{K}_{p,r}^{\alpha ,q}a_{\beta }^{s}}{.}
\end{equation*}%
$\mathrm{(ii)}$\ Let $\{b_{k,m}\}_{k\in \mathbb{N}_{0},m\in \mathbb{Z}^{n}}$
be a family of homogeneous\ smooth analysis molecules.\ Then for all\ $f\in 
\dot{K}_{q}^{\alpha _{2},p}A_{\beta }^{s}$ 
\begin{equation*}
{{\big\|}\{\langle f,b_{k,m}\rangle \}_{k\in \mathbb{N}_{0},m\in \mathbb{Z}%
^{n}}{\big\|}}_{\dot{K}_{p,r}^{\alpha ,q}a_{\beta }^{s}}{\lesssim {\big\|}f%
\big\|}_{\dot{K}_{p,r}^{\alpha ,q}A_{\beta }^{s}}{.}
\end{equation*}
\end{theorem}

\begin{proof}
The proof is a slight variant of \cite{FJ90}{. }We split the proof in two
steps.

\textit{Step 1. Proof of }\textrm{(i)}. By\ $\mathrm{\eqref{proc2}}$ we can
write 
\begin{equation*}
\varrho _{v,h}=\sum_{k=0}^{\infty }2^{-kn}\sum_{m\in \mathbb{Z}^{n}}%
\widetilde{\varphi }_{k}\ast \varrho _{v,h}(2^{-k}m)\psi _{k}(\cdot -2^{-k}m)
\end{equation*}%
for\ any\ $v\in \mathbb{N}_{0},h\in \mathbb{Z}^{n}$.\ Therefore, 
\begin{equation*}
f=\sum_{k=0}^{\infty }\sum_{m\in \mathbb{Z}^{n}}S_{k,m}\psi _{k,m}=T_{\psi
}S,
\end{equation*}%
where $S=\{S_{k,m}\}_{k\in \mathbb{N}_{0},m\in \mathbb{Z}^{n}}$, with 
\begin{equation*}
S_{k,m}=2^{-k\frac{n}{2}}\sum_{v=0}^{\infty }\sum_{h\in \mathbb{Z}^{n}}%
\widetilde{\varphi }_{k}\ast \varrho _{v,h}(2^{-k}m)\lambda _{v,h}.
\end{equation*}%
From Theorem {\ref{phi-tran-lorentz}, }we have 
\begin{equation*}
{{\big\|}f{\big\|}}_{\dot{K}_{p,r}^{\alpha ,q}A_{\beta }^{s}}{={\big\|}%
T_{\psi }S\big\|}_{\dot{K}_{p,r}^{\alpha ,q}A_{\beta }^{s}}{\lesssim {\big\|}%
S\big\|}_{\dot{K}_{p,r}^{\alpha ,q}a_{\beta }^{s}}{.}
\end{equation*}%
But 
\begin{equation*}
S_{k,m}=\sum_{v=0}^{\infty }\sum_{h\in \mathbb{Z}^{n}}a_{Q_{k,m}P_{v,h}}%
\lambda _{v,h},
\end{equation*}%
with 
\begin{equation*}
a_{Q_{k,m}P_{v,h}}=\langle \varrho _{v,h},\widetilde{\varphi }_{k,m}\rangle
,\quad k,v\in \mathbb{N}_{0},m,h\in \mathbb{Z}^{n}.
\end{equation*}%
Applying Lemma {\ref{matrix-est} and Theorem \ref{almost-diag-est} we find
that} 
\begin{equation*}
{{\big\|}S{\big\|}}_{\dot{K}_{p,r}^{\alpha ,q}a_{\beta }^{s}}{\lesssim {%
\big\|}\lambda {\big\|}}_{\dot{K}_{p,r}^{\alpha ,q}a_{\beta }^{s}}{.}
\end{equation*}

\textit{Step 2. Proof of }$\mathit{\mathrm{(ii)}}$\textit{.} We have 
\begin{align*}
\langle f,b_{k,m}\rangle & =\sum_{v=0}^{\infty }2^{-vn}\sum_{m\in \mathbb{Z}%
^{n}}\langle \psi _{v}(\cdot -2^{-v}h),b_{k,m}\rangle \widetilde{\varphi }%
_{v}\ast f(2^{-v}h) \\
& =\sum_{v=0}^{\infty }\sum_{m\in \mathbb{Z}^{n}}\langle \psi
_{v,h},b_{k,m}\rangle \lambda _{v,h} \\
& =\sum_{v=0}^{\infty }\sum_{h\in \mathbb{Z}^{n}}a_{Q_{k,m}P_{v,h}}\lambda
_{v,h},
\end{align*}%
{where } 
\begin{equation*}
a_{Q_{k,m}P_{v,h}}=\langle \psi _{v,h},b_{k,m}\rangle ,\quad \lambda
_{v,h}=2^{-v\frac{n}{2}}\widetilde{\varphi }_{v}\ast f(2^{-v}h).
\end{equation*}%
Again by Lemma {\ref{matrix-est} and Theorem \ref{almost-diag-est} we find
that} 
\begin{align*}
{\big\|\{\langle f,b_{k,m}\rangle \}_{k\in \mathbb{N}_{0},m\in \mathbb{Z}%
^{n}}\big\|}_{\dot{K}_{p,r}^{\alpha ,q}a_{\beta }^{s}}& {\lesssim }{\big\|%
\{\lambda _{v,h}\}_{v\in \mathbb{N}_{0},h\in \mathbb{Z}^{n}}\big\|}_{\dot{K}%
_{p,r}^{\alpha ,q}a_{\beta }^{s}} \\
& =c{\big\|\{(S_{\varphi })_{v,h}\}_{v\in \mathbb{N}_{0},h\in \mathbb{Z}^{n}}%
\big\|}_{\dot{K}_{p,r}^{\alpha ,q}a_{\beta }^{s}}{.}
\end{align*}%
Applying Theorem {\ref{phi-tran-lorentz} we find that} 
\begin{equation*}
{{\big\|}\{\langle f,b_{k,m}\rangle \}_{k\in \mathbb{N}_{0},m\in \mathbb{Z}%
^{n}}{\big\|}}_{\dot{K}_{p,r}^{\alpha ,q}a_{\beta }^{s}}{\lesssim {\big\|}f{%
\big\|}}_{\dot{K}_{p,r}^{\alpha ,q}A_{\beta }^{s}}{.}
\end{equation*}%
The proof is complete.
\end{proof}

Now we turn to the notion of a smooth atom for $\dot{K}_{p,r}^{\alpha
,q}A_{\beta }^{s}$.

\begin{definition}
\label{atom}Let $s\in \mathbb{R},0<p<\infty ,0<r,q\leq \infty ,0<\beta
<\infty ,\alpha >-\frac{n}{p}$ and\ $N=\max \{\lfloor J-n-s\rfloor ,-1\}$. A
function $\varrho _{Q_{k,m}}$ is called an homogeneous smooth atom for $\dot{%
K}_{p,r}^{\alpha ,q}A_{\beta }^{s}$ supported near $Q_{k,m}$, $k\in \mathbb{N%
}_{0}$ and $m\in \mathbb{Z}^{n}$, if 
\begin{equation}
\mathrm{supp}\varrho _{Q_{k,m}}\subseteq 3Q_{k,m}  \label{supp-cond}
\end{equation}

\begin{equation}
|\partial ^{\gamma }\varrho _{Q_{k,m}}(x)|\leq 2^{k(|\gamma |+\frac{n}{2})}%
\text{\quad if\quad }0\leq |\gamma |\leq \max (0,1+\lfloor s\rfloor ),\quad
x\in \mathbb{R}^{n}  \label{diff-cond}
\end{equation}%
and if 
\begin{equation}
\int_{\mathbb{R}^{n}}x^{\gamma }\varrho _{Q_{k,m}}(x)dx=0\text{\quad if\quad 
}0\leq |\gamma |\leq N\quad \text{and}\quad k\in \mathbb{N}.
\label{mom-cond1}
\end{equation}%
A collection $\{\varrho _{Q_{k,m}}\}_{k\in \mathbb{N}_{0},m\in \mathbb{Z}%
^{n}}$\ is called a family of homogeneous smooth atoms for $\dot{K}%
_{p,r}^{\alpha ,q}A_{\beta }^{s}$, if each $a_{Q_{k,m}}$ is an homogeneous
smooth atom for $\dot{K}_{p,r}^{\alpha ,q}A_{\beta }^{s}$ supported near $%
Q_{v,m}$.
\end{definition}

The moment condition $\mathrm{\eqref{mom-cond1}}$ can be strengthened into
that%
\begin{equation*}
\int_{\mathbb{R}^{n}}x^{\gamma }\varrho _{Q_{k,m}}(x)dx=0\text{\quad if\quad 
}0\leq |\gamma |\leq \tilde{N}\quad \text{and}\quad k\in \mathbb{N}
\end{equation*}%
and the regularity condition \eqref{diff-cond} can be strengthened into that%
\begin{equation*}
|\partial ^{\gamma }\varrho _{Q_{k,m}}(x)|\leq 2^{k(|\gamma |+\frac{n}{2})}%
\text{\quad if\quad }0\leq |\gamma |\leq \tilde{K},\quad x\in \mathbb{R}^{n},
\end{equation*}%
where $\tilde{K}$ and $\tilde{N}$ are arbitrary fixed integer satisfying $%
\tilde{K}\geq \max (0,1+\lfloor s\rfloor )$ and $\tilde{N}\geq \max
\{\lfloor J-n-s\rfloor ,-1\}$. If an atom $\varrho $ is supported near $%
Q_{k,m}$, then we denote it by $\varrho _{k,m}$. If $N=-1$, then $\mathrm{%
\eqref{mom-cond1}}$ means that no moment conditions are required. We see
that every inhomogeneous smooth atom for $\dot{K}_{p,r}^{\alpha ,q}A_{\beta
}^{s}$ is a multiple of an inhomogeneous smooth synthesis molecule for $\dot{%
K}_{p,r}^{\alpha ,q}A_{\beta }^{s}.$

Now we come to the atomic decomposition theorem{.}

\begin{theorem}
\label{atomic-dec}Let $s\in \mathbb{R},0<p<\infty ,0<r,q\leq \infty ,0<\beta
<\infty ,\alpha >-\frac{n}{p}$. Then for each $f\in \dot{K}_{p,r}^{\alpha
,q}A_{\beta }^{s}$, there exist a family\ $\{\varrho _{k,m}\}_{k\in \mathbb{N%
}_{0},m\in \mathbb{Z}^{n}}$ of homogeneous smooth atoms for $\dot{K}%
_{p,r}^{\alpha ,q}A_{\beta }^{s}$ and $\lambda =\{\lambda _{k,m}\}_{k\in 
\mathbb{N}_{0},m\in \mathbb{Z}^{n}}\in \dot{K}_{p,r}^{\alpha ,q}a_{\beta
}^{s}$ such that 
\begin{equation}
f=\sum\limits_{k=0}^{\infty }\sum\limits_{m\in \mathbb{Z}^{n}}\lambda
_{k,m}\varrho _{k,m},\text{\quad converging in }\mathcal{S}^{\prime }(%
\mathbb{R}^{n})  \label{atom-dec}
\end{equation}%
and 
\begin{equation*}
{{\big\|}\{\lambda _{k,m}\}_{k\in \mathbb{N}_{0},m\in \mathbb{Z}^{n}}{\big\|}%
}_{\dot{K}_{p,r}^{\alpha ,q}a_{\beta }^{s}}{\lesssim {\big\|}f\big\|}_{\dot{K%
}_{p,r}^{\alpha ,q}A_{\beta }^{s}}{.}
\end{equation*}%
Conversely, for any family of homogeneous smooth atoms for $\dot{K}%
_{p,r}^{\alpha ,q}A_{\beta }^{s}$ and 
\begin{equation*}
\lambda =\{\lambda _{k,m}\}_{k\in \mathbb{N}_{0},m\in \mathbb{Z}^{n}}\in 
\dot{K}_{p,r}^{\alpha ,q}a_{\beta }^{s},
\end{equation*}%
we have 
\begin{equation*}
{{\big\|}\sum\limits_{k=0}^{\infty }\sum\limits_{m\in \mathbb{Z}^{n}}\lambda
_{k,m}\varrho _{k,m}\big\|}_{\dot{K}_{p,r}^{\alpha ,q}A_{\beta }^{s}}{%
\lesssim {\big\|}\{\lambda _{k,m}\}_{k\in \mathbb{N}_{0},m\in \mathbb{Z}^{n}}%
{\big\|}}_{\dot{K}_{p,r}^{\alpha ,q}a_{\beta }^{s}}{.}
\end{equation*}
\end{theorem}

\begin{remark}
Let $s\in \mathbb{R},0<p<\infty ,0<r,q\leq \infty ,0<\beta <\infty ,\alpha >-%
\frac{n}{p}$ and $f\in \dot{K}_{p,r}^{\alpha ,q}A_{\beta }^{s}$. Let $%
\{\varrho _{k,m}\}_{k\in \mathbb{N}_{0},m\in \mathbb{Z}^{n}}$ be a family of
homogeneous smooth atoms for $\dot{K}_{p,r}^{\alpha ,q}A_{\beta }^{s}$. From
Theorem \ref{atom-dec} there exist $\lambda =\{\lambda _{k,m}\}_{k\in 
\mathbb{N}_{0},m\in \mathbb{Z}^{n}}\in \dot{K}_{p,r}^{\alpha ,q}a_{\beta
}^{s}$ such that 
\begin{equation*}
f=\sum\limits_{k=0}^{\infty }\sum\limits_{m\in \mathbb{Z}^{n}}\lambda
_{k,m}\varrho _{k,m},\text{\quad converging in }\mathcal{S}^{\prime }(%
\mathbb{R}^{n}),
\end{equation*}%
which can be written as%
\begin{align*}
f& =\sum\limits_{k=0}^{\infty }\sum\limits_{m\in \mathbb{Z}^{n}}2^{(s-\frac{n%
}{p}+\frac{n}{2})k}\lambda _{k,m}2^{-(s-\frac{n}{p}+\frac{n}{2})k}\varrho
_{k,m} \\
& =\sum\limits_{k=0}^{\infty }\sum\limits_{m\in \mathbb{Z}^{n}}\tilde{\lambda%
}_{k,m}\tilde{\varrho}_{k,m}.
\end{align*}%
Observe that%
\begin{equation*}
\int_{\mathbb{R}^{n}}x^{\gamma }\tilde{\varrho}_{k,m}(x)dx=0\text{\quad
if\quad }0\leq |\gamma |\leq \tilde{N}\quad \text{and}\quad k\in \mathbb{N}
\end{equation*}%
and the regularity condition \eqref{diff-cond} can be strengthened into that%
\begin{equation*}
|\partial ^{\gamma }\tilde{\varrho}_{k,m}(x)|\leq 2^{-(s-\frac{n}{p}%
)k+|\beta |k}\text{\quad if\quad }0\leq |\gamma |\leq \tilde{K},\quad x\in 
\mathbb{R}^{n},
\end{equation*}%
where $\tilde{K}$ and $\tilde{N}$ are arbitrary fixed integer satisfying $%
\tilde{K}\geq \max (0,1+\lfloor s\rfloor )$ and $\tilde{N}\geq \max
\{\lfloor J-n-s\rfloor ,-1\}$.
\end{remark}

\begin{definition}
\label{atom2}Let $s\in \mathbb{R},0<p<\infty ,0<r,q\leq \infty ,0<\beta
<\infty ,\alpha >-\frac{n}{p}$ and\ $K,N\in \mathbb{N}_{0}$. A function $%
\varrho _{k,m}$ $k\in \mathbb{N}_{0},m\in \mathbb{Z}^{n}$ are called $(s,p)$%
-atoms if 
\begin{equation*}
\mathrm{supp}\varrho _{k,m}\subseteq 3Q_{k,m}
\end{equation*}%
there exist all (classical) derivatives $\partial ^{\gamma }\varrho _{k,m}$
with $|\gamma |\leq K$ such that%
\begin{equation*}
|\partial ^{\gamma }\varrho _{k,m}(x)|\leq 2^{-(s-\frac{n}{p})k+|\gamma |k}%
\text{\quad if\quad }0\leq |\gamma |\leq K,\quad x\in \mathbb{R}^{n},
\end{equation*}%
and%
\begin{equation*}
\int_{\mathbb{R}^{n}}x^{\gamma }\varrho _{k,m}(x)dx=0\text{\quad if\quad }%
0\leq |\gamma |\leq N\text{\quad and}\quad k\in \mathbb{N},m\in \mathbb{Z}%
^{n}.
\end{equation*}
\end{definition}

Let $\lambda =\{\lambda _{k,m}\}_{k\in \mathbb{N}_{0},m\in \mathbb{Z}%
^{n}}\subset \mathbb{C}$ be a complex valued sequence. We set 
\begin{equation*}
\big\|\lambda \big\|_{\widetilde{\dot{K}_{p,r}^{\alpha ,q}b_{\beta }^{s}}}=%
\Big(\sum_{k=0}^{\infty }2^{k\frac{n\beta }{p}}\big\|\sum\limits_{m\in 
\mathbb{Z}^{n}}\lambda _{k,m}\chi _{k,m}\big\|_{\dot{K}_{p,r}^{\alpha
,q}}^{\beta }\Big)^{1/\beta }
\end{equation*}%
and 
\begin{equation*}
\big\|\lambda \big\|_{\widetilde{\dot{K}_{p,r}^{\alpha ,q}f_{\beta }^{s}}}=%
\Big\|\Big(\sum_{k=0}^{\infty }\sum\limits_{m\in \mathbb{Z}^{n}}2^{k\frac{%
n\beta }{p}}|\lambda _{k,m}|^{\beta }\chi _{k,m}\Big)^{1/\beta }\Big\|_{\dot{%
K}_{p,r}^{\alpha ,q}},\quad 0<p,q<\infty .
\end{equation*}%
From Theorem \ref{atomic-dec} we get the following result:

\begin{theorem}
\label{atomic-decv2}Let $s\in \mathbb{R},0<p<\infty ,0<r,q\leq \infty
,0<\beta <\infty ,\alpha >-\frac{n}{p}$. Let $K,N\in \mathbb{N}_{0}$ with%
\begin{equation*}
K>s\quad \text{and}\quad N\geq \max \{\lfloor J-n-s\rfloor ,-1\}.
\end{equation*}%
Then $f\in \dot{K}_{p,r}^{\alpha ,q}A_{\beta }^{s}$ if, and only if, it can
be represented as 
\begin{equation}
f=\sum\limits_{k=0}^{\infty }\sum\limits_{m\in \mathbb{Z}^{n}}\lambda
_{k,m}\varrho _{k,m},\text{\quad converging in }\mathcal{S}^{\prime }(%
\mathbb{R}^{n}),  \label{rep2}
\end{equation}%
where $\varrho _{k,m}$ $k\in \mathbb{N}_{0},m\in \mathbb{Z}^{n}$ are $(s,p)$%
-atoms. Furthermore, 
\begin{equation*}
{{\big\|}f\big\|}_{\dot{K}_{p,r}^{\alpha ,q}A_{\beta }^{s}}\approx \inf {{%
\big\|}\{\lambda _{k,m}\}_{k\in \mathbb{N}_{0},m\in \mathbb{Z}^{n}}{\big\|}}%
_{\widetilde{\dot{K}_{p,r}^{\alpha ,q}a_{\beta }^{s}}}{,}
\end{equation*}%
are equivalent quasi-norms where the infimum is taken over all admissible
representations \eqref{rep2}.
\end{theorem}

\subsection{Wavelet characterization}

Using the characterizations of $\dot{K}_{p,r}^{\alpha ,q}A_{\beta }^{s}$
spaces by atom obtained in Section 5.1, we establish characterizations of $%
\dot{K}_{p,r}^{\alpha ,q}A_{\beta }^{s}$ by wavelets. We begin with
recalling the notion of kernels; see \cite[Definition 9]{triebel08}.

\begin{definition}
\label{kernel}Let $A,B\in \mathbb{N}_{0}$ and $C>0$. Then $L_{\infty }$%
-functions $k_{j,m}:\mathbb{R}^{n}\rightarrow \mathbb{C}$ with $j\in \mathbb{%
N}_{0},m\in \mathbb{Z}^{n}$, are called kernels if%
\begin{equation*}
\mathrm{supp}k_{j,m}\subset CQ_{j,m},\text{\quad if\quad }j\in \mathbb{N}%
_{0},\quad m\in \mathbb{Z}^{n};
\end{equation*}%
there exist all (classical) derivatives $\partial ^{\beta }k_{j,m}$ with $%
|\beta |\leq A$ such that%
\begin{equation}
|\partial ^{\beta }k_{j,m}(x)|\leq 2^{j(n+|\beta |)},\text{\quad if\quad }%
|\beta |\leq A,\quad j\in \mathbb{N}_{0},\quad m\in \mathbb{Z}^{n};
\label{normalisation2}
\end{equation}%
and%
\begin{equation}
\int_{\mathbb{R}^{n}}x^{\beta }k_{j,m}(x)dx=0,\text{\quad if\quad }|\beta
|<B,\quad j\in \mathbb{N},\quad m\in \mathbb{Z}^{n}.  \label{mom-condT}
\end{equation}
\end{definition}

\begin{remark}
When $B=0$ or $j=0$, there are no moment conditions \eqref{mom-condT} on the
kernels. In view the Definition \ref{atom} for atoms we have different
normalisations in \eqref{diff-cond} and in \eqref{normalisation2}.
\end{remark}

\begin{definition}
\label{sequence-space copy(1)}Let $\alpha ,s\in \mathbb{R},0<p<\infty
,0<r,q\leq \infty $ and $0<\beta \leq \infty $.\newline
$\mathrm{(i)}$ The\ space $\dot{K}_{p,r}^{\alpha ,q}\bar{b}_{\beta }^{s}$\
is defined to be the set of all complex valued sequences $\lambda =\{\lambda
_{k,m}\}_{k\in \mathbb{N}_{0},m\in \mathbb{Z}^{n}}\subset \mathbb{C}$ such
that%
\begin{equation*}
\big\|\lambda \big\|_{\dot{K}_{p,r}^{\alpha ,q}\bar{b}_{\beta }^{s}}=\Big(%
\sum_{k=0}^{\infty }2^{ks\beta }\big\|\sum\limits_{m\in \mathbb{Z}%
^{n}}\lambda _{k,m}\chi _{k,m}\big\|_{\dot{K}_{p,r}^{\alpha ,q}}^{\beta }%
\Big)^{1/\beta }<\infty .
\end{equation*}%
$\mathrm{(ii)}$ Let $0<p,q<\infty $. The\ space $\dot{K}_{p,r}^{\alpha ,q}%
\bar{f}_{\beta }^{s}$\ is defined to be the set of all complex valued
sequences $\lambda =\{\lambda _{k,m}\}_{k\in \mathbb{N}_{0},m\in \mathbb{Z}%
^{n}}\subset \mathbb{C}$ such that 
\begin{equation*}
\big\|\lambda \big\|_{\dot{K}_{p,r}^{\alpha ,q}\bar{f}_{\beta }^{s}}=\Big\|%
\Big(\sum_{k=0}^{\infty }\sum\limits_{m\in \mathbb{Z}^{n}}2^{ks\beta
}|\lambda _{k,m}|^{\beta }\chi _{k,m}\Big)^{1/\beta }\Big\|_{\dot{K}%
_{p,r}^{\alpha ,q}}<\infty .
\end{equation*}
\end{definition}

\begin{definition}
Let $f\in \dot{K}_{p,r}^{\alpha ,q}A_{\beta }^{s},\alpha ,s\in \mathbb{R}%
,0<p<\infty ,0<r,q\leq \infty $ and $0<\beta \leq \infty $. Let $%
k_{j,m},j\in \mathbb{N}_{0},m\in \mathbb{Z}^{n}$ be kernels according to
Definition \ref{kernel} with $A>\max (J-n,0)-s$ and $B\in \mathbb{N}_{0}$.
We set%
\begin{equation}
k_{j,m}(f)=\langle f,k_{j,m}\rangle =\int_{\mathbb{R}^{n}}k_{j,m}(y)f(y)dy,%
\quad j\in \mathbb{N}_{0},m\in \mathbb{Z}^{n},  \label{kernel2}
\end{equation}%
where $\langle \cdot ,\cdot \rangle $ denotes the duality bracket between $%
\mathcal{S}(\mathbb{R}^{n})$ and $\mathcal{S}^{\prime }(\mathbb{R}^{n})$.
Furthermore,%
\begin{equation*}
k(f)=\{k_{j,m}(f):j\in \mathbb{N}_{0},m\in \mathbb{Z}^{n}\}.
\end{equation*}
\end{definition}

\begin{remark}
First, assume that the expression \eqref{kernel} makes sense, at least
formally. Later on we will justify the dual pairing.
\end{remark}

\begin{theorem}
\label{kernel3}Let $s\in \mathbb{R},0<p<\infty ,0<r,q\leq \infty ,0<\beta
\leq \infty $ and $\alpha >-\frac{n}{p}$. Let $k_{j,m},j\in \mathbb{N}%
_{0},m\in \mathbb{Z}^{n}$ be kernels according to Definition \ref{kernel}\
where $C>0$ is fixed, $A>\max (J-n,0)-s$ and $B>s$.$\newline
\mathrm{(i)}$\ It holds%
\begin{equation*}
\big\|k(f)\big\|_{\dot{K}_{p,r}^{\alpha ,q}\bar{b}_{\beta }^{s}}\lesssim %
\big\|f\big\|_{\dot{K}_{p,r}^{\alpha ,q}B_{\beta }^{s}}
\end{equation*}%
for all $f\in \dot{K}_{p,r}^{\alpha ,q}B_{\beta }^{s}.\newline
\mathrm{(ii)}$ Let $0<p,q<\infty $. It holds 
\begin{equation}
\big\|k(f)\big\|_{\dot{K}_{p,r}^{\alpha ,q}\bar{f}_{\beta }^{s}}\lesssim %
\big\|f\big\|_{\dot{K}_{p,r}^{\alpha ,q}F_{\beta }^{s}}
\label{rep-new-lorentz}
\end{equation}%
for all $f\in \dot{K}_{p,r}^{\alpha ,q}F_{\beta }^{s}.$
\end{theorem}

\begin{proof}
We will proceed in two steps.

\textit{Step 1.} {Let $\varphi $ be a continuous function with a compact
support in the unit ball having }all classical continuous derivatives of
order%
\begin{equation*}
\partial ^{\beta }\varphi ,\quad \frac{\partial }{\partial x_{i}}\partial
^{\beta }\varphi ,\quad \left\vert \beta \right\vert \leq N,\quad i=1,...,n.
\end{equation*}%
Let $f\in \dot{K}_{p,r}^{\alpha ,q}B_{\beta }^{s}$ be expanded according to %
\eqref{atom-dec}. We get from the moment conditions $\mathrm{%
\eqref{mom-cond1}}$ for fixed $j\in \mathbb{N}_{0}$%
\begin{align}
& \sum\limits_{j=0}^{\infty }\int\limits_{\mathbb{R}^{n}}\sum\limits_{m\in 
\mathbb{Z}^{n}}\lambda _{j,m}\varrho _{j,m}(y)\varphi (y)dy
\label{converges-atom} \\
& =\sum\limits_{j=0}^{\infty }\int\limits_{\mathbb{R}^{n}}\sum\limits_{m\in 
\mathbb{Z}^{n}}\lambda _{j,m}\varrho _{j,m}(y)\Big(\varphi
(y)-\sum\limits_{\left\vert \beta \right\vert <N}(y-2^{-j}m)^{\beta }\frac{%
\partial ^{\beta }\varphi (2^{-j}m)}{\beta !}\Big)dy  \notag \\
& =\sum\limits_{j=0}^{\infty }\sum\limits_{v=-\infty }^{\infty
}\int\limits_{R_{v}}\sum\limits_{m\in \mathbb{Z}^{n}}\lambda _{j,m}\varrho
_{j,m}(y)\Omega _{j,m}(y)dy,  \notag
\end{align}%
Let us estimate the sum 
\begin{equation}
\sum\limits_{v=-\infty }^{0}\int\limits_{R_{v}}\sum\limits_{m\in \mathbb{Z}%
^{n}}\lambda _{j,m}\varrho _{j,m}(y)\Omega _{j,m}(y)dy.  \label{sum-ne}
\end{equation}%
We use the Taylor expansion of $\varphi $ up to order $N-1$\ with respect to
the off-points $2^{-j}m$, we obtain%
\begin{equation*}
\Omega _{j,m}(y)=\sum\limits_{\left\vert \beta \right\vert
=N}(y-2^{-j}m)^{\beta }\frac{\partial ^{\beta }\varphi (\xi )}{\beta !},
\end{equation*}%
with $\xi $ on the line segment joining $y$ and $2^{-j}m$. Since 
\begin{equation*}
1+\left\vert y\right\vert \leq \left( 1+\left\vert \xi \right\vert \right)
\left( 1+\left\vert y-2^{-j}m\right\vert \right) ,
\end{equation*}%
we estimate%
\begin{align*}
\big|\Omega _{j,m}(y)\big|& \leq \left( 1+\left\vert y-2^{-j}m\right\vert
\right) ^{N}\sum\limits_{\left\vert \beta \right\vert =N}\frac{|\partial
^{\beta }\varphi (\xi )|}{\beta !} \\
& \leq \left( 1+\left\vert y-2^{-j}m\right\vert \right) ^{N}(1+\left\vert
\xi \right\vert )^{-S}\big\|\varphi \big\|_{\mathcal{S}_{n}(\mathbb{R}^{n})}
\\
& \leq c\left( 1+\left\vert y\right\vert \right) ^{-S}\left( 1+\left\vert
y-2^{-j}m\right\vert \right) ^{N+S},
\end{align*}%
where $S>0$ is at our disposal. Let $-\frac{\alpha p}{n}<t<\min (1,p)=1+p-%
\frac{p}{\min (1,p)}$ and $h=s+\frac{n}{p}(t-1)$ be such that 
\begin{equation*}
n(1-\frac{1}{\min (1,p)})+s>h>-1-N.
\end{equation*}%
Since $\varrho _{j,m}$ are atoms, then 
\begin{equation*}
2^{-j(N+1)}\left\vert \varrho _{j,m}(y)\right\vert \lesssim
2^{jh}2^{-j(N+1+h)}\left( 1+2^{j}\left\vert y-2^{-j}m\right\vert \right)
^{-M},
\end{equation*}%
where the implicit constant is independent of $j$ and $m$. Therefore, the
sum $\mathrm{\eqref{sum-ne}}$ can be estimated by%
\begin{equation}
c\text{ }2^{-j(N+1+h)}\sum\limits_{v=-\infty
}^{0}\int\limits_{R_{v}}\sum\limits_{m\in \mathbb{Z}^{n}}2^{hj}\left\vert
\lambda _{j,m}\right\vert \left( 1+2^{j}\left\vert y-2^{-j}m\right\vert
\right) ^{N+S-M}(1+\left\vert y\right\vert )^{-S}dy.  \label{Max-estimate}
\end{equation}%
Since $M$ can be taken large enough, by the same arguments as in Lemma \ref%
{almost-diag-est} we obtain%
\begin{equation*}
\sum\limits_{m\in \mathbb{Z}^{n}}\left\vert \lambda _{j,m}\right\vert \left(
1+2^{j}\left\vert y-2^{-j}m\right\vert \right) ^{N+S-M}\leq c\mathcal{M}%
_{\tau }\big(\sum\limits_{m\in \mathbb{Z}^{n}}\left\vert \lambda
_{j,m}\right\vert \chi _{j,m}\big)(y)
\end{equation*}%
for any $y\in R_{v}\cap Q_{j,l}$ with $l\in \mathbb{Z}^{n}$ where $0<\tau
<\min (1,\frac{p}{t},\frac{n}{\alpha +\frac{nt}{p}})$. We split $S$ into $%
R+T $ with $R+\alpha <0$ and $T$ large enough such that $T>\max (-R,\frac{%
n(p-t)}{p})$. Then $\mathrm{\eqref{Max-estimate}}$ is bounded by%
\begin{equation*}
c\ 2^{-j(N+1+h)}\sum\limits_{v=-\infty }^{0}2^{-vR}\int\limits_{R_{v}}%
\mathcal{M}_{\tau }\big(\sum\limits_{m\in \mathbb{Z}^{n}}2^{jh}\left\vert
\lambda _{j,m}\right\vert \chi _{j,m}\big)(y)(1+\left\vert y\right\vert
)^{-T}dy.
\end{equation*}%
Since we have in addition the factor $(1+\left\vert y\right\vert )^{-T}$, it
follows by H\"{o}lder's inequality that this expression is dominated by%
\begin{align*}
& c\ 2^{-j(N+1+h)}\sum\limits_{v=-\infty }^{0}2^{-vR}\Big\|\mathcal{M}_{\tau
}\Big(\sum\limits_{m\in \mathbb{Z}^{n}}2^{hj}\left\vert \lambda
_{j,m}\right\vert \chi _{j,m}\Big)\chi _{v}\Big\|_{L^{p/t,\infty }} \\
& \leq c\text{ }2^{-j(N+1+h)}\sum\limits_{v=-\infty }^{0}2^{-v(\alpha +R)}%
\Big\|\sum\limits_{m\in \mathbb{Z}^{n}}2^{hj}\left\vert \lambda
_{j,m}\right\vert \chi _{j,m}\Big\|_{\dot{K}_{p/t,\infty }^{\alpha ,\infty }}
\\
& \leq c\text{ }2^{-j(N+1+h)}\big\|\lambda \big\|_{\dot{K}_{p/t,\infty
}^{\alpha ,q}b_{\infty }^{h}},
\end{align*}%
where the first inequality follows by the boundedness of the
Hardy-Littlewood maximal operator $\mathcal{M}_{\tau }$ on $\dot{K}%
_{p/t,\infty }^{\alpha ,\infty }$. Using a combination of the arguments used
above, the sum 
\begin{equation*}
\sum\limits_{v=1}^{\infty }\int\limits_{R_{v}}\sum\limits_{m\in \mathbb{Z}%
^{n}}\lambda _{j,m}\varrho _{j,m}(y)\Omega _{j,m}(y)dy
\end{equation*}%
can be estimated from above by%
\begin{equation*}
c2^{-j(N+1+h)}\big\|\lambda \big\|_{\dot{K}_{p/t,\infty }^{\alpha
,q}b_{\infty }^{h}},
\end{equation*}%
where the positive constant $c$ is independent of $j$ and $\lambda $. We
claim that%
\begin{equation}
\dot{K}_{p,r}^{\alpha ,q}b_{\infty }^{s}\hookrightarrow \dot{K}_{p/t,\infty
}^{\alpha ,q}b_{\infty }^{h}.  \label{claim-emb}
\end{equation}%
Since $N+1+h>0$, by the embeddings \eqref{claim-emb}\ we obtain that %
\eqref{converges-atom} converges absolutely in $\mathbb{C}$; see Theorem \ref%
{embeddings3-lorentz}, and $\langle f,\Psi \rangle $ makes sense as a dual
pairing. The $F$-counterpart follows by the embeddings 
\begin{equation*}
\dot{K}_{p,\infty }^{\alpha ,q}f_{\infty }^{s}\hookrightarrow \dot{K}%
_{p,\infty }^{\alpha ,q}b_{\infty }^{s}\hookrightarrow \dot{K}_{p/t,\infty
}^{\alpha ,q}b_{\infty }^{h}.
\end{equation*}

\textit{Step 2.} By Step 1, the local means $k_{j,m}(f)$ make sense. By
similarity, we consider only the spaces $\dot{K}_{p,\infty }^{\alpha
,q}F_{\beta }^{s}$. Let 
\begin{equation}
f=\sum\limits_{k=0}^{\infty }\sum\limits_{z\in \mathbb{Z}^{n}}\lambda
_{k,z}\varrho _{k,z},  \label{atom3}
\end{equation}%
be an atomic decomposition of $f\in \dot{K}_{p,\infty }^{\alpha ,q}F_{\beta
}^{s}$, where $\tilde{K}=B>s$ and $\tilde{N}=A>\max (J-n,0)-s$ and $\varrho
_{k,m}$ $k\in \mathbb{N}_{0},m\in \mathbb{Z}^{n}$ are $(s,p)$-atoms. Let $%
j\in \mathbb{N}_{0}$. We split \eqref{atom3} into 
\begin{equation*}
f=\sum\limits_{v=0}^{j}\sum\limits_{z\in \mathbb{Z}^{n}}\lambda
_{v,z}\varrho _{v,z}+\sum\limits_{v=j+1}^{\infty }\sum\limits_{z\in \mathbb{Z%
}^{n}}\lambda _{v,z}\varrho _{v,z}.
\end{equation*}%
We set%
\begin{equation*}
V_{1,j,m}=\int_{\mathbb{R}^{n}}k_{j,m}(y)\sum\limits_{v=0}^{j}\sum\limits_{z%
\in \mathbb{Z}^{n}}\lambda _{v,z}\varrho _{v,z}(y)dy
\end{equation*}%
and%
\begin{equation*}
V_{2,j,m}=\int_{\mathbb{R}^{n}}k_{j,m}(y)\sum\limits_{v=j+1}^{\infty
}\sum\limits_{z\in \mathbb{Z}^{n}}\lambda _{v,z}\varrho _{v,z}(y)dy.
\end{equation*}

\textit{Estimate of }$V_{1,j,m}$. Let $y\in CQ_{j,m}\cap 3Q_{v,z}$. Then%
\begin{equation*}
|z-2^{v}y|\lesssim 1\quad \text{and}\quad |2^{v-j}m-2^{v}y|\lesssim 1,\quad
j\geq v,
\end{equation*}%
where the implicit constant is independent of $j,v,y$ and $m$. Thus $y$ is
located in the set%
\begin{equation*}
{\digamma _{j,v,m}=\{z\in \mathbb{Z}^{n}:|z-2^{v-j}m|\lesssim 1\}}
\end{equation*}%
and 
\begin{equation*}
V_{1,j,m}=\sum\limits_{v=0}^{j}\sum\limits_{z\in {\digamma _{j,v,m}}}\lambda
_{v,z}\int_{\mathbb{R}^{n}}k_{j,m}(y)\varrho _{v,z}(y)dy.
\end{equation*}%
We use the Taylor expansion of $\varrho _{v,z}$ up to order $B-1$\ with
respect to the off-points $2^{-j}m$, we obtain%
\begin{equation*}
V_{1,j,m}=\sum\limits_{v=0}^{j}\sum\limits_{z\in {\digamma _{j,v,m}}}\lambda
_{v,z}\sum\limits_{\left\vert \beta \right\vert =B}\int_{\mathbb{R}%
^{n}}k_{j,m}(y)(y-2^{-j}m)^{\beta }\frac{\partial ^{\beta }\varrho
_{v,z}(\xi )}{\beta !}dy,
\end{equation*}%
with $\xi $ on the line segment joining $y$ and $2^{-j}m$. Therefore%
\begin{align}
& \big|V_{1,j,m}\big|  \notag \\
& \leq \sum\limits_{v=0}^{j}\sum\limits_{z\in {\digamma _{j,v,m}}}|\lambda
_{v,z}|\sum\limits_{\left\vert \beta \right\vert =B}\sup_{x\in \mathbb{R}%
^{n}}\big|\partial ^{\beta }\varrho _{v,z}(x)\big|\int_{\mathbb{R}%
^{n}}|k_{j,m}(y)||y-2^{-j}m|^{B}dy  \notag \\
& \lesssim \sum\limits_{v=0}^{j}2^{(v-j)B-(s-\frac{n}{p})v}\sum\limits_{z\in 
{\digamma _{j,v,m}}}|\lambda _{v,z}|  \notag \\
& \lesssim 2^{-js}\sum\limits_{v=0}^{j}2^{(v-j)(B-s)}2^{v\frac{n}{p}%
}\sum\limits_{z\in {\digamma _{j,v,m}}}|\lambda _{v,z}|.  \label{sumv1}
\end{align}%
Let $x\in Q_{j,m}$ and $y\in Q_{v,z}$ with $|z-2^{v-j}m|\lesssim 1$. We have%
\begin{equation*}
|x-y|\leq |x-2^{-j}m|+|2^{-j}m-2^{-v}z|+|y-2^{-v}z|\lesssim 2^{-v},
\end{equation*}%
which implies that $y$ is located in the ball $B(x,2^{-v})$. Let $0<\tau
<\min (p,\beta ,\frac{n}{n+\frac{n}{p}})$. Then%
\begin{align}
|\lambda _{v,z}|& =\Big(\frac{1}{|Q_{v,z}|}\int_{Q_{v,z}}|\lambda
_{v,z}|^{\tau }\chi _{v,z}(y)dy\Big)^{\frac{1}{\tau }}  \notag \\
& =\Big(\frac{1}{|Q_{v,z}|}\int_{Q_{v,z}}\sum\limits_{h\in \mathbb{Z}%
^{n}}|\lambda _{v,h}|^{\tau }\chi _{v,h}(y)dy\Big)^{\frac{1}{\tau }}  \notag
\\
& \lesssim \mathcal{M}_{\tau }\big(\sum\limits_{h\in \mathbb{Z}^{n}}|\lambda
_{v,h}|\chi _{v,h}\big)(x),\quad x\in Q_{j,m}.  \label{sumv1.1}
\end{align}%
Plug \eqref{sumv1.1}\ in \eqref{sumv1}, and since the sum with respect to $%
z\in \mathbb{Z}^{n}$ such that $|z-2^{v-j}m|\lesssim 1$ in $\mathrm{%
\eqref{sumv1}}$ has always less than $C_{2}$ independent of $m$, we obtain%
\begin{equation*}
\big|V_{1,j,m}\big|\lesssim 2^{-js}\sum\limits_{v=0}^{j}2^{(v-j)(B-s)}2^{v%
\frac{n}{p}}\mathcal{M}_{\tau }\big(\sum\limits_{h\in \mathbb{Z}%
^{n}}|\lambda _{v,h}|\chi _{v,h}\big)(x)
\end{equation*}%
for any $x\in Q_{j,m},j\in \mathbb{N}_{0},m\in \mathbb{Z}^{n}$.\ By Lemmas %
\ref{lem:lq-inequality} and \ref{Maximal-Inq copy(2)-lorentz}, we get 
\begin{align*}
& \Big\|\Big(\sum_{j=0}^{\infty }2^{js\beta }\sum\limits_{m\in \mathbb{Z}%
^{n}}\big|V_{1,j,m}\big|^{\beta }\chi _{j,m}\Big)^{1/\beta }\Big\|_{\dot{K}%
_{p,r}^{\alpha ,q}} \\
& \lesssim \Big\|\Big(\sum_{v=0}^{\infty }2^{v\frac{n}{p}\beta }\big(%
\mathcal{M}_{\tau }\big(\sum\limits_{h\in \mathbb{Z}^{n}}|\lambda
_{v,h}|\chi _{v,h}\big)\big)^{\beta }\Big)^{1/\beta }\Big\|_{\dot{K}%
_{p,r}^{\alpha ,q}} \\
& \lesssim \Big\|\Big(\sum_{v=0}^{\infty }2^{v\frac{n}{p}\beta
}\sum\limits_{h\in \mathbb{Z}^{n}}|\lambda _{v,h}|^{\beta }\chi _{v,h}\Big)%
^{1/\beta }\Big\|_{\dot{K}_{p,r}^{\alpha ,q}}.
\end{align*}

\textit{Estimate of }$V_{2,j,m}$. Let $y\in CQ_{j,m}\cap 3Q_{v,z}$. Then%
\begin{align*}
|2^{-v}z-2^{-j}m|& \leq |y-2^{-v}z|+|y-2^{-j}m| \\
& \lesssim 2^{-v}+2^{-j},
\end{align*}%
which yields that $|z-2^{v-j}m|\lesssim 2^{v-j}$. Hence $y$ is located in
the set 
\begin{equation*}
{\Gamma _{j,v,m}=\{z\in \mathbb{Z}^{n}:|z-2^{v-j}m|\lesssim 2^{v-j}\}.}
\end{equation*}%
Again, by the Taylor expansion of $k_{j,m}$ up to order $A-1$\ with respect
to the off-points $2^{-v}z$, we obtain%
\begin{equation*}
V_{2,j,m}=\sum\limits_{v=j+1}^{\infty }\sum\limits_{z\in {\Gamma _{j,v,m}}%
}\lambda _{v,z}\sum\limits_{\left\vert \beta \right\vert =B}\int_{\mathbb{R}%
^{n}}\varrho _{v,z}(y)(y-2^{-v}z)^{\beta }\frac{\partial ^{\beta
}k_{j,m}(\xi )}{\beta !}dy,
\end{equation*}%
with $\xi $ on the line segment joining $y$ and $2^{-v}z$. Hence%
\begin{align}
& \big|V_{2,j,m}\big|  \notag \\
& \leq \sum\limits_{v=j+1}^{\infty }\sum\limits_{z\in {\Gamma _{j,v,m}}%
}|\lambda _{v,z}|\sum\limits_{\left\vert \beta \right\vert =A}\sup_{x\in 
\mathbb{R}^{n}}\big|\partial ^{\beta }k_{j,m}(x)\big|\int_{\mathbb{R}^{n}}%
\frac{|\varrho _{v,z}(y)|}{|y-2^{-v}z|^{-A}}dy  \notag \\
& \lesssim \sum\limits_{v=j+1}^{\infty }2^{(j-v)(A+n)}2^{-v(s-\frac{n}{p}%
)}\sum\limits_{z\in {\Gamma _{j,v,m}}}|\lambda _{v,z}|  \notag \\
& \lesssim 2^{-js}\sum\limits_{v=j+1}^{\infty }2^{(j-v)(A+s+n)}2^{v\frac{n}{p%
}}\sum\limits_{z\in {\Gamma _{j,v,m}}}|\lambda _{v,z}|.  \label{sumv2.1}
\end{align}%
Let $x\in Q_{j,m}$ and $y\in Q_{v,z}$ with $|z-2^{v-j}m|\lesssim 2^{v-j}$.
Let $M>0$. We have%
\begin{align*}
\sum\limits_{z\in {\Gamma _{j,v,m}}}|\lambda _{v,z}|& \lesssim
\sum\limits_{z\in {\Gamma _{j,v,m}}}\frac{|\lambda _{v,z}|}{%
(1+2^{j}|2^{-v}z-2^{-j}m|)^{M}} \\
& \lesssim \sum\limits_{z\in \mathbb{Z}^{n}}\frac{|\lambda _{v,z}|}{%
(1+2^{j}|2^{-v}z-2^{-j}m|)^{M}} \\
& =c\mathcal{V}_{j,v,m},
\end{align*}%
{For each }$i,j,v\in \mathbb{N}$\ and $m\in \mathbb{Z}^{n}${\ we define } 
\begin{equation*}
{\Omega _{i,j,v,m}=\{z\in \mathbb{Z}^{n}:2^{i-1}<2^{j}|2^{-v}z-2^{-j}m|\leq
2^{i}\}}
\end{equation*}%
{and } 
\begin{equation*}
{\Omega _{0,j,v,m}=\{z\in \mathbb{Z}^{n}:2^{j}|2^{-v}z-2^{-j}m|\leq 1\}.}
\end{equation*}%
Let $0<\tau <\min (1,p,\beta ,\frac{n}{\alpha +\frac{n}{p}})$. Rewrite $%
\mathcal{V}_{j,v,m}$ as follows 
\begin{align*}
\mathcal{V}_{j,v,m}& =\sum\limits_{i=0}^{\infty }\sum\limits_{z\in {\Omega
_{i,j,v,m}}}\frac{|\lambda _{v,z}|}{(1+2^{j}|2^{-v}z-2^{-j}m|)^{M}} \\
& \leq \sum\limits_{i=0}^{\infty }2^{-Mi}\sum\limits_{z\in {\Omega _{i,j,v,m}%
}}|\lambda _{v,z}|.
\end{align*}%
By the embedding $\ell _{\tau }\hookrightarrow \ell _{1}$ we deduce that 
\begin{align*}
\mathcal{V}_{j,v,m}& \leq \sum\limits_{i=0}^{\infty }2^{-Mi}\big(%
\sum\limits_{z\in {\Omega _{j,k,m}}}|\lambda _{v,z}|^{\tau }\big)^{1/\tau }
\\
& =\sum\limits_{i=0}^{\infty }2^{(\frac{n}{\tau }-M)i}\Big(%
2^{(v-i)n}\int\limits_{\cup _{h\in {\Omega _{i,j,v,m}}}Q_{v,h}}\sum%
\limits_{z\in {\Omega _{i,j,v,m}}}|\lambda _{v,z}|^{\tau }\chi _{v,z}(y)dy%
\Big)^{1/\tau }.
\end{align*}%
Let $y\in \cup _{h\in {\Omega _{i,j,v,m}}}Q_{v,h}$ and $x\in Q_{j,m}$ with $%
v\geq j$. It follows that $y\in Q_{v,h}$ for some $h\in {\Omega _{i,j,v,m}}$
and ${2^{i-1}<2^{j}|2^{-v}h-2^{-j}m|\leq 2^{i}}$. From this we obtain that 
\begin{align*}
\left\vert y-x\right\vert & \leq \left\vert y-{2}^{-v}h\right\vert
+\left\vert x-{2}^{-j}m\right\vert +\left\vert {2}^{-v}h-{2}^{-j}m\right\vert
\\
& \lesssim 2^{-v}+2^{-j}+2^{i-j}+ \\
& \leq 2^{i-j+\delta _{n}},\quad \delta _{n}\in \mathbb{N},
\end{align*}%
which implies that $y$ is located in the ball $B(x,2^{i-j+\delta _{n}})$. We
choose $M>0$. Then, we obtain 
\begin{equation}
\mathcal{V}_{j,v,m}\lesssim 2^{(v-j)\frac{n}{\tau }}\mathcal{M}_{\tau }\big(%
\sum\limits_{z\in \mathbb{Z}^{n}}|\lambda _{v,z}|\chi _{v,z}\big)(x),\quad
x\in Q_{j,m}.  \label{sumv2.4}
\end{equation}%
Inserting \eqref{sumv2.4}\ in \eqref{sumv2.1}, we obtain%
\begin{equation*}
\big|V_{2,j,m}\big|\lesssim 2^{-js}\sum\limits_{v=j+1}^{\infty }2^{(j-v)(A+s-%
\frac{n}{\tau }+n)}2^{v\frac{n}{p}}\mathcal{M}_{\tau }\big(\sum\limits_{z\in 
\mathbb{Z}^{n}}|\lambda _{v,z}|\chi _{v,z}\big)(x)
\end{equation*}%
We choose $\tau $ be such that 
\begin{equation*}
A+s-\frac{n}{\tau }+n>0,\quad A>\max (J-n,0)-s.
\end{equation*}%
By Lemmas \ref{lem:lq-inequality} and \ref{Maximal-Inq copy(2)-lorentz}, we
get 
\begin{align*}
& \Big\|\Big(\sum_{j=0}^{\infty }2^{js\beta }\sum\limits_{m\in \mathbb{Z}%
^{n}}\big|V_{2,j,m}\big|^{\beta }\chi _{j,m}\Big)^{1/\beta }\Big\|_{\dot{K}%
_{p,r}^{\alpha ,q}} \\
& \lesssim \Big\|\Big(\sum_{v=0}^{\infty }2^{v\frac{n}{p}\beta }\big(%
\mathcal{M}_{\tau }\big(\sum\limits_{z\in \mathbb{Z}^{n}}|\lambda
_{v,z}|\chi _{v,z}\big)\big)^{\beta }\Big)^{1/\beta }\Big\|_{\dot{K}%
_{p,r}^{\alpha ,q}} \\
& \lesssim \Big\|\Big(\sum_{v=0}^{\infty }2^{v\frac{n}{p}\beta
}\sum\limits_{z\in \mathbb{Z}^{n}}|\lambda _{v,z}|^{\beta }\chi _{v,z}\Big)%
^{1/\beta }\Big\|_{\dot{K}_{p,r}^{\alpha ,q}}.
\end{align*}%
Collecting the estimates\ obtained for $V_{1,j,m}$ and $V_{2,j,m}$, we
obtain \eqref{rep-new-lorentz}.

\textit{Step 3.} We prove our claim \eqref{claim-emb}. Let $\lambda \in \dot{%
K}_{p,r}^{\alpha ,q}b_{\infty }^{s}$ and $v\in \mathbb{N}$. We will estimate%
\begin{equation*}
\mathcal{H}_{v,1}=2^{(h+\frac{n}{2})v}\sup_{k\leq 1+c_{n}-v}\Big(2^{k\alpha }%
\Big\|\Big(\sum_{m\in \mathbb{Z}^{n}}|\lambda _{v,m}|\chi _{v,m}\chi _{k}%
\Big\|_{L^{p/t,\infty }}\Big).
\end{equation*}%
and%
\begin{equation*}
\mathcal{H}_{v,2}=2^{(h+\frac{n}{2})v}\sup_{k\geq 2+c_{n}-v}\Big(2^{k\alpha }%
\Big\|\Big(\sum_{m\in \mathbb{Z}^{n}}|\lambda _{v,m}|\chi _{v,m}\chi _{k}%
\Big\|_{L^{p/t,\infty }}\Big).
\end{equation*}%
\textit{Estimation of }$\mathcal{H}_{v,1}$\textit{.} Let $u>0,x\in R_{k}\cap
Q_{v,m}$ and $y\in Q_{v,m}$ with $k\leq 1+c_{n}-v$. As in Theorem, we obtain
that 
\begin{equation*}
\sum_{m\in \mathbb{Z}^{n}}|\lambda _{v,m}|\chi _{v,m}(x)\leq 2^{n\frac{v}{u}}%
\Big\|\sum_{m\in \mathbb{Z}^{n}}|\lambda _{v,m}|\chi _{v,m}\chi
_{B(0,2^{c_{n}-v+2})}\Big\|_{L^{u,u}}.
\end{equation*}%
This yields%
\begin{align}
\mathcal{H}_{v,1} &\lesssim 2^{v(h+\frac{n}{u}+\frac{n}{2}-n\frac{t}{p}%
-\alpha )}\Big\|\sum_{m\in \mathbb{Z}^{n}}|\lambda _{v,m}|\chi _{v,m}\chi
_{B(0,2^{c_{n}-v+2})}\Big\|_{L^{u,u}}  \notag \\
&\lesssim 2^{v(s+\frac{n}{u}+\frac{n}{2}-\frac{n}{p}-\alpha )}\Big(%
\sum_{i\leq -v}\Big\|\sum_{m\in \mathbb{Z}^{n}}|\lambda _{v,m}|\chi
_{v,m}\chi _{i+c_{n}+2}\Big\|_{L^{u,u}}^{\varkappa }\Big)^{1/\varkappa },
\label{sobolev-lorentz2}
\end{align}%
where $\varkappa =\min (1,u)$ and we have used Lemma \ref{Lp-estimate}, and
the implicit\ constant is independent of $v$. We may choose $u>0$ such that $%
\frac{1}{u}>\max (\frac{1}{p},\frac{1}{r},\frac{1}{p}+\frac{\alpha }{n})$
and 
\begin{equation*}
\frac{n}{u}=\frac{n}{p}+\frac{n}{l}=\frac{n}{\infty }+\frac{n}{u},\quad 
\frac{n}{l}=\alpha +\frac{n}{d},\quad 0<d<\infty .
\end{equation*}%
By H\"{o}lder's inequality and \eqref{est-function1}, we obtain%
\begin{align*}
& \Big\|\sum_{m\in \mathbb{Z}^{n}}|\lambda _{v,m}|\chi _{v,m}\chi
_{i+c_{n}+2}\Big\|_{L^{u,u}} \\
& \lesssim \Big\|\sum_{m\in \mathbb{Z}^{n}}|\lambda _{v,m}|\chi _{v,m}\chi
_{i+c_{n}+2}\Big\|_{L^{p,\infty }}\big\|\chi _{i+c_{n}+2}\big\|_{L^{l,u}} \\
& \lesssim 2^{i(\frac{n}{d}+\alpha )}\Big\|\sum_{m\in \mathbb{Z}%
^{n}}|\lambda _{v,m}|\chi _{v,m}\chi _{i+c_{n}+2}\Big\|_{L^{p,\infty }} \\
& \lesssim 2^{i(\frac{n}{d}+\alpha )-(s+\frac{n}{2})v}\sup_{j\in \mathbb{N}%
_{0}}\Big\|2^{(s+\frac{n}{2})j}\sum_{m\in \mathbb{Z}^{n}}|\lambda
_{j,m}|\chi _{j,m}\chi _{i+c_{n}+2}\Big\|_{L^{p,\infty }},
\end{align*}%
where the implicit constant is independent of $i$\ and $v$. Inserting this
estimate in \eqref{sobolev-lorentz2}, we get 
\begin{align*}
\mathcal{H}_{v,1}& \lesssim 2^{v\frac{n}{d}}\Big(\sum_{i\leq -v}2^{i(\frac{n%
}{d}+\alpha )\varkappa }\sup_{j\in \mathbb{N}_{0}}\Big\|2^{(s+\frac{n}{2}%
)j}\sum_{m\in \mathbb{Z}^{n}}|\lambda _{j,m}|\chi _{j,m}\chi _{i+c_{n}+2}%
\Big\|_{L^{p,\infty }}^{\varkappa }\Big)^{1/\varkappa } \\
& \lesssim \sup_{i\in \mathbb{Z}}2^{-\alpha i}\Big\|\sup_{j\in \mathbb{N}%
_{0}}\Big(2^{(s+\frac{n}{2})j}\sum_{m\in \mathbb{Z}^{n}}|\lambda _{j,m}|\chi
_{j,m}\chi _{-i+c_{n}+2}\Big)\Big\|_{L^{p,\infty }} \\
& \lesssim \big\|\lambda \big\|_{\dot{K}_{p,r}^{\alpha ,q}b_{\infty }^{s}}.
\end{align*}

\textit{Estimation of }$\mathcal{H}_{v,2}$\textit{.} As in the proof of
Theorem \ref{embeddings3 copy(1)-lorentz}, we obtain 
\begin{align*}
\vartheta _{v,k} &=2^{v(s+\frac{n}{2})+k\alpha }\sum_{m\in \mathbb{Z}%
^{n}}|\lambda _{v,m}|\chi _{v,m} \\
&\leq 2^{v(s_{2}+\frac{n}{2})+k\alpha }\sum_{m\in \mathbb{Z}^{n}}|\lambda
_{v,m}|\chi _{v,m}\chi _{\breve{R}_{k}} \\
&=\hslash _{k,v}..
\end{align*}%
where $\breve{R}_{k}=\cup _{i=-2}^{3}R_{k+i}$. Since $h-s=\frac{nt}{p}-\frac{%
n}{p}$, we get 
\begin{equation}
2^{v(h-s)}\big\|\vartheta _{v,k}\big\|_{L^{p/t,\infty }}=2^{v(\frac{nt}{p}-%
\frac{n}{p})}\big\|\vartheta _{v,k}\big\|_{L^{p/t,\infty }}
\label{Lorentz2.2}
\end{equation}%
for any $v\in \mathbb{N}_{0},k\in \mathbb{Z}$. Using duality, the right-hand
side of \eqref{Lorentz2.2} is dominated by%
\begin{equation*}
c\sup \int_{\mathbb{R}^{n}}2^{v(\frac{nt}{p}-\frac{n}{p})}\vartheta
_{v,k}(x)g(x)dx,
\end{equation*}%
where the supremum is taken\ over all $g\in L^{(p/t)^{\prime },1}$ such that 
$\big\|g\big\|_{L^{(p/t)^{\prime },1}}\leq 1$. It follows from Lemma \ref%
{Hardy-Littlewood inequality} that 
\begin{equation*}
2^{v(\frac{nt}{p}-\frac{n}{p})}\int_{\mathbb{R}^{n}}\vartheta
_{v,k}(x)g(x)dx\leq 2^{v(\frac{nt}{p}-\frac{n}{p})}\int_{0}^{\infty
}\vartheta _{v,k}^{\ast }(t)g^{\ast }(t)dt.
\end{equation*}%
We have%
\begin{equation}
\int_{0}^{\infty }\vartheta _{v,k}^{\ast }(t)g^{\ast
}(t)dt=\int_{0}^{2^{-vn}}\vartheta _{v,k}^{\ast }(t)g^{\ast
}(t)dt+\sum_{l=0}^{\infty }\int_{2^{(l-v)n}}^{2^{(l-v)n+n}}\vartheta
_{v,k}^{\ast }(t)g^{\ast }(t)dt.  \label{Lorentz2.1.1}
\end{equation}

Since $\vartheta _{v,k}^{\ast }$ is constant in $[0,2^{-vn})$\ and $%
\vartheta _{v,k}^{\ast }\leq \hslash _{k,v}^{\ast }$, the first term on the
right-hand side of \eqref{Lorentz2.1.1} is bounded by%
\begin{align}
\vartheta _{v,k}^{\ast }(2^{-vn-1})\int_{0}^{2^{-vn}}g^{\ast }(t)dt& \leq
2^{-vn}\vartheta _{v,k}^{\ast }(2^{-vn-1})g^{\ast \ast }(2^{-vn})  \notag \\
& \leq 2^{-vn}\hslash _{k,v}^{\ast }(2^{-vn-1})g^{\ast \ast }(2^{-vn}) 
\notag \\
& \leq 2^{-vn(1-\frac{1}{p})}\sup_{j\in \mathbb{N}_{0}}\big(2^{-\frac{jn}{p}%
}\hslash _{k,v}^{\ast }(2^{-jn-1})\big)g^{\ast \ast }(2^{-vn})  \notag \\
& \leq 2^{v(\frac{n}{p}-\frac{tn}{p})}\sup_{j\in \mathbb{Z}}\big(2^{-\frac{jn%
}{p}}\hslash _{k,v}^{\ast }(2^{-jn-1})\big)\sup_{v\in \mathbb{Z}}\big(%
2^{-vn(1-\frac{t}{p})}g^{\ast \ast }(2^{-vn})\big)  \notag \\
& \leq 2^{v(\frac{n}{p}-\frac{tn}{p})}\big\|\hslash _{k,v}\big\|%
_{L^{p,\infty }}\big\|g\big\|_{L^{(p/t)^{\prime },1}}.  \label{lorentz5.1}
\end{align}%
The second term on the right-hand side of \eqref{Lorentz2.1.1} can be
estimated from above by%
\begin{align}
& c\sum_{l=0}^{\infty }\hslash _{k,v}^{\ast }(2^{(l-v)n})2^{(l-v)n}g^{\ast
}(2^{(l-v)n})  \notag \\
& =c2^{v(\frac{n}{p}-\frac{tn}{p})}\sum_{l=0}^{\infty }2^{(l-v)\frac{n}{p}%
}\hslash _{k,v}^{\ast }(2^{(l-v)n})2^{(l-v)n(1-\frac{1}{p})}2^{v(\frac{tn}{p}%
-\frac{n}{p})}g^{\ast }(2^{(l-v)n}).  \label{Lorentz3.1}
\end{align}%
The term inside the sum in \eqref{Lorentz3.1} is dominated by%
\begin{align}
& \sup_{j\in \mathbb{N}_{0}}\big(2^{(l-j)n\frac{1}{p}}\hslash _{k,v}^{\ast
}(2^{(l-j)n})\big)\sup_{v\in \mathbb{N}_{0}}\Big(2^{(l-v)n(1-\frac{1}{p}%
)}2^{v(\frac{tn}{p}-\frac{n}{p})}g^{\ast }(2^{(l-v)n})\Big)  \notag \\
& \leq 2^{l(\frac{tn}{p}-\frac{n}{p})\theta _{1}}\big\|\hslash _{k,v}\big\|%
_{L^{p,\infty }}\sup_{v\in \mathbb{N}_{0}}\big(2^{(l-v)n(1-\frac{t}{p}%
)}g^{\ast }(2^{(l-v)n})\big)  \notag \\
& \leq 2^{l(\frac{tn}{p}-\frac{n}{p})}\big\|\hslash _{k,v}\big\|%
_{L^{p,\infty }}\big\|g\big\|_{L^{(p/t)^{\prime },1}}.  \label{Lorentz4.1}
\end{align}%
Collecting the estimates \eqref{Lorentz4.1} and \eqref{lorentz5.1} we get 
\begin{equation*}
\mathcal{H}_{v,2}\lesssim \big\|\lambda \big\|_{\dot{K}_{p,r}^{\alpha
,q}b_{\infty }^{s}}.
\end{equation*}

The proof is complete.
\end{proof}

Let $u\in \mathbb{N}$ and $\psi _{F},\psi _{M}\in C^{u}(\mathbb{R})$ be
real-valued compactly supported Daubechies wavelets with%
\begin{equation*}
\mathcal{F}\psi _{F}(0)=(2\pi )^{-\frac{1}{2}},\quad \int_{\mathbb{R}%
}x^{l}\psi _{M}(x)dx=0,\quad l\in \{0,...,u-1\}
\end{equation*}%
and%
\begin{equation*}
\big\|\psi _{F}\big\|_{2}=\big\|\psi _{M}\big\|_{2}=1.
\end{equation*}%
We have that%
\begin{equation*}
\{\psi _{F}(x-m),2^{\frac{j}{2}}\psi _{M}(2^{j}x-m)\}_{j\in \mathbb{N}%
_{0},m\in \mathbb{Z}^{n}}
\end{equation*}%
is an orthonormal basis in $L^{2}(\mathbb{R})$. This orthonormal basis can
be generalized to the $\mathbb{R}^{n}$ by the usual multiresolution
procedure. Let%
\begin{equation*}
G=\{G_{1},...,G_{n}\}\in G^{0}=\{F,M\}^{n}
\end{equation*}%
which means that $G_{r}$ is either $F$ or $M$. Let%
\begin{equation*}
G=\{G_{1},...,G_{n}\}\in G^{j}=\{F,M\}^{n^{\ast }},\quad j\in \mathbb{N},
\end{equation*}%
where $^{\ast }$ indicates that at least one of the components of $G$ must
be an $M$. Hence $G^{0}$ has $2^{n}$ elements, whereas $G^{j}$ with $j\in 
\mathbb{N}$ has $2^{n}-1$ elements. Let%
\begin{equation*}
\Psi _{G,m}^{j}(x)=2^{j\frac{n}{2}}\prod_{r=1}^{n}\psi
_{G_{r}}(2^{j}x_{r}-m_{r}),\quad G\in G^{j},m\in \mathbb{Z}^{n},x\in \mathbb{%
R}^{n},j\in \mathbb{N}_{0}.
\end{equation*}%
We always assume that $\psi _{F}\ $and $\psi _{M}$ have $L^{2}$-norm $1$.
Then%
\begin{equation}
\Psi =\{\Psi _{G,m}^{j}:\quad j\in \mathbb{N}_{0},G\in G^{j},m\in \mathbb{Z}%
^{n}\}  \label{orthogonal}
\end{equation}%
is an orthonormal basis in $L^{2}(\mathbb{R}^{n})$ (for any $u\in \mathbb{N}$%
) and.%
\begin{equation*}
f=\sum_{j=0}^{\infty }\sum_{G\in G^{j}}\sum_{m\in \mathbb{Z}^{n}}\lambda
_{j,m}^{G}2^{-j\frac{n}{2}}\Psi _{G,m}^{j},
\end{equation*}%
with%
\begin{equation*}
\lambda _{j,m}^{G}=\lambda _{j,m}^{G}(f)=2^{j\frac{n}{2}}\langle f,\Psi
_{G,m}^{j}\rangle ,
\end{equation*}%
is the corresponding expansion.

Let $\alpha ,s\in \mathbb{R},0<p<\infty ,0<r,q\leq \infty $\ and $0<\beta
\leq \infty $. We set%
\begin{equation*}
\overline{\dot{K}_{p,r}^{\alpha ,q}b_{\beta }^{s}}=\{\lambda =\{\lambda
_{j,m}^{G}\}_{j\in \mathbb{N}_{0},G\in G^{j},m\in \mathbb{Z}^{n}}\subset 
\mathbb{C}:\big\|\lambda \big\|_{\overline{\dot{K}_{p,r}^{\alpha ,q}b_{\beta
}^{s}}}<\infty \},
\end{equation*}%
and%
\begin{equation*}
\overline{\dot{K}_{p,r}^{\alpha ,q}f_{\beta }^{s}}=\{\lambda =\{\lambda
_{j,m}^{G}\}_{j\in \mathbb{N}_{0},G\in G^{j},m\in \mathbb{Z}^{n}}\subset 
\mathbb{C}:\big\|\lambda \big\|_{\overline{\dot{K}_{p,r}^{\alpha ,q}f_{\beta
}^{s}}}<\infty \},
\end{equation*}%
where%
\begin{equation*}
\big\|\lambda \big\|_{\overline{\dot{K}_{p,r}^{\alpha ,q}b_{\beta }^{s}}}=%
\Big(\sum_{j=0}^{\infty }2^{js\beta }\Big\|\sum_{G\in G^{j}}\sum_{m\in 
\mathbb{Z}^{n}}|\lambda _{j,m}^{G}|\chi _{j,m}\Big\|_{\dot{K}_{p,r}^{\alpha
,q}}^{\beta }\Big)^{1/\beta }.
\end{equation*}%
and%
\begin{equation*}
\big\|\lambda \big\|_{\overline{\dot{K}_{p,r}^{\alpha ,q}f_{\beta }^{s}}}=%
\Big\|\Big(\sum_{j=0}^{\infty }\sum_{G\in G^{j}}\sum_{m\in \mathbb{Z}%
^{n}}2^{js\beta }|\lambda _{j,m}^{G}|^{\beta }\chi _{j,m}\Big)^{1/\beta }%
\Big\|_{\dot{K}_{p,r}^{\alpha ,q}}.
\end{equation*}

\begin{theorem}
\label{wavelet}\textit{Let }$\alpha ,s\in \mathbb{R},0<p<\infty ,0<r,q\leq
\infty ,0<\beta \leq \infty $ and $\alpha >-\frac{n}{p}$. Let $\{\Psi
_{G,m}^{j}\}$ be the wavelet system with 
\begin{equation}
u>\max (J-s,s).  \label{u-assumption}
\end{equation}%
Let $f\in \mathcal{S}^{\prime }\mathcal{(}\mathbb{R}^{n})$. Then $f\in \dot{K%
}_{p,r}^{\alpha ,q}A_{\beta }^{s}$ if and only if 
\begin{equation}
f=\sum_{j=0}^{\infty }\sum_{G\in G^{j}}\sum_{m\in \mathbb{Z}^{n}}\lambda
_{j,m}^{G}2^{-j\frac{n}{2}}\Psi _{G,m}^{j},\quad \lambda \in \overline{\dot{K%
}_{p,r}^{\alpha ,q}a_{\beta }^{s}}  \label{representation}
\end{equation}%
with unconditional convergence in $\mathcal{S}^{\prime }\mathcal{(}\mathbb{R}%
^{n})$ and in any space $\dot{K}_{p,r}^{\alpha ,q}A_{\beta }^{\sigma }$ with 
$\sigma <s$. The representation \eqref{representation} is unique, 
\begin{equation*}
\lambda _{j,m}^{G}=\lambda _{j,m}^{G}(f)=2^{j\frac{n}{2}}\langle f,\Psi
_{G,m}^{j}\rangle
\end{equation*}%
and 
\begin{equation*}
I:\quad f\longmapsto \{\lambda _{j,m}^{G}(f)\}
\end{equation*}%
is an isomorphic map from $\dot{K}_{p,r}^{\alpha ,q}A_{\beta }^{s}$ onto $%
\overline{\dot{K}_{p,r}^{\alpha ,q}a_{\beta }^{s}}$. In particular, it holds 
\begin{equation*}
\big\|f\big\|_{\dot{K}_{p,r}^{\alpha ,q}F_{\beta }^{s}}\approx \big\|\lambda %
\big\|_{\overline{\dot{K}_{p,r}^{\alpha ,q}a_{\beta }^{s}}}.
\end{equation*}%
If, in addition, $q<\infty $, then $\{\Psi _{G,m}^{j}\}$ is an unconditional
basis in $\dot{K}_{p,r}^{\alpha ,q}A_{\beta }^{s}$.
\end{theorem}

\begin{proof}
We will do the proof in four steps.

\textit{Step 1.} Let $f\in \mathcal{S}^{\prime }\mathcal{(}\mathbb{R}^{n})$
be given by \eqref{representation}. Then%
\begin{equation*}
\varrho _{j,m}^{G}=2^{-(s-\frac{n}{p}+\frac{n}{2})j}\Psi _{G,m}^{j},\quad
j\in \mathbb{N}_{0},G\in G^{j},m\in \mathbb{Z}^{n}
\end{equation*}%
are $(s,p)$-atoms according to Definition \ref{atom2} with $K=L=u$ (up to
unimportant constants). We set%
\begin{equation*}
\lambda =\{\lambda _{j,m}^{G}:j\in \mathbb{N}_{0},G\in G^{j},m\in \mathbb{Z}%
^{n}\}.
\end{equation*}%
From Theorem \ref{atomic-decv2} and \eqref{u-assumption} we obtain $f\in 
\dot{K}_{p,r}^{\alpha ,q}A_{\beta }^{s}$ and 
\begin{equation}
\big\|f\big\|_{\dot{K}_{p,r}^{\alpha ,q}A_{\beta }^{s}}\lesssim \big\|%
\lambda \big\|_{\overline{\dot{K}_{p,r}^{\alpha ,q}a_{\beta }^{s}}}.
\label{key-estimate}
\end{equation}

\textit{Step 2. }Let $f\in \dot{K}_{p,r}^{\alpha ,q}A_{\beta }^{s}$. Then 
\begin{equation*}
k_{j,m}^{G}=2^{j\frac{n}{2}}\Psi _{G,m}^{j},\quad j\in \mathbb{N}_{0},G\in
G^{j},m\in \mathbb{Z}^{n}
\end{equation*}%
are kernels according to Definition \ref{kernel} with $A=B=u$. We set 
\begin{equation*}
k(f)=\{k_{j,m}^{G}(f)=\langle f,k_{j,m}^{G}\rangle :j\in \mathbb{N}_{0},G\in
G^{j},m\in \mathbb{Z}^{n}\}.
\end{equation*}%
All conditions on $k_{j,m}^{G}$ are fulfilled by \eqref{u-assumption} and
the compact support of the wavelets we get by Theorem \ref{kernel3} 
\begin{equation*}
\big\|k(f)\big\|_{\overline{\dot{K}_{p,r}^{\alpha ,q}a_{\beta }^{s}}%
}\lesssim \big\|f\big\|_{\dot{K}_{p,r}^{\alpha ,q}A_{\beta }^{s}}.
\end{equation*}

\textit{Step 3. }We prove the unconditional convergence of %
\eqref{representation} in $\mathcal{S}^{\prime }\mathcal{(}\mathbb{R}^{n})$
and in any space $\dot{K}_{p,r}^{\alpha ,q}A_{\beta }^{\sigma }$ with $%
\sigma <s$. First assume that $0<q<\infty $\ and $0<\beta <\infty $. By %
\eqref{key-estimate} and the properties of the sequence spaces $\dot{K}%
_{p,r}^{\alpha ,q}a_{\beta }^{s}$,\ we get the unconditional convergence of %
\eqref{representation} in $\dot{K}_{p,r}^{\alpha ,q}A_{\beta }^{s}$ and
hence in $\mathcal{S}^{\prime }\mathcal{(}\mathbb{R}^{n})$ and in any space $%
\dot{K}_{p,r}^{\alpha ,q}A_{\beta }^{\sigma }$ with $\sigma <s$. The
structure of the sequence spaces $\dot{K}_{p,r}^{\alpha ,q}a_{\beta }^{s}$
and $\sigma <s$, yields the unconditional convergence of $f$ given by %
\eqref{representation} in $\dot{K}_{p,r}^{\alpha ,q}A_{\beta }^{\sigma }$
with $\sigma <s$ and hence in $\mathcal{S}^{\prime }\mathcal{(}\mathbb{R}%
^{n})$.

\textit{Step 4. }We will prove the uniqueness of the coefficients. It
follows by Step 1 that%
\begin{equation*}
g=\sum_{j=0}^{\infty }\sum_{G\in G^{j}}\sum_{m\in \mathbb{Z}^{n}}\lambda
_{j,m}^{G}2^{-j\frac{n}{2}}\Psi _{G,m}^{j}\in \dot{K}_{p,r}^{\alpha
,q}A_{\beta }^{s}.
\end{equation*}%
From \eqref{u-assumption} the dual pairing of $g$ and any wavelet $\Psi
_{G^{\prime },m^{\prime }}^{j^{\prime }}$ makes sense. Since %
\eqref{orthogonal} is an orthonormal basis in $L^{2}(\mathbb{R}^{n})$ one
gets%
\begin{equation}
\langle g,\Psi _{G^{\prime },m^{\prime }}^{j^{\prime }}\rangle
=\sum_{j=0}^{\infty }\sum_{G\in G^{j}}\sum_{m\in \mathbb{Z}^{n}}\lambda
_{j,m}^{G}2^{-j\frac{n}{2}}\langle \Psi _{G,m}^{j},\Psi _{G^{\prime
},m^{\prime }}^{j^{\prime }}\rangle =\langle f,\Psi _{G^{\prime },m^{\prime
}}^{j^{\prime }}\rangle .  \label{g=f}
\end{equation}%
This holds also for finite linear combinations of $\Psi _{G^{\prime
},m^{\prime }}^{j^{\prime }}$. If $\varphi \in \mathcal{S}(\mathbb{R}^{n})$
then one has the unique $L^{2}(\mathbb{R}^{n})$-representation%
\begin{equation*}
\varphi =\sum_{j=0}^{\infty }\sum_{G\in G^{j}}\sum_{m\in \mathbb{Z}%
^{n}}\langle \varphi ,\Psi _{G,m}^{j}\rangle \Psi _{G,m}^{j}.
\end{equation*}%
By Step 1 of Theorem \ref{kernel3} this representation converges also in the
dual space of $\dot{K}_{p,r}^{\alpha ,q}A_{\beta }^{s}$. We get by %
\eqref{g=f} that $\langle g,\varphi \rangle =\langle f,\varphi \rangle $ for
all $\varphi \in \mathcal{S}(\mathbb{R}^{n})$ and hence $g=f$.
\end{proof}

\begin{remark}
We refer the reader to \cite{Xu09} for an atomic, molecular and\ wavelet
characterizations of the spaces $\dot{K}_{p}^{\alpha ,q}A_{\beta }^{s}$.
\end{remark}

\section{Several equivalent characterizations}

In this part, we establish characterizations of $\dot{K}_{p,r}^{\alpha
,q}A_{\beta }^{s}$ by Peetre maximal function, by ball mean of differences
and we will present some useful examples.

\subsection{Maximal function characterization}

Let $\{\varphi _{j}\}_{j\in \mathbb{N}_{0}}$ be the smooth dyadic resolution
of unity. Let $a>0$ and $f\in \mathcal{S}^{\prime }(\mathbb{R}^{n})$. Then
we define the Peetre maximal function as follows: 
\begin{equation*}
\varphi _{k}^{\ast ,a}f(x)=\sup_{y\in \mathbb{R}^{n}}\frac{\left\vert 
\mathcal{F}^{-1}\varphi _{k}\ast f(y)\right\vert }{\left( 1+2^{k}\left\vert
x-y\right\vert \right) ^{a}},\qquad x\in \mathbb{R}^{n},k\in \mathbb{N}_{0}.
\end{equation*}%
We now present a fundamental characterization of the spaces under
consideration.

\begin{theorem}
\label{fun-char-lorentz}Let $s\in \mathbb{R},0<p<\infty ,0<r,q\leq \infty
,0<\beta \leq \infty $ and $\alpha >-\frac{n}{p}$. \newline
$\mathrm{(i)}$ Let $a>\frac{n}{\min \big(p,\frac{n}{\alpha +\frac{n}{p}}\big)%
}$. Then 
\begin{equation*}
\big\|f\big\|_{\dot{K}_{p,r}^{\alpha ,q}B_{\beta }^{s}}^{\ast }=\Big(%
\sum\limits_{k=0}^{\infty }2^{ks\beta }\big\|\varphi _{k}^{\ast ,a}f\big\|_{%
\dot{K}_{p,r}^{\alpha ,q}}^{\beta }\Big)^{1/\beta },
\end{equation*}%
\textit{is an equivalent quasi-norm in }$\dot{K}_{p,r}^{\alpha ,q}B_{\beta
}^{s}$, with the obvious modification if $\beta =\infty $.\newline
$\mathrm{(ii)}$ Let $0<q<\infty $ and $a>\frac{n}{\min \big(\min (p,\beta ),%
\frac{n}{\alpha +\frac{n}{p}}\big)}$. Then 
\begin{equation*}
\big\|f\big\|_{\dot{K}_{p,r}^{\alpha ,q}F_{\beta }^{s}}^{\ast }=\Big\|\Big(%
\sum\limits_{k=0}^{\infty }2^{ks\beta }|\varphi _{k}^{\ast ,a}f|^{\beta }%
\Big)^{1/\beta }\Big\|_{\dot{K}_{p,r}^{\alpha ,q}},
\end{equation*}%
\textit{is an equivalent quasi-norm in }$\dot{K}_{p,r}^{\alpha ,q}F_{\beta
}^{s}$, with the obvious modification if \textit{\ }$\beta =\infty $.
\end{theorem}

\begin{proof}
By similarity, we only consider the spaces $\dot{K}_{p,r}^{\alpha
,q}F_{\beta }^{s}$. It is easy to see that for any $f\in \mathcal{S}^{\prime
}(\mathbb{R}^{n})$ with $\big\|f\big\|_{\dot{K}_{p,r}^{\alpha ,q}F_{\beta
}^{s}}^{\ast }<\infty $ and any $x\in \mathbb{R}^{n},k\in \mathbb{N}_{0}$ we
have%
\begin{equation*}
|\mathcal{F}^{-1}\varphi _{k}\ast f(x)|\leq \varphi _{k}^{\ast ,a}f(x)\text{.%
}
\end{equation*}%
This shows that $\big\|f\big\|_{\dot{K}_{p,r}^{\alpha ,q}F_{\beta }^{s}}\leq %
\big\|f\big\|_{\dot{K}_{p,r}^{\alpha ,q}F_{\beta }^{s}}^{\ast }$. We will
prove that there is a constant $C>0$ such that for every $f\in \dot{K}%
_{p,r}^{\alpha ,q}F_{\beta }^{s}$%
\begin{equation*}
\big\|f\big\|_{\dot{K}_{p,r}^{\alpha ,q}F_{\beta }^{s}}^{\ast }\leq C\big\|f%
\big\|_{\dot{K}_{p,r}^{\alpha ,q}F_{\beta }^{s}}.
\end{equation*}%
Let $0<\tau <\infty $ be such that $a>\frac{n}{\tau }>\frac{n}{\min \big(%
\min (p,\beta ),\frac{n}{\alpha +\frac{n}{p}}\big)}$. By Lemmas \ref{r-trick}
and \ref{est-maximal}, the estimate 
\begin{equation}
\left\vert \mathcal{F}^{-1}\varphi _{k}\ast f(y)\right\vert \leq C_{1}\text{ 
}\big(\eta _{k,\delta \tau }\ast |\mathcal{F}^{-1}\varphi _{k}\ast f|^{\tau
}(y)\big)^{1/\tau }  \label{esti-conv1-lorentz}
\end{equation}%
is true for any $y\in \mathbb{R}^{n}$, $\delta >\frac{n}{\tau }$ and $k\in 
\mathbb{N}_{0}$. Now dividing both sides of \eqref{esti-conv1-lorentz} by $%
\left( 1+2^{k}\left\vert x-y\right\vert \right) ^{a}$, in the right-hand
side we use the inequality%
\begin{equation*}
\left( 1+2^{k}\left\vert x-y\right\vert \right) ^{-a}\leq \left(
1+2^{k}\left\vert x-z\right\vert \right) ^{-a}\left( 1+2^{k}\left\vert
y-z\right\vert \right) ^{a},\quad x,y,z\in \mathbb{R}^{n},
\end{equation*}%
while in the left-hand side we take the supremum over $y\in \mathbb{R}^{n}$,
we find that%
\begin{align*}
\varphi _{k}^{\ast ,a}f(x)\lesssim & \big(\eta _{k,a\tau }\ast |\mathcal{F}%
^{-1}\varphi _{k}\ast f|^{\tau }(x)\big)^{1/\tau } \\
\lesssim & \mathcal{M}_{\tau }(\mathcal{F}^{-1}\varphi _{k}\ast f)(x),
\end{align*}%
where the implicit constant is independent of $x,k$ and $f$. Applying Lemma %
\ref{Maximal-Inq copy(2)-lorentz}, we deduce that 
\begin{equation*}
\big\|f\big\|_{\dot{K}_{p,r}^{\alpha ,q}F_{\beta }^{s}}^{\ast }\lesssim %
\Big\|\Big(\sum\limits_{k=0}^{\infty }2^{ks\beta }|\mathcal{F}^{-1}\varphi
_{k}\ast f|^{\beta }\Big)^{1/\beta }\Big\|_{\dot{K}_{p,r}^{\alpha
,q}}\lesssim \big\|f\big\|_{\dot{K}_{p,r}^{\alpha ,q}F_{\beta }^{s}}.
\end{equation*}%
The proof of Theorem \ref{fun-char-lorentz} is complete.
\end{proof}

Let us consider $k_{0},k\in \mathcal{S}\left( \mathbb{R}^{n}\right) $ and $%
S\geq -1$ an integer such that for an $\varepsilon >0$%
\begin{align}
\left\vert \mathcal{F}k_{0}\left( \xi \right) \right\vert & >0\text{\quad
for\quad }\left\vert \xi \right\vert <2\varepsilon  \label{T-cond1} \\
\left\vert \mathcal{F}k\left( \xi \right) \right\vert & >0\text{\quad
for\quad }\frac{\varepsilon }{2}<\left\vert \xi \right\vert <2\varepsilon
\label{T-cond2}
\end{align}%
and%
\begin{equation}
\int_{\mathbb{R}^{n}}x^{\alpha }k(x)dx=0\text{\quad for any\quad }\left\vert
\alpha \right\vert \leq S.  \label{moment-cond}
\end{equation}%
Here $\mathrm{\eqref{T-cond1}}$ and $\mathrm{\eqref{T-cond2}}$\ are
Tauberian conditions, while $\mathrm{\eqref{moment-cond}}$ are moment
conditions on $k$. We recall the notation%
\begin{equation*}
k_{t}(x)=t^{-n}k(t^{-1}x)\text{,}\quad k_{j}\left( x\right) =k_{2^{-j}}(x)%
\text{,\quad for}\quad t>0\quad \text{and}\quad j\in \mathbb{N}.
\end{equation*}%
For any $a>0$, $f\in \mathcal{S}^{\prime }\left( \mathbb{R}^{n}\right) $ and 
$x\in \mathbb{R}^{n}$ we denote%
\begin{equation*}
k_{j}^{\ast ,a}f(x)=\sup_{y\in \mathbb{R}^{n}}\frac{\left\vert k_{j}\ast
f(y)\right\vert }{(1+2^{j}\left\vert x-y\right\vert )^{a}},\quad j\in 
\mathbb{N}_{0}.
\end{equation*}%
Usually $k_{j}\ast f$ is called local mean.

We are able now to state the main result of this section.

\begin{theorem}
\label{loc-mean-char-lorentz}Let $0<p<\infty ,0<r,q\leq \infty ,0<\beta \leq
\infty $ and $\alpha >-\frac{n}{p}$. Let $s<S+1$\newline
$\mathrm{(i)}$ Let $a>\frac{n}{\min \big(p,\frac{n}{\alpha +\frac{n}{p}}\big)%
}$. Then 
\begin{equation*}
\big\|f\big\|_{\dot{K}_{p,r}^{\alpha ,q}B_{\beta }^{s}}^{\bullet }=\Big(%
\sum\limits_{j=0}^{\infty }2^{js\beta }\big\|k_{j}^{\ast ,a}f\big\|_{\dot{K}%
_{p,r}^{\alpha ,q}}^{\beta }\Big)^{1/\beta }
\end{equation*}%
and%
\begin{equation*}
\big\|f\big\|_{\dot{K}_{p,r}^{\alpha ,q}B_{\beta }^{s}}^{\star }=\Big(%
\sum\limits_{j=0}^{\infty }2^{js\beta }\big\|k_{j}\ast f\big\|_{\dot{K}%
_{p,r}^{\alpha ,q}}^{\beta }\Big)^{1/\beta }
\end{equation*}%
\textit{are an equivalent quasi-norm in }$\dot{K}_{p,r}^{\alpha ,q}B_{\beta
}^{s}$, with the obvious modification if \textit{\ }$\beta =\infty $.\newline
$\mathrm{(ii)}$ Let $0<q<\infty $ and $a>\frac{n}{\min \big(\min (p,\beta ),%
\frac{n}{\alpha +\frac{n}{p}}\big)}$. Then 
\begin{equation*}
\big\|f\big\|_{\dot{K}_{p,r}^{\alpha ,q}F_{\beta }^{s}}^{\bullet }=\Big\|%
\Big(\sum\limits_{j=0}^{\infty }2^{js\beta }|k_{j}^{\ast ,a}f|^{\beta }\Big)%
^{1/\beta }\Big\|_{\dot{K}_{p,r}^{\alpha ,q}},
\end{equation*}%
and%
\begin{equation*}
\big\|f\big\|_{\dot{K}_{p,r}^{\alpha ,q}F_{\beta }^{s}}^{\star }=\Big\|\Big(%
\sum\limits_{j=0}^{\infty }2^{js\beta }|k_{j}\ast f|^{\beta }\Big)^{1/\beta }%
\Big\|_{\dot{K}_{p,r}^{\alpha ,q}},
\end{equation*}%
\textit{are an equivalent quasi-norm in }$\dot{K}_{p,r}^{\alpha ,q}F_{\beta
}^{s}$, with the obvious modification if\textit{\ }$\beta =\infty $.
\end{theorem}

\begin{proof}
The proof is very similar as in Rychkov \cite{Ry01}.
\end{proof}

\subsection{Characterizations by ball mean of differences}

Let $0<p<\infty ,0<r,\beta \leq \infty $. For later use we introduce the
following abbreviations:%
\begin{equation*}
\sigma _{p}=n\max \big(\frac{1}{p}-1,0\big)\quad \text{and}\quad \sigma
_{p,\beta }=n\max \big(\frac{1}{p}-1,\frac{1}{\beta }-1,0\big).
\end{equation*}%
In the next we shall interpret $L_{\mathrm{loc}}^{1}(\mathbb{R}^{n})$ as the
set of regular distributions.

\begin{theorem}
\label{regular-distribution1-lorentz}\textit{Let }$0<p<\infty ,0<r,q,\beta
\leq \infty ,\alpha >\max (-n,-\frac{n}{p}),\alpha _{0}=n-\frac{n}{p}$ and 
\begin{equation*}
s>\max (\sigma _{p},\alpha -\alpha _{0}).
\end{equation*}%
Then 
\begin{equation*}
\dot{K}_{p,r}^{\alpha ,q}A_{\beta }^{s}\hookrightarrow L_{\mathrm{loc}}^{1}(%
\mathbb{R}^{n}),
\end{equation*}%
where $0<q<\infty $ in the case of Herz-type Triebel-Lizorkin spaces.
\end{theorem}

\begin{proof}
Let $\{\varphi _{j}\}_{j\in \mathbb{N}_{0}}$ be a smooth dyadic resolution
of unity. We set 
\begin{equation*}
\varrho _{k}=\sum\limits_{j=0}^{k}\mathcal{F}^{-1}\varphi _{j}\ast f,\quad
k\in \mathbb{N}_{0}.
\end{equation*}%
For technical reasons, we split the proof into two steps.

\textit{Step 1.} We consider the case $1\leq p<\infty $. In order to prove
we additionally do it into the four Substeps 1.1, 1.2, 1.3 and 1.4.

\textit{Substep 1.1.} $-\frac{n}{p}<\alpha <\alpha _{0}$. First assume that $%
1<p<\infty $. Let $1<p_{0}<\infty $ be such that 
\begin{equation*}
p<p_{0}<\frac{n}{\max (0,\frac{n}{p}-s)},
\end{equation*}%
which is possible because of $s>0$. From Theorem \ref{embeddings3-lorentz}
we obtain%
\begin{equation*}
\dot{K}_{p,r}^{\alpha ,q}B_{\beta }^{s}\hookrightarrow \dot{K}%
_{p_{0}}^{\alpha ,q}B_{\beta }^{s+\frac{n}{p_{0}}-\frac{n}{p}%
}\hookrightarrow \dot{K}_{p_{0}}^{\alpha ,\max (1,q)}B_{\beta }^{s+\frac{n}{%
p_{0}}-\frac{n}{p}}.
\end{equation*}%
We have 
\begin{equation*}
\sum_{j=0}^{\infty }\big\|\mathcal{F}^{-1}\varphi _{j}\ast f\big\|_{\dot{K}%
_{p_{0}}^{\alpha ,\max (1,q)}}\lesssim \big\|f\big\|_{\dot{K}_{p,r}^{\alpha
,q}A_{\beta }^{s}}.
\end{equation*}%
Then, the sequence $\{\varrho _{k}\}_{k\in \mathbb{N}_{0}}$ converges to $%
g\in \dot{K}_{p_{0}}^{\alpha ,\max (1,q)}$. Let $\varphi \in \mathcal{S}(%
\mathbb{R}^{n})$. Write 
\begin{equation*}
\langle f-g,\varphi \rangle =\langle f-\varrho _{N},\varphi \rangle +\langle
g-\varrho _{N},\varphi \rangle ,\quad N\in \mathbb{N}_{0}.
\end{equation*}%
Here $\langle \cdot ,\cdot \rangle $ denotes the duality bracket between $%
\mathcal{S}^{\prime }(\mathbb{R}^{n})$ and $\mathcal{S}(\mathbb{R}^{n})$.
Clearly, the first term tends to zero as $N\rightarrow \infty $, while by H%
\"{o}lder's\ inequality there exists a constant $C>0$ independent of $N$
such that 
\begin{equation*}
|\langle g-\varrho _{N},\varphi \rangle |\leq C\big\|g-\varrho _{N}\big\|_{%
\dot{K}_{p_{0}}^{\alpha ,\max (1,q)}},
\end{equation*}%
which tends to zero as $N\rightarrow \infty $. From this and $\dot{K}%
_{p_{0}}^{\alpha ,\max (1,q)}\hookrightarrow L_{\mathrm{loc}}^{1}(\mathbb{R}%
^{n})$, because of $\alpha <n-\frac{n}{p_{0}}$, see Lemma\ \ref%
{regular-distribution-lorentz}, we deduce the desired result. In addition,
we obtain 
\begin{equation*}
\dot{K}_{p,r}^{\alpha ,q}B_{\beta }^{s}\hookrightarrow \dot{K}%
_{p_{0}}^{\alpha ,\max (1,q)}.
\end{equation*}%
The case of the $F$-spaces follows simply from the embedding%
\begin{equation*}
\dot{K}_{p,r}^{\alpha ,q}F_{\infty }^{s}\hookrightarrow \dot{K}%
_{p,r}^{\alpha ,q}B_{\infty }^{s},
\end{equation*}%
Now, we study the case $p=1$. Let $d>1$ be such that 
\begin{equation*}
1<d<\min \big(\frac{n}{\max (0,n-s)},\frac{n}{-\alpha }\big).
\end{equation*}%
From Theorems \ref{embeddings2-lorentz} and \ref{embeddings3-lorentz}, we
obtain 
\begin{equation*}
\dot{K}_{1,r}^{\alpha ,q}A_{\beta }^{s}\hookrightarrow \dot{K}_{1,r}^{\alpha
,q}B_{\infty }^{s}\hookrightarrow \dot{K}_{d}^{\alpha ,q}B_{\infty }^{s+%
\frac{n}{d}-n}\hookrightarrow L_{\mathrm{loc}}^{1}(\mathbb{R}^{n}),
\end{equation*}%
where the last embedding follows since $s+\frac{n}{d}-n>0$ and $-\frac{n}{d}%
<\alpha <0.$

\textit{Substep 1.2.} $\alpha \geq \alpha _{0}$ and $1<p<\infty $. Let $%
1<p_{1}<\infty $ be such that 
\begin{equation*}
s>\alpha +\frac{n}{p}-\frac{n}{p_{1}}.
\end{equation*}%
We distinguish two cases:

$\bullet $ $p_{1}=p$. By Theorem \ref{embeddings3-lorentz}, we obtain 
\begin{equation*}
\dot{K}_{p,r}^{\alpha ,q}B_{\beta }^{s}\hookrightarrow \dot{K}%
_{p,p}^{0,p}B_{\beta }^{s-\alpha }=B_{p,\beta }^{s-\alpha }\hookrightarrow
L_{\mathrm{loc}}^{1}(\mathbb{R}^{n}).
\end{equation*}%
where the last embedding follows by the fact that 
\begin{equation}
B_{p,\beta }^{s-\alpha }\hookrightarrow L^{p},  \label{Substep1.2.1}
\end{equation}%
because of $s-\alpha >0$. The Lorentz Herz-type Triebel-Lizorkin case
follows by Theorem \ref{embeddings2-lorentz}.

$\bullet $ $1<p_{1}<p<\infty $ or $1<p<p_{1}<\infty $. If we assume the
first possibility then Theorem \ref{embeddings3-lorentz} and Substep 1.1
yield 
\begin{equation*}
\dot{K}_{p,r}^{\alpha ,q}B_{\beta }^{s}\hookrightarrow \dot{K}%
_{p_{1}}^{0,q}B_{\beta }^{s-\alpha -\frac{n}{p}+\frac{n}{p_{1}}%
}\hookrightarrow L_{\mathrm{loc}}^{1}(\mathbb{R}^{n}),
\end{equation*}%
since $\alpha +\frac{n}{p}>\frac{n}{p_{1}}$. The latter possibility follows
again by Theorem \ref{embeddings3-lorentz}. Indeed, we have 
\begin{equation*}
\dot{K}_{p,r}^{\alpha ,q}B_{\beta }^{s}\hookrightarrow \dot{K}_{p,r}^{\alpha
_{0},q}B_{\beta }^{s+\alpha _{0}-\alpha }\hookrightarrow \dot{K}%
_{p_{1}}^{0,p_{1}}B_{\beta }^{s-\alpha -\frac{n}{p}+\frac{n}{p_{1}}%
}=B_{p_{1},\beta }^{s-\alpha -\frac{n}{p}+\frac{n}{p_{1}}}\hookrightarrow L_{%
\mathrm{loc}}^{1}(\mathbb{R}^{n}),
\end{equation*}%
where the last embedding follows\ by the fact that 
\begin{equation}
B_{p_{1},\beta }^{s-\alpha -\frac{n}{p}+\frac{n}{p_{1}}}\hookrightarrow
L^{p_{1}}.  \label{Substep1.2.2}
\end{equation}%
Therefore from Theorem \ref{embeddings2-lorentz} we obtain the desired
embeddings.

\textit{Substep 1.3.} $p=1$ and $\alpha >0$. We have 
\begin{equation*}
\dot{K}_{1,r}^{\alpha ,q}B_{\beta }^{s}\hookrightarrow \dot{K}%
_{1}^{0,1}B_{\beta }^{s-\alpha }=B_{1,\beta }^{s-\alpha }\hookrightarrow
L^{1},
\end{equation*}%
since $s>\alpha $.

\textit{Substep 1.4.} $p=1$ and $\alpha =0$. Let $\alpha _{3}$ be a real
number such that 
\begin{equation*}
\max (-n,-s)<\alpha _{3}<0.
\end{equation*}%
From Theorems \ref{embeddings3-lorentz} and \ref{embeddings3 copy(4)-lorentz}%
, we get 
\begin{equation*}
\dot{K}_{1,r}^{0,q}A_{\beta }^{s}\hookrightarrow \dot{K}_{1,r}^{\alpha
_{3},q}B_{\infty }^{s+\alpha _{3}}\hookrightarrow L_{\mathrm{loc}}^{1}(%
\mathbb{R}^{n})
\end{equation*}%
by Substep 1.1.

\textit{Step 2.} We consider the case $0<p<1$.

\textit{Substep 2.1.}\ $-n<\alpha <0$. By Lemma \ref%
{Bernstein-Herz-ine1-lorentz}, we obtain 
\begin{equation*}
\sum_{j=0}^{\infty }\big\|\mathcal{F}^{-1}\varphi _{j}\ast f\big\|_{\dot{K}%
_{1}^{\alpha ,\max (1,q)}}\lesssim \sum_{j=0}^{\infty }2^{j(\frac{n}{p}-n)}%
\big\|\mathcal{F}^{-1}\varphi _{j}\ast f\big\|_{\dot{K}_{p,r}^{\alpha
,q}}\lesssim \big\|f\big\|_{\dot{K}_{p,r}^{\alpha ,q}A_{\beta }^{s}},
\end{equation*}%
since $s>\frac{n}{p}-n$. The desired embedding follows by the fact that 
\begin{equation*}
\dot{K}_{1}^{\alpha ,\max (1,q)}\hookrightarrow L_{\mathrm{loc}}^{1}(\mathbb{%
R}^{n})
\end{equation*}%
and the arguments in Substep 1.1. In addition, we obtain 
\begin{equation}
\dot{K}_{p,r}^{\alpha ,q}A_{\beta }^{s}\hookrightarrow \dot{K}_{1}^{\alpha
,\max (1,q)}.  \label{q-less1}
\end{equation}

\textit{Substep 2.2.} $\alpha \geq 0$. Let $\alpha _{4}$ be a real number
such that 
\begin{equation*}
\max \big(-n,-s+\frac{n}{p}-n+\alpha \big)<\alpha _{4}<0.
\end{equation*}%
From Theorem \ref{embeddings3-lorentz} , we get 
\begin{align*}
\dot{K}_{p,r}^{\alpha ,q}A_{\beta }^{s}& \hookrightarrow \dot{K}%
_{1}^{0,q}A_{\beta }^{s-\frac{n}{p}+n-\alpha } \\
& \hookrightarrow \dot{K}_{1}^{\alpha _{4},q}A_{\beta }^{s-\frac{n}{p}%
+n-\alpha +\alpha _{4}} \\
& \hookrightarrow \dot{K}_{1}^{\alpha _{4},\max (1,q)}A_{\beta }^{s-\frac{n}{%
p}+n-\alpha +\alpha _{4}}.
\end{align*}%
As in Substep 1.4, we easily obtain that 
\begin{equation*}
\dot{K}_{p,r}^{\alpha ,q}A_{\beta }^{s}\hookrightarrow L_{\mathrm{loc}}^{1}(%
\mathbb{R}^{n}).
\end{equation*}%
Therefore, under the hypothesis of this theorem, every $f\in \dot{K}%
_{p,r}^{\alpha ,q}A_{\beta }^{s}$ is a regular distribution. This finishes
the proof.
\end{proof}

\begin{remark}
In \cite[Theorem 2.4]{Dr-EMJ}, we have used the assumption $\alpha >-\frac{n%
}{q}$ but the correct is $\alpha >\max (-n,-\frac{n}{q})$.
\end{remark}

Using the same schema as in \cite{DrPolonais} with the help of Theorem \ref%
{regular-distribution1-lorentz} and the dilation identity %
\eqref{dilation-lorentz}, we obtain the following statement.

\begin{theorem}
Let\ $0<p<\infty ,0<r,\beta ,q\leq \infty ,\alpha >\max (-n,-\frac{n}{p})$\
and $s>\max (\sigma _{p},\alpha -n+\frac{n}{p})$. Then there exists a
positive constant $c$ independent of $\lambda $ such that 
\begin{equation*}
\big\|f(\lambda \cdot )\big\|_{\dot{K}_{p,r}^{\alpha ,q}A_{\beta }^{s}}\leq c%
\text{ }\lambda ^{s-\frac{n}{p}-\alpha }\big\|f\big\|_{\dot{K}_{p,r}^{\alpha
,q}A_{\beta }^{s}}
\end{equation*}%
holds for all $\lambda $ with $1\leq \lambda <\infty $\ and all $f\in \dot{K}%
_{p,r}^{\alpha ,q}A_{\beta }^{s}$.
\end{theorem}

Let $f$ be an arbitrary function on $\mathbb{R}^{n}$ and $x,h\in \mathbb{R}%
^{n}$. Then%
\begin{equation*}
\Delta _{h}f(x)=f(x+h)-f(x),\quad \Delta _{h}^{M+1}f(x)=\Delta _{h}(\Delta
_{h}^{M}f)(x),\quad M\in \mathbb{N}.
\end{equation*}%
These are the well-known differences of functions which play an important
role in the theory of function spaces. Using mathematical induction one can
show the explicit formula%
\begin{equation*}
\Delta _{h}^{M}f(x)=\sum_{j=0}^{M}\left( -1\right)
^{j}C_{j}^{M}f(x+(M-j)h),\quad x\in \mathbb{R}^{n},
\end{equation*}%
where $C_{j}^{M}$ are the binomial coefficients. By ball means of
differences we mean the quantity%
\begin{equation*}
d_{t}^{M}f(x)=t^{-n}\int_{|h|\leq t}\left\vert \Delta
_{h}^{M}f(x)\right\vert dh=\int_{B}\left\vert \Delta
_{th}^{M}f(x)\right\vert dh,\quad x\in \mathbb{R}^{n}.
\end{equation*}%
Here $B=\{y\in \mathbb{R}^{n}:|h|\leq 1\}$ is the unit ball of $\mathbb{R}%
^{n}$ and $t>0$ is a real number. We set%
\begin{equation*}
\big\|f\big\|_{\dot{K}_{p,r}^{\alpha ,q}B_{\beta }^{s}}^{\ast }=\big\|f\big\|%
_{\dot{K}_{p,r}^{\alpha ,q}}+\Big(\int_{0}^{\infty }t^{-s\beta }\big\|%
d_{t}^{M}f\big\|_{\dot{K}_{p,r}^{\alpha ,q}}^{\beta }\frac{dt}{t}\Big)%
^{1/\beta }
\end{equation*}%
and%
\begin{equation*}
\big\|f\big\|_{\dot{K}_{p,r}^{\alpha ,q}F_{\beta }^{s}}^{\ast }=\big\|f\big\|%
_{\dot{K}_{p,r}^{\alpha ,q}}+\Big\|\Big(\int_{0}^{\infty }t^{-s\beta
}(d_{t}^{M}f)^{\beta }\frac{dt}{t}\Big)^{1/\beta }\Big\|_{\dot{K}%
_{p,r}^{\alpha ,q}}.
\end{equation*}

\begin{theorem}
\label{means-diff-cha-lorentz}\textit{Let }$0<p<\infty ,0<r,q,\beta \leq
\infty ,\alpha >\max (-n,-\frac{n}{p}),\alpha _{0}=n-\frac{n}{p}\ $and $M\in 
\mathbb{N}\backslash \{0\}.\newline
\mathrm{(i)}$ Assume that 
\begin{equation*}
\max (\sigma _{p},\alpha -\alpha _{0})<s<M.
\end{equation*}%
Then $\big\|\cdot \big\|_{\dot{K}_{p,r}^{\alpha ,q}B_{\beta }^{s}}^{\ast }$
is an equivalent quasi-norm on $\dot{K}_{p,r}^{\alpha ,q}B_{\beta }^{s}$.$%
\newline
\mathrm{(ii)}$ Let\ $0<q<\infty $. Assume that 
\begin{equation*}
\max (\sigma _{p,\beta },\alpha -\alpha _{0})<s<M.
\end{equation*}%
Then $\big\|\cdot \big\|_{\dot{K}_{p,r}^{\alpha ,q}F_{\beta }^{s}}^{\ast }$
is an equivalent quasi-norm on $\dot{K}_{p,r}^{\alpha ,q}F_{\beta }^{s}$.
\end{theorem}

\begin{proof}
Let $\{\varphi _{j}\}_{j\in \mathbb{N}_{0}}$ be a smooth dyadic resolution
of unity. For ease of presentation, we split the proof into three steps.

\textit{Step 1}. We will prove that 
\begin{equation*}
\big\|f\big\|_{\dot{K}_{p,r}^{\alpha ,q}}\lesssim \big\|f\big\|_{\dot{K}%
_{p,r}^{\alpha ,q}A_{\beta }^{s}}
\end{equation*}%
for all $f\in \dot{K}_{p,r}^{\alpha ,q}A_{\beta }^{s}$. We employ the same
notations as in Theorem \ref{regular-distribution1-lorentz}. Recall that 
\begin{equation*}
\varrho _{k}=\sum\limits_{j=0}^{k}\mathcal{F}^{-1}\varphi _{j}\ast f,\quad
k\in \mathbb{N}_{0}.
\end{equation*}%
Obviously, $\{\varrho _{k}\}_{k\in \mathbb{N}_{0}}$ converges to $f$ in $%
\mathcal{S}^{\prime }(\mathbb{R}^{n})$ and $\{\varrho _{k}\}_{k\in \mathbb{N}%
_{0}}\subset \dot{K}_{p,r}^{\alpha ,q}$ for any $0<p<\infty ,0<q\leq \infty
\ $and any $\alpha >-\frac{n}{p}$. Furthermore, $\{\varrho _{k}\}_{k\in 
\mathbb{N}_{0}}$ is a Cauchy sequence in $\dot{K}_{p,r}^{\alpha ,q}$ and
hence it converges to a function $g\in \dot{K}_{p,r}^{\alpha ,q}$, and 
\begin{equation*}
\big\|g\big\|_{\dot{K}_{p,r}^{\alpha ,q}}\lesssim \big\|f\big\|_{\dot{K}%
_{p,r}^{\alpha ,q}A_{\beta }^{s}}.
\end{equation*}%
Let us prove that $g=f$ a.e. We will do this into four cases.

\textit{Case 1.} $-\frac{n}{p}<\alpha <\alpha _{0}$ and $1\leq p<\infty $.
First assume that $1<p<\infty $. Let $1<p_{0}<\infty $ be as in Theorem \ref%
{regular-distribution1-lorentz}. Let $\varphi \in \mathcal{D}(\mathbb{R}%
^{n}) $. We write 
\begin{equation*}
\langle f-g,\varphi \rangle =\langle f-\varrho _{N},\varphi \rangle +\langle
g-\varrho _{N},\varphi \rangle ,\quad N\in \mathbb{N}_{0}.
\end{equation*}%
Here $\langle \cdot ,\cdot \rangle $ denotes the duality bracket between $%
\mathcal{S}^{\prime }(\mathbb{R}^{n})$ and $\mathcal{S}(\mathbb{R}^{n})$. By
H\"{o}lder's\ inequality there exists a constant $C>0$ independent of $N$
such that 
\begin{equation*}
|\langle g-\varrho _{N},\varphi \rangle |\leq C\big\|g-\varrho _{N}\big\|_{%
\dot{K}_{p_{0}}^{\alpha ,\max (1,q)}},
\end{equation*}%
which tends to zero as $N\rightarrow \infty $. Let $1<d<\infty $ be as in
Theorem \ref{regular-distribution1-lorentz}. We obtain%
\begin{equation*}
|\langle g-\varrho _{N},\varphi \rangle |\leq C\big\|g-\varrho _{N}\big\|_{%
\dot{K}_{d}^{\alpha ,\max (1,q)}},\quad N\in \mathbb{N}.
\end{equation*}%
Observe that%
\begin{equation*}
\dot{K}_{d}^{\alpha ,q}B_{\infty }^{s+\frac{n}{d}-n}\hookrightarrow \dot{K}%
_{d}^{\alpha ,\max (1,q)}.
\end{equation*}%
Then, with the help of Substep 1.1 of the proof of Theorem \ref%
{regular-distribution1-lorentz}, we have $g=f$ almost everywhere.

\textit{Case 2.} $\alpha \geq \alpha _{0}$ and $1<p<\infty $. Let $%
1<p_{1}<\infty $ be as in Theorem \ref{regular-distribution1-lorentz}. From %
\eqref{Substep1.2.1} and \eqref{Substep1.2.2}, we derive in this case, that
every $f\in \dot{K}_{p,r}^{\alpha ,q}A_{\beta }^{s}$ is a regular
distribution, $\{\varrho _{k}\}_{k\in \mathbb{N}_{0}}$ converges to $f$ in $%
L^{p_{1}}$ and 
\begin{equation*}
\big\|f\big\|_{p_{1}}\lesssim \big\|f\big\|_{\dot{K}_{p,r}^{\alpha
,q}A_{\beta }^{s}}.
\end{equation*}%
Indeed, from the embeddings \eqref{Substep1.2.2} and since $f\in
B_{p_{1},\beta }^{s-\alpha \frac{n}{p_{1}}-\frac{n}{p}}$, it follows that $%
\{\varrho _{k}\}_{k\in \mathbb{N}_{0}}$ converges to a function $h$ $\in $ $%
L^{p_{1}}$. Similarly as in Case 1, we conclude that $f=h$ a.e. It remains
to prove that $g=f$ a.e. We have 
\begin{equation*}
\big\|f-g\big\|_{\dot{K}_{p,r}^{\alpha ,q}}\lesssim \big\|f-\varrho _{k}%
\big\|_{\dot{K}_{p,r}^{\alpha ,q}}+\big\|g-\varrho _{k}\big\|_{\dot{K}%
_{p,r}^{\alpha ,q}},\quad k\in \mathbb{N}_{0}
\end{equation*}%
and 
\begin{equation*}
\big\|f-\varrho _{k}\big\|_{\dot{K}_{p,r}^{\alpha ,q}}^{\sigma }\leq
\sum\limits_{j=k+1}^{\infty }\big\|\mathcal{F}^{-1}\varphi _{j}\ast f\big\|_{%
\dot{K}_{p,r}^{\alpha ,q}}^{\sigma }\leq \big\|f\big\|_{\dot{K}%
_{p,r}^{\alpha ,q}A_{\beta }^{s}}^{\sigma }\sum\limits_{j=k+1}^{\infty
}2^{-js\sigma },
\end{equation*}%
where $\sigma <\min (1,p,q,r)$. Letting $k$ tends to infinity, we get $g=f$
a.e.

\textit{Case 3.} $p=1$ and $\alpha \geq 0.$

\textit{Subcase 3.1. }$p=1$ and $\alpha >0$. We have 
\begin{equation*}
\dot{K}_{1,r}^{\alpha ,q}B_{\beta }^{s}\hookrightarrow L^{1},
\end{equation*}%
since $s>\alpha $, see Theorem \ref{regular-distribution1-lorentz}, Substep
1.3. Now one can continue as in Case 2.

\textit{Subcase 3.2.} $p=1$ and $\alpha =0$. Let $\alpha _{3}$ be a real
number such that $\max (-n,-s)<\alpha _{3}<0$. From Theorems \ref%
{embeddings3-lorentz} and \ref{embeddings3 copy(4)-lorentz}, we get 
\begin{equation*}
\dot{K}_{1,r}^{0,q}A_{\beta }^{s}\hookrightarrow \dot{K}_{1}^{\alpha
_{3},q}A_{\beta }^{s+\alpha _{3}}.
\end{equation*}%
We have 
\begin{equation*}
\sum_{k=0}^{\infty }\big\|\mathcal{F}^{-1}\varphi _{k}\ast f\big\|_{\dot{K}%
_{1}^{\alpha _{3},\max (1,q)}}\lesssim \big\|f\big\|_{\dot{K}_{1,r}^{\alpha
_{3},q}A_{\beta }^{s+\alpha _{3}}}\lesssim \big\|f\big\|_{\dot{K}%
_{1,r}^{0,q}A_{\beta }^{s}},
\end{equation*}%
since $\alpha _{3}+s>0$. Hence the sequence $\{\varrho _{k}\}_{k\in \mathbb{N%
}_{0}}$ converges to $f$ in $\dot{K}_{1}^{\alpha _{3},\max (1,q)}$, see Case
1. As in Case 2, we obtain $g=f$ a.e.

\textit{Case 4.} $0<p<1$.

\textit{Subcase 4.1. }$-n<\alpha <0$. From the embedding \eqref{q-less1} and
the fact that $s>\frac{n}{p}-n$, the sequence $\{\varrho _{k}\}_{k\in 
\mathbb{N}_{0}}$ converge to $f$ in $\dot{K}_{1}^{\alpha ,\max (1,q)}$. As
above we prove that $g=f$ a.e.

\textit{Subcase 4.2. }$\alpha \geq 0$. Recall that 
\begin{equation*}
\dot{K}_{p,r}^{\alpha ,q}A_{\beta }^{s}\hookrightarrow \dot{K}_{1}^{\alpha
_{4},\max (1,q)}A_{\beta }^{s-\frac{n}{p}+n-\alpha +\alpha _{4}},
\end{equation*}%
see Substep 2.2 of the proof of Theorem \ref{regular-distribution1-lorentz}.
As in Subcase 3.2 the sequence $\{\varrho _{k}\}_{k\in \mathbb{N}_{0}}$
converges to $f$ in $\dot{K}_{1}^{\alpha _{4},\max (1,q)}$. The same
arguments above one can conclude that: $g=f$ a.e..

\textit{Step 2.} In this step we prove that 
\begin{equation*}
\big\|f\big\|_{\dot{K}_{p,r}^{\alpha ,q}F_{\beta }^{s}}^{\ast }=\Big\|\Big(%
\int_{0}^{\infty }t^{-s\beta }(d_{t}^{M}f)^{\beta }\frac{dt}{t}\Big)%
^{1/\beta }\Big\|_{\dot{K}_{p,r}^{\alpha ,q}}\lesssim \big\|f\big\|_{\dot{K}%
_{p,r}^{\alpha ,q}F_{\beta }^{s}},\quad f\in \dot{K}_{p,r}^{\alpha
,q}F_{\beta }^{s}.
\end{equation*}%
Thus, we need to prove that 
\begin{equation*}
\Big\|\Big(\sum_{k=-\infty }^{\infty }2^{sk\beta }|d_{2^{-k}}^{M}f|^{\beta }%
\Big)^{1/\beta }\Big\|_{\dot{K}_{p,r}^{\alpha ,q}}
\end{equation*}%
does not exceed $c\big\|f\big\|_{\dot{K}_{p,r}^{\alpha ,q}F_{\beta }^{s}}$.
The proof is a slight variant of \cite{Dr-EMJ}. For the convenience of the
reader, we give some details. In order to prove we additionally do it into
the two Substeps 2.1 and 2.2. The estimate for the space $\dot{K}%
_{p,r}^{\alpha ,q}B_{\beta }^{s}$ is similar.

\textit{Substep 2.1.} We will estimate 
\begin{equation*}
\Big\|\Big(\sum_{k=0}^{\infty }2^{sk\beta }|d_{2^{-k}}^{M}f|^{\beta }\Big)%
^{1/\beta }\Big\|_{\dot{K}_{p,r}^{\alpha ,q}}.
\end{equation*}%
Obviously, we need to estimate 
\begin{equation}
\Big\{2^{ks}\sum\limits_{j=0}^{k}d_{2^{-k}}^{M}(\mathcal{F}^{-1}\varphi
_{j}\ast f)\Big\}_{k\in \mathbb{N}_{0}}  \label{First-term}
\end{equation}%
and 
\begin{equation}
\Big\{2^{ks}\sum\limits_{j=k+1}^{\infty }d_{2^{-k}}^{M}(\mathcal{F}%
^{-1}\varphi _{j}\ast f)\Big\}_{k\in \mathbb{N}_{0}}.  \label{second-term}
\end{equation}%
As\ in \cite{T83}, we arrive at the estimate 
\begin{equation}
d_{2^{-k}}^{M}(\mathcal{F}^{-1}\varphi _{j}\ast f)\lesssim 2^{\left(
j-k\right) M}\varphi _{j}^{\ast ,a}f\left( x\right)  \label{third-term}
\end{equation}%
if $a>0$, $0\leq j\leq k,k\in \mathbb{N}_{0}$ and $x\in \mathbb{R}^{n}$,
where the implicit constant is independent of $j,k$ and $x$. We choose $a>%
\frac{n}{\min (\min (p,\beta ),\frac{n}{\frac{n}{p}+\alpha })}$. Since $s<M$%
, \eqref{First-term} in $\ell ^{\beta }$-quasi-norm does not exceed 
\begin{equation}
\Big(\sum\limits_{j=0}^{\infty }2^{js\beta }(\varphi _{j}^{\ast ,a}f)^{\beta
}\Big)^{1/\beta }.  \label{First-term1}
\end{equation}%
By Theorem \ref{fun-char-lorentz}, the $\dot{K}_{p,r}^{\alpha ,q}$
-quasi-norm of \eqref{First-term1} is bounded by $c\big\|f\big\|_{\dot{K}%
_{p,r}^{\alpha ,q}F_{\beta }^{s}}$. Now, we estimate \eqref{second-term}. We
can distinguish two cases as follows:

\textit{Case 1. }$\min (p,\beta )\leq 1$. If $-\frac{n}{p}<\alpha <n(1-\frac{%
1}{p})$, then $s>\frac{n}{\min (p,\beta )}-n$. We choose 
\begin{equation}
\max \Big(0,1-\frac{s\min (p,\beta )}{n}\Big)<\lambda <\min (p,\beta ),
\label{lamda5}
\end{equation}%
which is possible because of 
\begin{equation*}
s>\frac{n}{\min (p,\beta )}-n=\frac{n}{\min (p,\beta )}\Big(1-\min (p,\beta )%
\Big).
\end{equation*}%
Let $\frac{n}{\min (p,\beta )}<a<\frac{s}{1-\lambda }$. Then $s>a(1-\lambda
) $. Now, assume that $\alpha \geq n(1-\frac{1}{p})$. Therefore 
\begin{equation*}
s>\max \Big(\frac{n}{\min (p,\beta )}-n,\frac{n}{p}+\alpha -n\Big).
\end{equation*}%
If $\min (p,\beta )\leq \frac{n}{\frac{n}{p}+\alpha }$, then we choose $%
\lambda $ as in \eqref{lamda5}. If $\min (p,\beta )>\frac{n}{\frac{n}{p}%
+\alpha }$, then we choose 
\begin{equation}
\max \Big(0,1-\frac{s}{\frac{n}{p}+\alpha }\Big)<\lambda <\frac{n}{\frac{n}{p%
}+\alpha }  \label{lamda6}
\end{equation}%
be a strict positive real number, which is possible because of 
\begin{equation*}
s>\frac{n}{p}+\alpha -n=\big(\frac{n}{p}+\alpha \big)\big(1-\frac{n}{\frac{n%
}{p}+\alpha }\big).
\end{equation*}%
In that case, we choose $\frac{n}{p}+\alpha <a<\frac{s}{1-\lambda }$. We set 
\begin{equation*}
J_{2,k}(f)=2^{ks}\sum\limits_{j=k+1}^{\infty }d_{2^{-k}}^{M}(\mathcal{F}%
^{-1}\varphi _{j}\ast f),\quad k\in \mathbb{N}_{0}.
\end{equation*}%
Recalling the definition of $d_{2^{-k}}^{M}(\varphi _{j}\ast f)$, we have 
\begin{align}
d_{2^{-k}}^{M}(\mathcal{F}^{-1}\varphi _{j}\ast f)& =\int_{B}\big|\Delta
_{2^{-k}h}^{M}(\mathcal{F}^{-1}\varphi _{j}\ast f)\big|dh  \notag \\
& \leq \int_{B}\big|\Delta _{2^{-k}h}^{M}(\mathcal{F}^{-1}\varphi _{j}\ast f)%
\big|^{\lambda }dh\sup_{h\in B}\big|\Delta _{2^{-k}h}^{M}(\mathcal{F}%
^{-1}\varphi _{j}\ast f)\big|^{1-\lambda }.  \label{term2}
\end{align}%
Observe that 
\begin{equation}
\big|\mathcal{F}^{-1}\varphi _{j}\ast f(x+(M-i)2^{-k}h)\big|\leq c2^{\left(
j-k\right) a}\varphi _{j}^{\ast ,a}f\left( x\right) ,\quad |h|\leq 1
\label{term3}
\end{equation}%
and 
\begin{equation}
\int_{B}\big|\mathcal{F}^{-1}\varphi _{j}\ast f(x+(M-i)2^{-k}h)\big|%
^{\lambda }dh\leq c\mathcal{M}(|\mathcal{F}^{-1}\varphi _{j}\ast f|^{\lambda
})(x).  \label{term4}
\end{equation}%
if $j>k,i\in \{0,...,M\}$ and $x\in \mathbb{R}^{n}$. Therefore 
\begin{equation*}
d_{2^{-k}}^{M}(\mathcal{F}^{-1}\varphi _{j}\ast f)\leq c2^{\left( j-k\right)
a(1-\lambda )}(\varphi _{j}^{\ast ,a}f)^{1-\lambda }\mathcal{M}(|\mathcal{F}%
^{-1}\varphi _{j}\ast f|^{\lambda })
\end{equation*}%
for any $j>k$, where the positive constant $c$ is independent of $j$ and $k$%
. Hence 
\begin{equation*}
J_{2,k}(f)\leq c2^{ks}\sum\limits_{j=k+1}^{\infty }2^{\left( j-k\right)
a(1-\lambda )}(\varphi _{j}^{\ast ,a}f)^{1-\lambda }\mathcal{M}(|\mathcal{F}%
^{-1}\varphi _{j}\ast f|^{\lambda }).
\end{equation*}%
Using Lemma \ref{lem:lq-inequality}, we obtain that \eqref{second-term} in $%
\ell ^{\beta }$-quasi-norm can be estimated from above by 
\begin{align*}
& c\Big(\sum\limits_{j=0}^{\infty }2^{js\beta }(\varphi _{j}^{\ast
,a}f)^{(1-\lambda )\beta }(\mathcal{M}(|\mathcal{F}^{-1}\varphi _{j}\ast
f|^{\lambda }))^{\beta }\Big)^{1/\beta } \\
& \lesssim \Big(\sum\limits_{j=0}^{\infty }2^{js\beta }(\varphi _{j}^{\ast
,a}f)^{\beta }\Big)^{(1-\lambda )/\beta }\Big(\sum\limits_{j=0}^{\infty
}2^{js\beta }(\mathcal{M}(|\mathcal{F}^{-1}\varphi _{j}\ast f|^{\lambda
}))^{\beta /\lambda }\Big)^{\lambda /\beta }.
\end{align*}%
Applying the $\dot{K}_{p,r}^{\alpha ,q}$-quasi-norm and using H\"{o}lder's
inequality we obtain that 
\begin{equation*}
\Big\|\Big(\sum\limits_{j=0}^{\infty }(J_{2,k}(f))^{\beta }\Big)^{1/\beta }%
\Big\|_{\dot{K}_{p,r}^{\alpha ,q}}
\end{equation*}%
is bounded by 
\begin{align*}
& c\Big\|\Big(\sum\limits_{j=0}^{\infty }2^{js\beta }(\varphi _{j}^{\ast
,a}f)^{\beta }\Big)^{(1-\lambda )/\beta }\Big\|_{\dot{K}_{\frac{p}{1-\lambda 
},\frac{r}{1-\lambda }}^{\alpha (1-\lambda ),\frac{q}{1-\lambda }}}\times \\
& \Big\|\Big(\sum\limits_{j=0}^{\infty }2^{js\beta }\big(\mathcal{M}(|%
\mathcal{F}^{-1}\varphi _{j}\ast f|^{\lambda })\big)^{\beta /\lambda }\Big)%
^{\lambda /\beta }\Big\|_{\dot{K}_{\frac{p}{\lambda },\frac{r}{\lambda }%
}^{\alpha \lambda ,\frac{q}{\lambda }}} \\
& \lesssim \Big\|\Big(\sum\limits_{j=0}^{\infty }2^{js\beta }(\varphi
_{j}^{\ast ,a}f)^{\beta }\Big)^{1/\beta }\Big\|_{\dot{K}_{p,r}^{\alpha
,q}}^{1-\lambda }\Big\|\Big(\sum\limits_{j=0}^{\infty }2^{js\beta }|\mathcal{%
F}^{-1}\varphi _{j}\ast f|^{\beta }\Big)^{1/\beta }\Big\|_{\dot{K}%
_{p,r}^{\alpha ,q}}^{\lambda } \\
& \lesssim \big\|f\big\|_{\dot{K}_{p,r}^{\alpha ,q}F_{\beta }^{s}},
\end{align*}%
where we have used Lemma \ref{Maximal-Inq copy(2)-lorentz} and Theorem \ref%
{fun-char-lorentz}.

\textit{Case 2. }$\min (p,\beta )>1$. Assume that $\alpha \geq n(1-\frac{1}{p%
})$. Then we choose $\lambda $ as in \eqref{lamda6}\ and $\frac{n}{p}+\alpha
<a<\frac{s}{1-\lambda }$. If $-\frac{n}{p}<\alpha <n(1-\frac{1}{p})$, then
we choose $\lambda =1$. The desired estimate can be done in the same manner
as in Case 1.

\textit{Substep 2.2.} We will estimate 
\begin{equation*}
\Big\|\Big(\sum_{k=-\infty }^{-1}2^{sk\beta }|d_{2^{-k}}^{M}f|^{\beta }\Big)%
^{1/\beta }\Big\|_{\dot{K}_{p,r}^{\alpha ,q}}.
\end{equation*}%
We employ the same notations as in Substep 1.1. Define 
\begin{equation*}
H_{2,k}(f)(x)=\int_{B}\big|\sum_{j=0}^{\infty }\Delta _{z2^{-k}}^{M}(%
\mathcal{F}^{-1}\varphi _{j}\ast f)(x)\big|dz,\quad k\leq 0,x\in \mathbb{R}%
^{n}.
\end{equation*}%
As in the estimation of $J_{2,k}$, we obtain that 
\begin{equation*}
H_{2,k}(f)\lesssim 2^{k(s-a(1-\lambda ))}\sup_{j\in \mathbb{N}_{0}}\Big(\big(%
2^{js}(\varphi _{j}^{\ast ,a}f\big)^{1-\lambda }\mathcal{M}\big(2^{js}|%
\mathcal{F}^{-1}\varphi _{j}\ast f|\big)^{\lambda }\Big)
\end{equation*}%
and this yields that 
\begin{equation*}
\Big(\sum_{k=-\infty }^{-1}2^{sk\beta }|H_{2,k}|^{\beta }\Big)^{1/\beta
}\lesssim \sup_{j\in \mathbb{N}_{0}}\Big(\big(2^{js}(\varphi _{j}^{\ast ,a}f%
\big)^{1-\lambda }\mathcal{M}\big(2^{js}|\varphi _{j}\ast f|\big)^{\lambda }%
\Big).
\end{equation*}%
By the same arguments as used in Substep 2.1 we obtain the desired estimate.

\textit{Step 3.} Let $f\in \dot{K}_{p,r}^{\alpha ,q}A_{\beta }^{s}$. We will
prove that 
\begin{equation*}
\big\|f\big\|_{\dot{K}_{p,r}^{\alpha ,q}A_{\beta }^{s}}\lesssim \big\|f\big\|%
_{\dot{K}_{p,r}^{\alpha ,q}A_{\beta }^{s}}^{\ast }.
\end{equation*}%
As the proof for $\dot{K}_{p,r}^{\alpha ,q}B_{\beta }^{s}$ is similar, we
only consider $\dot{K}_{p,r}^{\alpha ,q}F_{\beta }^{s}$. The proof is very
similar as in \cite{Dr-EMJ}. We present some details, because we need it in
the next theorem. Let $\Psi $ be a function in $\mathcal{S}(\mathbb{R}^{n})$
satisfying $\Psi (x)=1$ for $\lvert x\rvert \leq 1$ and $\Psi (x)=0$ for $%
\lvert x\rvert \geq \frac{3}{2}$, and in addition radialsymmetric. We make
use of an observation made by Nikol'skij \cite{Nikolskii1975}, see also \cite%
[Section 3.3.2]{T83}. We put 
\begin{equation*}
\psi (x)=(-1)^{M+1}\sum\limits_{i=0}^{M-1}(-1)^{i}C_{i}^{M}\Psi (x\left(
M-i\right) ),
\end{equation*}%
where $C_{i}^{M}$, $i\in \{0,...,M-1\}$ are the binomial coefficients. The
function $\psi $ satisfies $\psi \left( x\right) =1$ for $\left\vert
x\right\vert \leq \frac{1}{M}$ and $\psi \left( x\right) =0$ for $\left\vert
x\right\vert \geq \frac{3}{2}$. Then, taking $\varphi _{0}(x)=\psi
(x),\varphi _{1}(x)=\psi (\frac{x}{2})-\psi (x)$ and $\varphi
_{j}(x)=\varphi _{1}(2^{-j+1}x)$ for $j=2,3,...$, we obtain that $\{\varphi
_{j}\}_{j\in \mathbb{N}_{0}}$\ is a smooth dyadic resolution of unity. This
yields that 
\begin{equation*}
\Big\|\Big(\sum_{j=0}^{\infty }2^{js\beta }|\mathcal{F}^{-1}\varphi _{j}\ast
f|^{\beta }\Big)^{1/\beta }\Big\|_{\dot{K}_{p,r}^{\alpha ,q}}
\end{equation*}%
is a quasi-norm equivalent in $\dot{K}_{p,r}^{\alpha ,q}{F_{\beta }^{s}}$.
Let us prove that the last expression is bounded by 
\begin{equation}
C\big\Vert f\big\Vert_{\dot{K}_{p,r}^{\alpha ,q}{F_{\beta }^{s}}}^{\ast }.
\label{second-est}
\end{equation}%
We observe that 
\begin{equation*}
\mathcal{F}^{-1}\varphi _{0}\ast f(x)=\left( -1\right) ^{M+1}\int_{\mathbb{R}%
^{n}}\mathcal{F}^{-1}\Psi \left( z\right) \Delta _{-z}^{M}f(x)dz+f(x)\int_{%
\mathbb{R}^{n}}\mathcal{F}^{-1}\Psi \left( z\right) dz
\end{equation*}%
Moreover, it holds for $x\in \mathbb{R}^{n}$ and $j=1,2,...$ 
\begin{equation*}
\mathcal{F}^{-1}\varphi _{j}\ast f\left( x\right) =\left( -1\right)
^{M+1}\int_{\mathbb{R}^{n}}\Delta _{2^{-j}y}^{M}f\left( x\right) \widetilde{%
\Psi }\left( y\right) dy,
\end{equation*}%
with $\widetilde{\Psi }=\mathcal{F}^{-1}\Psi -2^{-n}\mathcal{F}^{-1}\Psi
(\cdot /2)$. Now, for $j\in \mathbb{N}_{0}$ we have 
\begin{align}
& \int_{\mathbb{R}^{n}}|\Delta _{2^{-j}y}^{M}f(x)||\widetilde{\Psi }(y)|dy 
\notag \\
& =\int_{\left\vert y\right\vert \leq 1}|\Delta _{2^{-j}y}^{M}f(x)||%
\widetilde{\Psi }(y)|dy+\int_{\left\vert y\right\vert >1}|\Delta
_{2^{-j}y}^{M}f(x)||\widetilde{\Psi }(y)|dy.  \label{diff2.1}
\end{align}%
Thus, we need only to estimate the second term of \eqref{diff2.1}. We write 
\begin{align}
& 2^{sj}\int_{\left\vert y\right\vert >1}|\Delta _{2^{-j}y}^{M}f(x)||%
\widetilde{\Psi }(y)|dy  \notag \\
& =2^{sj}\sum\limits_{k=0}^{\infty }\int_{2^{k}<\left\vert y\right\vert \leq
2^{k+1}}|\Delta _{2^{-j}y}^{M}f(x)||\widetilde{\Psi }(y)|dy  \notag \\
& \leq c2^{sj}\sum\limits_{k=0}^{\infty }2^{nj-Nk}\int_{2^{k-j}<\left\vert
h\right\vert \leq 2^{k-j+1}}|\Delta _{h}^{M}f(x)|dh  \label{diff}
\end{align}%
where $N>0$ is at our disposal and we have used the properties of the
function $\widetilde{\Psi }$, $|\widetilde{\Psi }(x)|\leq c(1+\left\vert
x\right\vert )^{-N},$ for any $x\in \mathbb{R}^{n}$ and any $N>0$. Without
lost of generality, we may assume $1\leq \beta \leq \infty $. Now, the
right-hand side of \eqref{diff} in $\ell ^{\beta }$-norm is bounded by 
\begin{equation}
c\sum\limits_{k=0}^{\infty }2^{-Nk}\Big(\sum\limits_{j=0}^{\infty
}2^{(s+n)j\beta }\Big(\int_{\left\vert h\right\vert \leq 2^{k-j+1}}|\Delta
_{h}^{M}f(x)|dh\Big)^{\beta }\Big)^{1/\beta }.  \label{diff1}
\end{equation}%
After a change of variable $j-k-1=v$, we estimate \eqref{diff1} by 
\begin{equation*}
c\sum\limits_{k=0}^{\infty }2^{(s+n-N)k}\Big(\sum\limits_{v=-k-1}^{\infty
}2^{sv\beta }\big(d_{2^{-v}}^{M}f(x)\big)^{\beta }\Big)^{1/\beta }\lesssim %
\Big(\sum\limits_{v=-{\infty }}^{\infty }2^{sv\beta }\big(d_{2^{-v}}^{M}f(x)%
\big)^{\beta }\Big)^{1/\beta },
\end{equation*}%
where we choose $N>n+s$. Taking the $\dot{K}_{p,r}^{\alpha ,q}$-quasi-norm
we obtain the desired estimate \eqref{second-est}. The proof is complete.
\end{proof}

\begin{remark}
In \cite[Theorem 2.5]{Dr-EMJ}, we have used the assumption $\alpha >-\frac{n%
}{q}$\ but the correct is $\alpha >\max (-n,-\frac{n}{q})$.
\end{remark}

We set%
\begin{equation*}
\big\|f\big\|_{\dot{K}_{p,r}^{\alpha ,q}B_{\beta }^{s}}^{\ast \ast }=\big\|f%
\big\|_{\dot{K}_{p}^{\alpha ,q}}+\Big(\int_{0}^{1}t^{-s\beta }\big\|%
d_{t}^{M}f\big\|_{\dot{K}_{p,r}^{\alpha ,q}}^{\beta }\frac{dt}{t}\Big)%
^{1/\beta }
\end{equation*}%
and%
\begin{equation*}
\big\|f\big\|_{\dot{K}_{p,r}^{\alpha ,q}F_{\beta }^{s}}^{\ast \ast }=\big\|f%
\big\|_{\dot{K}_{p,r}^{\alpha ,q}}+\Big\|\Big(\int_{0}^{1}t^{-s\beta
}(d_{t}^{M}f)^{\beta }\frac{dt}{t}\Big)^{1/\beta }\Big\|_{\dot{K}%
_{p,r}^{\alpha ,q}}.
\end{equation*}

We have also another equivalent quasi-norm on $\dot{K}_{p,r}^{\alpha
,q}A_{\beta }^{s}$.

\begin{theorem}
\label{means-diff-cha copy(1)-lorentz}\textit{Let }$0<p<\infty ,0<r,q,\beta
\leq \infty ,\alpha >\max (-n,-\frac{n}{p}),\alpha _{0}=n-\frac{n}{p}\ $and $%
M\in \mathbb{N}\backslash \{0\}.\newline
\mathrm{(i)}$ Assume that 
\begin{equation*}
\max (\sigma _{p},\alpha -\alpha _{0})<s<M.
\end{equation*}%
Then $\big\|\cdot \big\|_{\dot{K}_{p,r}^{\alpha ,q}B_{\beta }^{s}}^{\ast
\ast }$ is an equivalent quasi-norm on $\dot{K}_{p,r}^{\alpha ,q}B_{\beta
}^{s}$.$\newline
\mathrm{(ii)}$ Let\ $0<q<\infty $. Assume that 
\begin{equation*}
\max (\sigma _{p,\beta },\alpha -\alpha _{0})<s<M.
\end{equation*}%
Then $\big\|\cdot \big\|_{\dot{K}_{p,r}^{\alpha ,q}F_{\beta }^{s}}^{\ast
\ast }$ is an equivalent quasi-norm on $\dot{K}_{p,r}^{\alpha ,q}F_{\beta
}^{s}$.
\end{theorem}

\begin{proof}
We employ the same notations as in Theorem \ref{means-diff-cha-lorentz}. By
similarity, we will consider only the spaces $\dot{K}_{p,r}^{\alpha
,q}F_{\beta }^{s}$. Let $f\in \dot{K}_{p,r}^{\alpha ,q}F_{\beta }^{s}$.
Immediately it follows%
\begin{equation*}
\big\|f\big\|_{\dot{K}_{p,r}^{\alpha ,q}F_{\beta }^{s}}^{\ast \ast }\lesssim %
\big\|f\big\|_{\dot{K}_{p,r}^{\alpha ,q}F_{\beta }^{s}}^{\ast }\lesssim %
\big\|f\big\|_{\dot{K}_{p,r}^{\alpha ,q}F_{\beta }^{s}}.
\end{equation*}%
We will prove that 
\begin{equation*}
\big\|f\big\|_{\dot{K}_{p,r}^{\alpha ,q}F_{\beta }^{s}}\lesssim \big\|f\big\|%
_{\dot{K}_{p,r}^{\alpha ,q}F_{\beta }^{s}}^{\ast \ast }.
\end{equation*}%
In view of Step 3 of the proof of Theorem \ref{means-diff-cha-lorentz}, we
need only to estimate 
\begin{equation*}
V=\sum\limits_{k=0}^{\infty }2^{(s+n-N)k}\Big(\sum\limits_{v=-k-1}^{\infty
}2^{sv\beta }\big(d_{2^{-v}}^{M}f\big)^{\beta }\Big)^{1/\beta }
\end{equation*}%
in $\dot{K}_{p,r}^{\alpha ,q}$-quasi-norm. We see that $V$ can be estimated
from above by $V_{1}+V_{2}$, where 
\begin{equation*}
V_{1}=c\sum\limits_{k=0}^{\infty }2^{(s+n-N)k}\Big(\sum%
\limits_{v=-k-1}^{0}2^{sv\beta }\big(d_{2^{-v}}^{M}f\big)^{\beta }\Big)%
^{1/\beta }
\end{equation*}%
and%
\begin{equation*}
V_{2}=c\sum\limits_{k=0}^{\infty }2^{(s+n-N)k}\Big(\sum\limits_{v=1}^{\infty
}2^{sv\beta }\big(d_{2^{-v}}^{M}f\big)^{\beta }\Big)^{1/\beta }.
\end{equation*}%
We have%
\begin{equation*}
d_{2^{-v}}^{M}(f)\leq \sum_{j=0}^{\infty }d_{2^{-v}}^{M}(\mathcal{F}%
^{-1}\varphi _{j}\ast f),\quad v\geq -k-1.
\end{equation*}%
We choose $N>0$ sufficiently large such that $N>s+n$.

\textit{Estimate of }$V_{1}$. Using \eqref{term2}, \eqref{term3},%
\eqref{term4}, we obtain%
\begin{align*}
2^{sv}d_{2^{-v}}^{M}(f)& \leq c2^{-v(a(1-\lambda
)-s)}\sum\limits_{j=0}^{\infty }2^{j(a(1-\lambda )-s)}(2^{sj}\varphi
_{j}^{\ast ,a}f)^{1-\lambda }2^{sj\lambda }\mathcal{M}(|\mathcal{F}%
^{-1}\varphi _{j}\ast f|^{\lambda }) \\
& \leq c2^{-v(a(1-\lambda )-s)}\sup_{j\in \mathbb{N}_{0}}\Big((2^{sj}\varphi
_{j}^{\ast ,a}f)^{1-\lambda }2^{sj\lambda }\mathcal{M}(|\mathcal{F}%
^{-1}\varphi _{j}\ast f|^{\lambda })\Big),
\end{align*}%
where the positive constant $c$ is independent of $v$ and $k$. Thus, $V_{1}$
can be estimated from above by 
\begin{align}
& c\sup_{j\in \mathbb{N}_{0}}\Big((2^{sj}\varphi _{j}^{\ast ,a}f)^{1-\lambda
}2^{sj\lambda }\mathcal{M}(|\mathcal{F}^{-1}\varphi _{j}\ast f|^{\lambda })%
\Big)  \notag \\
& \times \sum\limits_{k=0}^{\infty }2^{(s+n-N)k}\Big(\sum%
\limits_{v=-k-1}^{0}2^{-\beta v(\frac{a}{\sigma }(1-\lambda )-s)}\Big)%
^{1/\beta }  \notag \\
& \lesssim \sup_{j\in \mathbb{N}_{0}}\Big((2^{sj}\varphi _{j}^{\ast
,a}f)^{1-\lambda }2^{sj\lambda }\mathcal{M}(|\mathcal{F}^{-1}\varphi
_{j}\ast f|^{\lambda })\Big)  \notag \\
& \lesssim \sup_{j\in \mathbb{N}_{0}}\big((2^{sj}\varphi _{j}^{\ast
,a}f)^{1-\lambda }\big)\sup_{j\in \mathbb{N}_{0}}\big(2^{sj\lambda }\mathcal{%
M}(|\mathcal{F}^{-1}\varphi _{j}\ast f|^{\lambda })\big),  \label{V1.1}
\end{align}%
since $N>s+n$. Taking the $\dot{K}_{p,r}^{\alpha ,q}$-quasi-norm in both
sides of \eqref{V1.1} and using H\"{o}lder's inequality, we obtain that $%
\big\|V_{1}\big\|_{\dot{K}_{p,r}^{\alpha ,q}}$ is bounded by 
\begin{align*}
& c\Big\|\big(\sup_{j\in \mathbb{N}_{0}}2^{js}(\varphi _{j}^{\ast ,a}f)\big)%
^{1-\lambda }\Big\|_{\dot{K}_{\frac{p}{1-\lambda },\frac{r}{1-\lambda }%
}^{\alpha (1-\lambda ),\frac{q}{1-\lambda }}}\Big\|\sup_{j\in \mathbb{N}_{0}}%
\big(2^{js\lambda }\mathcal{M}(|\mathcal{F}^{-1}\varphi _{j}\ast f|^{\lambda
})\big)\Big\|_{\dot{K}_{\frac{p}{\lambda },\frac{r}{\lambda }}^{\alpha
\lambda ,\frac{q}{\lambda }}} \\
& \lesssim \Big\|\sup_{j\in \mathbb{N}_{0}}(2^{js}\varphi _{j}^{\ast ,a}f)%
\Big\|_{\dot{K}_{p,r}^{\alpha ,q}}^{1-\lambda }\Big\|\sup_{j\in \mathbb{N}%
_{0}}\big(2^{js}(\mathcal{F}^{-1}\varphi _{j}\ast f)\big)\Big\|_{\dot{K}%
_{p,r}^{\alpha ,q}}^{\lambda } \\
& \lesssim \big\|f\big\|_{\dot{K}_{p,r}^{\alpha ,q}F_{\beta }^{s}},
\end{align*}%
where we have used Lemma \ref{Maximal-Inq copy(2)-lorentz} and Theorem \ref%
{fun-char-lorentz}.

\textit{Estimate of }$V_{2}$. We set%
\begin{equation*}
V_{3,k}=\Big(\sum\limits_{v=1}^{\infty }2^{sv\beta }\Big(%
\sum_{j=0}^{v}d_{2^{-v}}^{M}(\mathcal{F}^{-1}\varphi _{j}\ast f)\Big)^{\beta
}\Big)^{1/\beta }
\end{equation*}%
and%
\begin{equation*}
V_{4,k}=\Big(\sum\limits_{v=1}^{\infty }2^{sv\beta }\Big(\sum_{j=v+1}^{%
\infty }d_{2^{-v}}^{M}(\mathcal{F}^{-1}\varphi _{j}\ast f)\Big)^{\beta }\Big)%
^{1/\beta }.
\end{equation*}%
By \eqref{third-term} and Lemma \ref{lem:lq-inequality}, we get%
\begin{align*}
V_{3,k}& =\Big(\sum\limits_{v=1}^{\infty }\Big(\sum_{j=0}^{v}2^{\left(
j-v\right) (M-s)}2^{sj}\varphi _{j}^{\ast ,a}f\left( x\right) \Big)^{\beta }%
\Big)^{1/\beta } \\
& \lesssim \Big(\sum\limits_{j=0}^{\infty }\big(2^{sj}\varphi _{j}^{\ast
,a}f\left( x\right) \big)^{\beta }\Big)^{1/\beta },
\end{align*}%
where the implicit constant is independent of $k$. Theorem \ref%
{fun-char-lorentz} yields that 
\begin{equation*}
\Big\|\sum\limits_{k=0}^{\infty }2^{(s+n-N)k}V_{3,k}\Big\|_{\dot{K}%
_{p,r}^{\alpha ,q}}\lesssim \big\|f\big\|_{\dot{K}_{p,r}^{\alpha ,q}F_{\beta
}^{s}}.
\end{equation*}%
Now, using \eqref{term2}, \eqref{term3},\eqref{term4} and Lemma \ref%
{lem:lq-inequality}, we obtain%
\begin{align*}
V_{4,k}& \lesssim \Big(\sum\limits_{v=1}^{\infty }\Big(\sum_{j=v+1}^{\infty
}2^{(j-v)(a(1-\lambda )-s)}(2^{sj}\varphi _{j}^{\ast ,a}f)^{1-\lambda
}2^{sj\lambda }\mathcal{M}(|\mathcal{F}^{-1}\varphi _{j}\ast f|^{\lambda })%
\Big)^{\beta }\Big)^{1/\beta } \\
& \lesssim \Big(\sum\limits_{j=0}^{\infty }\Big(\big(2^{sj}\varphi
_{j}^{\ast ,a}f\big)^{1-\lambda }2^{sj\lambda }\mathcal{M}(|\mathcal{F}%
^{-1}\varphi _{j}\ast f|^{\lambda })\Big)^{\beta }\Big)^{1/\beta },
\end{align*}%
where the implicit constant is independent of $k$. The same schema as in the
estimation of $V_{1}$ applies%
\begin{equation*}
\Big\|\sum\limits_{k=0}^{\infty }2^{(s+n-N)k}V_{3,k}\Big\|_{\dot{K}%
_{p,r}^{\alpha ,q}}\lesssim \big\|f\big\|_{\dot{K}_{p,r}^{\alpha ,q}F_{\beta
}^{s}}.
\end{equation*}%
Therefore, 
\begin{equation*}
\big\|V_{2}\big\|_{\dot{K}_{p,r}^{\alpha ,q}}\lesssim \big\|f\big\|_{\dot{K}%
_{p,r}^{\alpha ,q}F_{\beta }^{s}}.
\end{equation*}%
Hence the proof is complete.
\end{proof}

We define the discretized counterpart of $\big\|f\big\|_{\dot{K}%
_{p,r}^{\alpha ,q}F_{\beta }^{s}}^{\ast }$\ and $\big\|f\big\|_{\dot{K}%
_{p,r}^{\alpha ,q}F_{\beta }^{s}}^{\ast \ast }$ by 
\begin{equation*}
\big\|f\big\|_{\dot{K}_{p,r}^{\alpha ,q}F_{\beta }^{s}}^{\ast ,1}=\big\|f%
\big\|_{\dot{K}_{p,r}^{\alpha ,q}}+\Big\|\Big(\sum_{k=-\infty }^{\infty
}2^{sk\beta }|d_{2^{-k}}^{M}f|^{\beta }\Big)^{1/\beta }\Big\|_{\dot{K}%
_{p,r}^{\alpha ,q}}
\end{equation*}%
and 
\begin{equation*}
\big\|f\big\|_{\dot{K}_{p,r}^{\alpha ,q}F_{\beta }^{s}}^{\ast \ast ,1}=\big\|%
f\big\|_{\dot{K}_{p,r}^{\alpha ,q}}+\Big\|\Big(\sum_{k=0}^{\infty
}2^{sk\beta }|d_{2^{-k}}^{M}f|^{\beta }\Big)^{1/\beta }\Big\|_{\dot{K}%
_{p,r}^{\alpha ,q}}.
\end{equation*}%
While for Lorentz-Herz-type Besov spaces, we put 
\begin{equation*}
\big\|f\big\|_{\dot{K}_{p,r}^{\alpha ,q}B_{\beta }^{s}}^{\ast ,1}=\big\|f%
\big\|_{\dot{K}_{p,r}^{\alpha ,q}}+\Big(\sum_{k=-\infty }^{\infty
}2^{sk\beta }\big\|d_{2^{-k}}^{M}f\big\|_{\dot{K}_{p,r}^{\alpha ,q}}^{\beta }%
\Big)^{1/\beta }
\end{equation*}%
and 
\begin{equation*}
\big\|f\big\|_{\dot{K}_{p,r}^{\alpha ,q}B_{\beta }^{s}}^{\ast \ast ,1}=\big\|%
f\big\|_{\dot{K}_{p,r}^{\alpha ,q}}+\Big(\sum_{k=0}^{\infty }2^{sk\beta }%
\big\|d_{2^{-k}}^{M}f\big\|_{\dot{K}_{p,r}^{\alpha ,q}}^{\beta }\Big)%
^{1/\beta }.
\end{equation*}

The Theorems \ref{means-diff-cha-lorentz} and \ref{means-diff-cha
copy(1)-lorentz} give immediately the following equivalent quasi-norms for
the spaces $\dot{K}_{p,r}^{\alpha ,q}A_{\beta }^{s}.$

\begin{corollary}
\label{means-diff-cha copy(2)}\textit{Let }$0<p<\infty ,0<r,q,\beta \leq
\infty ,\alpha >\max (-n,-\frac{n}{p}),\alpha _{0}=n-\frac{n}{p}\ $and $M\in 
\mathbb{N}\backslash \{0\}.\newline
\mathrm{(i)}$ Assume that 
\begin{equation*}
\max (\sigma _{p},\alpha -\alpha _{0})<s<M.
\end{equation*}%
Then $\big\|\cdot \big\|_{\dot{K}_{p,r}^{\alpha ,q}B_{\beta }^{s}}^{\ast ,1}$
and $\big\|\cdot \big\|_{\dot{K}_{p,r}^{\alpha ,q}B_{\beta }^{s}}^{\ast \ast
,1}$ are an equivalent quasi-norm on $\dot{K}_{p,r}^{\alpha ,q}B_{\beta
}^{s} $.$\newline
\mathrm{(ii)}$ Let\ $0<q<\infty $. Assume that 
\begin{equation*}
\max (\sigma _{p,\beta },\alpha -\alpha _{0})<s<M.
\end{equation*}%
Then $\big\|\cdot \big\|_{\dot{K}_{p,r}^{\alpha ,q}F_{\beta }^{s}}^{\ast ,1}$
and $\big\|\cdot \big\|_{\dot{K}_{p,r}^{\alpha ,q}F_{\beta }^{s}}^{\ast \ast
,1}$ are an equivalent quasi-norm on $\dot{K}_{p,r}^{\alpha ,q}F_{\beta
}^{s} $.
\end{corollary}

\subsection{Examples}

We investigate a series of examples which play an important role in the
study of function spaces and composition operators in
Besov-Triebel-Lizorkin-type spaces.

\label{Si-funct1}Let $0<p,q<\infty ,0<r,\beta \leq \infty ,\alpha >\max (-n,-%
\frac{n}{p})$, $\alpha _{0}=n-\frac{n}{p}\ $and $s>\max (\sigma _{p},\alpha
-\alpha _{0})$. We put%
\begin{equation}
f_{\mu ,\delta }(x)=\theta (x)|x|^{\mu }(-\log |x|)^{-\delta },
\label{f-alpha-delta}
\end{equation}%
where $\mu ^{2}+\delta ^{2}>0,\delta \geq 0$, $\mu \neq 0$ and $\theta $ is
a smooth cut-off function with $\mathrm{supp}$ $\theta \subset \{x:|x|\leq
\vartheta \}$, $\vartheta >0$ sufficiently small.$\newline
\mathrm{(i)}$ Let $\delta >0$ and 
\begin{equation}
s<\frac{n}{p}+\alpha +\mu \quad \text{or\quad }s=\frac{n}{p}+\alpha +\mu
\quad \text{and}\quad \beta \delta >1.  \label{si-bou1}
\end{equation}%
Then $f_{\mu ,\delta }\in \dot{K}_{p,r}^{\alpha ,q}B_{\beta }^{s}$. If $\mu
<1$, then \eqref{si-bou1}\ become necessary. $\newline
\mathrm{(ii)}$ We have $f_{\mu ,0}\in \dot{K}_{p,r}^{\alpha ,q}B_{\beta
}^{s} $ if 
\begin{equation}
s<\frac{n}{p}+\alpha +\mu \quad \text{or\quad }s=\frac{n}{p}+\alpha +\mu
\quad \text{and}\quad \beta =\infty .  \label{si-bou2}
\end{equation}%
If $\mu <1$, then \eqref{si-bou2}\ become necessary.

\begin{proof}
We will present the proof in two steps.

\textit{Step 1.} \textit{Proof of sufficiency in }$\mathrm{(i)}$\textit{\
and }$\mathrm{(ii)}$\textit{.} We have to divide this step into two substeps.

\textit{Substep 1.1.} $-\frac{n}{p}<\alpha \leq 0$. First our assumptions
guarantee that\ $\mu >-n$. Let $0<\tau <\min \big(p,q,\frac{n}{\max (-\mu ,0)%
}\big)$. From Theorem \ref{embeddings5-lorentz} we know 
\begin{equation*}
B_{\tau ,\beta }^{\frac{n}{\tau }+\mu }\hookrightarrow \dot{K}_{p,r}^{\alpha
,q}B_{\beta }^{\frac{n}{p}+\alpha +\mu }.
\end{equation*}%
Notice that $\frac{n}{\tau }+\mu >\sigma _{\tau }$ and $f_{\mu ,\delta }\in
B_{\tau ,\beta }^{\frac{n}{\tau }+\mu }$ if $\beta \delta >1$, see \cite[%
2.3.1, p. 44]{RuSi96}. This finishes the proof of this case.

\textit{Substep 1.2.\ }$\alpha >0$. Our estimate use partially some
decomposition techniques already used in \cite[2.3.1, p. 44]{RuSi96}. Let $M$
be a natural number large enough. Let $0<t<\frac{\vartheta }{2M}$ and 
\begin{equation*}
B(t)=\mathbb{R}^{n}\backslash B(0,2Mt).
\end{equation*}%
We will estimate%
\begin{equation}
\sum\limits_{k=-\infty }^{\infty }2^{k\alpha q}\big\|(d_{t}^{M}f_{\mu
,\delta })\chi _{R_{k}}\big\|_{L^{p,r}}^{q}.  \label{main-est}
\end{equation}%
We split the integral $\big\|(d_{t}^{M}f_{\mu ,\delta })\chi _{R_{k}}\big\|%
_{L^{p,r}}^{q}$, $k\in \mathbb{Z}$ into two parts, one integral over the set 
$B(0,2Mt)$ and one over its complement. It holds%
\begin{equation*}
\sum\limits_{k=-\infty }^{\infty }2^{k\alpha q}\big\|(d_{t}^{M}f_{\mu
,\delta })\chi _{B(0,2Mt)\cap R_{k}}\big\|_{L^{p,r}}^{q}
\end{equation*}%
is just 
\begin{equation}
\sum\limits_{k\in \mathbb{Z},2^{k}\leq 2Mt}2^{k\alpha q}\big\|%
(d_{t}^{M}f_{\mu ,\delta })\chi _{B(0,2Mt)\cap R_{k}}\big\|_{L^{p,r}}^{q}=I.
\label{est-I}
\end{equation}%
Observe that $|x+jh|\leq 3Mt$, $j\in \{0,1,...,M\},x\in B(0,2Mt)\cap R_{k}$
and $|h|<t$, $k\in \mathbb{Z}$, which yields%
\begin{equation*}
d_{t}^{M}f_{\mu ,\delta }(x)\lesssim t^{-n}\int_{|v|\leq 3Mt}\big|f_{\mu
,\delta }(v)\big|dv\lesssim t^{\mu }(-\log t)^{-\delta },
\end{equation*}%
because of $\mu >-n$, where the implicit constant is independent of $x$ and $%
t$. Putting this into \eqref{est-I} and using 
\begin{equation*}
\big\|\chi _{B(0,2Mt)}\big\|_{L^{p,r}}\lesssim t^{\frac{n}{p}},
\end{equation*}%
where the implicit constant is independent of $t$ and $k$, we arrive at%
\begin{equation}
I\lesssim t^{q(\mu +\frac{n}{p})}(-\log t)^{-q\delta }\sum\limits_{k\in 
\mathbb{Z},2^{k}\leq 2Mt}2^{k\alpha q}\lesssim t^{q(\mu +\frac{n}{p}+\alpha
)}(-\log t)^{-q\delta },  \label{est-I1}
\end{equation}%
since $\alpha >0$. We easily seen that%
\begin{align*}
\sum\limits_{k=-\infty }^{\infty }2^{k\alpha q}\big\|(d_{t}^{M}f_{\mu
,\delta })\chi _{B(t)\cap R_{k}}\big\|_{L^{p,r}}^{q}& =\sum\limits_{k\in 
\mathbb{Z},2^{k}\geq 2Mt}2^{k\alpha q}\big\|(d_{t}^{M}f_{\mu ,\delta })\chi
_{B(t)\cap R_{k}}\big\|_{L^{p,r}}^{q} \\
& =J.
\end{align*}%
We set%
\begin{equation*}
B_{1}(t)=\{x\in \mathbb{R}^{n}:2Mt\leq \left\vert x\right\vert
<2Mt+\vartheta \}.
\end{equation*}%
Since $d_{t}^{M}f_{\mu ,\delta }(x)=0$ if $\left\vert x\right\vert \geq
2Mt+\vartheta $, we obtain%
\begin{equation*}
J\leq \sum\limits_{k\in \mathbb{Z},2^{k}\geq 2Mt}2^{k\alpha q}\big\|%
(d_{t}^{M}f_{\mu ,\delta })\chi _{B_{1}(t)\cap R_{k}}\big\|_{L^{p,r}}^{q}.
\end{equation*}%
Using the fact that%
\begin{equation*}
|\Delta _{h}^{M}f_{\mu ,\delta }(x)|\lesssim |h|^{M}\max_{|\gamma
|=M}\sup_{|x-y|\leq M|h|}|D^{\gamma }f_{\mu ,\delta }(y)|
\end{equation*}%
if $0\notin \{y\in \mathbb{R}^{n}:|x-y|\leq M|h|\}$ and%
\begin{equation*}
|D^{\gamma }f_{\mu ,\delta }(x)|\lesssim |x|^{\mu -M}(-\log |x|)^{-\delta
},\quad |\gamma |=M\geq 1,
\end{equation*}%
we find%
\begin{equation*}
\big\|(d_{t}^{M}f_{\mu ,\delta })\chi _{B_{1}(t)\cap R_{k}}\big\|%
_{L^{p,r}}\lesssim t^{M}\big\||x|^{(\mu -M)}(-\log |x|)^{-\delta }\chi
_{B_{1}(t)\cap R_{k}}\big\|_{L^{p,r}}.
\end{equation*}%
Let $i\in \mathbb{Z},i_{0}\in \mathbb{N}$ be such that $2^{i-1}\leq
2Mt<2^{i} $ and $2^{i_{0}-1}\leq \frac{\vartheta }{2tM}<2^{i_{0}}$. Then $J$
can be estimated from above by%
\begin{align}
& \sum\limits_{k\in \mathbb{Z},2^{k}\geq 2Mt}2^{k\alpha q}\big\|%
(d_{t}^{M}f_{\mu ,\delta })\chi _{B_{i,i_{0}}\cap R_{k}}\big\|_{L^{p,r}}^{q}
\notag \\
& \leq \sum\limits_{l=i}^{i+i_{0}+1}\sum\limits_{k\in \mathbb{Z},2^{k}\geq
2Mt}2^{k\alpha q}\big\|(d_{t}^{M}f_{\mu ,\delta })\chi _{R_{l}\cap R_{k}}%
\big\|_{L^{p,r}}^{q}  \notag \\
& \lesssim \sum\limits_{l=i}^{i+i_{0}+1}2^{l\alpha q}\big\||x|^{\mu
-M}(-\log |x|)^{-\delta }\chi _{R_{l}}\big\|_{L^{p,r}}^{q}  \notag \\
& \lesssim \sum\limits_{l=i}^{i+i_{0}+1}2^{l(\mu -M+\frac{n}{p}+\alpha
)q}(-l)^{-\delta q},  \label{est-J2lorentz}
\end{align}%
where $B_{i,i_{0}}=\{x\in \mathbb{R}^{n}:2^{i-1}\leq \left\vert x\right\vert
<2^{i+i_{0}+1}\}$. One easily checks%
\begin{align}
& \sum\limits_{l=i}^{i+i_{0}+1}2^{l(\mu -M+\frac{n}{p}+\alpha
)q}(-l)^{-\delta q}  \notag \\
& =(-i)^{-\delta q}2^{i(\mu -M+\frac{n}{p}+\alpha )q}\sum\limits_{\kappa
=0}^{i_{0}+1}2^{\kappa (\mu -M+\frac{n}{p}+\alpha )q}(1+\frac{k}{-i-\kappa }%
)^{\delta q}  \notag \\
& \lesssim (-i)^{-\delta q}2^{i(\mu -M+\frac{n}{p}+\alpha )q},
\label{est-J1}
\end{align}%
since $-i-\kappa \geq -i_{0}-i\geq -1-\log _{2}2\vartheta $ and $M$ is
sufficiently large. Inserting the estimation \eqref{est-J1}\ into %
\eqref{est-J2lorentz}, we get 
\begin{equation}
J\leq t^{Mq}(-i)^{-\delta q}2^{i(\mu -M+\frac{n}{p}+\alpha )q}\lesssim
t^{q(\mu +\frac{n}{p}+\alpha )}(-\log t)^{-q\delta }.  \label{est-J3}
\end{equation}%
Plugging \eqref{est-I1}\ and \eqref{est-J3} into \eqref{main-est}, we obtain 
\begin{equation*}
\big\|d_{t}^{M}f_{\mu ,\delta }\big\|_{\dot{K}_{p,r}^{\alpha ,q}}^{q}.\leq
ct^{q(\mu +\frac{n}{p}+\alpha )}(-\log t)^{-q\delta }
\end{equation*}%
for some constant $c$ independent of $t$. Consequently we obtain%
\begin{equation*}
\int_{0}^{\frac{\vartheta }{2M}}t^{-(\mu +\frac{n}{p}+\alpha )\beta }\big\|%
d_{t}^{M}f_{\mu ,\delta }\big\|_{\dot{K}_{p,r}^{\alpha ,q}}^{\beta }\frac{dt%
}{t}\lesssim \int_{0}^{\frac{\vartheta }{2M}}(-\log t)^{-\delta \beta }\frac{%
dt}{t}<\infty
\end{equation*}%
if and only if $\delta \beta >1$.

\textit{Step 2}\textbf{.} \textit{Necessity in part }$\mathrm{(i)}$\textit{\
and }$\mathrm{(ii)}$. Let $p_{1}>0$ be such that 
\begin{equation*}
\max (1,p)<p_{1}<\frac{n}{\max ((-\alpha )_{+},(-\mu -\alpha )_{+})}.
\end{equation*}%
Let $\alpha _{1}\in \mathbb{R}$ be such that 
\begin{equation*}
\max \Big(-\mu -\frac{n}{p_{1}},-\frac{n}{p_{1}}\Big)<\alpha _{1}<\min \Big(%
\alpha ,-\mu -\frac{n}{p_{1}}+1\Big).
\end{equation*}%
We claim that $f_{\mu ,\delta }\notin \dot{K}_{p_{1},r}^{\alpha
_{1},q}B_{\beta }^{\frac{n}{p_{1}}+\alpha _{1}+\mu }$, which implies that $f$
does not belong to $\dot{K}_{p,r}^{\alpha ,q}B_{\beta }^{\frac{n}{p}+\alpha
+\mu }$, since%
\begin{equation*}
\dot{K}_{p,r}^{\alpha ,q}B_{\beta }^{\frac{n}{p}+\alpha +\mu
}\hookrightarrow \dot{K}_{p_{1},r}^{\alpha _{1},q}B_{\beta }^{\frac{n}{p_{1}}%
+\alpha _{1}+\mu },
\end{equation*}%
see Theorem \ref{embeddings3-lorentz}. Let us prove our claim. Let $H>0$ and 
$0<t<\varepsilon $, where $\varepsilon $ is sufficiently small. Let $i\in 
\mathbb{Z}$ be such that $2^{i-1}\leq \frac{t}{2H}<2^{i}$. It is easily seen
that%
\begin{align*}
\sum\limits_{k=-\infty }^{\infty }2^{k\alpha _{1}q}\big\|(d_{t}^{1}f_{\mu
,\delta })\chi _{R_{k}}\big\|_{L^{p_{1},r}}^{q}& \geq 2^{(i-1)\alpha _{1}q}%
\big\|(d_{t}^{1}f_{\mu ,\delta })\chi _{B(0,\frac{t}{2H})\cap R_{i-1}}\big\|%
_{L^{p_{1},r}}^{q} \\
& \geq ct^{\alpha _{1}q}\big\|(d_{t}^{1}f_{\mu ,\delta })\chi _{B(0,\frac{t}{%
2H})\cap R_{i-1}}\big\|_{L^{p_{1},r}}^{q},
\end{align*}%
where $c$ is independent of $t$ and $i$. Let $A=\{x=(x_{1},...,x_{n}):x_{i}%
\geq 0,i=1,...,n\}$ and $x\in B(0,\frac{t}{2H})\cap R_{i-1}\cap A$. By the
inequality (27) in \cite[2.3.1, p. 45]{RuSi96}, we obtain%
\begin{equation*}
d_{t}^{1}f_{\mu ,\delta }(x)\geq t^{-n}\int_{\frac{t}{2}\leq |h|<t}|\Delta
_{h}^{1}f_{\mu ,\delta }(x)|\chi _{M}(h)dh\geq ct^{\mu }(-\log \frac{t}{2}%
)^{-\delta }
\end{equation*}%
for some positive constant $c$ independent of $h$ where $M=%
\{h=(h_{1},...,h_{n}):h_{i}\geq 0\}$. Therefore%
\begin{align*}
\big\|(d_{t}^{1}f_{\mu ,\delta })\chi _{B(0,\frac{t}{2H})\cap R_{i-1}}\big\|%
_{L^{p_{1},r}}^{q}& \geq ct^{\mu q}(-\log \frac{t}{2})^{-\delta q}\Big(%
\int_{R_{i-1}\cap A}dx\Big)^{\frac{q}{p_{1}}} \\
& \geq ct^{(\mu +\frac{n}{p_{1}})q}(-\log \frac{t}{2})^{-\delta q}.
\end{align*}%
As a consequence of the last estimate, we get%
\begin{equation*}
\int_{0}^{\varepsilon }t^{(-\frac{n}{p_{1}}-\alpha _{1}-\mu )\beta }\big\|%
d_{t}^{1}f_{\mu ,\delta }\big\|_{\dot{K}_{p_{1},r}^{\alpha _{1},q}}^{\beta }%
\frac{dt}{t}\geq c\int_{0}^{\varepsilon }(-\log \frac{t}{2})^{-1}\frac{dt}{t}%
=\infty .
\end{equation*}%
This yields the desired result. The proof is complete.
\end{proof}

\begin{remark}
If $\alpha =0$ and $p=q$, then Lemma \ref{Si-funct1} reduces to the result
given in \cite[Lemma 2.3.1/1]{RuSi96}.
\end{remark}

Let $\varrho $ be a $C^{\infty }$ function on $\mathbb{R}$ such that $%
\varrho (x)=1$ for $x\leq e^{-3}$ and $\varrho (x)=0$ for $x\geq e^{-2}$.
Let $(\lambda ,\sigma )\in \mathbb{R}^{2}$ and%
\begin{equation}
f_{\lambda ,\sigma }(x)=|\log |x||^{\lambda }|\log |\log |x|||^{-\sigma
}\varrho (|x|).  \label{triebel-function}
\end{equation}%
As in \cite{Bo4} let $U_{\beta }$ be the set of $(\lambda ,\sigma )\in 
\mathbb{R}^{2}$ such that:

\textbullet\ $\lambda =1-\frac{1}{\beta }$ and $\sigma >\frac{1}{\beta }$,
or $\lambda <1-\frac{1}{\beta }$, in case $1<\beta <\infty $,

\textbullet\ $\lambda =0$ and $\sigma >0$, or $\lambda <0$, in case $\beta
=1 $,

\textbullet\ $\lambda =1$ and $\sigma \geq 0$, or $\lambda <1$, in case $%
\beta =\infty .$

\begin{lemma}
\label{Bourdaud-Triebel}Let $(\lambda ,\sigma )\in \mathbb{R}^{2},0<p<\infty
,0<r,q\leq \infty ,1\leq \beta \leq \infty ,\alpha >-\frac{n}{p}$ and 
\begin{equation}
(\lambda ,\sigma )\in U_{\beta }.  \label{condition1}
\end{equation}%
Let $f_{\lambda ,\sigma }$ be the function defined by %
\eqref{triebel-function}\textrm{. }$\newline
\mathrm{(i)}$\ We have $f_{\lambda ,\sigma }\in \dot{K}_{p,r}^{\alpha
,q}B_{\beta }^{\alpha +\frac{n}{p}}$. In the case $\alpha \geq 0$, the
condition \eqref{condition1} becomes necessary.$\newline
\mathrm{(ii)}$ Let $1\leq r,q<\infty ,0<\beta \leq \infty $. Let $(\lambda
,\sigma )\in U_{p_{2}}$ where%
\begin{equation*}
p_{2}=\left\{ 
\begin{array}{ccc}
q, & \text{if} & q\leq r, \\ 
r, & \text{if} & q>r.%
\end{array}%
\right.
\end{equation*}%
Then $f_{\lambda ,\sigma }\in \dot{K}_{p,r}^{\alpha ,q}F_{\beta }^{\alpha +%
\frac{n}{p}}.$
\end{lemma}

\begin{proof}
For clarity, we split the proof into two steps.

\textit{Step 1.} \textit{Sufficiency in part in }$\mathrm{(i)}$. Let $%
\{\varphi _{j}\}_{j\in \mathbb{N}_{0}}$ be a partition of unity. Notice that 
\begin{equation*}
\big\|\mathcal{F}^{-1}\varphi _{j}\ast f_{\lambda ,\sigma }\big\Vert_{\dot{K}%
_{p,r}^{\alpha ,q}}<\infty ,\quad j\in \{0,1\}.
\end{equation*}%
Indeed, we have%
\begin{equation*}
|\mathcal{F}^{-1}\varphi _{j}(x-y)|\leq c\eta _{m}(x)\eta _{-m}(y),\quad
x,y\in \mathbb{R}^{n},j=0,1,m>0,
\end{equation*}%
where the positive constant $c$ is independent of $x$ and $y$, and $\eta
_{m}(x)=(1+|x|)^{-m},x\in \mathbb{R}^{n}$. We choose $m>\alpha +\frac{n}{p}$%
. Since $f_{\lambda ,\sigma }$ is an integrable function, we obtain%
\begin{equation*}
\big\|\mathcal{F}^{-1}\varphi _{j}\ast f_{\lambda ,\sigma }\big\Vert_{\dot{K}%
_{p,r}^{\alpha ,q}}\lesssim \big\|\eta _{m}\big\Vert_{\dot{K}_{p,r}^{\alpha
,q}}\int_{|y|\leq e^{-2}}|f_{\lambda ,\sigma }(y)|\eta _{-m}(y)dy<\infty
,\quad j\in \{0,1\},
\end{equation*}%
Therefore it suffices to prove the following:%
\begin{equation*}
\sum_{j=2}^{\infty }2^{j(n+\frac{n}{p})\beta }\big\|\mathcal{F}^{-1}\varphi
_{j}\ast f_{\lambda ,\sigma }\big\Vert_{\dot{K}_{p,r}^{\alpha ,q}}^{\beta
}<\infty ,
\end{equation*}%
From \cite[p. 272]{Bo4},%
\begin{equation*}
|x|^{2v}|\mathcal{F}^{-1}\varphi _{j}\ast f_{\lambda ,\sigma }(x)|\lesssim
2^{-2jv}\varepsilon _{j},\quad x\in \mathbb{R}^{n},j\geq 2,v\in \mathbb{N}%
_{0},
\end{equation*}%
with%
\begin{equation*}
\varepsilon _{j}=j^{\lambda -1}(\log j)^{-\sigma }\text{\quad if\quad }%
\lambda \neq 0,\quad \varepsilon _{j}=j^{-1}(\log j)^{-\sigma -1}\text{\quad
if\quad }\lambda =0,
\end{equation*}%
which belongs to $\ell ^{\beta }$ if and only if $(\lambda ,\sigma )\in
U_{\beta }$. Then we split%
\begin{equation*}
\sum_{k=-\infty }^{\infty }2^{k\alpha q}\big\|(\mathcal{F}^{-1}\varphi
_{j}\ast f_{\lambda ,\sigma })\chi _{k}\big\|_{L^{p,r}}^{q}=I_{1,j}+I_{2,j},%
\text{\quad }j\geq 2,
\end{equation*}%
where%
\begin{equation*}
I_{1,j}=\sum_{k=-\infty }^{-j}2^{k\alpha q}\big\|(\mathcal{F}^{-1}\varphi
_{j}\ast f_{\lambda ,\sigma })\chi _{k}\big\|_{L^{p,r}}^{q}
\end{equation*}%
and 
\begin{equation*}
I_{2,j}=\sum_{k=-j+1}^{\infty }2^{k\alpha q}\big\|(\mathcal{F}^{-1}\varphi
_{j}\ast f_{\lambda ,\sigma })\chi _{k}\big\|_{L^{p,r}}^{q}.
\end{equation*}%
It is easily seen that $I_{1,j}\lesssim \varepsilon _{j}^{q}\sum_{k=-\infty
}^{-j}2^{k(\alpha +\frac{n}{p})q},j\geq 2$. Therefore%
\begin{equation*}
\sum_{j=2}^{\infty }2^{j(\alpha +\frac{n}{p})\beta }(I_{1,j})^{\beta
/q}\lesssim \sum_{j=2}^{\infty }\varepsilon _{j}^{\beta }\Big(%
\sum_{k=-\infty }^{-j}2^{(k+j)(\alpha +\frac{n}{p})q}\Big)^{\beta
/q}\lesssim \sum_{j=2}^{\infty }\varepsilon _{j}^{\beta }<\infty .
\end{equation*}%
Now%
\begin{equation*}
I_{2,j}\lesssim \varepsilon _{j}^{q}\sum_{k=-j+1}^{\infty }2^{(k\alpha
-2jv)q}\big\||\cdot |^{-2v}\chi _{k}\big\|_{L^{p,r}}^{q}\lesssim \varepsilon
_{j}^{q}\sum_{k=-j+1}^{\infty }2^{k(\alpha -2v+\frac{n}{p})q-2jvq}
\end{equation*}%
for any $j\geq 2$. Hence%
\begin{equation*}
\sum_{j=2}^{\infty }2^{j(\alpha +\frac{n}{p})\beta }(I_{2,j})^{\beta
}\lesssim \sum_{j=2}^{\infty }\varepsilon _{j}^{\beta }\Big(%
\sum_{k=-j+1}^{\infty }2^{(k+j)(\alpha -2v+\frac{n}{p})}\Big)^{\beta
/q}\lesssim \sum_{j=2}^{\infty }\varepsilon _{j}^{\beta }<\infty ,
\end{equation*}%
by taking $v>\frac{\alpha +\frac{n}{p}}{2}$.

\textit{Step 2. Necessity part in }$\mathrm{(i)}$\textit{. }Let us assume $%
(\lambda ,\sigma )\notin U_{\beta }$ and $\alpha \geq 0$. We are going to
prove that $f_{\lambda ,\sigma }\notin \dot{K}_{p,r}^{\alpha ,q}B_{\beta
}^{\alpha +\frac{n}{p}}$, but this follows by the embeddings\ 
\begin{equation*}
\dot{K}_{p,r}^{\alpha ,q}B_{\beta }^{\alpha +\frac{n}{p}}\hookrightarrow 
\dot{K}_{p_{0}}^{0,q}B_{\beta }^{\frac{n}{p_{0}}}\hookrightarrow B_{\infty
,\beta }^{0},\quad 0<p<p_{0}<\infty
\end{equation*}
and $f_{\lambda ,\sigma }\notin B_{\infty ,\beta }^{0}$ for any $(\lambda
,\sigma )\notin U_{\beta }$, see \cite[Proposition 2]{Bo4}.

\textit{Step 3. Proof of }$\mathrm{(ii)}$\textit{. }Let $0<p_{1}<p<\infty $.
According to Theorem \ref{embeddings6.1-lorentz} the following embedding
holds:%
\begin{equation*}
\dot{K}_{p_{1},r}^{\alpha ,q}B_{p_{2}}^{\alpha +\frac{n}{p_{1}}%
}\hookrightarrow \dot{K}_{p,r}^{\alpha ,q}F_{\beta }^{\alpha +\frac{n}{p}},
\end{equation*}%
where%
\begin{equation*}
p_{2}=\left\{ 
\begin{array}{ccc}
q, & \text{if} & q\leq r, \\ 
r, & \text{if} & q>r.%
\end{array}%
\right.
\end{equation*}%
This proves (ii).
\end{proof}

\begin{remark}
If $\alpha =0$ and $p=q$, then Lemma \ref{Bourdaud-Triebel} reduces to the
result given in \cite[Proposition 2]{Bo4} and \cite{Triebel93}.
\end{remark}

Now, we present the last example.

\begin{proposition}
\label{Key-1-lorentz}Let $\beta >0,1\leq p,q,r<\infty ,-\frac{n}{p}<\alpha
<n-\frac{n}{p},$%
\begin{equation*}
0<\max \Big(\delta +\frac{n}{p},\delta +\frac{n}{p}+\alpha \Big)<2(\beta
+1)\quad \text{and}\quad \sigma =\frac{\delta +\frac{n}{p}+\alpha }{\beta +1}%
.
\end{equation*}%
Let $g\in B_{\infty ,\infty }^{\gamma }(\mathbb{R})$ for some $\sigma
<\gamma $. The function 
\begin{equation*}
f(x)=\left\vert x\right\vert ^{\delta }g(\left\vert x\right\vert ^{-\beta
})\varrho \left( \left\vert x\right\vert \right)
\end{equation*}%
belongs to $\dot{K}_{p,r}^{\alpha ,q}B_{\infty }^{\sigma }${. }
\end{proposition}

\begin{proof}
Observe that $f\in \dot{K}_{p,r}^{\alpha ,q}$. From Theorem \ref%
{means-diff-cha-lorentz}, we need to prove that%
\begin{equation*}
\sup_{0<t\leq \frac{1}{2}e^{-2}}t^{-\sigma }\big\Vert d_{t}^{m}f\big\Vert_{%
\dot{K}_{p,r}^{\alpha ,q}}<\infty ,\quad 0<\sigma <m\leq 2.
\end{equation*}%
We will divide the proof into three steps.

\textit{Step 1.}\ We\ will\ prove that $f\in \dot{K}_{p,r}^{\alpha
,q}B_{\infty }^{\sigma }$ with $0<\sigma <1$ and $\delta \neq 1-\frac{n}{p}%
-\alpha $. We can only assume that $\gamma <1$. Let us estimate $\big\Vert %
d_{t}^{1}f\big\Vert_{\dot{K}_{p,r}^{\alpha ,q}}$ for any $0<t\leq \frac{1}{2}%
e^{-2}$. Obviously, $d_{t}^{1}f(x)=0$ for any $x\in \mathbb{R}^{n}$ such
that $\left\vert x\right\vert \geq 2e^{-2}$ and $0<t\leq \frac{1}{2}e^{-2}$.
We see that \textit{\ }%
\begin{align}
\big\|(d_{t}^{1}f)\chi _{B(0,2e^{-2})}\big\|_{\dot{K}_{p,r}^{\alpha
,q}}^{q}=& \sum\limits_{k=-\infty }^{\infty }2^{k\alpha q}\big\|%
(d_{t}^{1}f)\chi _{B(0,2e^{-2})\cap R_{k}}\big\|_{L^{p,r}}^{q}  \notag \\
=& H_{1}(t)+H_{2}(t),  \label{est-H-lorentz}
\end{align}%
where 
\begin{equation*}
H_{1}(t)=c\sum\limits_{k\in \mathbb{Z},2^{k}<4t^{\frac{1}{\beta +1}%
}}2^{k\alpha q}\big\|(d_{t}^{1}f)\chi _{B(0,2e^{-2})\cap R_{k}}\big\|%
_{L^{p,r}}^{q}
\end{equation*}%
and%
\begin{equation*}
H_{2}(t)=c\sum\limits_{k\in \mathbb{Z},2^{k}\geq 4t^{\frac{1}{\beta +1}%
}}2^{k\alpha q}\big\|(d_{t}^{1}f)\chi _{B(0,2e^{-2})\cap R_{k}}\big\|%
_{L^{p,r}}^{q}.
\end{equation*}%
In what follows, we estimate each term on the right hand side of %
\eqref{est-H-lorentz}. To do this, note first 
\begin{align*}
& \big\|(d_{t}^{1}f)\,\chi _{B(0,2e^{-2})\cap R_{k}}\big\|_{L^{p,r}} \\
& \leq \big\|(d_{t}^{1}f)\chi _{B(0,2e^{-2})\cap B(0,4t^{\frac{1}{\beta +1}%
})\cap R_{k}}\big\|_{L^{p,r}}+\big\|(d_{t}^{1}f)\chi _{(\mathbb{R}%
^{n}\backslash B(0,4t^{\frac{1}{\beta +1}}))\cap R_{k}}\big\|_{L^{p,r}} \\
& =T_{1,k}(t)+T_{2,k}(t),\quad k\in \mathbb{Z}.
\end{align*}%
For clarity, we split this step into two substeps and conclusion.

\textit{Substep 1.1.}\textbf{\ }\textit{Estimation of\ }$H_{1}$. Since $%
T_{2,k}(t)=0$ if $2^{k}<4t^{\frac{1}{\beta +1}},0<t\leq \frac{1}{2}e^{-2}\ $%
and $k\in \mathbb{Z}$, we need only to estimate $T_{1,k}(t)$. Let $x\in
B(0,4t^{\frac{1}{\beta +1}})\cap B(0,2e^{-2})\cap R_{k}$ and 
\begin{equation*}
\max \Big(0,-\delta \frac{p}{n}\Big)<\frac{1}{\tau }<\min \Big(1,1+\frac{%
\alpha p}{n}\Big).
\end{equation*}%
By H\"{o}lder's inequality, we get 
\begin{align}
T_{1,k}(t)& \leq \big\|\chi _{R_{k}}\big\|_{L^{p\tau ^{\prime },\infty }}%
\big\|(d_{t}^{1}f)\chi _{B(0,2e^{-2})\cap B(0,4t^{\frac{1}{\beta +1}})\cap
R_{k}}\big\|_{L^{p\tau ,r}}  \notag \\
& \lesssim 2^{k\frac{n}{p\tau ^{\prime }}}\big\|(d_{t}^{1}f)\chi
_{B(0,2e^{-2})\cap B(0,4t^{\frac{1}{\beta +1}})\cap R_{k}}\big\|_{L^{p\tau
,r}}.  \label{estimate-t1-lorentz}
\end{align}%
To estimate the right-hand side of \eqref{estimate-t1-lorentz} we
distinguish between the following two cases: $\delta \geq 0$\ and\ $-\frac{n%
}{p}<\delta <0$.

\textit{Case 1.} $\delta \geq 0$. We have%
\begin{equation*}
d_{t}^{1}f(x)\lesssim t^{-n}\int_{|h|<t}\left\vert f\left( x+h\right)
\right\vert dh+\left\vert f\left( x\right) \right\vert .
\end{equation*}%
Using the fact that $g,\varrho \in L^{\infty }(\mathbb{R})$ and 
\begin{equation*}
|x+h|\leq |x|+|h|<5t^{\frac{1}{\beta +1}},
\end{equation*}%
whenever $x\in B(0,2e^{-2})\cap B(0,4t^{\frac{1}{\beta +1}})\cap R_{k}$ and $%
\left\vert h\right\vert <t$, we obtain 
\begin{equation}
d_{t}^{1}f(x)\lesssim t^{\frac{\delta }{\beta +1}},  \label{est-dt-lorentz}
\end{equation}%
where the implicit constant is independent of $x$ and $t$. By %
\eqref{est-dt-lorentz}, we get%
\begin{align*}
\big\|(d_{t}^{1}f)\chi _{B(0,2e^{-2})\cap B(0,4t^{\frac{1}{\beta +1}})\cap
R_{k}}\big\|_{L^{p\tau ,r}}& \lesssim t^{\frac{\delta }{\beta +1}}\big\|\chi
_{B(0,4t^{\frac{1}{\beta +1}})}\big\|_{L^{p\tau ,r}} \\
& \lesssim t^{\frac{\delta +\frac{n}{p\tau }}{\beta +1}}.
\end{align*}

\textit{Case 2.} $-\frac{n}{p}<\delta <0$.

\textit{Subcase 2.1. }$-\frac{n}{p}<\delta <0$ and $1<p<\infty $. We see
that 
\begin{equation}
\Big\|t^{-n}\int_{|h|<t}\left\vert f\left( \cdot +h\right) \right\vert \chi
_{B(-h,5t^{\frac{1}{\beta +1}})}dh\Big\|_{L^{p\tau ,r}}  \label{dual-lorentz}
\end{equation}%
is compared to%
\begin{equation*}
c\sup \int_{\mathbb{R}^{n}}t^{-n}\int_{|h|<t}\left\vert f\left( x+h\right)
\right\vert \chi _{B(-h,5t^{\frac{1}{\beta +1}})}(x)dhw(x)dx,
\end{equation*}%
where the supremum is taken over all $w\in L^{(p\tau )^{\prime },r^{\prime
}} $ such that $\big\|w\big\|_{L^{(p\tau )^{\prime },r^{\prime }}}\leq 1$.
By H\"{o}lder's inequality, we get%
\begin{equation*}
\int_{\mathbb{R}^{n}}\left\vert f\left( x+h\right) \right\vert \chi
_{B(-h,5t^{\frac{1}{\beta +1}})}(x)w(x)dx\lesssim \big\|f\chi _{B(0,5t^{%
\frac{1}{\beta +1}})}\big\|_{L^{p\tau ,r}}.
\end{equation*}%
Hence, the right-hand side of \eqref{estimate-t1-lorentz} can be estimated
from above by%
\begin{equation}
c\big\|f\chi _{B(0,5t^{\frac{1}{\beta +1}})}\big\|_{L^{p\tau ,r}}.
\label{estimate-t2-lorentz}
\end{equation}%
Put $\omega (x)=|x|^{\delta }\chi _{B(0,5t^{\frac{1}{\beta +1}})}(x)$. A
simple calculation yields%
\begin{equation*}
\omega ^{\ast }(z)=\left\{ 
\begin{array}{ccc}
z^{\frac{\delta }{n}}, & \text{if} & 0<z<5^{n}t^{\frac{n}{\beta +1}}, \\ 
0, & \text{if} & z\geq 5^{n}t^{\frac{n}{\beta +1}}.%
\end{array}%
\right.
\end{equation*}%
This implies%
\begin{align*}
\big\|\omega \big\|_{L^{p\tau ,r}}& =\Big(\int_{0}^{5^{n}t^{\frac{n}{\beta +1%
}}}z^{\frac{r}{p\tau }}(\omega ^{\ast }(z))^{r}\frac{dz}{z}\Big)^{1/r} \\
& \lesssim \Big(\int_{0}^{5^{n}t^{\frac{n}{\beta +1}}}z^{(\frac{1}{p\tau }+%
\frac{\delta }{n})r}\frac{dz}{z}\Big)^{1/r} \\
& \lesssim t^{\frac{\delta +\frac{n}{p\tau }}{\beta +1}}
\end{align*}%
since $\frac{n}{\tau p}+\delta >0$. Consequently, \eqref{estimate-t2-lorentz}
does not exceed $ct^{\frac{\delta +\frac{n}{p\tau }}{\beta +1}}$, where the
positive constant is independent of $t.$

\textit{Subcase 2.2. }$-\frac{n}{p}<\delta <0$\ and $p=1$. Using the
embeddings $L^{1}\hookrightarrow L^{1,r}$, it is easy to see that $%
T_{1,k}(t) $\ can be estimated from above%
\begin{equation}
c\text{ }t^{-n}\int_{|h|<t}\int_{B(-h,5t^{\frac{1}{\beta +1}})\cap
B(0,2e^{-2})\cap R_{k}}\left\vert f(x+h)\right\vert dxdh+\int_{B(0,4t^{\frac{%
1}{\beta +1}})\cap B(0,2e^{-2})\cap R_{k}}\left\vert f(x)\right\vert dx,
\label{est-H2-lorentz}
\end{equation}%
where the positive constant is independent of $t,h$ and $k$. In this case we
choose 
\begin{equation*}
\max \Big(0,-\delta \frac{1}{n}\Big)<\frac{1}{\tau }<\min \Big(1,1+\frac{%
\alpha }{n}\Big).
\end{equation*}%
By H\"{o}lder's inequality and since $g,\varrho \in L^{\infty }(\mathbb{R})$%
, we obtain 
\begin{align}
\int_{B(-u,5t^{\frac{1}{\beta +1}})\cap B(0,2e^{-2})\cap R_{k}}\left\vert
f(x+u)\right\vert dx\lesssim & 2^{k\frac{n}{\tau ^{\prime }}}\Big(%
\int_{B(-u,5t^{\frac{1}{\beta +1}})}\left\vert f(x+u)\right\vert ^{\tau }dx%
\Big)^{\frac{1}{\tau }}  \notag \\
\lesssim & 2^{k\frac{n}{\tau ^{\prime }}}\Big(\int_{|z|<5t^{\frac{1}{\beta +1%
}}}\left\vert f(z)\right\vert ^{\tau }dz\Big)^{\frac{1}{\tau }}  \notag \\
\lesssim & 2^{k\frac{n}{\tau ^{\prime }}}\Big(\int_{0}^{5t^{\frac{1}{\beta +1%
}}}r^{\tau \delta +n-1}dr\Big)^{\frac{1}{\tau }}  \notag \\
\lesssim & t^{\frac{\delta +\frac{n}{\tau }}{\beta +1}}2^{k\frac{n}{\tau
^{\prime }}},  \label{est-H1-lorentz}
\end{align}%
since $\delta +\frac{n}{\tau }>0$, where $u\in \{0,h\}\ $and the implicit
constant is independent of $k$ and $t$. Plugging \eqref{est-H1-lorentz}\
into \eqref{est-H2-lorentz}, we obtain%
\begin{equation*}
T_{1,k}(t)\lesssim t^{\frac{\delta +\frac{n}{\tau }}{\beta +1}}2^{k\frac{n}{%
\tau ^{\prime }}}.
\end{equation*}%
In any case, we end up with 
\begin{equation}
H_{1}(t)\leq ct^{\sigma q}\sum\limits_{k\in \mathbb{Z},2^{k}<4t^{\frac{1}{%
\beta +1}}}\Big(\frac{2^{k}}{t^{\frac{1}{\beta +1}}}\Big)^{(\alpha +\frac{n}{%
p\tau ^{\prime }})q}\leq ct^{\sigma q},  \label{key-est11-lorentz}
\end{equation}%
since $\frac{n}{\tau ^{\prime }}+\alpha p>0$, where $c>0$ is independent of $%
t$.

\textit{Substep 1.2.}\textbf{\ }\textit{Estimation of\ }$H_{2}$. The
situation is quite different and more complicated. As in Substep 1.1, more
precisely with $\tau =1$, one finds that%
\begin{equation*}
T_{1,k}(t)\lesssim t^{\frac{\delta +\frac{n}{p}}{\beta +1}}.
\end{equation*}%
Therefore%
\begin{equation*}
\sum\limits_{k\in \mathbb{Z},2^{k}\geq 4t^{\frac{1}{\beta +1}}}2^{k\alpha
q}(T_{1,k}(t))^{q}\leq ct^{(\delta +\frac{n}{p})\frac{q}{\beta +1}%
}\sum\limits_{k\in \mathbb{Z},4t^{\frac{1}{\beta +1}}\leq 2^{k}\leq 8t^{%
\frac{1}{\beta +1}}}2^{k\alpha q}\leq ct^{\sigma q},
\end{equation*}%
where $c>0$ is independent of $t$.

\textit{Estimation of }$T_{2,k}(t)$. We decompose $\triangle _{h}^{1}f$\
into three parts%
\begin{equation*}
\triangle _{h}^{1}f(x)=\omega _{1}(x,h)+\omega _{2}(x,h)+\omega _{3}(x,h),
\end{equation*}%
where%
\begin{equation*}
\omega _{1}(x,h)=\left\vert x\right\vert ^{\delta }\big(g(\left\vert
x+h\right\vert ^{-\beta })-g(\left\vert x\right\vert ^{-\beta })\big)\varrho
\left( \left\vert x+h\right\vert \right) ,
\end{equation*}%
\begin{equation*}
\omega _{2}(x,h)=\big(|x+h|^{\delta }-|x|^{\delta })g(\left\vert
x+h\right\vert ^{-\beta }\big)\varrho (\left\vert x+h\right\vert )
\end{equation*}%
and%
\begin{equation*}
\omega _{3}(x,h)=\left\vert x\right\vert ^{\delta }g(\left\vert x\right\vert
^{-\beta })\big(\varrho (\left\vert x+h\right\vert )-\varrho (\left\vert
x\right\vert )\big).
\end{equation*}%
Define%
\begin{equation*}
\tilde{\omega}_{i}(x,t)=t^{-n}\int_{|h|<t}|\omega _{i}(x,h)|dh,\quad i\in
\{1,2,3\}.
\end{equation*}%
Let $x\in \mathbb{R}^{n}$ be such that $\left\vert x\right\vert \geq
2\left\vert h\right\vert ^{\frac{1}{\beta +1}}$. By the mean value theorem
we have%
\begin{equation*}
\big||x+h|^{-\beta }-|x|^{-\beta }\big|\leq c|h||x|^{-\beta -1},
\end{equation*}%
which together with the fact that $g\in B_{\infty ,\infty }^{\gamma }$ we
obtain that%
\begin{equation*}
\big|g(|x+h|^{-\beta })-g(|x|^{-\beta })\big|\leq c|h|^{\gamma }\big\|g\big\|%
_{B_{\infty ,\infty }^{\gamma }}|x|^{-\gamma (\beta +1)},
\end{equation*}%
where $c>0$ is independent of $h$. Therefore,%
\begin{equation*}
\tilde{\omega}_{1}(x,t)\lesssim t^{\gamma }\big\|g\big\|_{B_{\infty ,\infty
}^{\gamma }}|x|^{\delta -\gamma (\beta +1)},
\end{equation*}%
which yields%
\begin{align*}
\big\|\tilde{\omega}_{1}(\cdot ,t)\chi _{(\mathbb{R}^{n}\backslash B(0,4t^{%
\frac{1}{\beta +1}}))\cap R_{k}}\big\|_{L^{p,r}}& \lesssim t^{\gamma }\big\|%
|x|^{\delta -\gamma (\beta +1)}\chi _{k}\big\|_{L^{p,r}} \\
& \lesssim t^{\gamma }2^{k(\delta -\gamma (\beta +1)}\big\|\chi _{k}\big\|%
_{L^{p,r}} \\
& \lesssim t^{\gamma }2^{k(\delta -\gamma (\beta +1)+\frac{n}{p})}
\end{align*}%
Consequently%
\begin{align}
\sum\limits_{k\in \mathbb{Z},2^{k}\geq 4t^{\frac{1}{\beta +1}}}2^{k\alpha q}%
\big\|\tilde{\omega}_{1}(\cdot ,t)\chi _{k}\big\|_{L^{p,r}}^{q}\lesssim &
t^{\gamma q}\sum\limits_{k\in \mathbb{Z},2^{k}\geq 4t^{\frac{1}{\beta +1}%
}}2^{k(\delta -\gamma (\beta +1)+\frac{n}{p}+\alpha )q}  \notag \\
\lesssim & t^{\sigma q}\sum\limits_{k\in \mathbb{Z},2^{k}\geq 4t^{\frac{1}{%
\beta +1}}}\Big(\frac{2^{k}}{t^{\frac{1}{\beta +1}}}\Big)^{(\sigma -\gamma
)(\beta +1)q}  \notag \\
\lesssim & t^{\sigma q},  \label{omega1-lorentz}
\end{align}%
since $\sigma <\gamma $. We have%
\begin{equation}
\big|\left\vert x+h\right\vert ^{\delta }-\left\vert x\right\vert ^{\delta }%
\big|\leq c|h||x+\theta h|^{\delta -1},\quad 0<\theta <1,
\label{est-omega2-lorentz}
\end{equation}%
because of $\left\vert x\right\vert \geq 2t^{\frac{1}{\beta +1}}>2\left\vert
h\right\vert ^{\frac{1}{\beta +1}}$, where the positive constant $c$ is
independent of $x,h$ and $t$. From 
\begin{equation}
\frac{1}{2}|x|\leq |x+\theta h|\leq \frac{3}{2}|x|,\quad g,\varrho \in
L^{\infty }(\mathbb{R})  \label{x-lorentz}
\end{equation}%
and \eqref{est-omega2-lorentz} we immediately deduce that 
\begin{equation*}
\sum\limits_{k\in \mathbb{Z},2^{k}\geq 4t^{\frac{1}{\beta +1}}}2^{k\alpha q}%
\big\|\tilde{\omega}_{2}(\cdot ,t)\,\chi _{\{x:|x|\leq \frac{2}{e^{2}}\}\cap
R_{k}}\big\|_{L^{p,r}}^{q}\lesssim t^{q}\sum\limits_{k\in \mathbb{Z},4t^{%
\frac{1}{\beta +1}}\leq 2^{k}\leq 4e^{-2}}2^{k(\delta -1+\frac{n}{p}+\alpha
)q}
\end{equation*}%
which is bounded by%
\begin{equation}
S(t)=c\left\{ 
\begin{array}{ccc}
t^{(1-\frac{1}{\beta +1}+\sigma )q} & \text{if} & \delta -1+\frac{n}{p}%
+\alpha <0, \\ 
t^{q} & \text{if} & \delta -1+\frac{n}{p}+\alpha >0, \\ 
t^{q}\log \frac{1}{t} & \text{if} & \delta -1+\frac{n}{p}+\alpha =0,%
\end{array}%
\right.  \label{omega2-lorentz}
\end{equation}%
for sufficiently small $t>0$. Obviously, 
\begin{equation}
\sum\limits_{k\in \mathbb{Z},2^{k}\geq 2t^{\frac{1}{\beta +1}}}2^{k\alpha q}%
\big\|\tilde{\omega}_{3}(\cdot ,t)\,\chi _{\{x:|x|\leq 2e^{-2}\}\cap R_{k}}%
\big\|_{L^{p,r}}^{q}\lesssim t^{q}\sum\limits_{k\in \mathbb{Z},2^{k}\leq
4e^{-2}}2^{k(\delta +\frac{n}{p}+\alpha )q}\lesssim t^{q}.
\label{omega3-lorentz}
\end{equation}%
Collecting the estimations \eqref{omega1-lorentz}, \eqref{omega2-lorentz}
and \eqref{omega3-lorentz}, we derive 
\begin{equation}
H_{2}(t)\lesssim t^{\sigma q}+S(t).  \label{key-est22-lorentz}
\end{equation}

\textit{Conclusion. }Combining the two estimates \eqref{key-est11-lorentz}
and \eqref{key-est22-lorentz} we obtain $f\in \dot{K}_{p,r}^{\alpha
,q}B_{\infty }^{\sigma }$ but with 
\begin{equation*}
0<\sigma <1\text{\quad and\quad }\delta -1+\frac{n}{p}+\alpha \neq 0.
\end{equation*}

\textit{Step 2.\ }In this step we prove that $f$ belongs to $f\in \dot{K}%
_{p,r}^{\alpha ,q}B_{\infty }^{\sigma }$ with $1\leq \sigma <2$. We can only
assume that $\sigma <\gamma <2$. Then we split%
\begin{equation*}
\big\|(d_{t}^{2}f)\chi _{B(0,2e^{-2})}\big\|_{\dot{K}_{p,r}^{\alpha
,q}}^{q}=I_{1}+I_{2},
\end{equation*}%
where 
\begin{equation*}
I_{1}(t)=\sum\limits_{k\in \mathbb{Z},2^{k}<4t^{\frac{1}{\beta +1}%
}}2^{k\alpha q}\big\|(d_{t}^{2}f)\chi _{B(0,2e^{-2})\cap R_{k}}\big\|%
_{L^{p,r}}^{q}
\end{equation*}%
and%
\begin{equation*}
I_{2}(t)=\sum\limits_{k\in \mathbb{Z},2^{k}\geq 4t^{\frac{1}{\beta +1}%
}}2^{k\alpha q}\big\|(d_{t}^{2}f)\chi _{B(0,2e^{-2})\cap R_{k}}\big\|%
_{L^{p,r}}^{q}.
\end{equation*}%
We use the following estimate: 
\begin{align*}
& \big\|(d_{t}^{2}f\,)\chi _{B(0,2e^{-2})\cap R_{k}}\big\|_{L^{p,r}} \\
\lesssim & \big\|(d_{t}^{2}f\,)\chi _{B(0,2e^{-2})\cap B(0,4t^{\frac{1}{%
\beta +1}})\cap R_{k}}\big\|_{L^{p,r}}\,+\big\|(d_{t}^{2}f\,)\chi
_{B(0,2e^{-2})\cap (\mathbb{R}^{n}\backslash B(0,4t^{\frac{1}{\beta +1}%
}))\cap R_{k}}\big\|_{L^{p,r}} \\
=& V_{1,k}(t)+V_{2,k}(t),\quad 0<t<1,k\in \mathbb{Z}.
\end{align*}%
We will divide the proof into two Substeps 2.1 and 2.2.

\textit{Substep 2.1.}$\ $\textit{Estimation of }$I_{1}$. Obviously, $%
V_{2,k}(t)=0$ if $2^{k}<4t^{\frac{1}{\beta +1}}$ and $k\in \mathbb{Z}$. We
have%
\begin{equation*}
\triangle _{h}^{2}f\left( x\right) =f\left( x+2h\right) +f\left( x\right)
-2f\left( x+h\right)
\end{equation*}%
and 
\begin{equation*}
|x+2h|\leq |x|+2|h|<4t^{\frac{1}{\beta +1}},
\end{equation*}%
if $x\in B(0,4t^{\frac{1}{\beta +1}})$ and $\left\vert h\right\vert <t$. In
this case, we use an argument similar to that used in Step 1 we find $%
I_{1}(t)\lesssim t^{\sigma q}.$

\textit{Substep 2.2.}$\ $\textit{Estimation of}\textbf{\ }$I_{2}$. Using the
same type of arguments as in Step 1 it is easy to see that $V_{1,k}(t)\leq
ct^{\frac{\delta +\frac{n}{p}}{\beta +1}}$, where $c>0$ is independent of $k$
and $t$ and 
\begin{equation*}
\sum\limits_{k\in \mathbb{Z},2^{k}\geq 4t^{\frac{1}{\beta +1}}}2^{k\alpha
q}(V_{1,k}(t))^{q}\leq ct^{\sigma q}.
\end{equation*}%
We decompose $\triangle _{h}^{2}f(x)$ into $\sum_{i=1}^{5}\varpi _{i}(x,h)$,
where%
\begin{equation*}
\varpi _{1}(x,h)=\left\vert x+h\right\vert ^{\delta }\big(g(\left\vert
x+2h\right\vert ^{-\beta })+g(\left\vert x\right\vert ^{-\beta
})-2g(\left\vert x+h\right\vert ^{-\beta })\big)\varrho \left( \left\vert
x+2h\right\vert \right) ,
\end{equation*}%
\begin{align*}
\varpi _{2}(x,h)=& (\left\vert x+2h\right\vert ^{\delta }-\left\vert
x+h\right\vert ^{\delta })g(\left\vert x+2h\right\vert ^{-\beta })\varrho
(\left\vert x+2h\right\vert ), \\
\varpi _{3}(x,h)=& (\left\vert x\right\vert ^{\delta }-\left\vert
x+h\right\vert ^{\delta })g(\left\vert x\right\vert ^{-\beta })\varrho
(\left\vert x\right\vert ),
\end{align*}%
\begin{equation*}
\varpi _{4}(x,h)=2\left\vert x+h\right\vert ^{\delta }g(\left\vert
x+h\right\vert ^{-\beta })\big(\varrho (\left\vert x+2h\right\vert )-\varrho
(\left\vert x+h\right\vert )\big)
\end{equation*}%
and%
\begin{equation*}
\varpi _{5}(x,h)=\left\vert x+h\right\vert ^{\delta }g(\left\vert
x\right\vert ^{-\beta })\big(\varrho (\left\vert x\right\vert )-\varrho
(\left\vert x+2h\right\vert )\big).
\end{equation*}%
Obviously we need only to estimate $\varpi _{1}$. From %
\eqref{est-omega2-lorentz} and \eqref{x-lorentz}, we obtain%
\begin{equation*}
\left\vert x+h\right\vert ^{\delta }\lesssim \left\vert x\right\vert
^{\delta },\quad |\left\vert x+2h\right\vert ^{\delta }-\left\vert
x+h\right\vert ^{\delta }|\lesssim \left\vert x\right\vert ^{\delta }
\end{equation*}%
if $\left\vert x\right\vert \geq 4t^{\frac{1}{\beta +1}}>4\left\vert
h\right\vert ^{\frac{1}{\beta +1}}$. We split%
\begin{equation*}
2g(\left\vert x+2h\right\vert ^{-\beta })+2g(\left\vert x\right\vert
^{-\beta })-4g(\left\vert x+h\right\vert ^{-\beta })
\end{equation*}%
into three terms i.e., $\vartheta _{1}(x,h)+\vartheta _{2}(x,h)+\vartheta
_{3}(x,h)$, where%
\begin{align*}
& \vartheta _{1}(x,h) \\
=& g(\left\vert x+2h\right\vert ^{-\beta })-g(2\left\vert x+h\right\vert
^{-\beta }-\left\vert x\right\vert ^{-\beta })+g(\left\vert x\right\vert
^{-\beta }) \\
& -g(2\left\vert x+h\right\vert ^{-\beta }-\left\vert x+2h\right\vert
^{-\beta }),
\end{align*}%
\begin{equation*}
\vartheta _{2}(x,h)=g(\left\vert x+2h\right\vert ^{-\beta })+g(2\left\vert
x+h\right\vert ^{-\beta }-\left\vert x+2h\right\vert ^{-\beta
})-2g(\left\vert x+h\right\vert ^{-\beta })
\end{equation*}%
and%
\begin{equation*}
\vartheta _{3}(x,h)=g(\left\vert x\right\vert ^{-\beta })+g(2\left\vert
x+h\right\vert ^{-\beta }-\left\vert x\right\vert ^{-\beta })-2g(\left\vert
x+h\right\vert ^{-\beta }).
\end{equation*}%
Define%
\begin{equation*}
\tilde{\vartheta}_{i}(x,t)=t^{-n}\int_{|h|<t}|\vartheta _{i}(x,h)|dh,\quad
i\in \{1,2,3\}.
\end{equation*}%
Let%
\begin{equation*}
J_{i,k}(t)=\big\|\left\vert x\right\vert ^{\delta }|\tilde{\vartheta}%
_{i}(x,t)|\chi _{B(0,2e^{-2})\cap (\mathbb{R}^{n}\backslash B(0,4t^{\frac{1}{%
\beta +1}}))\cap R_{k}}\big\|_{L^{p,r}},\quad i\in \{1,2,3\}.
\end{equation*}%
Observe that $g^{(1)}\in B_{\infty ,\infty }^{\gamma -1}(\mathbb{R}%
)\hookrightarrow L^{\infty }(\mathbb{R})$. Again by the mean value theorem;%
\begin{equation*}
\left\vert \left\vert x+2h\right\vert ^{-\beta }+\left\vert x\right\vert
^{-\beta }-2\left\vert x+h\right\vert ^{-\beta }\right\vert \leq
c|h|^{2}|x|^{-\beta -2},\quad \left\vert x\right\vert \geq 2\left\vert
h\right\vert ^{\frac{1}{\beta +1}},
\end{equation*}%
which yields that%
\begin{align*}
J_{1,k}(t)& \lesssim t^{2}\big\||x|^{\delta -(\beta +2)}\chi _{R_{k}}\big\|%
_{L^{p,r}} \\
& \lesssim t^{2}\big\||x|^{\delta -(\beta +2)}\chi _{R_{k}}\big\|_{L^{p,r}}
\\
& \lesssim t^{2}2^{k(\delta -(\beta +2)+\frac{n}{p})}.
\end{align*}%
We also obtain%
\begin{equation*}
J_{i,k}(t)\lesssim t^{\gamma }\big\||x|^{\delta -\gamma (\beta +1)}\chi
_{R_{k}}\big\|_{L^{p,r}}\lesssim t^{\gamma }2^{k(\delta -\gamma (\beta +1)+%
\frac{n}{p})},
\end{equation*}%
where $i\in \{2,3\}$. Therefore%
\begin{align*}
\sum\limits_{k\in \mathbb{Z},4t^{\frac{1}{\beta +1}}\leq 2^{k}\leq
4e^{-2}}2^{k\alpha q}(J_{1,k}(t))^{q}\lesssim & t^{2q}\sum\limits_{k\in 
\mathbb{Z},4t^{\frac{1}{\beta +1}}\leq 2^{k}\leq 4e^{-2}}2^{k(\delta +\frac{n%
}{p}+\alpha -(\beta +2))q} \\
\lesssim & \max \big(t^{(2+\sigma -\frac{\beta +2}{\beta +1})q},t^{2q}\big)
\end{align*}%
and%
\begin{align*}
\sum\limits_{k\in \mathbb{Z},2^{k}\geq 4t^{\frac{1}{\beta +1}}}2^{k\alpha
q}(J_{i,k}(h))^{q}\lesssim & t^{\gamma q}\sum\limits_{k\in \mathbb{Z}%
,2^{k}\geq 4t^{\frac{1}{\beta +1}}}2^{k(\delta +\frac{n+\alpha }{p}-\gamma
(\beta +1))q} \\
\lesssim & t^{\sigma q},
\end{align*}%
since $\sigma <\gamma $, where $i\in \{2,3\}$. Hence%
\begin{equation*}
I_{2}(t)\lesssim t^{\sigma q}+\max \big(t^{(2+\sigma -\frac{\beta +2}{\beta
+1})q},t^{2q}\big).
\end{equation*}%
Collecting the estimates of $I_{1}$ and $I_{2}$ we have proved $f\in \dot{K}%
_{p,r}^{\alpha ,q}B_{\infty }^{\sigma }$ with $1\leq \sigma <2$.

\textit{Step 3. }We\ will\ prove that $f\in \dot{K}_{p,r}^{\alpha
,q}B_{\infty }^{\sigma }$ and\ $\delta =1-\frac{n}{p}-\alpha $.\ Let $1\leq
p_{1},p,p_{2}<\infty $ be such that%
\begin{equation*}
\max \Big(\frac{n}{n-\alpha },\frac{n}{2(\beta +1)-\alpha _{+}-\delta }\Big)%
<p_{1}<p<p_{2}<\frac{n}{(-\alpha )_{+}},
\end{equation*}%
where $\alpha _{+}=\max (0,\alpha )$ and $(-\alpha )_{+}=\max (0,-\alpha )$.
We set%
\begin{equation*}
\sigma _{i}=\frac{\delta +\frac{n}{p_{i}}+\alpha }{\beta +1},\quad i\in
\{1,2\},\quad \frac{1}{p}=\frac{\theta }{p_{1}}+\frac{1-\theta }{p_{2}}%
,\quad 0<\theta <1.
\end{equation*}%
Observe that $\delta -1+\frac{n}{p_{1}}+\alpha >0$ and $\delta -1+\frac{n}{%
p_{2}}+\alpha <0$, which yield that $f\in \dot{K}_{p_{i}}^{\alpha
,q}B_{\infty }^{\sigma _{i}}$, $i\in \{1,2\}$. By H\"{o}lder's inequality,
we obtain%
\begin{equation}
\big\|f\big\|_{\dot{K}_{p,r}^{\alpha ,q}B_{\infty }^{\sigma }}\leq \big\|f%
\big\|_{\dot{K}_{p_{1},r}^{\alpha ,q}B_{\infty }^{\sigma _{1}}}^{\theta }%
\big\|f\big\|_{\dot{K}_{p_{2},r}^{\alpha ,q}B_{\infty }^{\sigma
_{2}}}^{1-\theta }.  \label{interpolation-lorentz}
\end{equation}%
This ensures that $f\in \dot{K}_{p,r}^{\alpha ,q}B_{\infty }^{\sigma }$ but
for $p>1$. Now assume that $p=1$. Let $-n<\alpha _{1}<\alpha <\alpha _{2}<0$%
.\ We put%
\begin{equation*}
\sigma _{i}=\frac{\delta +n+\alpha _{i}}{\beta +1},\quad i\in \{1,2\},
\end{equation*}%
which yield that $f\in \dot{K}_{1,r}^{\alpha _{i},q}B_{\infty }^{\sigma
_{i}} $, $i\in \{1,2\}$. An interpolation inequality as in %
\eqref{interpolation-lorentz} gives that $f\in \dot{K}_{1,r}^{\alpha
,q}B_{\infty }^{\sigma },0<\sigma <1$.

The proof is complete.
\end{proof}

\begin{remark}
$\mathrm{(i)}$ If $\alpha =0$ and $p=q$, then Lemma \ref{Key-1-lorentz}
reduces to the result given in \cite[Proposition 3]{Bo4}.\newline
$\mathrm{(ii)}$ We can use Theorem \ref{Maximal-Inq3-lorentz} to estimate %
\eqref{dual-lorentz}. Indeed, we have%
\begin{align*}
t^{-n}\int_{|h|<t}\left\vert f\left( x+h\right) \right\vert \chi _{B(-h,5t^{%
\frac{1}{\beta +1}})}(x)dh& \leq t^{-n}\int_{|z-x|<t}\left\vert f\left(
z\right) \right\vert \chi _{B(0,5t^{\frac{1}{\beta +1}})}(z)dh \\
& \lesssim \mathcal{M}(f\chi _{B(0,5t^{\frac{1}{\beta +1}})})(x),
\end{align*}%
where $x\in B(0,4t^{\frac{1}{\beta +1}})\cap B(0,2e^{-2})\cap R_{k},k\in 
\mathbb{Z}.$
\end{remark}

\begin{remark}
$\mathrm{(i)}$\ It is well-known that Herz spaces have been widely applied
in harmonic analysis; see, for instance, \cite{Drappl}, \cite%
{FeichtingerWeisz08}, \cite{Rag09}-\cite{Rag12} and \cite{T11}. It is a
natural question to find more applications of Lorentz-Herz spaces in
harmonic analysis.\newline
$\mathrm{(ii)}$\ We think that it is interesting to develop a real-variable
theory of mixed-norm Lorentz-Herz spaces; see \cite{ZYZ22}. More precisely.
For $i\in \{1,...,n\}$ and $k_{i}\in \mathbb{Z}$\ let 
\begin{equation*}
R_{k_{i}}=\{x_{i}\in \mathbb{R}:2^{k_{i}-1}\leq |x_{i}|<2^{k_{i}}\}\quad 
\text{and}\quad \chi _{k_{i}}=\chi _{R_{k_{i}}}.
\end{equation*}%
Vectors $\vec{p}=(p_{1},...,p_{n})$ with $p_{i}\in (0,\infty ],i=1,...,n$
are written $0<\vec{p}\leq \infty $. Let $0<\vec{p},\vec{q},\vec{r}\leq
\infty $\ and $\vec{\alpha}=(\alpha _{1},...,\alpha _{n})\in \mathbb{R}^{n}$%
. The mixed-norm Lorentz Herz space $\dot{E}_{\vec{p},\vec{r}}^{\vec{\alpha},%
\vec{q}}(\mathbb{R}^{n})$ is defined to be the set of all measurable
functions $f$ such that%
\begin{equation*}
\big\|f\big\|_{\dot{E}_{\vec{p},\vec{r}}^{\vec{\alpha},\vec{q}}(\mathbb{R}%
^{n})}=\big\|\cdot \cdot \cdot \big\|f\big\|_{\dot{K}_{p_{1},r_{1}}^{\alpha
_{1},q_{1}}}\cdot \cdot \cdot \big\|_{\dot{K}_{p_{n},r_{n}}^{\alpha
_{n},q_{n}}}<\infty ,
\end{equation*}%
where%
\begin{equation*}
\big\|f\big\|_{\dot{K}_{p_{i},r_{i}}^{\alpha _{i},q_{i}}}=\Big(%
\sum\limits_{k_{i}\in \mathbb{Z}}2^{k_{i}{\alpha }_{i}{q}_{i}}\big\|f\,\chi
_{k_{i}}\big\|_{L^{p_{i},r_{i}}}^{q_{i}}\Big)^{1/q_{i}},\quad i\in
\{1,...,n\}.
\end{equation*}%
$\mathrm{(iii)}$ It is also interesting to develop a real-variable theory of
weighted Lorentz Herz-type Besov-Triebel-Lizorkin spaces.
\end{remark}

\bigskip \textbf{Acknowledgements}

This work is found by the General Direction of Higher Education and Training
under\ Grant No. C00L03UN280120220004 and by The General Directorate of
Scientific Research and Technological Development, Algeria.


\begin{thebibliography}{99}
\bibitem{AKS} M. Ashraf Bhat, P. Kolwicz and G. Sankara Raju Kosuru, K\"{o}%
the-Herz spaces: The Amalgam-type spaces of infinite direct sums.
arXiv:2209.05897v4.

\bibitem{BS85} A. Baernstein II and E. T. Sawyer, Embedding and multiplier
theorems for $H^{p}(\mathbb{R}^{n})$, Mem. Amer. Math. Soc. 53, no. 318,
1985.

\bibitem{Be64} A. Beurling, Construction and analysis of some convolution
algebras, Ann. Inst. Fourier Grenoble. \textbf{14} (1964), 1--32.

\bibitem{BM01} H. Brezis and P. Mironescu, Gagliardo-Nirenberg, composition
and products in fractional Sobolev spaces, J. Evol. Equ. \textbf{1}(4)
(2001), 387--404.

\bibitem{BM91} G. Bourdaud and Y. Meyer, Fonctions qui op\`{e}rent sur les
espaces de Sobolev, J. Funct. Anal. \textbf{97} (1991), 351--360.

\bibitem{Bo4} G. Bourdaud, A sharpness result for powers of Besov functions,
J. Funct. Spaces Appl. \textbf{2(3)} (2004), 267--277.

\bibitem{BoHo06} M. Bownik and K.-P. Ho, Atomic and molecular decompositions
of anisotropic Triebel-Lizorkin spaces, Trans. Amer. Math. Soc. \textbf{358}
(2006), 1469--1510.

\bibitem{Drihem1.13} D. Drihem, Embeddings properties on Herz-type Besov and
Triebel-Lizorkin spaces, Math. Ineq. and Appl. \textbf{16}(2) (2013),
439--460.

\bibitem{Drihem2.13} D. Drihem, Sobolev embeddings for Herz-type
Triebel-Lizorkin spaces. Function Spaces and Inequalities. P. Jain,
H.-J.Schmeisser (ed.). Springer Proceedings in Mathematics and Statistics.
Springer, 2017.

\bibitem{drihem2016jawerth} D. Drihem, Jawerth-Franke embeddings of
Herz-type Besov and Triebel-Lizorkin spaces, Funct. Approx. Comment. Math. 
\textbf{61}(2) (2019), 207--226.

\bibitem{Dr-Sobolev} D. Drihem, Herz-Sobolev spaces on domains, Le
Matematiche. \textbf{77}(2) (2022). 229--263.

\bibitem{Dr22.Banach} D. Drihem, Composition operators on Herz-type
Triebel-Lizorkin spaces with application to semilinear parabolic equations,
Banach J. Math. Anal. \textbf{16}, 29 (2022).

\bibitem{DrPolonais} D. Drihem, On the composition operators on Besov and
Triebel--Lizorkin spaces with power weights, Annales Polonici Mathematici. 
\textbf{129} (2022), 117--137.

\bibitem{Drappl} D. Drihem, Semilinear parabolic equations in Herz spaces,
Applicable Analysis. \textbf{102}(11) (2023), 3043--3063.

\bibitem{Dr-AOT23} D. Drihem, Triebel-Lizorkin spaces with general weights,
Adv. Oper. Theory. \textbf{8}(5) (2023), 69 pages.
https://doi.org/10.1007/s43036-022-00230-0

\bibitem{Dr-EMJ} D. Drihem, Caffarelli--Kohn--Nirenberg inequalities for
Besov and Triebel--Lizorkin-type spaces, Eurasian Mathematical Journal. 
\textbf{14} (2)(2023), 24--57.

\bibitem{Dr-inter-herz} D. Drihem, Real and complex interpolation of
Herz-type Besov-Triebel-Lizorkin spaces. Submitted.

\bibitem{Fe} H. G. Feightinger, An elementary approach to Wiener's third
Tauberian Theorem for the Euclidean n-spaces, Proceedings of Conference at
Cortona 1984, Symposia Mathematica, vol. 29, Academic Press, NewY ork, 1987,
pp. 267--301.

\bibitem{FeichtingerWeisz08} H. G. Feichtinger and F. Weisz, Herz spaces and
summability of Fourier transforms, Math. Nachr. \textbf{281}(3) (2008),
309--324.

\bibitem{Fr86} J. Franke, On the spaces $F_{p,q}^{s}$\ of Triebel-Lizorkin
type: pointwise multipliers and spaces on domains, Math Nachr. \textbf{125}
(1986), 29--68.

\bibitem{FR65} J. Franke and T. Runst, Regular elliptic boundary value
problems in Besov-Triebel-Lizorkin spaces, Math. Nachr. \textbf{174} (1995),
113--149.

\bibitem{FJ90} M. Frazier and B. Jawerth,\ A discrete transform and
decomposition of distribution spaces, J. Funct. Anal. \textbf{93} (1990),
34--170.

\bibitem{FPV} L.C.F. Ferreira, J. E. P\'{e}rez-L\'{o}pez and J. C.
Valencia-Guevara, On bilinear estimates and critical uniqueness classes for
Navier-Stokes equations, arXiv:2211.11122v1

\bibitem{L. Graf14} L. Grafakos, Classical Fourier analysis. 2nd Edition,
Springer, 2008

\bibitem{Herz68} C. Herz, Lipschitz spaces and Bernstein's theorem on
absolutely convergent Fourier transforms, J. Math. Mech. \textbf{18} (1968),
283--324.

\bibitem{HerYang99} E. Hernandez and D. Yang, Interpolation of Herz spaces
and applications, Math. Nachr. \textbf{205} (1999), 69--87.

\bibitem{Ho18} K.-P. Ho, Young's inequalities and Hausdorff-Young
inequalities on Herz spaces. Boll Unione Mat Ital. \textbf{11} (2018),
469--481.

\bibitem{Ja77} B. Jawerth, Some observations on Besov and Lizorkin-Triebel
spaces, Math. Scand. \textbf{40} (1977), 94--104.

\bibitem{LiYang96} X. Li and D. Yang, Boundedness of some sublinear
operators on Herz spaces, Illinois J. Math. \textbf{40} (1996), 484-501.

\bibitem{LuYang95} S. Lu and D. Yang, The decomposition of weighted Herz
space on $\mathbb{R}^{n}$ and its applications, Sci. China (Ser. A). \textbf{%
38} (1995), 147--158.

\bibitem{LuYang97} S. Lu, D. Yang, Herz-type Sobolev and Bessel potential
spaces and their applications, Sci. in China (Ser. A). \textbf{40} (1997),
113--129.

\bibitem{LYH08} S. Lu, D. Yang and G. Hu, Herz type spaces and their
applications, Beijing: Science Press, 2008

\bibitem{MM12} M. Meyries and M.C. Veraar, \textit{Sharp embedding results
for spaces of smooth functions with power weights}. Studia. Math. \textbf{%
208 }(3) (2012), 257--293

\bibitem{Nikolskii1975} S. M. Nikol'skij, Approximation of function of
several variables and imbedding Theorem, Springer, Berlin, Germany, 1975.

\bibitem{O'Neil} R. O'Neil, Convolution operaters and $L^{p,q}$ spaces. Duke
Math J. \textbf{30} (1963), 129--142.

\bibitem{RS} H. Rafeiro, S. Samko, Herz spaces meet Morrey type spaces and
complementary Morrey type spaces, J. Fourier Anal. Appl. \textbf{26} (2020),
Paper No. 74, 14 pp.

\bibitem{Rag09} M. A. Ragusa, Homogeneous Herz spaces and regularity
results, Nonlinear Anal. \textbf{71} (2009), e1909--e1914

\bibitem{Rag12} M. A. Ragusa, Parabolic Herz spaces and their applications,
Appl.Math. Lett. \textbf{25} (10) (2012), 1270--127.

\bibitem{RuSi96} T. Runst and W. Sickel, Sobolev spaces of fractional order,
Nemytskij operators, and Nonlinear Partial Differential equations. de
Gruyter Series in Nonlinear Analysis and Applications 3, Walter de Gruyter,
Berlin 1996.

\bibitem{Ry01} V.S. Rychkov, On a theorem of Bui, Paluszynski and Taibleson,
Proc. Steklov Inst. Math. \textbf{227} (1999), 280--292.

\bibitem{Sawano18} Y. Sawano, Theory of Besov spaces, Developments in Math.
56, Springer, Singapore, 2018.

\bibitem{ST19} A. Seeger and W. Trebels, Embeddings for spaces of
Lorentz--Sobolev type. Math. Ann. 373, 1017--1056 (2019)

\bibitem{SiTr} W. Sickel and H. Triebel, H\"{o}lder inequalities and sharp
embeddings in function spaces of $B_{p,q}^{s}$ and $F_{p,q}^{s}$ type. Z.
Anal. Anwendungen. \textbf{14} (1995), 105--140.

\bibitem{SW94} F. Soria and G. Weiss, A remark on singular integrals and
power weights, Indiana Univ. Math. J. \textbf{43} (1994), 187--204.

\bibitem{Vybiral08} J. Vyb\'{\i}ral, A new proof of the Jawerth-Franke
embedding, Rev. Mat. Complut. \textbf{21}(1) (2008), 75--82.

\bibitem{TangYang2000} L. Tang and D. Yang, Boundedness of vector-valued
operators on weighted Herz spaces, Approx. Th. \& its Appl. \textbf{16}
(2000), 58--70.

\bibitem{T83} H. Triebel, Theory of function spaces, Birkh\"{a}user, Basel
1983.

\bibitem{T2} H. Triebel, Theory of function spaces, II, Birkh\"{a}user,
Basel 1992.

\bibitem{Triebel93} H. Triebel, Approximation numbers and entropy numbers of
embeddings of fractional Besov-Sobolev spaces in Orlicz spaces, Proc. London
Math. Soc, \textbf{66} (1993), 589--618.

\bibitem{triebel08} H. Triebel, Local means and wavelets in function spaces,
Banach Center Publications. \textbf{79}(1) (2008), 215--234.

\bibitem{T11} Y. Tsutsui, The Navier-Stokes equations and weak Herz spaces,
Adv. Differential Equations. \textbf{16} (2011), 1049--1085.

\bibitem{XuYang03} J. Xu and D. Yang, Applications of Herz-type
Triebel-Lizorkin spaces, Acta. Math. Sci (Ser. B). \textbf{23} (2003),
328--338.

\bibitem{XuYang05} J. Xu, D. Yang, Herz-type Triebel-Lizorkin spaces. I,\
Acta. Math. Sci (English Ed.). \textbf{21}(3) (2005), 643--654.

\bibitem{Xu05} J. Xu. Equivalent norms of Herz-type Besov and
Triebel-Lizorkin spaces, J. Funct. Spaces. Appl. \textbf{3} (2005), 17--31.

\bibitem{Xu09} J. Xu, Decompositions of non-homogeneous Herz-type Besov and
Triebel-Lizorkin spaces, Sci. China. Math. \textbf{57}(2) (2014), 315--331.\ 

\bibitem{YSY10} W. Yuan, W. Sickel and D. Yang, Morrey and Campanato Meet
Besov, Lizorkin and Triebel, Lecture Notes in Mathematics 2005,
Springer-Verlag, Berlin, (2010).

\bibitem{ZYZ22} Y. Zhao, D. Yang, Y. Zhang, Mixed-norm Herz spaces and their
applications in related Hardy spaces, Anal. Appl. \textbf{21 (}2022),
1131--1222.
\end{thebibliography}
\end{document}